\documentclass[11pt,reqno]{amsart}
\usepackage[foot]{amsaddr}
\usepackage[margin=1in]{geometry}
\usepackage{hyperref}

\usepackage{amsmath}
\usepackage{amssymb}
\usepackage{amsthm}
\usepackage{algpseudocode}
\usepackage{algorithm}
\usepackage{enumerate}
\usepackage{graphicx}

\allowdisplaybreaks
\newtheorem{theorem}{Theorem}[section]
\newtheorem{lemma}[theorem]{Lemma}
\newtheorem{proposition}[theorem]{Proposition}
\newtheorem{corollary}[theorem]{Corollary}
\theoremstyle{definition}
\newtheorem{remark}[theorem]{Remark}

\newtheorem{definition}[theorem]{Definition}
\newtheorem{assumption}[theorem]{Assumption}
\numberwithin{equation}{section}
\numberwithin{figure}{section}

\newcommand{\N}{\mathcal{N}}
\newcommand{\RR}{\mathbb{R}}
\newcommand{\EE}{\mathbb{E}}
\newcommand{\PP}{\mathbb{P}}
\newcommand{\toW}{\overset{W}{\to}}
\newcommand{\toWtwo}{\overset{W_2}{\to}}
\newcommand{\toWp}{\overset{W_p}{\to}}
\newcommand{\mmse}{\operatorname{mmse}}
\newcommand{\diag}{\operatorname{diag}}
\newcommand{\Cov}{\operatorname{Cov}}
\newcommand{\Var}{\operatorname{Var}}
\newcommand{\Tr}{\operatorname{Tr}}
\newcommand{\PCA}{{\text{PCA}}}
\newcommand{\lb}{{\text{lb}}}
\newcommand{\eps}{\varepsilon}
\newcommand{\A}{\mathbf{A}}
\newcommand{\B}{\mathbf{B}}

\newcommand{\E}{\mathbf{E}}
\newcommand{\F}{\mathbf{F}}

\renewcommand{\H}{\mathbf{H}}
\newcommand{\I}{\mathbf{I}}
\newcommand{\J}{\mathbf{J}}
\renewcommand{\L}{\mathbf{L}}
\newcommand{\M}{\mathbf{M}}
\newcommand{\bN}{\mathbf{N}}
\renewcommand{\O}{\mathbf{O}}
\renewcommand{\P}{\mathbf{P}}
\newcommand{\Q}{\mathbf{Q}}
\newcommand{\R}{\mathbf{R}}

\newcommand{\T}{\mathbf{T}}
\newcommand{\U}{\mathbf{U}}
\newcommand{\V}{\mathbf{V}}
\newcommand{\W}{\mathbf{W}}
\newcommand{\X}{\mathbf{X}}
\newcommand{\Y}{\mathbf{Y}}
\newcommand{\Z}{\mathbf{Z}}
\renewcommand{\a}{\mathbf{a}}
\renewcommand{\b}{\mathbf{b}}
\newcommand{\bc}{\mathbf{c}}

\newcommand{\f}{\mathbf{f}}
\newcommand{\g}{\mathbf{g}}
\newcommand{\h}{\mathbf{h}}
\renewcommand{\i}{\mathbf{i}}
\renewcommand{\j}{\mathbf{j}}
\renewcommand{\l}{\mathbf{l}}
\newcommand{\p}{\mathbf{p}}
\newcommand{\q}{\mathbf{q}}
\renewcommand{\r}{\mathbf{r}}
\newcommand{\s}{\mathbf{s}}

\renewcommand{\u}{\mathbf{u}}
\renewcommand{\v}{\mathbf{v}}
\newcommand{\w}{\mathbf{w}}
\newcommand{\x}{\mathbf{x}}
\newcommand{\y}{\mathbf{y}}
\newcommand{\z}{\mathbf{z}}
\newcommand{\blambda}{\pmb{\lambda}}
\newcommand{\bgamma}{\pmb{\gamma}}
\newcommand{\bLambda}{\pmb{\Lambda}}
\newcommand{\proj}{\Pi}
\newcommand{\bSigma}{\mathbf{\Sigma}}
\newcommand{\bOmega}{\mathbf{\Omega}}
\newcommand{\bDelta}{\mathbf{\Delta}}
\newcommand{\bGamma}{\mathbf{\Gamma}}
\newcommand{\bPhi}{\mathbf{\Phi}}
\newcommand{\bPsi}{\mathbf{\Psi}}
\newcommand{\bTheta}{\mathbf{\Theta}}
\newcommand{\bXi}{\mathbf{\Xi}}
\newcommand{\bUpsilon}{\mathbf{\Upsilon}}
\newcommand{\bsigma}{\pmb{\sigma}}
\newcommand{\bomega}{\pmb{\omega}}
\newcommand{\bdelta}{\pmb{\delta}}
\newcommand{\bphi}{\pmb{\phi}}
\newcommand{\bpsi}{\pmb{\psi}}
\newcommand{\bmu}{\pmb{\mu}}
\newcommand{\bnu}{\pmb{\nu}}
\newcommand{\0}{\mathbf{0}}
\newcommand{\Id}{\mathrm{Id}}

\newcommand{\cF}{\mathcal{F}}

\newcommand{\cX}{\mathcal{X}}
\newcommand{\NC}{\operatorname{NC}}
\newcommand{\1}{\mathbf{1}}

\title{Approximate Message Passing algorithms for rotationally invariant matrices}

\author{Zhou Fan}
\address{Z.F.: Department of Statistics and Data Science \\ Yale University}
\email{zhou.fan@yale.edu}

\begin{document}

\begin{abstract}
Approximate Message Passing (AMP) algorithms have seen widespread use across a
variety of applications. However, the precise forms for their Onsager
corrections and state evolutions depend on properties of the underlying random
matrix ensemble, limiting the extent to which AMP algorithms derived for white
noise may be applicable to data matrices that arise in practice.

In this work, we study more general AMP algorithms for random matrices $\W$
that satisfy orthogonal rotational invariance in law, where $\W$ may have a
spectral distribution that is different from the semicircle and Marcenko-Pastur laws characteristic of white noise. The Onsager corrections and state
evolutions in these algorithms are defined by the free cumulants or
rectangular free cumulants of the spectral distribution of $\W$. Their forms
were derived previously by Opper, \c{C}akmak, and Winther using non-rigorous
dynamic functional theory techniques, and we provide rigorous proofs.

Our motivating application is a Bayes-AMP algorithm for Principal Components
Analysis, when there is prior structure for the principal components (PCs)
and possibly non-white noise. For sufficiently large signal strengths and any
non-Gaussian prior distributions for the PCs, we show that
this algorithm provably achieves higher estimation accuracy than the sample PCs.
\end{abstract}

\maketitle

\tableofcontents

\section{Introduction}

Approximate Message Passing (AMP) algorithms are a general family of iterative
algorithms that have seen widespread use in a variety of applications. 
First developed for Bayesian linear regression and compressed sensing in
\cite{kabashima2003cdma,donoho2009message,donoho2010messageI,donoho2010messageII}, they have since
been applied to many high-dimensional problems arising in statistics and
machine learning, including Lasso estimation and sparse linear regression
\cite{bayati2011lasso,maleki2013asymptotic}, generalized linear models and phase
retrieval \cite{rangan2011generalized,schniter2014compressive,sur2019modern},
robust linear regression \cite{donoho2016high}, sparse or structured principal
components analysis (PCA)
\cite{rangan2012iterative,deshpande2014information,dia2016mutual,montanari2017estimation}, group synchronization problems \cite{perry2018message},
deep learning \cite{borgerding2016onsager,borgerding2017amp,metzler2017learned},
and optimization in spin glass models
\cite{montanari2019optimization,gamarnik2019overlap,alaoui2020optimization}. We
refer to \cite{feng2021unifying} for a recent review.

In their basic form as described in \cite{bayati2011dynamics}, given a data
matrix $\W \in \RR^{m \times n}$ and an initialization $\u_1 \in \RR^m$, an AMP
algorithm consists of the iterative updates
\begin{align*}
\z_t&=\W^\top \u_t-b_t\v_{t-1}\\
\v_t&=v_t(\z_t)\\
\y_t&=\W\v_t-a_t\u_t\\
\u_{t+1}&=u_{t+1}(\y_t).
\end{align*}
Here, $a_t,b_t \in \RR$ are two sequences of debiasing coefficients, and
$v_t:\RR \to \RR$ and $u_{t+1}:\RR \to \RR$ are two sequences of 
functions applied entrywise to $\z_t \in \RR^n$ and $\y_t \in \RR^m$.
By appropriately designing these functions $v_t$ and $u_{t+1}$, possibly
to also depend on additional ``side information'' such as response variables in
regression problems, this basic iteration may be
applied to perform optimization or Bayes posterior-mean estimation in the
above applications.

A defining characteristic of the AMP algorithm is the subtraction of the
two ``memory'' terms $b_t\v_{t-1}$ and $a_t\u_t$ in the definitions of
$\z_t$ and $\y_t$, known as the \emph{Onsager corrections}.
This achieves the effect of removing a bias of 
$\W^\top \u_t$ and $\W\v_t$ in the directions of the preceding iterates, so
that as $m,n \to \infty$, the empirical distributions of
$\y_t$ and $\z_t$ converge to certain Gaussian limits
\begin{equation}\label{eq:gaussianlimitintro}
\y_t \to \N(0,\sigma_t^2) \quad \text{ and } \quad \z_t \to
\N(0,\omega_t^2).
\end{equation}
This was proven rigorously in the Sherrington-Kirkpatrick model by
Bolthausen in \cite{bolthausen2014iterative} and for general AMP algorithms of
the above form by Bayati and Montanari in \cite{bayati2011dynamics},
and various extensions have been established in
\cite{donoho2013information,javanmard2013state,bayati2015universality,berthier2020state,chen2020universality}. The description of the variances
$\sigma_t^2$ and $\omega_t^2$ across iterations is known as the algorithm's
\emph{state evolution}. This ability to characterize the distributions of the
iterates is a major appeal of the AMP approach, and has enabled a more
precise theoretical understanding of many high-dimensional statistical
estimators and the development of associated inference procedures that
quantify statistical uncertainty
\cite{su2017false,mousavi2018consistent,sur2019modern,sur2019likelihood,bu2019algorithmic}.

A drawback of AMP algorithms, however, is that the correct forms of the
debiasing coefficients $a_t,b_t$ and resulting variances
$\sigma_t^2,\omega_t^2$ depend on the properties of the
data matrix $\W$. When $\W$ has i.i.d.\ $\N(0,1/n)$ entries,
these quantities are given explicitly by
\[a_t=\langle v_t'(\z_t) \rangle, \qquad
b_t=\gamma \langle u_t'(\y_{t-1}) \rangle, \qquad
\sigma_t^2=\langle v_t(\z_t)^2 \rangle, \qquad
\omega_t^2=\gamma \langle u_t(\y_{t-1})^2 \rangle\]
where $\gamma=m/n$, $u_t'(\cdot),v_t'(\cdot),u_t(\cdot)^2,v_t(\cdot)^2$
denote the derivatives and squares of $u_t,v_t$ applied entrywise, and
$\langle \cdot \rangle$ denotes the empirical average of coordinates. It has
been shown in \cite{bayati2015universality,chen2020universality} that
these forms enjoy a certain amount of universality, being valid also for $\W$
having i.i.d.\ non-Gaussian entries. Extensions to $\W$ having independent
entries with several blocks of differing variances were derived in
\cite{donoho2013information,javanmard2013state}. Unfortunately, these results
do not apply to $\W$ with more complex correlation structure, which is
common in data applications. A sizeable body of work has developed alternative
algorithms or damping procedures to address this shortcoming
\cite{opper2001adaptive,opper2001tractable,minka2001family,opper2005expectation,kabashima2014signal,cakmak2014s,caltagirone2014convergence,fletcher2016expectation,schniter2016vector,ma2017orthogonal,takeuchi2017rigorous,rangan2019convergence,rangan2019vector},
and the connections between several of these algorithms were discussed recently in
\cite{maillard2019high}. However, many such algorithms are no longer
characterized by a rigorous state evolution, and some have been empirically
observed to exhibit slow convergence or divergent behavior.

\subsection{Contributions}

We develop a rigorous extension of general AMP procedures of the above form to
rotationally invariant matrices. We then apply these general
algorithms to a prototypical ``structured PCA''
problem of estimating a rank-one matrix in (possibly non-white) noise.
In this PCA application, we develop a Bayes-AMP algorithm that provably
achieves lower mean-squared-error than the rank-one estimate constructed from
the sample principal components (PCs), for any sufficiently large signal
strength and any prior distributions of the PCs that are not mean-zero Gaussian
laws.

Let us first describe the general AMP algorithm in the
simpler setting of symmetric square matrices. We study matrices
$\W \in \RR^{n \times n}$ that satisfy the equality in law
\[\W\overset{L}{=} \tilde{\O}^\top \W\tilde{\O}\]
for any deterministic orthogonal matrix $\tilde{\O} \in \RR^{n \times n}$.
Equivalently, such matrices admit the eigendecomposition $\W=\O^\top \bLambda\O$
where the eigenvectors $\O \in \RR^{n \times n}$ are independent of the
eigenvalues $\bLambda$ and are uniformly distributed on the orthogonal group.
The AMP algorithm will take the form
\begin{align}
\z_t&=\W\u_t-b_{t1}\u_1-b_{t2}\u_2-\ldots-b_{tt}\u_t\label{eq:AMPintro1}\\
\u_{t+1}&=u_{t+1}(\z_1,\z_2,\ldots,\z_t)\label{eq:AMPintro2}
\end{align}
where the coefficients $b_{ts}$ are defined so that each $\z_t$ has an
empirical Gaussian limit as in (\ref{eq:gaussianlimitintro}).
For greater generality and applicability,
we will allow $u_{t+1}:\RR^t \to \RR$ to be
a function of all previous iterates $\z_1,\ldots,\z_t$, rather than only the
preceding iterate $\z_t$. (Outside of the 
i.i.d.\ Gaussian setting, the full debiasing of $\W\u_t$ by
$\u_1,\ldots,\u_t$ is necessary even if $u_{t+1}(\cdot)$ 
depends only on $\z_t$.) The correct forms for $b_{t1},\ldots,b_{tt}$ and the
corresponding state evolution
\begin{equation}\label{eq:SEintro}
(\z_1,\z_2,\ldots,\z_t) \to \N(0,\bSigma_t)
\end{equation}
were derived previously by Opper, \c{C}akmak, and Winther using non-rigorous
dynamic functional theory techniques \cite{opper2016theory}. These forms
depend on the free cumulants of the eigenvalue distribution of $\W$, and we
describe them in Section \ref{sec:BSigma}. Our work provides a rigorous
proof of the validity of this state evolution.

In the rectangular setting, we study bi-rotationally invariant matrices
$\W \in \RR^{m \times n}$ satisfying the equality in law
\[\W\overset{L}{=} \tilde{\O}^\top \W\tilde{\Q}\]
for any deterministic orthogonal matrices $\tilde{\O} \in \RR^{m \times m}$ and
$\tilde{\Q} \in \RR^{n \times n}$. Equivalently, such matrices
admit the singular value decomposition $\W=\O\bLambda\Q^\top$ where the singular
vectors $\O \in \RR^{m \times m}$ and $\Q \in \RR^{n \times n}$ are independent
of the singular values $\bLambda$ and are both uniformly distributed over the
orthogonal groups. The analogous AMP algorithm takes the form
\begin{align}
\z_t&=\W^\top\u_t-b_{t1}\v_1-b_{t2}\v_2-\ldots-b_{t,t-1}\v_{t-1}\label{eq:AMPintrorect1}\\
\v_t&=v_t(\z_1,\z_2,\ldots,\z_t)\\
\y_t&=\W\v_t-a_{t1}\u_1-a_{t2}\u_2-\ldots-a_{tt}\u_t\\
\u_{t+1}&=u_{t+1}(\y_1,\y_2,\ldots,\y_t)\label{eq:AMPintrorect4}
\end{align}
We describe in Section \ref{sec:debiasingrect} the forms of the
debiasing coefficients $a_{ts},b_{ts}$ and the corresponding state evolutions
\[(\y_1,\ldots,\y_t) \to \N(0,\bSigma_t), \qquad (\z_1,\ldots,\z_t) \to
\N(0,\bOmega_t),\]
which are related to the rectangular free cumulants of the singular value
distribution of $\W$ as introduced in
\cite{benaych2009rectangular,benaych2009rectangularII}.
This algorithm has also been derived recently and independently
in \cite{ccakmak2020dynamical}, using a dynamic functional theory approach
similar to \cite{opper2016theory}.

These classes of rotationally invariant matrices include, but are not
restricted to, $\W$ having i.i.d.\ Gaussian entries. Importantly,
the spectral distribution of $\W$ can be arbitrary, rather than following the
behavior prescribed by the semicircle or Marcenko-Pastur law. Our primary
motivation for studying such rotationally invariant models is that we expect
the resulting AMP algorithms to be valid under a much larger universality class 
of matrices $\W$ than AMP algorithms derived in the i.i.d.\ Gaussian setting,
and that this class may provide a more flexible model for data matrices arising
in practice.

In the contexts of compressed sensing and generalized linear models,
alternative ``vector AMP'' or ``orthogonal AMP'' approaches for
rotationally-invariant matrices have been developed in
\cite{rangan2019vector,schniter2016vector,ma2017orthogonal,takeuchi2017rigorous}, and rigorous state evolutions for these algorithms were also derived.
These derivations are based on analyses of denoising functions that satisfy
the divergence-free conditions
\begin{equation}\label{eq:divergencefree}
\langle \partial_s v_t(\z_1,\ldots,\z_t) \rangle=0, \qquad \langle
\partial_s u_{t+1}(\y_1,\ldots,\y_t) \rangle=0 \qquad \text{ for all } s \leq t.
\end{equation}
A similar idea was used
in \cite{ccakmak2019memory} to develop an algorithm for solving the
TAP equations for Ising models with rotationally-invariant couplings.
Analyses of certain ``long-memory'' Convolutional AMP algorithms for
compressed sensing, related to our work, were recently carried out in
\cite{takeuchi2019unified,takeuchi2020convolutional,takeuchi2020bayes} by
mapping these algorithms to a divergence-free form. Our proofs build on the
insight in \cite{rangan2019vector,takeuchi2017rigorous} that Bolthausen's
conditioning technique may be applied to rotationally-invariant models. However,
we derive directly the forms of the Onsager corrections and state evolutions
for AMP algorithms that do not restrict $v_t(\cdot)$ and $u_{t+1}(\cdot)$ to be
divergence-free, in a general setting that extends beyond compressed sensing
applications. We clarify the relation between certain long-memory algorithms and
the AMP algorithms of \cite{bayati2011dynamics} for Gaussian matrices, by
relating their Onsager corrections and state evolutions to the free cumulants 
of the spectral distribution of $\W$.

\subsection{Organization of paper}
Section \ref{sec:prelim} establishes preliminary background and notation
on Wasserstein convergence of empirical measures and free cumulants. Section
\ref{sec:PCA} first discusses the specific application of structured PCA and
the Bayes-AMP algorithms for this application that specialize the more general
AMP algorithms to follow. Section \ref{sec:square} describes the general AMP
algorithm and state evolution for symmetric square matrices, and Section
\ref{sec:rect} describes the analogous general algorithm for rectangular
matrices. Section \ref{sec:proof} provides a high-level overview of the proofs,
which are contained in the Supplementary Appendices.

\section{Preliminaries on Wasserstein convergence and free
probability}\label{sec:prelim}

\subsection{Notation and conventions}

For vectors $\v \in \RR^n$ and $\w \in \RR^m$, we denote
\[\langle \v \rangle=\frac{1}{n}\sum_{i=1}^n v_i,
\qquad \langle \w \rangle=\frac{1}{m}\sum_{i=1}^m w_i.\]
For a matrix $(\v_1,\ldots,\v_k) \in \RR^{n \times k}$ and a function $f:\RR^k
\to \RR$, we write $f(\v_1,\ldots,\v_k) \in \RR^n$
as its row-wise evaluation.

For a weakly differentiable function $u:\RR^k \to \RR$, we denote by
$\partial_s u$ (any version of) its $s^\text{th}$ partial derivative. 
For a matrix $(\v_1,\ldots,\v_k) \in \RR^{n \times k}$, we write
$\proj_{(\v_1,\ldots,\v_k)} \in \RR^{n \times n}$ for the orthogonal projection
onto the linear span of $(\v_1,\ldots,\v_k)$, and
$\proj_{(\v_1,\ldots,\v_k)^\perp}=\Id-\proj_{(\v_1,\ldots,\v_k)}$ for the
projection onto its orthogonal complement. $\Id$ is the identity matrix, and we
write $\Id_{k \times k}$ to specify the dimension $k$. We will use the
convention
\[\M^0=\Id\]
for the zero-th power of any square matrix $\M$, even if some eigenvalues of
$\M$ may be 0.

Products over the empty set are equal to 1, and sums over the
empty set are equal to 0.
$\|\cdot\|$ is the $\ell_2$ norm for vectors and $\ell_2 \to \ell_2$
operator norm for matrices. $\|\v\|_\infty=\max_i |v_i|$ is the vector
$\ell_\infty$ norm, and $\|\M\|_F=(\sum_{i,j} m_{ij}^2)^{1/2}$ is the matrix
Frobenius norm.

\subsection{Wasserstein convergence of empirical
distributions}\label{sec:wasserstein}

\begin{definition}\label{def:toW}
For $p \geq 1$, a matrix
$(\v_1,\ldots,\v_k)=(v_{i,1},\ldots,v_{i,k})_{i=1}^n
\in \RR^{n \times k}$, and a probability distribution $\mathcal{L}$
over $\RR^k$ or a random vector $(V_1,\ldots,V_k) \sim \mathcal{L}$, we write
\[(\v_1,\ldots,\v_k) \overset{W_p}{\to} \mathcal{L} \qquad \text{ or }
\qquad (\v_1,\ldots,\v_k) \overset{W_p}{\to} (V_1,\ldots,V_k)\]
for the convergence of the empirical distribution of rows of
$(\v_1,\ldots,\v_k)$ to $\mathcal{L}$ in the Wasserstein space of order $p$.
This means, for any $C>0$ and continuous function $f:\RR^k \to \RR$
satisfying
\begin{equation}\label{eq:growth}
|f(v_1,\ldots,v_k)| \leq C\Big(1+\|(v_1,\ldots,v_k)\|^p\Big),
\end{equation}
as $n \to \infty$,
\begin{equation}\label{eq:LLN}
\frac{1}{n}\sum_{i=1}^n f(v_{i,1},\ldots,v_{i,k}) \to
\EE\Big[f(V_1,\ldots,V_k)\Big].
\end{equation}
Implicit in this notation is the finite moment condition
$\EE_{(V_1,\ldots,V_k) \sim \mathcal{L}}[\|(V_1,\ldots,V_k)\|^p]<\infty$. 

We write
\[(\v_1,\ldots,\v_k) \toW \mathcal{L} \qquad \text{ or } \qquad
(\v_1,\ldots,\v_k) \toW (V_1,\ldots,V_k)\]
to mean that this convergence holds for every fixed $p \geq 1$,
where $\mathcal{L}$ has finite moments of all orders.
\end{definition}

We will use a certain calculus associated to these notations
$\overset{W_p}{\to}$ and $\toW$, which we review in Appendix
\ref{appendix:wasserstein}. By \cite[Definition 6.7]{villani2008optimal}, to
show that (\ref{eq:LLN}) holds for all continuous
functions $f$ satisfying (\ref{eq:growth}),
it suffices to check that it holds for all bounded Lipschitz functions $f$
together with the function $f(v_1,\ldots,v_k)=\|(v_1,\ldots,v_k)\|^p$. See
Chapter 6 of \cite{villani2008optimal} for further background.

\subsection{Free cumulants}\label{sec:cumulants}

We briefly review the notion of free cumulants, and refer readers to
\cite{novak2014three} for a more thorough and motivated introduction.

Let $X$ be a random variable with finite moments of all orders, and denote
$m_k=\EE[X^k]$. In what follows, the law of $X$ will be the empirical
eigenvalue distribution of a symmetric matrix $\W \in \RR^{n \times n}$. Let
$\NC(k)$ be the set of all non-crossing partitions of $\{1,\ldots,k\}$. 
The free cumulants $\kappa_1,\kappa_2,\kappa_3,\ldots$ of $X$
are defined recursively by the moment-cumulant relations
\begin{equation}\label{eq:momentcumulant}
m_k=\sum_{\pi \in \NC(k)} \prod_{S \in \pi} \kappa_{|S|}
\end{equation}
where $|S|$ is the cardinality of the set $S \in \pi$.
The first four free cumulants may be computed to be
\begin{align*}
\kappa_1&=m_1=\EE[X]\\
\kappa_2&=m_2-m_1^2=\Var[X]\\
\kappa_3&=m_3-3m_2m_1+2m_1^3\\
\kappa_4&=m_4-4m_3m_1-2m_2^2+10m_2m_1^2-5m_1^4,
\end{align*}
where $\kappa_4$ is the first free cumulant that differs from the classical
cumulants. The free cumulants linearize free additive convolution, describing
the eigenvalue distribution of sums of freely independent symmetric square
matrices. If $X$ has the Wigner semicircle law supported on $[-2,2]$, then
\[\kappa_1=0, \qquad \kappa_2=1, \qquad \kappa_j=0 \quad
\text{ for all } j \geq 3.\]

Defining the formal generating functions
\[M(z)=1+\sum_{k=1}^\infty m_k z^k,
\qquad R(z)=\sum_{k=1}^\infty \kappa_k z^{k-1},\]
the relations (\ref{eq:momentcumulant}) are equivalent to an identity
of formal series (see \cite[Section 2.5]{novak2014three})
\[M(z)=1+zM(z) \cdot R(zM(z)).\]
Here, $R(z)$ is the R-transform of $X$.
Comparing the coefficients of $z^k$ on both sides,
each free cumulant $\kappa_k$ may be computed from $m_1,\ldots,m_k$ and
$\kappa_1,\ldots,\kappa_{k-1}$ as
\[\kappa_k=m_k-[z^k]\sum_{j=1}^{k-1}
\kappa_j \left(z+m_1z^2+m_2z^3+\ldots+m_{k-1}z^k\right)^j\]
where $[z^k](q(z))$ denotes the coefficient of $z^k$ in the polynomial $q(z)$.

\subsection{Rectangular free cumulants}\label{sec:cumulantsrect}

For rectangular matrices $\W \in \RR^{m \times n}$, we review the notion
of rectangular free cumulants developed in \cite{benaych2009rectangular}.
This is an example of the operator-valued free cumulants described in
\cite{speicher1998combinatorial}, where freeness is with amalgamation over a
2-dimensional subalgebra corresponding to the $2 \times 2$ block structure of
$\RR^{(m+n) \times (m+n)}$. 

We fix an aspect ratio parameter
\[\gamma=m/n>0.\]
Let $X$ be a random variable with finite moments of all orders, and denote the
even moments by $m_{2k}=\EE[X^{2k}]$. The law of $X^2$
will be the empirical eigenvalue distribution of
$\W\W^\top \in \RR^{m \times m}$, so that $m_{2k}$ is the $k^\text{th}$ moment
of this distribution. Define also an auxiliary sequence of even moments
\begin{equation}\label{eq:barmu}
\bar{m}_{2k}=\begin{cases} 1 & \text{ if } k=0 \\
\gamma \cdot m_{2k} & \text{ if } k \geq 1. \end{cases}
\end{equation}
Since the eigenvalues of $\W\W^\top$ and $\W^\top\W$ coincide up to the addition
or removal of $|m-n|$ zeros,
the value $\bar{m}_{2k}$ is the $k^\text{th}$ moment of the empirical
eigenvalue distribution of $\W^\top \W \in \RR^{n \times n}$.

Let $\NC'(2k)$ be the non-crossing partitions $\pi$ of
$\{1,\ldots,2k\}$ where each set $S \in \pi$ has even cardinality. Then we may
define two sequences of rectangular free cumulants
$\kappa_2,\kappa_4,\kappa_6,\ldots$ and
$\bar{\kappa}_2,\bar{\kappa}_4,\bar{\kappa}_6,\ldots$ by the moment-cumulant
relations
\begin{align*}
m_{2k}&=\sum_{\pi \in \NC'(2k)} \mathop{\prod_{S \in \pi}}_{\min S \text{ is
odd}} \kappa_{|S|} \cdot \mathop{\prod_{S \in \pi}}_{\min S \text{ is even}}
\bar{\kappa}_{|S|}\\
\bar{m}_{2k}&=\sum_{\pi \in \NC'(2k)} \mathop{\prod_{S \in \pi}}_{\min S \text{ is odd}} \bar{\kappa}_{|S|} \cdot \mathop{\prod_{S \in \pi}}_{\min S \text{ is even}} \kappa_{|S|}
\end{align*}
See \cite[Eqs.\ (8--9)]{benaych2009rectangular}. These cumulants have a simple
relation given by
\begin{equation}\label{eq:kappabarkappa}
\bar{\kappa}_{2k}=\gamma \cdot \kappa_{2k} \quad \text{ for all } k \geq 1,
\end{equation}
so outside of the proofs, we will always refer to the first sequence
$\{\kappa_{2k}\}_{k \geq 1}$ for simplicity.

Letting $e(\pi)$ be the number of sets $S \in \pi$ where the smallest element
of $S$ is even, and letting $o(\pi)$ be the number where the smallest element
is odd, applying (\ref{eq:kappabarkappa}) above implies
\begin{equation}\label{eq:momentcumulantrect}
m_{2k}=\sum_{\pi \in \NC'(2k)} \gamma^{e(\pi)}\prod_{S \in \pi} \kappa_{|S|},
\qquad \bar{m}_{2k}=\sum_{\pi \in \NC'(2k)} \gamma^{o(\pi)}\prod_{S \in \pi}
\kappa_{|S|}.
\end{equation}
See also \cite[Proposition 3.1]{benaych2009rectangular}.
The first four rectangular free cumulants may be computed as
\begin{align*}
\kappa_2&=m_2=\EE[X^2]\\
\kappa_4&=m_4-(1+\gamma)m_2^2\\
\kappa_6&=m_6-(3+3\gamma)m_4m_2
+(2+3\gamma+2\gamma^2)m_2^3\\
\kappa_8&=m_8-(4+4\gamma)m_6m_2
-(2+2\gamma)m_4^2+(10+16\gamma+10\gamma^2)m_4m_2^2\\
&\hspace{1in}-(5+10\gamma+10\gamma^2+5\gamma^3)m_2^4.
\end{align*}
The rectangular free cumulants linearize rectangular free additive
convolution, describing the singular value distribution of sums of freely
independent rectangular matrices. If $X^2$ has the Marcenko-Pastur law with
aspect ratio $\gamma$, then
\[\kappa_2=1, \qquad \kappa_{2j}=0 \quad \text{ for all } j \geq 2.\]

The rectangular free cumulants may be computed
from the following relation of generating functions: Let
\[M(z)=\sum_{k=1}^\infty m_{2k}z^k,
\qquad R(z)=\sum_{k=1}^\infty \kappa_{2k} z^k.\]
Here, $R(z)$ is the rectangular R-transform of $X$. Then
\begin{equation}\label{eq:rectseries}
M(z)=R\Big(z(\gamma M(z)+1)(M(z)+1)\Big),
\end{equation}
see \cite[Lemma 3.4]{benaych2009rectangular}.
Thus, comparing the coefficients of $z^k$ on both sides, each value
$\kappa_{2k}$ may be computed from $m_2,\ldots,m_{2k}$ and
$\kappa_2,\ldots,\kappa_{2k-2}$ as
\[\kappa_{2k}=m_{2k}-[z^k]\sum_{j=1}^{k-1}
\kappa_{2j}\Big(z(\gamma M(z)+1)(M(z)+1)\Big)^{2j}\]
where $[z^k](q(z))$ again denotes the coefficient of $z^k$ in the
polynomial $q(z)$.

\begin{remark}
The reasons for the appearance of the square/rectangular free cumulants 
in the forms of the Onsager corrections and state evolution for AMP are somewhat
opaque in our work, as they will arise from a certain combinatorial unfolding
of the moment-cumulant relations on the non-crossing partition lattice; we
discuss this further in Section \ref{sec:proof}. Their emergence is
conceptually clearer in the (non-rigorous, but illuminating)
analysis of the limit characteristic function of the AMP iterates in
\cite{opper2016theory}, where they arise instead from the evaluation of a
low-rank HCIZ integral over the Haar-orthogonal randomness in $\W$, and from the
coefficients of the series expansion of the R-transform that describes this
integral. See \cite{collins2003moments,guionnet2005fourier} and
\cite{benaych2011rectangular} for this connection in the square and rectangular
settings, respectively.
\end{remark}

\section{Structured Principal Components Analysis}\label{sec:PCA}

We study the problem of estimating a rank-one signal matrix in possibly
non-white noise, where the singular vectors of the rank-one signal have some
``prior'' structure. For sufficiently large signal strengths, we describe a 
Bayes-AMP algorithm that provably achieves lower mean-squared-error than the
rank-one estimate constructed from the sample principal components. This
extends the types of AMP algorithms that were studied for i.i.d.\ Gaussian
noise in \cite{rangan2012iterative,deshpande2014information,dia2016mutual,montanari2017estimation}.

\subsection{Symmetric square matrices}\label{sec:squarePCA}

Suppose first that we observe a symmetric data matrix
\[\X=\frac{\alpha}{n}\u_*\u_*^\top+\W \in \RR^{n \times n}\]
and seek to estimate $\u_* \in \RR^n$. Writing the eigendecomposition
$\W=\O^\top \bLambda\O$ where $\bLambda=\diag(\blambda)$, we assume that $\W$ is
rotationally-invariant in law and that as $n \to \infty$,
the empirical distributions of $\blambda$ and $\u_*$ satisfy
\begin{equation}\label{eq:Bayescondition1}
\blambda \toW \Lambda, \qquad \u_* \toW U_*
\end{equation}
for two limit laws $\Lambda$ and $U_*$. This notation $\toW$ denotes
Wasserstein convergence at all orders, as discussed in
Section \ref{sec:wasserstein}. To fix the scaling, we take $\|\u_*\|=\sqrt{n}$,
so that
\[\EE[U_*^2]=\lim_{n \to \infty} \frac{1}{n}\|\u_*\|^2=1.\]
Here, the law of $\Lambda$ is the limit spectral distribution of $\W$. The law
of $U_*$ represents a prior distribution for the entries of $\u_*$, which
may reflect assumptions of sparsity \cite{deshpande2014information},
non-negativity \cite{montanari2015non}, or a discrete support that encodes
cluster or community membership \cite{deshpande2017asymptotic}.

We assume for simplicity that we have an initialization
$\u_1 \in \RR^n$ independent of $\W$,
satisfying the joint empirical convergence
\begin{equation}\label{eq:Bayescondition2}
(\u_1,\u_*) \toW (U_1,U_*), \qquad \EE[U_1U_*]>0.
\end{equation}
We then estimate $\u_*$ by the iterates $\u_t$ of an AMP algorithm
\begin{align}
\f_t&=\X\u_t-b_{t1}\u_1-\ldots-b_{tt}\u_t\label{eq:AMPPCA1}\\
\u_{t+1}&=u_{t+1}(\f_t).\label{eq:AMPPCA2}
\end{align}
It will be shown that each iterate $\f_t$ behaves like $\u_*$ corrupted by
entrywise Gaussian noise, so we take each function $u_{t+1}(\cdot)$ to be a
scalar denoiser that estimates $\u_*$ from $\f_t$.

To describe the forms of the debiasing coefficients
$b_{t1},\ldots,b_{tt}$, let us write
$\lambda_1(\X) \geq \ldots \geq \lambda_n(\X)$ as the eigenvalues of
$\X$. For each $k \geq 1$, let
\begin{equation}\label{eq:PCAmk}
m_k=\frac{1}{n}\sum_{i=2}^n \lambda_i(\X)^k
\end{equation}
be the $k^\text{th}$ moment of the empirical eigenvalue distribution of $\X$
excluding its largest eigenvalue. Let $\{\kappa_k\}_{k \geq 1}$ be the free
cumulants corresponding to this sequence of moments $\{m_k\}_{k \geq 1}$, as
defined in Section \ref{sec:cumulants}. It is easy to check that under the
assumption (\ref{eq:Bayescondition1}), as $n \to \infty$,
\[m_k \to m_k^\infty=\EE[\Lambda^k],
\qquad \kappa_k \to \kappa_k^\infty\]
for each fixed $k \geq 1$, where these limits are the moments and free
cumulants of the limit spectral distribution $\Lambda$ of the noise $\W$. The
debiasing coefficients in (\ref{eq:AMPPCA1}) are set as
\begin{equation}\label{eq:PCAb}
b_{tt}=\kappa_1, \qquad b_{t,t-j}=\kappa_{j+1}\prod_{i=t-j+1}^t \langle
u_i'(\f_{i-1}) \rangle \text{ for } j=1,\ldots,t-1.
\end{equation}

The state evolution that describes the AMP iterations
(\ref{eq:AMPPCA1}--\ref{eq:AMPPCA2}) is
expressed in terms of a sequence of mean vectors
$\bmu_T^\infty=(\mu_t^\infty)_{1 \leq t \leq T}$ and covariance matrices
$\bSigma_T^\infty=(\sigma_{st}^\infty)_{1 \leq s,t \leq T}$, defined
recursively as follows: For $T=1$, we set
\[\mu_1^\infty=\alpha \cdot \EE[U_1U_*],
\qquad \sigma_{11}^\infty=\kappa_2^\infty \EE[U_1^2].\]
Having defined $\bmu_T^\infty$ and $\bSigma_T^\infty$, we denote
\[U_t=u_t(F_{t-1}) \text{ for } t=2,\ldots,T+1, \qquad
(F_1,\ldots,F_T)=\bmu_T^\infty \cdot U_*+(Z_1,\ldots,Z_T),
\text{ and }\]
\begin{equation}\label{eq:UFPCA}
(Z_1,\ldots,Z_T) \sim \N(0,\bSigma_T^\infty) \text{ independent of } (U_1,U_*).
\end{equation}
We then define $\bmu_{T+1}^\infty$ and $\bSigma_{T+1}^\infty$ to have the
entries, for $1 \leq s,t \leq T+1$,
\begin{align}
\mu_t^\infty&=\alpha \cdot \EE[U_tU_*]\nonumber\\
\sigma_{st}^\infty&=\sum_{j=0}^{s-1} \sum_{k=0}^{t-1}
\kappa_{j+k+2}^\infty \left(\prod_{i=s-j+1}^s \EE[u_i'(F_{i-1})] \right)
\left(\prod_{i=t-k+1}^t \EE[u_i'(F_{i-1})]\right)\EE[U_{s-j}U_{t-k}].
\label{eq:PCAsigma}
\end{align}
In the limit $n \to \infty$, the iterates of (\ref{eq:AMPPCA1}) will satisfy the
second-order Wasserstein convergence
\[(\f_1,\ldots,\f_T,\u_*) \toWtwo (F_1,\ldots,F_T,U_*).\]
Thus, the rows of $(\f_1,\ldots,\f_T)$ behave like Gaussian vectors with mean
$\bmu_T^\infty \cdot U_*$ and covariance $\bSigma_T^\infty$.

As one example of choosing the functions $u_{t+1}(\cdot)$,
let us analyze this state evolution for the following ``single-iterate
posterior mean'' denoisers: In the scalar Gaussian observation model
\begin{equation}\label{eq:scalarchannel}
F=\mu \cdot U_*+Z, \qquad Z \sim \N(0,\sigma^2) \text{ independent of } U_*,
\end{equation}
we denote the Bayes posterior-mean estimate of $U_*$ as
\begin{equation}\label{eq:denoiserexplicit}
\eta(f \mid \mu,\sigma^2)=\EE[U_* \mid F=f]
=\frac{\EE[U_* \exp(-(f-\mu \cdot U_*)^2/2\sigma^2)]}
{\EE[\exp(-(f-\mu \cdot U_*)^2/2\sigma^2)]}.
\end{equation}
We denote the Bayes-optimal mean-squared-error of this estimate as
\begin{equation}\label{eq:mmse}
\mmse(\mu^2/\sigma^2)
=\EE\Big[\big(U_*-\eta(F \mid \mu,\sigma^2)\big)^2\Big].
\end{equation}
The single-iterate posterior mean denoiser is the choice
\begin{equation}\label{eq:utBayes}
u_{t+1}(f_t)=\eta(f_t \mid \mu_t^\infty,\sigma_{tt}^\infty)
\end{equation}
where $\mu_t^\infty$ and $\sigma_{tt}^\infty$ are the above state evolution
parameters that describe the univariate Gaussian law of $F_t$. These
parameters may be replaced by consistent estimates in practice.

Let $R(x)$ be the R-transform of the limit spectral distribution $\Lambda$,
as discussed in Section \ref{sec:cumulants},
and let $R'(x)$ be its derivative. For small $|x|$, these may be defined by the
convergent series (see Proposition \ref{prop:cumulantbound})
\begin{equation}\label{eq:Rintro}
R(x)=\sum_{k=1}^\infty \kappa_k^\infty x^{k-1}.
\end{equation}

\begin{theorem}\label{thm:PCA}
Suppose $\W=\O^\top \bLambda\O \in \RR^{n \times n}$ where $\O$ is a
Haar-uniform orthogonal matrix. Let $\bLambda=\diag(\blambda)$, where
$(\blambda,\u_1,\u_*)$ are independent of $\O$,
$\|\u_*\|=\sqrt{n}$, and
\[\blambda \toW \Lambda, \qquad (\u_1,\u_*) \toW (U_1,U_*)\]
almost surely as $n \to \infty$. Suppose $\EE[U_1^2] \leq 1$,
$\EE[U_1U_*]=\eps>0$, and
$\|\blambda\|_\infty \leq C_0$ almost surely for all large $n$ and some
constants $C_0,\eps>0$.
\begin{enumerate}[(a)]
\item Let $\alpha \geq 0$, and let each function $u_{t+1}(\cdot)$ be
continuously differentiable and Lipschitz on $\RR$. Then for each fixed
$T \geq 1$, almost surely as $n \to \infty$,
\[(\u_1,\ldots,\u_{T+1},\f_1,\ldots,\f_T,\u_*) \toWtwo
(U_1,\ldots,U_{T+1},F_1,\ldots,F_T,U_*)\]
where the joint law of this limit is described by (\ref{eq:UFPCA}).
\item Suppose each function $u_{t+1}(\cdot)$ is the posterior-mean denoiser in
(\ref{eq:utBayes}), and suppose this is Lipschitz on $\RR$. Then
there exist constants $C,\alpha_0>0$ depending only
on $C_0,\eps$ such that for all $\alpha>\alpha_0$,
defining $I_\Delta=[1-C/\alpha^2,1]$ and
$I_\Sigma=[\kappa_2^\infty/2,3\kappa_2^\infty/2]$, there is a unique fixed point
$(\Delta_*,\Sigma_*) \in I_\Delta \times I_\Sigma$ to the equations
\begin{equation}\label{eq:PCAfixedpoint}
1-\Delta_*=\mmse\left(\frac{\alpha^2 \Delta_*^2}{\Sigma_*}\right),
\qquad \Sigma_*=\Delta_* R'\left(\frac{\alpha
\Delta_*(1-\Delta_*)}{\Sigma_*}\right).
\end{equation}
Furthermore,
\begin{equation}\label{eq:alignment}
\lim_{T \to \infty} \left(\lim_{n \to \infty} \frac{1}{n}\u_T^\top \u_*\right)
=\lim_{T \to \infty} \left(\lim_{n \to \infty} \frac{1}{n}\|\u_T\|^2\right)=\Delta_*.
\end{equation}
\end{enumerate}
\end{theorem}

The proof of this result is provided in Appendix \ref{appendix:PCA}. As
discussed in Section \ref{sec:wasserstein}, the notation $\toWtwo$ in part (a)
guarantees that for any continuous function $f:\RR^{2T+2} \to \RR$ satisfying
$\EE[f(U_1,\ldots,U_{T+1},Z_1,\ldots,Z_T,U_*)^2]<\infty$,
\[\big\langle f(\u_1,\ldots,\u_{T+1},\z_1,\ldots,\z_T,\u_*)\big\rangle
\to \EE\big[f(U_1,\ldots,U_{T+1},Z_1,\ldots,Z_T,U_*)\big]\]
where the left side is the empirical average of this function $f$ evaluated
across the $n$ rows.

\begin{remark}\label{remark:squarePCA}
Theorem \ref{thm:PCA}(b) implies that the asymptotic matrix mean-squared-error
of the rank-one estimate $\u_T\u_T^\top$ for $\u_*\u_*^\top$,
in the limit $T \to \infty$, is given by
\begin{align*}
\text{MSE} &\equiv \lim_{T \to \infty}\left(
\lim_{n \to \infty} \frac{1}{n^2}\|\u_T\u_T^\top-\u_*\u_*^\top\|_F^2\right)\\
&=\lim_{T \to \infty} \left(\lim_{n \to \infty} \frac{1}{n^2}(\|\u_T\|^2)^2
-\frac{2}{n^2}(\u_T^\top \u_*)^2+\frac{1}{n^2}(\|\u_*\|^2)^2\right)
=1-\Delta_*^2.
\end{align*}
Let us compare this with the matrix mean-squared-error of the best PCA estimate
$c \cdot \hat{\u}_\PCA\hat{\u}_\PCA^\top$ optimized over $c>0$,
where $\hat{\u}_\PCA$ is the leading sample eigenvector of $\X$. Normalizing
$\hat{\u}_\PCA$ such that $\|\hat{\u}_\PCA\|=\|\u_*\|=\sqrt{n}$, 
\cite[Theorem 2.2(a)]{benaych2011eigenvalues} shows for sufficiently large
$\alpha$ that
\begin{equation}\label{eq:squarePCAbehavior}
\lim_{n \to \infty} \left(\frac{1}{n}\hat{\u}_\PCA^\top \u_*\right)^2
=\Delta_\PCA \equiv\frac{-1}{\alpha^2 G'(G^{-1}(1/\alpha))},
\end{equation}
where $G(z)=\EE[(z-\Lambda)^{-1}]$ is the Cauchy transform of $\Lambda$, and
$G^{-1}(z)$ is the functional inverse of $G$ (which is
well-defined for small $|z|$). Then
\begin{align*}
\text{MSE}_\PCA &\equiv \min_{c>0}\left(\lim_{n \to \infty}
\frac{1}{n^2}\|c \cdot \hat{\u}_\PCA \hat{\u}_\PCA^\top-\u_*\u_*^\top\|_F^2
\right)\\
&=\min_{c>0} c^2-2c\Delta_\PCA+1=1-\Delta_\PCA^2,
\end{align*}
with the minimum attained at the rescaling $c=\Delta_\PCA<1$.

To see that $1-\Delta_*^2 \leq 1-\Delta_\PCA^2$,
observe that for any prior distribution $U_*$ satisfying our normalization
$\EE[U_*^2]=1$, we have
\begin{equation}\label{eq:mmsebound}
\mmse(\mu^2/\sigma^2) \leq \frac{1}{1+\mu^2/\sigma^2}.
\end{equation}
This is because under the scalar observation model (\ref{eq:scalarchannel}),
the right side of (\ref{eq:mmsebound}) is the risk $\EE[(\hat{U}-U_*)^2]$
of the linear estimator $\hat{U}=(\mu/(\sigma^2+\mu^2))F$,
which upper bounds the Bayes risk on the left side of (\ref{eq:mmsebound}).
Equality holds in (\ref{eq:mmsebound}) if and only if $\hat{U}$ is the Bayes
estimator in this model, i.e.\ if and only if the prior distribution
is $U_* \sim \N(0,1)$. Applying (\ref{eq:mmsebound}) to the
first equation of (\ref{eq:PCAfixedpoint}) and rearranging, we obtain
\[\frac{\alpha\Delta_*(1-\Delta_*)}{\Sigma_*} \leq \frac{1}{\alpha}.\]
Now applying this to the second equation of (\ref{eq:PCAfixedpoint}), and
using that $\kappa_2^\infty=\Var[\Lambda]>0$ so that the function
$xR'(x)=\kappa_2^\infty x+2\kappa_3^\infty x^2+3\kappa_4^\infty x^3+\ldots$ is
increasing in a neighborhood of 0, we have for $\alpha>\alpha_0$ sufficiently
large that
\[1-\Delta_*=\frac{1}{\alpha} \cdot
\frac{\alpha \Delta_*(1-\Delta_*)}{\Sigma_*}R'\left(
\frac{\alpha \Delta_*(1-\Delta_*)}{\Sigma_*}\right)
\leq \frac{1}{\alpha^2}R'\left(\frac{1}{\alpha}\right).\]
Differentiating the R-transform identity $R(x)=G^{-1}(x)-1/x$, this
is equivalently written as
\[\Delta_* \geq 1-\frac{1}{\alpha^2}R'\left(\frac{1}{\alpha}\right)
=\frac{-1}{\alpha^2 G'(G^{-1}(1/\alpha))}=\Delta_\PCA,\]
so that
\begin{equation}\label{eq:PCAlowerbound}
\text{MSE}=1-\Delta_*^2 \leq 1-\Delta_\PCA^2=\text{MSE}_\PCA
\end{equation}
as desired. Equality holds here if and only if equality holds in
(\ref{eq:mmsebound}), i.e.\ when $U_* \sim \N(0,1)$. Thus, for any
signal strength $\alpha>\alpha_0$ sufficiently large and any distribution of
$U_*$ other than $\N(0,1)$, the above AMP algorithm achieves strictly better
estimation accuracy than PCA.
\end{remark}

An illustration of the algorithm and state evolution is presented in the left
panel of Figure \ref{fig:PCA}, with noise eigenvalues drawn from a centered and
rescaled $\operatorname{Beta}(1,2)$ distribution. We observe a close agreement
with the state evolution predictions at sample size $n=2000$, and a significant
improvement in estimation accuracy over the naive principal components for this
prior distribution $U_* \sim \operatorname{Uniform}\{+1,-1\}$. Let us remark
that although carrying out many iterations of this AMP algorithm would require
estimating successively higher-order free cumulants of the spectral
distribution of $\W$, for large signal strengths $\alpha$ the algorithm only
needs a very small number of iterations to converge.  

\begin{remark}
In this algorithm, the Onsager corrections involving the free cumulants 
may be understood as iteratively constructing the series (\ref{eq:Rintro})
for $R(\alpha \Delta_*(1-\Delta_*)/\Sigma_*)$, whose derivative appears
in the characterization of the fixed-point in Theorem \ref{thm:PCA}. This is
somewhat analogous to the single-step-memory algorithm in
\cite{opper2016theory} for solving the TAP equations in a related Ising model,
which alternatively constructs a series for the inverse R-transform.

The convergence condition and final mean-squared-error of this algorithm
are likely not Bayes-optimal. For example, we believe that the convergence of
(\ref{eq:AMPPCA1}--\ref{eq:AMPPCA2}) requires convergence of the series
(\ref{eq:Rintro}) at $x=\alpha \Delta_*(1-\Delta_*)/\Sigma_*$, which (depending
on the spectral law of $\W$) may impose a stronger condition for the signal
strength $\alpha$ than the spectral phase transition. One
natural way to improve upon the algorithm is to consider more generally
\[u_{t+1}(\f_1,\ldots,\f_t)=\eta(c_{t1}\f_1+\ldots+c_{tt}\f_t \mid \bc_t^\top
\bmu_t^\infty,\;\bc_t^\top \bSigma_t^\infty \bc_t)\]
for a vector $\bc_t=(c_{t1},\ldots,c_{tt})$ in each iteration, or specialize
this to $\bc_t=(\bSigma_t^\infty)^{-1}\bmu_t^\infty$ to obtain
the posterior mean estimate of $U_*$ given all previous observations
$(F_1,\ldots,F_t)$. Our general results describe also the state evolution for
these extensions, but analyses of their fixed points are more
involved, and we will not pursue this in the current work.

These procedures differ from the ``Vector AMP'' or ``memory-free'' algorithms
of \cite{rangan2019vector,ccakmak2019memory}, whose forms may be derived from
the Expectation Propagation framework of \cite{minka2001family}. These latter
algorithms operate directly on a resolvent of $\W$ and use divergence-free
nonlinearities, corresponding to $\bPhi_t=\bPsi_t=0$ in our notations to follow.
Thus their state evolutions have simpler forms that depend on the first two
moments of the resolvent but not (explicitly) on the free cumulants
of $\W$. PCA differs from the applications in
\cite{opper2016theory,rangan2019vector,ccakmak2019memory} in two important
ways: First, the log-likelihood of $\X$ given $\u$ is not quadratic in $\u$
under general spectral laws of $\W$. Second, the noise matrix $\W$ is not
directly observed in PCA, and its resolvent cannot be directly computed.
Due to these differences, we believe that extending the algorithmic ideas
of \cite{rangan2019vector,ccakmak2019memory} to PCA may be an interesting
open question to study in future work.
\end{remark}

\subsection{Rectangular matrices}\label{sec:rectPCA}

\begin{figure}
\includegraphics[width=0.33\textwidth]{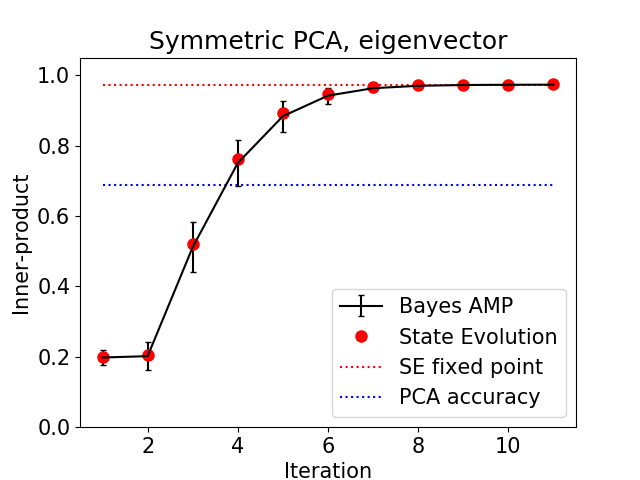}%
\includegraphics[width=0.33\textwidth]{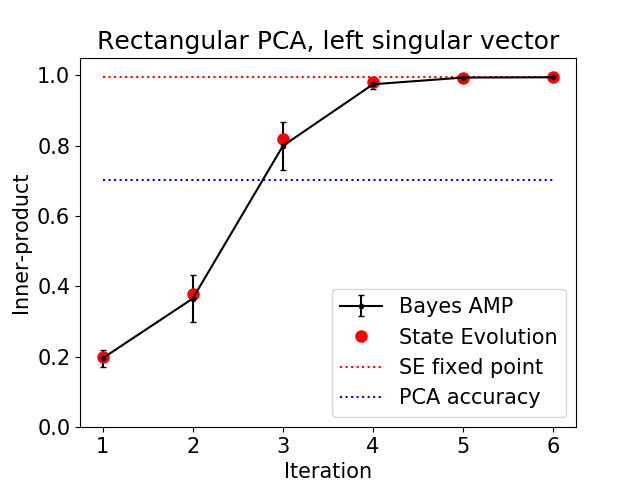}%
\includegraphics[width=0.33\textwidth]{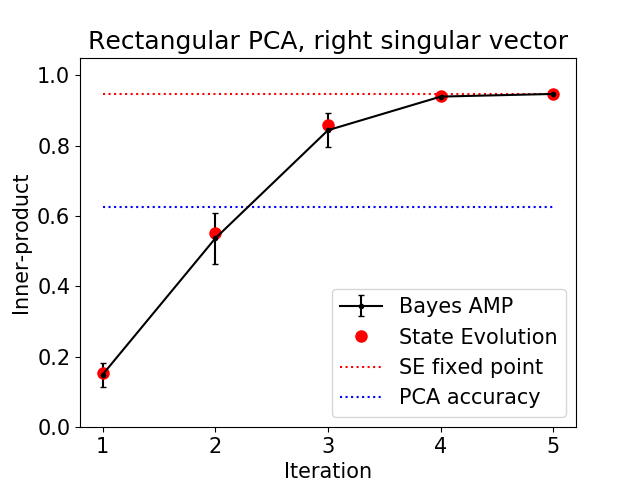}
\caption{Simulations of the Bayes-AMP algorithms for PCA, with priors
$U_*,V_* \sim \operatorname{Uniform}\{+1,-1\}$ and the single-iterate posterior
mean denoisers in (\ref{eq:utBayes}) and (\ref{eq:vtutBayes}). Shown are the
mean and std.\ dev.\ of $\langle \u_t\u_* \rangle$ and
$\langle \v_t\v_* \rangle$ across 100 simulations
in black, their state evolution predictions computed from (\ref{eq:PCAsigma}),
(\ref{eq:PCAsigmarect}), and (\ref{eq:PCAomegarect}) in red dots, and the
fixed points $\Delta_*$ and $\Gamma_*$ of (\ref{eq:PCAfixedpoint}) and
(\ref{eq:PCAfixedpointrect}) in dashed red. For comparison, $\Delta_\PCA$ and
$\Gamma_\PCA$ corresponding to the sample PCs are in dashed blue.
Left: $\u_1$ (initialization), $\u_2,\ldots,\u_{11}$
for symmetric square $\W$ with $n=2000$,
$\alpha=2.5$, and eigenvalue distribution given by centering and scaling
$\operatorname{Beta}(1,2)$ to mean 0 and variance 1.
Middle and right: $\u_1$ (initialization), $\u_2,\ldots,\u_6$
and $\v_1,\ldots,\v_5$ for
rectangular $\W$ with $m=2000$, $n=4000$, $\gamma=0.5$, $\alpha=1.5$,
and singular value distribution given by rescaling
$\operatorname{Beta}(1,2)$ to second-moment 1.}\label{fig:PCA}
\end{figure}

Consider now a rectangular data matrix
\[\X=\frac{\alpha}{m}\u_*\v_*^\top+\W \in \RR^{m \times n},\]
and the task of estimating $\u_* \in \RR^m$ and $\v_* \in \RR^n$.
Writing the singular value decomposition $\W=\O^\top \bLambda\Q$ where
$\bLambda=\diag(\blambda)$ and $\blambda \in \RR^{\min(m,n)}$, we assume that
$\W$ is bi-rotationally invariant in law and that
\[m/n=\gamma,
\qquad \blambda \toW \Lambda, \qquad \u_* \toW U_*, \qquad \v_* \toW V_*\]
as $m,n \to \infty$,
for some constant $\gamma \in (0,\infty)$ and some limit laws
$\Lambda,U_*,V_*$.
We fix the scalings $\|\u_*\|=\sqrt{m}$ and $\|\v_*\|=\sqrt{n}$, so that
\[\EE[U_*^2]=\EE[V_*^2]=1.\]
Note that the rank-one signal component $(\alpha/m)\u_*\v_*^\top$ has singular
value $\alpha/\sqrt{\gamma}$.

We again assume that we have an initialization $\u_1 \in \RR^m$ independent of
$\W$, for which
\[(\u_1,\u_*) \toW (U_1,U_*), \qquad \EE[U_1U_*]>0.\]
We then estimate $\u_*$ and $\v_*$ by the iterates $\u_t$ and $\v_t$ of an
AMP algorithm
\begin{align}
\g_t&=\X^\top \u_t-b_{t1}\v_1-\ldots-b_{t,t-1}\v_{t-1}\label{eq:AMPPCArect1}\\
\v_t&=v_t(\g_t)\label{eq:AMPPCArect2}\\
\f_t&=\X\v_t-a_{t1}\u_1-\ldots-a_{tt}\u_t\label{eq:AMPPCArect3}\\
\u_{t+1}&=u_{t+1}(\f_t),\label{eq:AMPPCArect4}
\end{align}
where $u_{t+1}(\cdot)$ and $v_t(\cdot)$ are scalar denoisers that estimate
$\u_*$ and $\v_*$ from $\f_t$ and $\g_t$.

To describe the forms of the debiasing coefficients $a_{ts}$ and $b_{ts}$, let 
us define $\blambda_m \in \RR^m$ to be $\blambda$ if $m \leq n$ or $\blambda$
extended by $m-n$ additional 0's if $m>n$. We will work instead with the limit
\[\blambda_m \toW \Lambda_m,\]
which is a mixture of $\Lambda$ and a point mass at 0 if $\gamma=m/n>1$.
Denoting the singular values of $\X$ by $\lambda_1(\X) \geq \ldots \geq
\lambda_{\min(m,n)}(\X)$, for each $k \geq 1$ we set
\[m_{2k}=\frac{1}{m}\sum_{i=2}^{\min(m,n)} \lambda_i(\X)^{2k}.\]
We then define $\{\kappa_{2k}\}_{k \geq 1}$ as the rectangular free cumulants
associated to these even moments $\{m_{2k}\}_{k \geq 1}$
and aspect ratio $\gamma$, as defined in
Section \ref{sec:cumulantsrect}. It is easily checked that as $m,n \to \infty$,
\[m_{2k} \to m_{2k}^\infty=\EE[\Lambda_m^{2k}],
\qquad \kappa_{2k} \to \kappa_{2k}^\infty,\]
where these limits are the even moments and rectangular free cumulants of
$\Lambda_m$. Then the debiasing coefficients in
(\ref{eq:AMPPCArect1}--\ref{eq:AMPPCArect4}) are set as
\begin{align*}
a_{t,t-j}&=\kappa_{2(j+1)}\langle v_t'(\g_t) \rangle
\prod_{i=t-j+1}^t \langle u_i'(\f_{i-1})\rangle \langle v_{i-1}'(\g_{i-1})
\rangle \quad \text{ for } \quad j=0,\ldots,t-1,\\
b_{t,t-j}&=\gamma \kappa_{2j} \langle u_t'(\f_{t-1}) \rangle
\prod_{i=t-j+1}^{t-1} \langle v_i'(\g_i) \rangle \langle u_i'(\f_{i-1}) \rangle
\quad \text{ for } \quad j=1,\ldots,t-1.
\end{align*}
We use the convention that empty products equal 1, so the first coefficients
here are simply
\[a_{tt}=\kappa_2\langle v_t'(\g_t) \rangle, \qquad
b_{t,t-1}=\gamma \kappa_2 \langle u_t'(\f_{t-1}) \rangle.\]

The state evolution for this algorithm may be expressed in terms of two
sequences of mean vectors $\bmu_T^\infty=(\mu_t^\infty)_{1 \leq t \leq T}$ and
$\bnu_T^\infty=(\nu_t^\infty)_{1 \leq t \leq T}$ and 
covariance matrices
$\bSigma_T^\infty=(\sigma_{st}^\infty)_{1 \leq s,t \leq T}$ and
$\bOmega_T^\infty=(\omega_{st}^\infty)_{1 \leq s,t \leq T}$, defined as follows:
For $T=1$ we set
\[\nu_1^\infty=\alpha \cdot \EE[U_1U_*],
\qquad \omega_{11}^\infty=\gamma \kappa_2^\infty \cdot \EE[U_1^2].\]
Having defined $\bmu_{T-1}^\infty$, $\bSigma_{T-1}^\infty$,
$\bnu_T^\infty$, and $\bOmega_T^\infty$, we denote
\[U_t=u_t(F_{t-1}) \text{ for } t=2,\ldots,T, \qquad
(F_1,\ldots,F_{T-1})=\bmu_{T-1}^\infty \cdot U_*+(Y_1,\ldots,Y_{T-1}),\]
\[(Y_1,\ldots,Y_{T-1}) \sim \N(0,\bSigma_{T-1}^\infty) \text{ independent of }
(U_1,U_*),\]
\[V_t=v_t(G_t) \text{ for } t=1,\ldots,T, \qquad
(G_1,\ldots,G_T)=\bnu_T^\infty \cdot V_*+(Z_1,\ldots,Z_T),\]
\begin{equation}\label{eq:UFPCArect}
(Z_1,\ldots,Z_T) \sim \N(0,\bOmega_T^\infty) \text{ independent of } V_*.
\end{equation}
We then define $\bmu_T^\infty$ and $\bSigma_T^\infty$ with the entries, for
$1 \leq s,t \leq T$,
\begin{align}
\mu_t^\infty&=(\alpha/\gamma) \cdot \EE[V_tV_*]\nonumber\\
\sigma_{st}^\infty&=\sum_{j=0}^{s-1} \sum_{k=0}^{t-1}
\left(\prod_{i=s-j+1}^s \EE[v_i'(G_i)]\EE[u_i'(F_{i-1})]\right)
\left(\prod_{i=t-k+1}^t \EE[v_i'(G_i)]\EE[u_i'(F_{i-1})]\right)\nonumber\\
&\hspace{0.2in}\cdot\left(\kappa_{2(j+k+1)}^\infty \EE[V_{s-j}V_{t-k}]
+\kappa_{2(j+k+2)}^\infty \EE[v_{s-j}'(G_{s-j})]
\EE[v_{t-k}'(G_{t-k})]\EE[U_{s-j}U_{t-k}]\right).\label{eq:PCAsigmarect}
\end{align}
Now having defined $\bmu_T^\infty$ and $\bSigma_T^\infty$, we extend
(\ref{eq:UFPCArect}) to 
\[U_t=u_t(F_{t-1}) \text{ for } t=2,\ldots,T+1, \qquad
(F_1,\ldots,F_T)=\bmu_T^\infty \cdot U_*+(Y_1,\ldots,Y_T),\]
\begin{equation}\label{eq:UFPCArect2}
(Y_1,\ldots,Y_T) \sim \N(0,\bSigma_T^\infty) \text{ independent of }
(U_1,U_*)
\end{equation}
and define $\bnu_{T+1}^\infty$ and $\bOmega_{T+1}^\infty$ with the entries, for
$1 \leq s,t \leq T+1$,
\begin{align}
\nu_t^\infty&=\alpha \cdot \EE[U_tU_*]\nonumber\\
\omega_{st}^\infty&=\gamma \sum_{j=0}^{s-1} \sum_{k=0}^{t-1}
\left(\prod_{i=s-j+1}^s \EE[u_i'(F_{i-1})]\EE[v_{i-1}'(G_{i-1})]\right)
\left(\prod_{i=t-k+1}^t \EE[u_i'(F_{i-1})]\EE[v_{i-1}'(G_{i-1})]\right)
\nonumber\\
&\hspace{0.2in}\cdot\left(\kappa_{2(j+k+1)}^\infty
\EE[U_{s-j}U_{t-k}]+\kappa_{2(j+k+2)}^\infty \EE[u_{s-j}'(F_{s-j-1})]
\EE[u_{t-k}'(F_{t-k-1})]\EE[V_{s-j-1}V_{t-k-1}]\right).\label{eq:PCAomegarect}
\end{align}
We use the convention $V_0=0$, so that the second term of
(\ref{eq:PCAomegarect}) is 0 for $j=s-1$ or $k=t-1$. 
In the limit $m,n \to \infty$, the iterates of
(\ref{eq:AMPPCArect1}--\ref{eq:AMPPCArect4}) will satisfy
\[(\f_1,\ldots,\f_T,\u_*) \toWtwo (F_1,\ldots,F_T,U_*),
\qquad (\g_1,\ldots,\g_T,\v_*) \toWtwo (G_1,\ldots,G_T,V_*).\]

As an example of choices for $v_t(\cdot)$ and $u_{t+1}(\cdot)$, let us again
analyze the single-iterate posterior mean denoisers given by
\begin{equation}\label{eq:vtutBayes}
v_t(g_t)=\eta(g_t \mid \nu_t,\omega_{tt}),
\qquad u_{t+1}(f_t)=\eta(f_t \mid \mu_t,\sigma_{tt}),
\end{equation}
where $\eta(\cdot)$ is as defined in (\ref{eq:denoiserexplicit}), and
$(\nu_t,\omega_{tt})$ and $(\mu_t,\sigma_{tt})$ are the state evolution
parameters describing the univariate Gaussian laws of $G_t$ and $F_t$.
We denote by $\mmse(\cdot)$ the scalar mean-squared-error function from
(\ref{eq:mmse}), and by $R(x)$ the rectangular R-transform of
$\Lambda_m$ with aspect ratio $\gamma$, as discussed in Section
\ref{sec:cumulantsrect}. This may be defined for small
$|x|$ by the convergent series (see Proposition \ref{prop:cumulantbound})
\[R(x)=\sum_{k=1}^\infty \kappa_{2k}^\infty x^k,\]
where $\kappa_{2k}^\infty$ are the rectangular free cumulants of $\Lambda_m$
above. We denote $R'(x)$ as its derivative, and
\[S(x)=\left(\frac{R(x)}{x}\right)'=\frac{xR'(x)-R(x)}{x^2}.\]

\begin{theorem}\label{thm:PCArect}
Suppose $\W=\O^\top \bLambda \Q \in \RR^{m \times n}$ where $\Q$ and $\O$ are
Haar-uniform orthogonal matrices. Let $\bLambda=\diag(\blambda)$, where
$(\blambda,\u_1,\u_*,\v_*)$ are independent of $(\O,\Q)$,
$\|\u_*\|=\sqrt{n}$, $\|\v_*\|=\sqrt{m}$, and
\[\blambda \toW \Lambda, \qquad (\u_1,\u_*) \toW (U_1,U_*),
\qquad \v_* \toW V_*, \qquad m/n=\gamma \in (0,\infty)\]
as $m,n \to \infty$. Suppose $\EE[U_1^2] \leq 1$, $\EE[U_1U_*]=\eps>0$,
and $\|\blambda\|_\infty \leq
C_0$, almost surely for all large $n$ and some constants $C_0,\eps>0$.
\begin{enumerate}[(a)]
\item Let $\alpha \geq 0$, and let each function $v_t(\cdot)$ and
$u_{t+1}(\cdot)$ be continuously differentiable and Lipschitz on $\RR$. Then
for each fixed $T \geq 1$, almost surely as $m,n \to \infty$,
\begin{align*}
(\u_1,\ldots,\u_{T+1},\f_1,\ldots,\f_T,\u_*) &\toW
(U_1,\ldots,U_{T+1},F_1,\ldots,F_T,U_*),\\
(\v_1,\ldots,\v_T,\g_1,\ldots,\g_T,\v_*) &\toW
(V_1,\ldots,V_T,G_1,\ldots,G_T,V_*),
\end{align*}
where these limits are as defined in (\ref{eq:UFPCArect}) and
(\ref{eq:UFPCArect2}).
\item Suppose $v_t(\cdot),u_{t+1}(\cdot)$ are the posterior-mean denoisers
in (\ref{eq:vtutBayes}) and are Lipschitz on $\RR$.
There exist constants $C,\alpha_0>0$ depending only on
$C_0,\eps,\gamma$ such that for all $\alpha>\alpha_0$, setting
\[I_\Delta=I_\Gamma=[1-C/\alpha^2,1],\quad 
I_\Sigma=[\kappa_2^\infty/2,3\kappa_2^\infty/2],
\quad I_\Omega=\gamma \cdot I_\Sigma,\]
there is a unique fixed point
$(\Delta_*,\Sigma_*,\Gamma_*,\Omega_*,X_*) \in I_\Delta \times I_\Sigma \times
I_\Gamma \times I_\Omega \times \RR$ to the equations
\begin{align}
X_*&=\frac{\alpha^2\Delta_*\Gamma_*(1-\Delta_*)(1-\Gamma_*)}
{\gamma \Sigma_*\Omega_*},\;\;
1-\Delta_*=\mmse\left(\frac{\alpha^2\Gamma_*^2}{\gamma^2\Sigma_*}\right),\;\;
1-\Gamma_*=\mmse\left(\frac{\alpha^2\Delta_*^2}{\Omega_*}\right),\nonumber\\
\Sigma_*&=\Gamma_*R'(X_*)+\frac{\alpha^2\Delta_*^3(1-\Gamma_*)^2}{\Omega_*^2}S(X_*),\;\;
\Omega_*=\gamma \Delta_*R'(X_*)
+\frac{\alpha^2\Gamma_*^3(1-\Delta_*)^2}{\gamma\Sigma_*^2}S(X_*).
\label{eq:PCAfixedpointrect}
\end{align}
Furthermore,
\begin{align*}
\lim_{T \to \infty} \left(\lim_{m,n \to \infty} \frac{1}{m}\u_T^\top \u_*\right)
&=\lim_{T \to \infty} \left(\lim_{m,n \to \infty} \frac{1}{m}\|\u_T\|^2\right)
=\Delta_*\\
\lim_{T \to \infty} \left(\lim_{m,n \to \infty} \frac{1}{n}\v_T^\top \v_*\right)
&=\lim_{T \to \infty} \left(\lim_{m,n \to \infty} \frac{1}{n}\|\v_T\|^2\right)
=\Gamma_*.
\end{align*}
\end{enumerate}
\end{theorem}
The proof of this result is provided in Appendix \ref{appendix:PCA}.

\begin{remark}
As in the symmetric square setting of Remark \ref{remark:squarePCA},
the above fixed points imply that the asymptotic matrix
mean-squared-error is given by
\[\text{MSE} \equiv \lim_{T \to \infty}\left(\lim_{m,n \to \infty}
\frac{1}{mn}\|\u_T\v_T^\top-\u_*\v_*^\top\|_F^2\right)=1-\Delta_*\Gamma_*.\]
We may compare this with the asymptotic error of the PCA estimate: Assume
without loss of generality that $\gamma=m/n \leq 1$. Let $\hat{\u}_\PCA$ and
$\hat{\v}_\PCA$ be the leading left and right singular vectors of $\X$, with the
scalings $\|\hat{\u}_\PCA\|=\|\u_*\|=\sqrt{m}$ and
$\|\hat{\v}_\PCA\|=\|\v_*\|=\sqrt{n}$. Recall that the singular value of the
rank-one signal $(\alpha/m)\u_*\v_*^\top$ is $\alpha/\sqrt{\gamma}$, and set
\[x=\gamma/\alpha^2.\]
Then \cite[Theorem 2.9]{benaych2012singular} shows
\begin{align}
\lim_{m,n \to \infty} \left(\frac{1}{m}\hat{\u}_\PCA^\top \u_*\right)^2
&=\Delta_\PCA \equiv \frac{-2x\varphi(D^{-1}(x))}{D'(D^{-1}(x))}
\label{eq:DeltaPCA}\\
\lim_{m,n \to \infty} \left(\frac{1}{n}\hat{\v}_\PCA^\top \v_*\right)^2
&=\Gamma_\PCA \equiv \frac{-2x\bar{\varphi}(D^{-1}(x))}{D'(D^{-1}(x))}
\label{eq:GammaPCA}
\end{align}
where
\begin{equation}\label{eq:BGNnotation}
\varphi(z)=\EE\left[\frac{z}{z^2-\Lambda_m^2}\right],
\quad \bar{\varphi}(z)=\gamma \varphi(z)+\frac{1-\gamma}{z},
\quad D(z)=\varphi(z)\bar{\varphi}(z),
\end{equation}
and $D^{-1}(z)$ is the functional inverse of $D$ for small $|z|$. Then the
matrix mean-squared-error for the best rescaling of the PCA estimate is
\begin{align*}
\text{MSE}_\PCA& \equiv \min_{c>0} \left(\lim_{m,n \to \infty}
\frac{1}{mn}\|c\cdot \hat{\u}_\PCA\hat{\v}_\PCA^\top
-\u_*\v_*^\top\|_F^2\right)\\
&=\min_{c>0}\left(c^2-2c\sqrt{\Delta_\PCA \Gamma_\PCA}+1\right)
=1-\Delta_\PCA\Gamma_\PCA
\end{align*}
with the minimum attained at $c=\sqrt{\Delta_\PCA \Gamma_\PCA}$.
We verify in Appendix \ref{subsec:rectPCAimproves} that for all
$\alpha>\alpha_0$ sufficiently large, the fixed points of
Theorem \ref{thm:PCArect}(b) satisfy
\begin{equation}\label{eq:rectPCAimproves}
\text{MSE}=1-\Delta_*\Gamma_* \leq 1-\Delta_\PCA\Gamma_\PCA=\text{MSE}_\PCA,
\end{equation}
and that equality holds if and only if both $U_* \sim \N(0,1)$ and $V_* \sim
\N(0,1)$. Thus, for sufficiently large signal strength and any
non-Gaussian prior for either $U_*$ or $V_*$, the above AMP algorithm achieves
strictly better estimation accuracy than PCA.
\end{remark}

An illustration of this AMP algorithm and its state evolution is presented in
the middle and right panels of Figure \ref{fig:PCA}, with noise singular values
drawn from a rescaled $\operatorname{Beta}(1,2)$ distribution. Again, close
agreement with the state evolution predictions is observed at these sample sizes
$(m,n)=(2000,4000)$ and $\gamma=1/2$.

\section{General AMP algorithm for symmetric square matrices}\label{sec:square}

We now describe the general AMP algorithm for symmetric square matrices
\begin{equation}\label{eq:squaremodel}
\W=\O^\top \bLambda \O \in \RR^{n \times n}, \qquad \bLambda=\diag(\blambda)
\end{equation}
and we state a formal theorem for its state evolution.

We consider an initialization $\u_1 \in \RR^n$, and also a possible matrix of
side information
\[\E \in \RR^{n \times k}\]
for a fixed dimension $k \geq 0$, both independent of $\W$.
(We may take $k=0$ if there is no such
side information.) Starting from this initialization $\u_1$,
the AMP algorithm takes the form
\begin{align}
\z_t&=\W \u_t-b_{t1}\u_1-b_{t2}\u_2-\ldots-b_{tt}\u_t\label{eq:AMPz}\\
\u_{t+1}&=u_{t+1}(\z_1,\ldots,\z_t,\E)\label{eq:AMPu}
\end{align}
Each function $u_{t+1}:\RR^{t+k} \to \RR$ is applied row-wise
to $(\z_1,\ldots,\z_t,\E) \in \RR^{n \times (t+k)}$.
The debiasing coefficients $b_{t1},\ldots,b_{tt} \in \RR$ are defined
to ensure the empirical convergence
\[(\z_1,\ldots,\z_t) \toW \N(0,\bSigma_t^\infty)\]
as $n \to \infty$. The forms of $b_{t1},\ldots,b_{tt}$ and $\bSigma_t^\infty$
were first described in \cite{opper2016theory}, and we review this
in the next section.

\subsection{Debiasing coefficients and limit covariance}\label{sec:BSigma}

Define the $t \times t$ matrices
\begin{equation}\label{eq:DeltaPhi}
\bDelta_t=\begin{pmatrix} \langle \u_1^2 \rangle & \langle \u_1 \u_2 \rangle
& \cdots & \langle \u_1 \u_t \rangle \\
\langle \u_2\u_1 \rangle & \langle \u_2^2 \rangle & \cdots & \langle \u_2 \u_t
\rangle\\
\vdots & \vdots & \ddots & \vdots \\ \langle \u_t\u_1 \rangle & \langle \u_t\u_2
\rangle & \cdots & \langle \u_t^2 \rangle \end{pmatrix},
\quad \bPhi_t=\begin{pmatrix} 0 & 0 & \cdots & 0 & 0 \\
\langle \partial_1 \u_2 \rangle & 0 & \cdots & 0 & 0 \\
\langle \partial_1 \u_3 \rangle & \langle \partial_2 \u_3 \rangle
& \cdots & 0 & 0 \\
\vdots & \vdots & \ddots & \vdots & \vdots \\
\langle \partial_1 \u_t \rangle & \langle \partial_2 \u_t \rangle
& \cdots & \langle \partial_{t-1} \u_t \rangle & 0 \end{pmatrix}
\end{equation}
where $\u_s\u_{s'} \in \RR^n$, $\u_s^2 \in \RR^n$, and $\partial_{s'} \u_s
\in \RR^n$ denote the entrywise product, square, and partial derivative with
respect to $z_{s'}$. For each $j \geq 0$, define
\begin{equation}\label{eq:Theta}
\bTheta^{(j)}_t=\sum_{i=0}^j \bPhi_t^i \bDelta_t (\bPhi_t^{j-i})^\top.
\end{equation}
For example,
\[\bTheta_t^{(0)}=\bDelta_t, \quad \bTheta_t^{(1)}=\bPhi_t\bDelta_t
+\bDelta_t\bPhi_t^\top, \quad \bTheta_t^{(2)}=\bPhi_t^2\bDelta_t
+\bPhi_t\bDelta_t\bPhi_t^\top+\bDelta_t(\bPhi_t^2)^\top.\]

Let $\{\kappa_k\}_{k \geq 1}$ be the free cumulants of the empirical
eigenvalue distribution of $\W$. These are the free cumulants
as defined in Section \ref{sec:cumulants}
corresponding to the empirical moments
\begin{equation}\label{eq:empiricalmoments}
m_k=\frac{1}{n}\sum_{i=1}^n \lambda_i^k,
\end{equation}
where $(\lambda_1,\ldots,\lambda_n)=\blambda$ are the eigenvalues of $\W$.
Then define two matrices $\B_t$ and $\bSigma_t$ by
\begin{equation}\label{eq:BSigma}
\B_t=\left(\sum_{j=0}^\infty \kappa_{j+1}\bPhi_t^j\right)^\top,
\qquad \bSigma_t=\sum_{j=0}^\infty \kappa_{j+2}\bTheta_t^{(j)}.
\end{equation}
Here, $\B_t$ may be interpreted as the R-transform applied to $\bPhi_t^\top$.
Note that we write these as infinite series for
convenience, but in fact the series are finite because $\bPhi_t^j=0$ for
all $j \geq t$, and hence also 
$\bTheta_t^{(j)}=0$ for all $j \geq 2t-1$. So for example,
\[\B_1=\kappa_1\Id_{1 \times 1}, \qquad \B_2=\kappa_1 \Id_{2 \times 2}
+\kappa_2\bPhi_2^\top,\]
\[\bSigma_1=\kappa_2\bTheta_1^{(0)}, \qquad
\bSigma_2=\kappa_2\bTheta_2^{(0)}+\kappa_3\bTheta_2^{(1)}
+\kappa_4\bTheta_2^{(2)}.\]

Each matrix $\B_t$ is upper-triangular, which we may write entrywise as
\[\B_t=\begin{pmatrix} b_{11} & b_{21} & \cdots & b_{t1} \\
& b_{22} & \cdots & b_{t2} \\
& & \ddots & \vdots \\
& & & b_{tt} \end{pmatrix}.\]
The debiasing coefficients in (\ref{eq:AMPz}) are defined to be
the last column of $\B_t$. Note that
the diagonal entries $b_{11},b_{22},\ldots,b_{tt}$ are all equal to $\kappa_1$,
corresponding to the subtraction of $\kappa_1\u_t$ in (\ref{eq:AMPz}) when the
eigenvalue distribution of $\W$ has mean $\kappa_1$. If $\kappa_1=0$, then the
debiasing for $\W\u_t$ depends only on the previous iterates
$\u_1,\ldots,\u_{t-1}$.

Under the conditions to be imposed in Assumption \ref{assump:main}, all of the
matrices $\bDelta_t$, $\bPhi_t$, $\B_t$, and
$\bSigma_t$ will converge to deterministic $t \times t$ matrices in the
$n \to \infty$ limit, which we denote as
\[(\bDelta_t^\infty,\bPhi_t^\infty,\B_t^\infty,\bSigma_t^\infty)
=\lim_{n \to \infty} (\bDelta_t,\bPhi_t,\B_t,\bSigma_t).\]
This matrix $\bSigma_t^\infty$ is the covariance defining the state
evolution of the iterates $(\z_1,\ldots,\z_t)$.
All of our results will hold equally if the debiasing coefficients in
(\ref{eq:AMPz}) are replaced by their limits $b_{ts}^\infty$, or by any
consistent estimates of these limits.

We make two observations regarding this construction:
\begin{enumerate}
\item From the lower-triangular form of $\bPhi_t$, one may check
that the upper-left $(t-1) \times (t-1)$ submatrix of $(\bPhi_t^j)^\top$
is $(\bPhi_{t-1}^j)^\top$, and similarly the upper-left $(t-1) \times (t-1)$
submatrix of $\bTheta_t^{(j)}$ is $\bTheta_{t-1}^{(j)}$. Thus, the upper-left
submatrices of $\B_t$ and $\bSigma_t$ coincide with
$\B_{t-1}$ and $\bSigma_{t-1}$.
\item For each iteration $t \geq 1$, $\B_t$
depends on $\blambda$ only via its first $t$ free cumulants
$\kappa_1,\ldots,\kappa_t$, and $\bSigma_t$ depends on $\blambda$
only via its first $2t$ free cumulants $\kappa_1,\ldots,\kappa_{2t}$.
\end{enumerate}

\begin{remark}\label{remark:GOE}
In the Gaussian setting of $\W \sim \operatorname{GOE}(n)$, where $\W$ has
independent $\N(0,1/n)$ entries above the diagonal and $\N(0,2/n)$
entries on the diagonal, the limit spectral distribution of $\W$ is the Wigner
semicircle law. The limits of the free cumulants $\kappa_1,\kappa_2,\ldots$ in
this case are
\[\kappa_1^\infty=0, \qquad \kappa_2^\infty=1, \qquad \kappa_j^\infty=0 \qquad
\text{ for all } j \geq 2.\]
This yields simply
\[\B_t^\infty=(\bPhi_t^\infty)^\top, \qquad \bSigma_t^\infty=\bDelta_t^\infty.\]
If we further specialize to an algorithm where each $\u_t$ depends only on
the previous iterate $\z_{t-1}$, then $\langle \partial_s \u_t \rangle=0$ for
$s \neq t-1$, and this yields the Gaussian AMP algorithm
\begin{align*}
\z_t&=\W\u_t-\langle \partial_{t-1} \u_t \rangle \u_{t-1}\\
\u_{t+1}&=u_{t+1}(\z_t,\E)
\end{align*}
as studied in \cite{bolthausen2014iterative} and
\cite[Section 4]{bayati2011dynamics}. Furthermore, the state
evolution is such that each iterate $\z_t$ has the empirical limit
$\N(0,\sigma_{tt}^\infty)$, where
$\sigma_{tt}^\infty=\lim_{n \to \infty} \langle \u_t^2 \rangle$.

Note that outside of this Gaussian setting, we do not in general have
the identity $\bSigma_t^\infty=\bDelta_t^\infty$, i.e.\ the empirical second
moments of $\z_1,\ldots,\z_t$ do not coincide with those of $\u_1,\ldots,\u_t$
in the large-$n$ limit, even if $\W$ is scaled so that $\kappa_2=1$.
\end{remark}

\subsection{Main result}

We impose the following assumptions on the model (\ref{eq:squaremodel})
and the AMP iterates (\ref{eq:AMPz}--\ref{eq:AMPu}). Note that here, we do not
require the functions $u_{t+1}(\cdot)$ to be Lipschitz, but instead impose only
the assumption (\ref{eq:growth}) of polynomial growth.

\begin{assumption}\label{assump:main}
\hspace{2em}
\begin{enumerate}[(a)]
\item $\O \in \RR^{n \times n}$ is a random and Haar-uniform orthogonal matrix.
\item $\blambda \in \RR^n$ is independent of $\O$ and satisfies
$\blambda \toW \Lambda$ almost surely as $n \to \infty$, for a random variable
$\Lambda$ having finite moments of all orders.
\item $\u_1 \in \RR^n$ and $\E \in \RR^{n \times k}$ are independent of $\O$
and satisfy $(\u_1,\E) \toW (U_1,E)$ almost surely
as $n \to \infty$, for a random vector
$(U_1,E) \equiv (U_1,E_1,\ldots,E_k)$ having finite moments of all orders.
\item Each function $u_{t+1}:\RR^{t+k} \to \RR$ satisfies (\ref{eq:growth}) for
some $C>0$ and $p \geq 1$.
Writing its argument as $(z,e)$ where $z \in \RR^t$ and $e \in \RR^k$,
$u_{t+1}$ is weakly differentiable in $z$ and continuous in $e$.
For each $s=1,\ldots,t$, $\partial_s u_{t+1}$ also
satisfies (\ref{eq:growth}) for some $C>0$ and $p \geq 1$,
and $\partial_s u_{t+1}(z,e)$ is continuous
at Lebesgue-a.e.\ $z \in \RR^t$ for every $e \in \RR^k$.
\item $\Var[\Lambda]>0$ and $\EE[U_1^2]>0$. Letting $(Z_1,\ldots,Z_t) \sim
\N(0,\bSigma_t^\infty)$ be independent of $(U_1,E)$, each function
$u_{t+1}$ is such that there do not exist constants
$\alpha_1,\ldots,\alpha_t,\beta_1,\ldots,\beta_t$ for which
\[u_{t+1}(Z_1,\ldots,Z_t,E)=\sum_{s=1}^t \alpha_s Z_s+\beta_1 U_1+\sum_{s=2}^t
\beta_s U_s(Z_1,\ldots,Z_{s-1},E)\]
with probability 1 over $(U_1,E,Z_1,\ldots,Z_t)$.
\end{enumerate}
\end{assumption}

We clarify that Theorem \ref{thm:main} below establishes the existence of the
limit $\bSigma_t^\infty$ provided that condition (e) holds for the
functions $u_2,\ldots,u_t$, and this limit $\bSigma_t^\infty$ then defines
condition (e) for the next function $u_{t+1}$. This condition (e) is a
non-degeneracy assumption that holds if each function $u_{t+1}$ has a
non-linear dependence on the preceding iterate $z_t$.

\begin{theorem}\label{thm:main}
Under Assumption \ref{assump:main}, for each fixed $t \geq 1$,
almost surely as $n \to \infty$: $\bSigma_t \to \bSigma_t^\infty$ for 
a deterministic non-singular matrix $\bSigma_t^\infty$, and
\[(\u_1,\ldots,\u_{t+1},\z_1,\ldots,\z_t,\E) \toW
(U_1,\ldots,U_{t+1},Z_1,\ldots,Z_t,E)\]
where $(Z_1,\ldots,Z_t) \sim \N(0,\bSigma_t^\infty)$,
this vector $(Z_1,\ldots,Z_t)$ is independent of $(U_1,E)$, and
$U_s=u_s(Z_1,\ldots,Z_{s-1},E)$ for each $s=2,\ldots,t+1$.
\end{theorem}

The proof of this result is provided in Appendix
\ref{appendix:squareproof}. The limit $\bSigma_t^\infty$ is
given by replacing $\langle \u_s\u_{s'} \rangle$,
$\langle \partial_{s'} \u_s \rangle$, and $\kappa_k$ in the
definitions (\ref{eq:DeltaPhi}) and (\ref{eq:BSigma}) with $\EE[U_sU_{s'}]$,
$\EE[\partial_{s'} u_s(Z_1,\ldots,Z_{s-1},E)]$, and the free cumulants
$\kappa_k^\infty$ of the limit spectral distribution $\Lambda$.

\subsection{Removing the non-degeneracy assumption}

The following corollary provides a version of Theorem \ref{thm:main}
without the non-degeneracy condition of Assumption \ref{assump:main}(e),
under the stronger condition that each
function $u_{t+1}$ is continuously-differentiable and Lipschitz.
Note that the convergence established is only in $W_2$, rather than
in $W_p$ for every order $p \geq 1$ as in Theorem \ref{thm:main}.

The proof follows the idea of
\cite{berthier2020state} by studying a perturbed AMP sequence and then taking
the limit of this perturbation to 0. We provide this proof in
Appendix \ref{appendix:degenerate}.

\begin{corollary}\label{cor:degenerate}
Suppose Assumption \ref{assump:main}(a--c) holds, $\limsup_{n \to \infty}
\|\blambda\|_\infty<\infty$, each function
$u_{t+1}:\RR^{t+k} \to \RR$ is continuously-differentiable, and
\[|u_{t+1}(z,e)-u_{t+1}(z',e)| \leq C\|z-z'\|\]
for a constant $C>0$ and all $z,z' \in \RR^t$ and $e \in \RR^k$.
Then for each fixed $t \geq 1$, almost surely as $n \to \infty$:
$\bSigma_t \to \bSigma_t^\infty$ for a deterministic (possibly singular)
matrix $\bSigma_t^\infty$, and
\[(\u_1,\ldots,\u_{t+1},\z_1,\ldots,\z_t,\E) \toWtwo
(U_1,\ldots,U_{t+1},Z_1,\ldots,Z_t,E)\]
where $(U_1,\ldots,U_{t+1},Z_1,\ldots,Z_t,E)$ is as defined in
Theorem \ref{thm:main}.
\end{corollary}

\section{AMP algorithm for rectangular matrices}\label{sec:rect}

In this section, we describe the form of the general AMP algorithm for
a rectangular matrix
\begin{equation}\label{eq:rectmodel}
\W=\O^\top \bLambda \Q \in \RR^{m \times n}, \qquad \bLambda=\diag(\blambda)
\end{equation}
and state a formal theorem for its state evolution. We denote
\begin{equation}\label{eq:lambdarect}
\blambda=(\lambda_1,\ldots,\lambda_{\min(m,n)}) \in \RR^{\min(m,n)}
\end{equation}
as the diagonal entries of $\bLambda$, which are the singular values of $\W$.

We consider an initialization $\u_1 \in \RR^m$, and two matrices of side
information
\[\E \in \RR^{m \times k} \qquad \text{ and }
\qquad \F \in \RR^{n \times \ell}\]
for fixed dimensions $k,\ell \geq 0$, all independent of $\W$.
(We may take $k,\ell=0$ if there is no
such side information.) Starting from this initialization, the AMP algorithm
takes the form
\begin{align}
\z_t&=\W^\top
\u_t-b_{t1}\v_1-b_{t2}\v_2-\ldots-b_{t,t-1}\v_{t-1}\label{eq:AMPrectz}\\
\v_t&=v_t(\z_1,\ldots,\z_t,\F)\label{eq:AMPrectv}\\
\y_t&=\W \v_t-a_{t1}\u_1-a_{t2}\u_2-\ldots-a_{tt}\u_t\label{eq:AMPrecty}\\
\u_{t+1}&=u_{t+1}(\y_1,\ldots,\y_t,\E)\label{eq:AMPrectu}
\end{align}
for functions $v_t:\RR^{t+\ell} \to \RR$ and $u_{t+1}:\RR^{t+k} \to \RR$.
In the first iteration $t=1$, (\ref{eq:AMPrectz}) is simply $\z_1=\W^\top \u_1$.
The debiasing coefficients $a_{t1},\ldots,a_{tt}$ and
$b_{t1},\ldots,b_{t,t-1}$ are defined to ensure that
\[(\y_1,\ldots,\y_t) \toW \N(0,\bSigma_t^\infty) \qquad \text{ and }
\qquad (\z_1,\ldots,\z_t) \toW \N(0,\bOmega_t^\infty)\]
as $m,n \to \infty$. We describe these debiasing coefficients and
state evolution in the next section, in terms of the rectangular free cumulants
of $\W$---these were also derived recently in \cite{ccakmak2020dynamical}.

\subsection{Debiasing coefficients and limit covariance}\label{sec:debiasingrect}

Define the $t \times t$ matrices
\begin{align}
\bDelta_t=\begin{pmatrix} \langle \u_1^2 \rangle & \langle \u_1\u_2 \rangle
& \cdots & \langle \u_1\u_t \rangle \\
\langle \u_2\u_1 \rangle & \langle \u_2^2 \rangle & \cdots & \langle \u_2\u_t
\rangle\\
\vdots & \vdots & \ddots & \vdots \\
\langle \u_t\u_1 \rangle & \langle \u_t\u_2 \rangle & \cdots &
\langle \u_t^2 \rangle \end{pmatrix},&
\quad \bPhi_t=\begin{pmatrix} 0 & 0 & \cdots & 0 & 0 \\
\langle \partial_1 \u_2 \rangle & 0 & \cdots & 0 & 0 \\
\langle \partial_1 \u_3 \rangle & \langle \partial_2 \u_3 \rangle
& \cdots & 0 & 0 \\
\vdots & \vdots & \ddots & \vdots & \vdots \\
\langle \partial_1 \u_t \rangle & \langle \partial_2 \u_t \rangle
& \cdots & \langle \partial_{t-1} \u_t \rangle & 0 \end{pmatrix},
\label{eq:DeltaPhirect}\\
\bGamma_t=\begin{pmatrix} \langle \v_1^2 \rangle & \langle \v_1\v_2 \rangle &
\cdots & \langle \v_1\v_t \rangle \\
\langle \v_2\v_1 \rangle & \langle \v_2^2 \rangle & \cdots & \langle \v_2\v_t
\rangle \\
\vdots & \vdots & \ddots & \vdots \\ \langle \v_t\v_1 \rangle & \langle \v_t\v_2
\rangle & \cdots & \langle \v_t^2 \rangle \end{pmatrix},&
\quad \bPsi_t=
\begin{pmatrix} \langle \partial_1 \v_1 \rangle & 0 & \cdots & 0 \\
\langle \partial_1 \v_2 \rangle & \langle \partial_2 \v_2 \rangle
& \cdots & 0 \\
\vdots & \vdots & \ddots & \vdots \\
\langle \partial_1 \v_t \rangle & \langle \partial_2 \v_t \rangle
& \cdots & \langle \partial_t \v_t \rangle \end{pmatrix}.
\label{eq:GammaPsirect}
\end{align}
For each $j \geq 0$, define
\begin{align}
\bTheta_t^{(j)}&=\sum_{i=0}^j (\bPhi_t\bPsi_t)^i \bDelta_t
(\bPsi_t^\top \bPhi_t^\top)^{j-i}
+\sum_{i=0}^{j-1} (\bPhi_t\bPsi_t)^i\bPhi_t \bGamma_t\bPhi_t^\top
(\bPsi_t^\top \bPhi_t^\top)^{j-1-i},\label{eq:Thetarect}\\
\bXi_t^{(j)}&=\sum_{i=0}^j (\bPsi_t\bPhi_t)^i \bGamma_t
(\bPhi_t^\top \bPsi_t^\top)^{j-i}
+\sum_{i=0}^{j-1} (\bPsi_t\bPhi_t)^i\bPsi_t \bDelta_t\bPsi_t^\top
(\bPhi_t^\top \bPsi_t^\top)^{j-1-i}.\label{eq:Xirect}
\end{align}
The second summations of
(\ref{eq:Thetarect}) and (\ref{eq:Xirect}) are not present
for $j=0$. So for example,
\begin{align*}
\bTheta_t^{(0)}&=\bDelta_t\\
\bTheta_t^{(1)}&=\bPhi_t\bPsi_t\bDelta_t+\bPhi_t \bGamma_t \bPhi_t^\top+
\bDelta_t\bPsi_t^\top\bPhi_t^\top\\
\bTheta_t^{(2)}&=\bPhi_t\bPsi_t\bPhi_t\bPsi_t\bDelta_t+\bPhi_t\bPsi_t\bPhi_t
\bGamma_t \bPhi_t^\top+\bPhi_t\bPsi_t\bDelta_t\bPsi_t^\top\bPhi_t^\top\\
&\hspace{0.5in}+\bPhi_t\bGamma_t \bPhi_t^\top\bPsi_t^\top\bPhi_t^\top
+\bDelta_t\bPsi_t^\top\bPhi_t^\top\bPsi_t^\top\bPhi_t^\top\\
\bXi_t^{(0)}&=\bGamma_t\\
\bXi_t^{(1)}&=\bPsi_t\bPhi_t\bGamma_t+\bPsi_t\bDelta_t\bPsi_t^\top
+\bGamma_t\bPhi_t^\top\bPsi_t^\top\\
\bXi_t^{(2)}&=\bPsi_t\bPhi_t\bPsi_t\bPhi_t\bGamma_t+\bPsi_t\bPhi_t\bPsi_t\bDelta_t\bPsi_t^\top
+\bPsi_t\bPhi_t\bGamma_t\bPhi_t^\top\bPsi_t^\top\\
&\hspace{0.5in}+\bPsi_t\bDelta_t\bPsi_t^\top\bPhi_t^\top\bPsi_t^\top
+\bGamma_t\bPhi_t^\top\bPsi_t^\top\bPhi_t^\top\bPsi_t^\top.
\end{align*}

Let $\{\kappa_{2k}\}_{k \geq 1}$ be the rectangular free cumulants
with aspect ratio $\gamma=m/n$ corresponding to the sequence of even moments
\begin{equation}\label{eq:momentsrect}
m_{2k}=\frac{1}{m}\sum_{i=1}^{\min(m,n)} \lambda_i^{2k},
\end{equation}
as defined in Section \ref{sec:cumulantsrect}. Note that we always use the
normalization $1/m$, so these are the moments of $\blambda$ padded by $m-n$
additional 0's if $m>n$.

Define the $t \times t$ matrices
\begin{equation}\label{eq:ABrect}
\A_t=\left(\sum_{j=0}^\infty
\kappa_{2(j+1)}\bPsi_t(\bPhi_t\bPsi_t)^j\right)^\top, \qquad
\B_t=\left(\gamma \sum_{j=0}^\infty \kappa_{2(j+1)}
\bPhi_t(\bPsi_t\bPhi_t)^j\right)^\top,
\end{equation}
\begin{equation}\label{eq:SigmaOmegarect}
\bSigma_t=\sum_{j=0}^\infty \kappa_{2(j+1)}\bXi_t^{(j)}, \qquad
\bOmega_t=\gamma \sum_{j=0}^\infty \kappa_{2(j+1)}\bTheta_t^{(j)}.
\end{equation}
These are in fact finite series, as it may be verified that
\begin{align*}
\bPsi_t(\bPhi_t\bPsi_t)^j&=0 \text{ for } j \geq t+1\\
\bPhi_t(\bPsi_t\bPhi_t)^j&=0 \text{ for } j \geq t\\
\bXi_t^{(j)}&=0 \text{ for } j \geq 2t\\
\bTheta_t^{(j)}&=0 \text{ for } j \geq 2t-1.
\end{align*}
So for example,
\begin{align*}
&\A_1=\kappa_2\bPsi_1^\top, \qquad
\A_2=\kappa_2\bPsi_2^\top+\kappa_4(\bPsi_2\bPhi_2\bPsi_2)^\top, \quad \ldots\\
&\B_1=0, \qquad \B_2=\gamma \kappa_2 \bPhi_2^\top,
\qquad \B_3=\gamma \kappa_2\bPhi_3^\top+\gamma
\kappa_4(\bPhi_3\bPsi_3\bPhi_3)^\top, \quad \ldots\\
&\bSigma_1=\kappa_2 \bXi_1^{(0)}+\kappa_4\bXi_1^{(1)},
\qquad \bSigma_2=\kappa_2\bXi_2^{(0)}+\kappa_4\bXi_2^{(1)}+\kappa_6\bXi_2^{(2)}+\kappa_8\bXi_2^{(3)}, \quad \ldots\\
&\bOmega_1=\gamma \kappa_2 \bTheta_1^{(0)},
\qquad \bOmega_2=\gamma \kappa_2\bTheta_2^{(0)}+\gamma
\kappa_4\bTheta_2^{(1)}+\gamma\kappa_6\bTheta_2^{(2)}, \quad \ldots
\end{align*}

The matrices $\A_t$ and $\B_t$ are upper-triangular, with the forms
\[\A_t=\begin{pmatrix} a_{11} & a_{21} & \cdots & a_{t1} \\
& a_{22} & \cdots & a_{t2} \\
& & \ddots & \vdots \\
& & & a_{tt} \end{pmatrix}, 
\qquad
\B_t=\begin{pmatrix} 0 & b_{21} & b_{31} & \cdots & b_{t1} \\
& 0 & b_{32} & \cdots & b_{t2} \\
& & \ddots & \ddots & \vdots \\
& & & 0 & b_{t,t-1}\\
& & & & 0 \end{pmatrix}.\]
The debiasing coefficients $a_{t1},\ldots,a_{tt},b_{t1},\ldots,b_{t,t-1}$ in
(\ref{eq:AMPrectz}) and (\ref{eq:AMPrecty}) are defined as the last columns of
$\A_t$ and $\B_t$. Under the conditions to be imposed in Assumption
\ref{assump:rect}, these matrices all have deterministic $t \times t$ limits
\[(\bDelta_t^\infty,\bGamma_t^\infty,\bPhi_t^\infty,\bPsi_t^\infty,
\A_t^\infty,\B_t^\infty,\bSigma_t^\infty,\bOmega_t^\infty)
=\lim_{m,n \to \infty} (\bDelta_t,\bGamma_t,\bPhi_t,\bPsi_t,\A_t,\B_t,
\bSigma_t,\bOmega_t).\]
The matrices $\bSigma_t^\infty$ and $\bOmega_t^\infty$ are the covariances
in the state evolutions for $(\y_1,\ldots,\y_t)$ and $(\z_1,\ldots,\z_t)$.
As in the symmetric square setting, the debiasing coefficients in
(\ref{eq:AMPrectz}) and (\ref{eq:AMPrecty}) may be replaced by their limits
$a_{ts}^\infty$ and $b_{ts}^\infty$, or by any consistent estimates of these
limits.

We make the following observations about the above definitions:
\begin{enumerate}
\item The upper-left $(t-1) \times (t-1)$ submatrices of
$\A_t,\B_t,\bSigma_t,\bOmega_t$ coincide with the matrices
$\A_{t-1},\B_{t-1},\bSigma_{t-1},\bOmega_{t-1}$.
\item For each $t \geq 1$, $\A_t,\B_t,\bSigma_t,\bOmega_t$ depend respectively
only on the rectangular free cumulants of $\blambda$ up to the orders
$\kappa_{2t},\kappa_{2t-2},\kappa_{4t},\kappa_{4t-2}$.
\item The matrices $\A_t,\bSigma_t$ depend on
$\u_1,\ldots,\u_t,\v_1,\ldots,\v_t$ and their derivatives.
The matrices $\B_t,\bOmega_t$ depend on
$\u_1,\ldots,\u_t,\v_1,\ldots,\v_{t-1}$ and their derivatives, but they
do not depend on $\v_t$ or its derivatives. (Thus the debiasing coefficients
and state evolution for $\z_t$ in (\ref{eq:AMPrectz}) are well-defined before
defining $\v_t$ in (\ref{eq:AMPrectv}).)
\end{enumerate}
The first two statements are analogous to our observations in the
symmetric square setting. The third statement holds from the definitions of
$\B_t$ and $\bOmega_t$ in (\ref{eq:ABrect}--\ref{eq:SigmaOmegarect}),
because the last column of
$\bPhi_t$ is 0, so $\bPhi_t\bPsi_t$ does not depend on the last row of
$\bPsi_t$, and $\bPhi_t\bGamma_t\bPhi_t^\top$ does not depend on the last row or
column of $\bGamma_t$.

\begin{remark}\label{remark:wishart}
In the Gaussian setting where $\W$ has i.i.d.\ $\N(0,1/n)$ entries, the
limit spectral distribution of $\W\W^\top$ is the Marcenko-Pastur law, with
limiting rectangular free cumulants
\[\kappa_2^\infty=1, \qquad \kappa_{2j}^\infty=0 \quad \text{ for all }
j \geq 2.\]
This yields simply
\[\A_t^\infty=(\bPsi_t^\infty)^\top, \qquad \B_t^\infty=\gamma
(\bPhi_t^\infty)^\top, \qquad \bSigma_t^\infty=\bGamma_t^\infty,
\qquad \bOmega_t=\gamma \bDelta_t^\infty.\]
If we further specialize to an algorithm where $v_t$ depends only on $z_t$ and
$u_{t+1}$ depends only on $y_t$, then $\langle \partial_s \u_t \rangle=0$ for
all $s \neq t-1$ and $\langle \partial_s \z_t \rangle=0$ for all $s \neq t$.
This yields the Gaussian AMP algorithm
\begin{align*}
\z_t&=\W^\top \u_t-\gamma \langle \partial_{t-1}\u_t \rangle\v_{t-1}\\
\v_t&=v_t(\z_t,\F)\\
\y_t&=\W \v_t-\langle \partial_t \v_t \rangle \u_t\\
\u_{t+1}&=u_{t+1}(\y_t,\E)
\end{align*}
as studied in \cite[Section 3]{bayati2011dynamics}. Furthermore,
the state evolution is such that $\z_t$ has the
empirical limit $\N(0,\omega_{tt}^\infty)$ where
$\omega_{tt}^\infty=\lim_{m,n \to \infty} \gamma \cdot \langle \u_t^2 \rangle$,
and $\y_t$ has the empirical limit $\N(0,\sigma_{tt}^\infty)$ where
$\sigma_{tt}^\infty=\lim_{m,n \to \infty} \langle \v_t^2
\rangle$.

Note that outside of this Gaussian setting, in general we do not have
the identities $\bSigma_t=\bGamma_t$ and $\bOmega_t=\gamma \bDelta_t$ even when
$\W$ is normalized such that $\kappa_2=1$.
\end{remark}

\subsection{Main result}

We impose the following assumptions on the model
(\ref{eq:rectmodel}--\ref{eq:lambdarect}) and the AMP
iterates (\ref{eq:AMPrectz}--\ref{eq:AMPrectu}). Again, we do not
require here $v_t(\cdot)$ and $u_{t+1}(\cdot)$ to be Lipschitz.

\begin{assumption}\label{assump:rect}
\hspace{2em}
\begin{enumerate}[(a)]
\item $m,n \to \infty$ such that $m/n=\gamma \in (0,\infty)$ is a fixed
constant.
\item $\O \in \RR^{m \times m}$ and $\Q \in \RR^{n \times n}$ are independent
random and Haar-uniform orthogonal matrices.
\item $\blambda \in \RR^{\min(m,n)}$ is independent of $\O,\Q$ and satisfies
$\blambda \toW \Lambda$ almost surely as $m,n \to \infty$, for a random
variable $\Lambda$ having finite moments of all orders.
\item $\u_1 \in \RR^m$, $\E \in \RR^{m \times k}$, and $\F \in \RR^{n \times
\ell}$ are independent of $\O,\Q$
and satisfy $(\u_1,\E) \toW (U_1,E)$ and $\F \toW F$ almost surely
as $m,n \to \infty$, where
$(U_1,E) \equiv (U_1,E_1,\ldots,E_k)$ and $F \equiv (F_1,\ldots,F_\ell)$
are random vectors having finite moments of all orders.
\item Each function $v_t:\RR^{t+\ell} \to \RR$ and
$u_{t+1}:\RR^{t+k} \to \RR$ satisfies (\ref{eq:growth}) for some $C>0$ and $p
\geq 1$. Writing their arguments
as $(z,f)$ and $(y,e)$ where $z,y \in \RR^t$, $f \in \RR^\ell$, and $e \in
\RR^k$, $v_t$ is weakly differentiable in $z$ and continuous in $f$, and
$u_{t+1}$ is weakly differentiable in $y$ and continuous in $e$. For each
$s=1,\ldots,t$, $\partial_s v_t$ and $\partial_s u_{t+1}$ also
satisfy (\ref{eq:growth}) for some $C>0$ and $p \geq 1$, where
$\partial_s v_t(z,f)$ is continuous at Lebesgue-a.e.\ $z
\in \RR^t$ for every $f \in \RR^\ell$, and
$\partial_s u_{t+1}(y,e)$ is continuous at Lebesgue-a.e.\ $y \in \RR^t$ for
every $e \in \RR^k$.
\item $\Var[\Lambda]>0$ and $\EE[U_1^2]>0$. Letting $(Z_1,\ldots,Z_t) \sim
\N(0,\bOmega_t^\infty)$ be independent of $F$, there do not exist constants
$\alpha_1,\ldots,\alpha_t,\beta_1,\ldots,\beta_{t-1}$ for which
\[v_t(Z_1,\ldots,Z_t,F)=\sum_{s=1}^t \alpha_s Z_s+\sum_{s=1}^{t-1}
\beta_s v_s(Z_1,\ldots,Z_s,F)\]
with probability 1 over $(F,Z_1,\ldots,Z_t)$. Letting
$(Y_1,\ldots,Y_t) \sim \N(0,\bSigma_t^\infty)$ be independent of $(U_1,E)$,
there do not exist constants
$\alpha_1,\ldots,\alpha_t,\beta_1,\ldots,\beta_t$ for which
\[u_{t+1}(Y_1,\ldots,Y_t,E)=\sum_{s=1}^t \alpha_s Y_s+\beta_1 U_1
+\sum_{s=2}^t \beta_s u_s(Y_1,\ldots,Y_{s-1},E)\]
with probability 1 over $(U_1,E,Y_1,\ldots,Y_t)$.
\end{enumerate}
\end{assumption}

As in the symmetric square setting, we clarify that Theorem \ref{thm:rect} below
establishes the existence of $\bOmega_t^\infty$ when condition (f) holds for
$u_1,\ldots,u_t$ and $v_1,\ldots,v_{t-1}$, and this limit $\bOmega_t^\infty$
then defines condition (f) for $v_t$. Similarly, the theorem establishes
the existence of $\bSigma_t^\infty$ when condition (f) holds for
$u_1,\ldots,u_t$ and $v_1,\ldots,v_t$, and this limit $\bSigma_t^\infty$
then defines the condition for $u_{t+1}$. This condition (f) is a non-degeneracy
assumption that will hold as long as $u_{t+1}(\cdot)$ and $v_t(\cdot)$ depend
non-linearly on $y_t$ and $z_t$, respectively.

\begin{theorem}\label{thm:rect}
Under Assumption \ref{assump:rect}, for each fixed $t \geq 1$, almost surely as
$n \to \infty$: $\bSigma_t \to \bSigma_t^\infty$ and
$\bOmega_t \to \bOmega_t^\infty$ for some deterministic non-singular
matrices $\bSigma_t^\infty$ and $\bOmega_t^\infty$. Also,
\begin{align*}
(\u_1,\ldots,\u_{t+1},\y_1,\ldots,\y_t,\E) &\toW
(U_1,\ldots,U_{t+1},Y_1,\ldots,Y_t,E)\\
(\v_1,\ldots,\v_t,\z_1,\ldots,\z_t,\F) &\toW
(V_1,\ldots,V_t,Z_1,\ldots,Z_t,F)
\end{align*}
where $(Y_1,\ldots,Y_t) \sim \N(0,\bSigma_t^\infty)$ is independent of
$(U_1,E)$; $(Z_1,\ldots,Z_t) \sim \N(0,\bOmega_t^\infty)$ is independent of $F$;
$U_s=u_s(Z_1,\ldots,Z_{s-1},E)$ for each $s=2,\ldots,t+1$;
and $V_s=v_s(Z_1,\ldots,Z_s,F)$ for each $s=1,\ldots,t$.
\end{theorem}

The limits $\bSigma_t^\infty$ and $\bOmega_t^\infty$ are
given by replacing $\langle \u_s\u_{s'} \rangle$, $\langle \v_s\v_{s'} \rangle$,
$\langle \partial_{s'} \u_s \rangle$, $\langle \partial_{s'} \v_s \rangle$,
and $\kappa_{2k}$
in the definitions (\ref{eq:DeltaPhirect}--\ref{eq:GammaPsirect}) and
(\ref{eq:SigmaOmegarect}) with
$\EE[U_sU_{s'}]$, $\EE[V_sV_{s'}]$,
$\EE[\partial_{s'} u_s(Y_1,\ldots,Y_{s-1},E)]$,
$\EE[\partial_{s'} v_s(Z_1,\ldots,Z_s,F)]$, and $\kappa_{2k}^\infty$.

The proof of this result is provided in Appendix
\ref{appendix:rectproof}.
As in Corollary \ref{cor:degenerate}, we may remove the non-degeneracy
condition in Assumption \ref{assump:rect}(f) if $v_t$ and $u_{t+1}$ are
continuously-differentiable and Lipschitz.
This is stated in the following corollary. The proof
follows the same argument as that of Corollary \ref{cor:degenerate}, and we omit
this for brevity.

\begin{corollary}\label{cor:degeneraterect}
Suppose Assumption \ref{assump:rect}(a--d) holds, $\limsup_{n \to \infty}
\|\blambda\|_\infty<\infty$, each function $v_t:\RR^{t+\ell} \to \RR$ and
$u_{t+1}:\RR^{t+k} \to \RR$ is continuously-differentiable, and
\[|v_t(z,f)-v_t(z',f)| \leq C\|z-z'\|, \qquad
|u_{t+1}(y,e)-u_{t+1}(y',e)| \leq C\|y-y'\|\]
for a constant $C>0$ and all $z,z',y,y' \in \RR^t$, $e \in \RR^k$, and $f \in
\RR^\ell$. Then for each fixed $t \geq 1$, almost surely as $n \to \infty$:
$\bSigma_t \to \bSigma_t^\infty$ and $\bOmega_t \to \bOmega_t^\infty$ for some
deterministic (possibly singular) matrices $\bSigma_t^\infty$ and
$\bOmega_t^\infty$, and
\begin{align*}
(\u_1,\ldots,\u_{t+1},\y_1,\ldots,\y_t,\E) &\toWtwo
(U_1,\ldots,U_{t+1},Y_1,\ldots,Y_t,E)\\
(\v_1,\ldots,\v_t,\z_1,\ldots,\z_t,\F) &\toWtwo
(V_1,\ldots,V_t,Z_1,\ldots,Z_t,F)
\end{align*}
where these limits are as defined in Theorem \ref{thm:rect}.
\end{corollary}

\section{Proof ideas}\label{sec:proof}

We describe here the main ideas of the proofs. In the
setting of a symmetric square matrix $\W \in \RR^{n \times n}$, the basic
strategy is to write $\W=\O^\top \bLambda \O$, and to express the AMP iterations
(\ref{eq:AMPintro1}--\ref{eq:AMPintro2}) in an expanded form as
\begin{align}
\r_t&=\O\u_t\label{eq:AMPintroexp1}\\
\s_t&=\O^\top \bLambda \r_t\label{eq:AMPintroexp2}\\
\z_t&=\s_t-b_{t1}\u_1-\ldots-b_{tt}\u_t\label{eq:AMPintroexp3}\\
\u_{t+1}&=u_{t+1}(\z_1,\ldots,\z_t).\label{eq:AMPintroexp4}
\end{align}
All analyses are performed conditional on $\u_1$ and $\bLambda$, so that the
only randomness is in the Haar-orthogonal matrix $\O$. We apply Bolthausen's
conditioning technique \cite{bolthausen2014iterative}, analyzing sequentially
each iterate $\r_1,\s_1,\z_1,\u_2,\r_2,\ldots$ conditional on all preceding
iterates. This requires understanding the law of $\O$ conditional on events of
the form
\[\O\X=\Y,\]
which was shown in \cite{rangan2019vector,takeuchi2017rigorous} to be 
\begin{equation}\label{eq:Oconditioningintro}
\O|_{\O\X=\Y}
\overset{L}{=}\X(\X^\top \X)^{-1}\Y^\top+\proj_{\X^\perp} \tilde{\O}
\proj_{\Y^\perp}^\top.
\end{equation}
Here, $\proj_{\X^\perp}$ and $\proj_{\Y^\perp}$ are matrices with
orthonormal columns spanning the orthogonal complements of
the column spans of $\X$ and $\Y$, and $\tilde{\O}$ is an independent
Haar-orthogonal matrix. Applying
(\ref{eq:Oconditioningintro}) to the appearances of $\O$ in
(\ref{eq:AMPintroexp1}--\ref{eq:AMPintroexp2}), we will exhibit decompositions
\[\r_t=\r_\parallel+\r_\perp, \qquad \s_t=\s_\parallel+\s_\perp.\]
The vectors $\r_\perp$ and $\s_\perp$ arise from the second term of
(\ref{eq:Oconditioningintro}) and have empirical distributions that are
approximately Gaussian conditional on the preceding iterates. The vectors
$\r_\parallel$ and $\s_\parallel$ arise from the first term of
(\ref{eq:Oconditioningintro}), are deterministic conditional on the preceding
iterates, and represent biases respectively in the directions of
$(\r_1,\ldots,\r_{t-1},\bLambda\r_{t-1},\ldots,\bLambda\r_{t-1})$ and
$(\u_1,\ldots,\u_t,\z_1,\ldots,\z_{t-1})$. The Onsager correction
by $b_{t1}\u_1+\ldots+b_{tt}\u_t$ in (\ref{eq:AMPintroexp3}) is defined to
exactly cancel the component of this bias $\s_\parallel$ in $(\u_1,\ldots,\u_t)$, so that $(\z_1,\ldots,\z_t)$ has
an approximate joint Gaussian law. When the spectrum of $\W$ converges to
Wigner's semicircle law,
the forms of $\r_\parallel$ and $\s_\parallel$ and variances of
$\r_\perp$ and $\s_\perp$ are more straightforward to track across iterations,
and this produces a slightly different proof of the AMP analyses in
\cite{bayati2011dynamics,bolthausen2014iterative}.

When the spectrum of $\W$ does not converge to the
semicircle law, two difficulties arise in carrying out this conditional
analysis. First, the forms of $\r_\parallel,\r_\perp,\s_\parallel,\s_\perp$
in iteration $T$ will depend on
\[n^{-1} \u_s^\top\W^k\u_t \equiv n^{-1} \r_s^\top \bLambda^k\r_t
\qquad \text{ for } k=1,2 \text{ and } s,t \leq T.\]
These values will in turn depend on
\[n^{-1}\u_s^\top \W^k \u_t \equiv n^{-1}\r_s^\top \bLambda^k\r_t
\qquad \text{ for } k=1,\ldots,4 \text{ and } s,t \leq T-1,\]
which will in turn depend on
\[n^{-1}\u_s^\top \W^k \u_t \equiv n^{-1}\r_s^\top \bLambda^k\r_t
\qquad \text{ for } k=1,\ldots,6 \text{ and } s,t \leq T-2,\]
and so forth. The final dependence is
on $n^{-1}\u_1^\top \W^k \u_1$ for $k=1,\ldots,2T$, whose large-$n$
limits are given by
the first $2T$ moments of the limit spectral distribution of $\W$, because
the initialization $\u_1$ is independent of $\W$ which is rotationally invariant
in law. The free cumulants of $\W$ that appear in the final forms
of the Onsager correction and state evolution emerge by
tracking these dependences. To provide an inductive argument that can describe
these dependences for arbitrary iterations, our proof establishes a
precise form of
\[\lim_{n \to \infty} n^{-1} \u_s^\top \W^k \u_t\]
for every fixed moment $k \geq 0$ and all fixed iterates $s,t \geq 1$.
These forms depend on combinatorial coefficients that we call ``partial
moment coefficients'', defined by summing over certain subsets of the
non-crossing partition lattice, and which interpolate between the moments and
free cumulants of the spectral distribution of $\W$.
We define these coefficients in Appendix \ref{sec:partialmoments}.

A second technical difficulty which arises is that for the resulting
conditioning events $\O\X=\Y$, the form of the matrix $\X^\top \X$ in
(\ref{eq:Oconditioningintro}) becomes complicated, depending on series of
matrices with these partial moment coefficients, and $(\X^\top
\X)^{-1}$ does not admit a tractable description. Instead, we handle
matrix-vector products $(\X^\top \X)^{-1}\v$ arising in the computation by
``guessing'' the form $\w$ for this product, and then verifying that $(\X^\top
\X)\w=\v$. This type of verification is contained in Lemma \ref{lemma:upsilon},
and relies on combinatorial identities for these partial moment coefficients.

The proof ideas in the rectangular setting are similar: We write $\W=\O^\top
\bLambda \Q$ and express (\ref{eq:AMPintrorect1}--\ref{eq:AMPintrorect4}) in an
expanded form analogous to (\ref{eq:AMPintroexp1}--\ref{eq:AMPintroexp4}) above.
A key component of the proof is then to identify the large-$(m,n)$ limits of the
four quantities
\[m^{-1}\u_s^\top (\W\W^\top)^k\u_t,
\quad m^{-1}\v_s^\top \W^\top(\W\W^\top)^k\u_t,\]
\[n^{-1}\u_s^\top \W(\W^\top\W)^k\v_t,
\quad n^{-1}\v_s^\top (\W^\top\W)^k\v_t\]
for all fixed moments $k \geq 0$ and iterates $s,t \geq 1$. These will depend
on certain partial moment coefficients that interpolate between the moments and
rectangular free cumulants of the limit singular value distribution
of $\W$, and which are defined by summing over subsets of the lattice of
non-crossing partitions of sets with even cardinality. These coefficients are
defined in Appendix \ref{sec:partialmomentrect}, and the corresponding
identities involving $(\X^\top\X)^{-1}$ are contained in Lemma
\ref{lemma:upsilonrect}.

For the analyses of the Bayes-AMP algorithms for PCA in Section \ref{sec:PCA},
part (a) of Theorems \ref{thm:PCA} and \ref{thm:PCArect} are straightforward
consequences of the results for the general AMP algorithms. Part (b) of these
theorems require an analysis of the state evolutions for the single-iterate
posterior mean denoisers, which we carry out in Appendix
\ref{subsec:PCApartb}. This analysis applies a contractive
mapping argument to show that for sufficiently large signal strengths,
the matrices $\bDelta_t$, $\bSigma_t$, $\bGamma_t$, and $\bOmega_t$ all converge
as $t \to \infty$ in a space of ``infinite matrices'' equipped with a
weighted $\ell_\infty$ metric.

\appendix

\section{Proof for symmetric square matrices}\label{appendix:squareproof}

In this appendix, we prove Theorem \ref{thm:main}.
Recalling $\W=\O^\top \bLambda \O$ where $\bLambda=\diag(\blambda)$,
we may write the iterations (\ref{eq:AMPz}--\ref{eq:AMPu}) equivalently as
\begin{align}
\r_t&=\O \u_t\label{eq:AMPextr}\\
\s_t&=\O^\top \bLambda \r_t\label{eq:AMPexts}\\
\z_t&=\s_t-b_{t1}\u_1-\ldots-b_{tt}\u_t\label{eq:AMPextz}\\
\u_{t+1}&=u_{t+1}(\z_1,\ldots,\z_t,\E)\label{eq:AMPextu}
\end{align}

As discussed in Section \ref{sec:proof}, we will wish to
identify the almost-sure limits
\begin{equation}\label{eq:keylimits}
\lim_{n \to \infty} n^{-1}\r_s^\top \bLambda^k \r_{s'} \equiv 
\lim_{n \to \infty} n^{-1}\u_s^\top \W^k\u_{s'}
\end{equation}
for each fixed pair $s,s' \geq 1$ and fixed order $k \geq 0$.
In Section \ref{sec:partialmoments} below, we first
define certain ``partial moment'' coefficients $c_{k,j}$ corresponding to the
free cumulants $\{\kappa_k\}_{k \geq 1}$ of $\blambda$. We then define,
for each iteration $t \geq 1$ and each order $k \geq 0$, the $t \times t$ matrix
\begin{equation}\label{eq:Lk}
\L_t^{(k)}=\sum_{j=0}^\infty c_{k,j} \bTheta_t^{(j)}
\end{equation}
where $\bTheta_t^{(j)}$ is defined by (\ref{eq:Theta}). As in the definitions
(\ref{eq:BSigma}), this series is in fact finite because $\bTheta_t^{(j)}=0$
for $j \geq 2t-1$. The limits (\ref{eq:keylimits}) will be identified as the
entries of $\L_t^{(k,\infty)}=\lim_{n \to \infty} \L_t^{(k)}$.

\subsection{Coefficients for ``partial moments''}\label{sec:partialmoments}

Let $\{m_k\}_{k \geq 1}$ and $\{\kappa_k\}_{k \geq 1}$ be the moments and
free cumulants of $\blambda$, as defined in
Section \ref{sec:BSigma}. For notational convenience, we identify
\begin{equation}\label{eq:kappa0}
\kappa_0=1.
\end{equation}
We then define a doubly-indexed sequence of coefficients
$(c_{k,j})_{k,j \geq 0}$ by
\begin{equation}\label{eq:ckj}
c_{0,0}=1, \quad c_{0,j}=0 \text{ for } j \geq 1,
\quad c_{k,j}=\sum_{m=0}^{j+1} c_{k-1,m}\,\kappa_{j+1-m} \text{ for }
k \geq 1.
\end{equation}
These coefficients admit the following combinatorial interpretation: Let
\[\NC(k,\ell)=\Big\{\pi \in \NC(k):S \cap \{1,\ldots,\ell\} \neq S
\text{ for all } S \in \pi\Big\}.\]
This is the subset of non-crossing partitions $\pi \in \NC(k)$
where no set $S \in \pi$ is contained in $\{1,\ldots,\ell\}$.
For $\ell=0$, $\NC(k,0)=\NC(k)$ is the set of all non-crossing partitions.
The following lemma shows that
$c_{k,j}$ corresponds to the part of the sum (\ref{eq:momentcumulant})
that enumerates only over the partitions
belonging to the subset $\NC(k+j,j)$ of $\NC(k+j)$.

\begin{lemma}\label{lemma:ckj}
For each $k \geq 1$,
\begin{equation}\label{eq:ckjinterp}
c_{k,j}=\sum_{\pi \in \NC(k+j,j)} \prod_{S \in \pi} \kappa_{|S|}.
\end{equation}
In particular, $c_{1,j}=\kappa_{j+1}$ for each $j \geq 0$, and $c_{k,0}=m_k$
for each $k \geq 1$.
\end{lemma}
\begin{proof}
For $k=1$, the only non-zero term in the sum (\ref{eq:ckj})
corresponds to $m=0$. This gives $c_{1,j}=\kappa_{j+1}$.
The only partition of $\{1,\ldots,j+1\}$ where
no set belongs to $\{1,\ldots,j\}$ is the partition consisting of a single
set with all $j+1$ elements. Thus $\NC(j+1,j)$ consists of this single
partition, so the right side of (\ref{eq:ckjinterp}) is simply $\kappa_{j+1}$.
This verifies (\ref{eq:ckjinterp}) for $k=1$.

Suppose inductively that (\ref{eq:ckjinterp}) holds for $k-1$ (and all $j$).
Consider
\begin{equation}\label{eq:ckjdecomp}
c_{k,j}=c_{k-1,j+1}+\sum_{m=0}^j c_{k-1,m}\kappa_{j+1-m}.
\end{equation}
By this induction hypothesis, the first term is
\begin{equation}\label{eq:ck1j1}
c_{k-1,j+1}=\sum_{\pi \in \NC(k+j,j+1)} \prod_{S \in \pi} \kappa_{|S|}.
\end{equation}
To analyze the second term of (\ref{eq:ckjdecomp}), note that
if $\pi \in \NC(k+j,j)$ but $\pi \notin \NC(k+j,j+1)$, then there is some set
$S \in \pi$ containing $j+1$ and also belonging to $\{1,\ldots,j+1\}$.
This set $S \in \pi$ must consist of consecutive elements of
$\{1,\ldots,j+1\}$, because if there is a gap in the elements of $S$, then the
elements in this gap must form their own sets of $\pi$ as
$\pi$ is non-crossing, and this contradicts $\pi \in \NC(k+j,j)$. Thus
$S=\{m+1,\ldots,j+1\}$ for some $m \in \{0,\ldots,j\}$.
Removing $S$ from $\pi$ establishes a
bijection between such partitions $\pi$ and the
non-crossing partitions $\pi' \in \NC(k-1+m,m)$ of the $k-1+m$ remaining
elements, such that
no set of $\pi'$ is contained in $\{1,\ldots,m\}$. Summing over such partitions
$\pi'$ and applying the induction hypothesis, we have
\[c_{k-1,m}=\sum_{\pi' \in \NC(k-1+m,m)} \prod_{S' \in \pi'} \kappa_{|S'|}.\]
Then applying this bijection and including back $\{m+1,\ldots,j+1\}$
(of size $j+1-m$) into $\pi$,
\[c_{k-1,m}\kappa_{j+1-m}=\mathop{\sum_{\pi \in \NC(k+j,j) \setminus
\NC(k+j,j+1)}}_{\{m+1,\ldots,j+1\} \in \pi}\;\;\prod_{S \in \pi} \kappa_{|S|}.\]
Summing this over all possible values $m \in \{0,\ldots,j\}$ gives
\[\sum_{m=0}^j c_{k-1,m}\kappa_{j+1-m}=\sum_{\pi \in \NC(k+j,j) \setminus
\NC(k+j,j+1)}\;\;\prod_{S \in \pi} \kappa_{|S|},\]
and combining with (\ref{eq:ckjdecomp}) and
(\ref{eq:ck1j1}) yields (\ref{eq:ckjinterp}). This completes
the induction, establishing (\ref{eq:ckjinterp}) for all $k$.

Finally, the statement $c_{k,0}=m_k$ follows from specializing
(\ref{eq:ckjinterp}) to $j=0$, and applying
$\NC(k,0)=\NC(k)$ and the moment-cumulant
relations (\ref{eq:momentcumulant}).
\end{proof}

\subsection{Partial moment identities}

Recalling the definition of $\L_t^{(k)}$ in (\ref{eq:Lk}), 
we now establish several identities that are derived from the recursion for
$c_{k,j}$ in (\ref{eq:ckj}).

\begin{lemma}\label{lemma:Lidentities}
For every $t \geq 1$,
\begin{align}
\L_t^{(0)}&=\bDelta_t\label{eq:L0}\\
\L_t^{(1)}&=\bDelta_t\B_t+\bPhi_t\bSigma_t\nonumber\\
&=\B_t^\top \bDelta_t+\bSigma_t\bPhi_t^\top\label{eq:L1}\\
\L_t^{(2)}&=\B_t^\top \bDelta_t\B_t
+\B_t^\top\bPhi_t\bSigma_t+\bSigma_t\bPhi_t^\top \B_t+\bSigma_t\label{eq:L2}
\end{align}
\end{lemma}
\begin{proof}
For $k=0$, we have $c_{0,0}=1$ and $c_{0,j}=0$ for all $j \geq 1$. We also have
$\bTheta_t^{(0)}=\bDelta_t$. Hence (\ref{eq:L0}) follows from (\ref{eq:Lk}).

For $k=1$, we have $c_{1,j}=\kappa_{j+1}$ by Lemma \ref{lemma:ckj}. Then
\[\L_t^{(1)}=\sum_{j=0}^\infty \kappa_{j+1}\bTheta_t^{(j)}
=\sum_{j=0}^\infty \kappa_{j+1}\sum_{i=0}^j \bPhi_t^i \bDelta_t
(\bPhi_t^{j-i})^\top.\]
Note that all series throughout this proof are actually finite, so we may freely
exchange orders of summation.
Separating the terms that begin with $\bDelta_t$ from those that begin with
$\bPhi_t$,
\begin{align*}
\L_t^{(1)}&=\sum_{j=0}^\infty \kappa_{j+1}\bDelta_t (\bPhi_t^j)^\top
+\sum_{j=1}^\infty \kappa_{j+1}\sum_{i=1}^j \bPhi_t^i \bDelta_t
(\bPhi_t^{j-i})^\top\\
&=\bDelta_t\B_t+\bPhi_t\sum_{j=0}^\infty \kappa_{j+2}\sum_{i=0}^j \bPhi_t^i
\bDelta_t(\bPhi_t^{j-i})^\top=\bDelta_t\B_t+\bPhi_t\bSigma_t.
\end{align*}
Since $\L_t^{(1)}$, $\bDelta_t$, and $\bSigma_t$ are symmetric,
we must also have $\L_t^{(1)}=\B_t^\top
\bDelta_t+\bSigma_t\bPhi_t^\top$, and this yields both identities in
(\ref{eq:L1}).

For $k=2$, applying $c_{1,m}=\kappa_{m+1}$ and the recursion (\ref{eq:ckj}),
we have
\[\L_t^{(2)}=\sum_{j=0}^\infty c_{2,j}\bTheta_t^{(j)}
=\sum_{j=0}^\infty \left(\sum_{m=0}^{j+1} \kappa_{m+1}\kappa_{j+1-m}\right)
\cdot \left(\sum_{i=0}^j \bPhi_t^i \bDelta_t(\bPhi_t^{j-i})^\top\right).\]
Collecting terms by powers of $\bPhi_t$ and $\bPhi_t^\top$,
\begin{align*}
\L_t^{(2)}&=\sum_{i=0}^\infty \sum_{p=0}^\infty
\left(\sum_{m=0}^{i+p+1} \kappa_{m+1}\kappa_{i+p+1-m}\right)
\bPhi_t^i \bDelta_t (\bPhi_t^p)^\top.
\end{align*}
Substituting $q=i+p-m$, we may write
\[\sum_{m=0}^{i+p+1} \kappa_{m+1}\kappa_{i+p+1-m}
=\kappa_{i+1}\kappa_{p+1}+\kappa_{i+p+2}\kappa_0
+\sum_{m=0}^{i-1} \kappa_{m+1}\kappa_{i+p+1-m}
+\sum_{q=0}^{p-1} \kappa_{i+p+1-q}\kappa_{q+1}\]
where the last two sums may be empty if $i=0$ or $p=0$. Recalling the notation
$\kappa_0=1$ from (\ref{eq:kappa0}), and identifying
\begin{align*}
\B_t^\top \bDelta_t\B_t&=\left(\sum_{i=0}^\infty \kappa_{i+1}\bPhi_t^i\right)
\bDelta_t\left(\sum_{p=0}^\infty \kappa_{p+1} (\bPhi_t^p)^\top\right)
=\sum_{i=0}^\infty \sum_{p=0}^\infty
\kappa_{i+1}\kappa_{p+1} \bPhi_t^i \bDelta_t (\bPhi_t^p)^\top\\
\bSigma_t&=\sum_{j=0}^\infty \kappa_{j+2}\sum_{i=0}^j \bPhi_t^i \bDelta_t
(\bPhi_t^{j-i})^\top=\sum_{i=0}^\infty \sum_{p=0}^\infty
\kappa_{i+p+2}\kappa_0 \bPhi_t^i \bDelta_t (\bPhi_t^p)^\top\\
\B_t^\top \bPhi_t\bSigma_t&=
\left(\sum_{m=0}^\infty \kappa_{m+1}\bPhi_t^m\right)\bPhi_t
\left(\sum_{j=0}^\infty \kappa_{j+2}\sum_{p=0}^j
\bPhi_t^{j-p}\bDelta_t(\bPhi_t^p)^\top\right)\\
&=\sum_{i=1}^\infty \sum_{p=0}^\infty
\left(\sum_{m=0}^{i-1} \kappa_{m+1}\kappa_{i+p+1-m}\right)
\bPhi_t^i \bDelta_t (\bPhi_t^p)^\top\\
\bSigma_t\bPhi_t^\top \B_t&=\left(\sum_{j=0}^\infty \kappa_{j+2}
\sum_{i=0}^j \bPhi_t^i \bDelta_t (\bPhi_t^{j-i})^\top\right)
\bPhi_t^\top \left(\sum_{q=0}^\infty \kappa_{q+1} (\bPhi_t^q)^\top\right)\\
&=\sum_{i=0}^\infty \sum_{p=1}^\infty
\left(\sum_{q=0}^{p-1} \kappa_{i+p+1-q}\kappa_{q+1}\right)
\bPhi_t^i \bDelta_t (\bPhi_t^p)^\top,
\end{align*}
this yields (\ref{eq:L2}).
\end{proof}

\begin{lemma}\label{lemma:upsilon}
Define
\begin{equation}\label{eq:upsilondef}
\bUpsilon_t=\begin{pmatrix} \bDelta_t & \bDelta_t\B_t+\bPhi_t\bSigma_t \\
\bPhi_t^\top & \bPhi_t^\top \B_t+\Id \end{pmatrix}.
\end{equation}
For every $t \geq 1$ and $k \geq 0$,
\begin{align}
\begin{pmatrix} \L_t^{(k)} & \L_t^{(k+1)} \end{pmatrix}
&=\begin{pmatrix}
\sum_{j=0}^\infty c_{k,j} \bPhi_t^j &
\sum_{j=0}^\infty c_{k,j+1} \bTheta_t^{(j)}
\end{pmatrix}\bUpsilon_t\label{eq:upsiloninv}\\
\begin{pmatrix} \L_t^{(k)} & \L_t^{(k+1)} \\
\L_t^{(k+1)} & \L_t^{(k+2)} \end{pmatrix}
&=c_{k,0}\begin{pmatrix} \L_t^{(0)} & \L_t^{(1)} \\
\L_t^{(1)} & \L_t^{(2)} \end{pmatrix}
+\bUpsilon_t^\top \begin{pmatrix} 0 & \sum_{j=0}^\infty
c_{k,j+1}(\bPhi_t^j)^\top \\
\sum_{j=0}^\infty c_{k,j+1} \bPhi_t^j & \sum_{j=0}^\infty
c_{k,j+2}\bTheta_t^{(j)} \end{pmatrix}\bUpsilon_t \label{eq:upsiloninv2}
\end{align}
\end{lemma}

\begin{proof}
Applying (\ref{eq:L1}) and the definitions of $\L_t^{(1)}$ and $\B_t$,
and recalling the notation $\kappa_0=1$
from (\ref{eq:kappa0}) and $c_{1,j}=\kappa_{j+1}$ from Lemma
\ref{lemma:ckj},
\begin{equation}\label{eq:upsilonform}
\bUpsilon_t^\top=\begin{pmatrix} \bDelta_t & \bPhi_t \\
\L_t^{(1)} & \B_t^\top \bPhi_t+\Id
\end{pmatrix}
=\begin{pmatrix} \bDelta_t & \bPhi_t \\
\sum_{j=0}^\infty \kappa_{j+1}\bTheta_t^{(j)} &
\sum_{j=0}^\infty \kappa_j \bPhi_t^j \end{pmatrix}.
\end{equation}

For (\ref{eq:upsiloninv}), applying the definition of $\bTheta_t^{(j)}$,
we compute
\begin{align}
\begin{pmatrix} \bDelta_t & \bPhi_t \end{pmatrix}
\begin{pmatrix} \sum_{j=0}^\infty c_{k,j}(\bPhi_t^j)^\top \\
\sum_{j=0}^\infty c_{k,j+1}\bTheta_t^{(j)} \end{pmatrix}
&=\bDelta_t \cdot \sum_{j=0}^\infty c_{k,j} (\bPhi_t^j)^\top
+\bPhi_t \cdot \sum_{j=0}^\infty c_{k,j+1} \sum_{i=0}^j \bPhi_t^i \bDelta_t
(\bPhi_t^{j-i})^\top\nonumber\\
&=\sum_{j=0}^\infty c_{k,j} \sum_{i=0}^j \bPhi_t^i
\bDelta_t(\bPhi_t^{j-i})^\top\nonumber\\
&=\sum_{j=0}^\infty c_{k,j} \bTheta_t^{(j)}=\L_t^{(k)}.\label{eq:Lkprod}
\end{align}
We also compute
\begin{align}
&\begin{pmatrix} \sum_{j=0}^\infty \kappa_{j+1}\bTheta_t^{(j)}
& \sum_{j=0}^\infty \kappa_j \bPhi_t^j \end{pmatrix}
\begin{pmatrix} \sum_{j=0}^\infty c_{k,j}(\bPhi_t^j)^\top \nonumber\\
\sum_{j=0}^\infty c_{k,j+1}\bTheta_t^{(j)} \end{pmatrix}\\
&=\left(\sum_{j=0}^\infty \kappa_{j+1} \sum_{i=0}^j \bPhi_t^i \bDelta_t
(\bPhi_t^{j-i})^\top\right)
\cdot \left(\sum_{p=0}^\infty c_{k,p}(\bPhi_t^p)^\top\right)\nonumber\\
&\hspace{1in}+\left(\sum_{j=0}^\infty \kappa_j \bPhi_t^j\right)
\cdot \left(\sum_{p=0}^\infty c_{k,p+1}\sum_{q=0}^p \bPhi_t^{p-q} \bDelta_t
(\bPhi_t^q)^\top\right)\nonumber\\
&=\sum_{i=0}^\infty \sum_{j=i}^\infty \sum_{p=0}^\infty
\kappa_{j+1}c_{k,p} \bPhi_t^i \bDelta_t(\bPhi_t^{j-i+p})^\top
+\sum_{q=0}^\infty \sum_{p=q}^\infty\sum_{j=0}^\infty 
\kappa_j c_{k,p+1} \bPhi_t^{j+p-q}\bDelta_t(\bPhi_t^q)^\top\nonumber\\
&=\sum_{i=0}^\infty\sum_{r=0}^\infty
\left(\sum_{j=i}^{r+i} \kappa_{j+1}c_{k,r-j+i}\right)
\bPhi_t^i \bDelta_t(\bPhi_t^r)^\top
+\sum_{q=0}^\infty \sum_{\ell=0}^\infty \left(\sum_{p=q}^{\ell+q}
\kappa_{\ell-p+q}c_{k,p+1}\right) \bPhi_t^\ell\bDelta_t(\bPhi_t^q)^\top
\nonumber\\
&=\sum_{i=0}^\infty \sum_{r=0}^\infty
\Big((\kappa_{i+1}c_{k,r}+\kappa_{i+2}c_{k,r-1}+
+\ldots+\kappa_{i+r+1}c_{k,0})\nonumber\\
&\hspace{1in}+(\kappa_ic_{k,r+1}+\kappa_{i-1}c_{k,r+2}
+\ldots+\kappa_0c_{k,i+r+1})\Big)\bPhi_t^i\bDelta_t(\bPhi_t^r)^\top\nonumber\\
&=\sum_{i=0}^\infty \sum_{r=0}^\infty
\left(\sum_{j=0}^{i+r+1} \kappa_j c_{k,i+r+1-j}\right)
\bPhi_t^i\bDelta_t(\bPhi_t^r)^\top.\label{eq:Lkplus1prod}
\end{align}
From the recursion for $c_{k,j}$ in (\ref{eq:ckj}), this is equal to
\[\sum_{i=0}^\infty \sum_{r=0}^\infty
c_{k+1,i+r}\bPhi_t^i\bDelta_t(\bPhi_t^r)^\top=\L_t^{(k+1)}.\]
Combining this with (\ref{eq:Lkprod}) and (\ref{eq:upsilonform}) and taking the
transpose yields (\ref{eq:upsiloninv}).

For (\ref{eq:upsiloninv2}), applying (\ref{eq:upsilonform}),
first observe that
\begin{align*}
\bUpsilon_t^\top\begin{pmatrix} 0 \\
\sum_{j=0}^\infty c_{k,j+1}\bPhi_t^j \end{pmatrix}
&=\begin{pmatrix} \sum_{j=0}^\infty c_{k,j+1}\bPhi_t^{j+1} \\
\sum_{j=0}^\infty \kappa_j \bPhi_t^j \cdot \sum_{p=0}^\infty
c_{k,p+1}\bPhi_t^p
\end{pmatrix}\\
&=\begin{pmatrix} \sum_{\ell=1}^\infty c_{k,\ell}\bPhi_t^\ell \\
\sum_{\ell=0}^\infty \left(\sum_{j=0}^\ell \kappa_j c_{k,\ell-j+1}\right)
\bPhi_t^\ell \end{pmatrix}
=\begin{pmatrix} \sum_{\ell=0}^\infty (c_{k,\ell}-c_{k,0}c_{0,\ell})\bPhi_t^\ell \\
\sum_{\ell=0}^\infty (c_{k+1,\ell}-c_{k,0}c_{1,\ell})\bPhi_t^\ell \end{pmatrix}
\end{align*}
where the last equality applies $c_{0,0}=1$, $c_{0,\ell}=0$ for $\ell \geq 1$,
$c_{1,\ell}=\kappa_{\ell+1}$, and the recursion (\ref{eq:ckj}).
Next, applying the same computations as leading to
(\ref{eq:Lkprod}) and (\ref{eq:Lkplus1prod}), we obtain
\begin{align*}
\bUpsilon_t^\top \begin{pmatrix}
\sum_{j=0}^\infty c_{k,j+1}(\bPhi_t^j)^\top \\
\sum_{j=0}^\infty c_{k,j+2}\bTheta_t^{(j)} \end{pmatrix}
&=\begin{pmatrix}
\sum_{j=0}^\infty c_{k,j+1}\bTheta_t^{(j)}\\
\sum_{i=0}^\infty \sum_{r=0}^\infty \sum_{j=0}^{i+r+1}
\kappa_j c_{k,i+r+2-j}\bPhi_t^i\bDelta_t(\bPhi_t^r)^\top
\end{pmatrix}\\
&=\begin{pmatrix}
\sum_{\ell=0}^\infty (c_{k,\ell+1}-c_{k,0}c_{0,\ell+1})\bTheta_t^{(\ell)} \\
\sum_{\ell=0}^\infty (c_{k+1,\ell+1}-c_{k,0}c_{1,\ell+1})\bTheta_t^{(\ell)}
\end{pmatrix}
\end{align*}
where the second equality again applies $c_{0,\ell+1}=0$ for $\ell \geq 0$,
$c_{1,\ell+1}=\kappa_{\ell+2}$, and the recursion (\ref{eq:ckj}).
Combining these two identities, we get
\begin{align*}
&\bUpsilon_t^\top \begin{pmatrix}
0 & \sum_{j=0}^\infty c_{k,j+1}(\bPhi_t^j)^\top \\
\sum_{j=0}^\infty c_{k,j+1}\bPhi_t^j
& \sum_{j=0}^\infty c_{k,j+2}\bTheta_t^{(j)} \end{pmatrix}\\
&=\begin{pmatrix} \sum_{\ell=0}^\infty c_{k,\ell}\bPhi_t^\ell
& \sum_{\ell=0}^\infty c_{k,\ell+1} \bTheta_t^{(\ell)} \\
\sum_{\ell=0}^\infty c_{k+1,\ell}\bPhi_t^\ell
& \sum_{\ell=0}^\infty c_{k+1,\ell+1} \bTheta_t^{(\ell)} \end{pmatrix}
-c_{k,0}\begin{pmatrix} \sum_{\ell=0}^\infty c_{0,\ell}\bPhi_t^\ell
& \sum_{\ell=0}^\infty c_{0,\ell+1} \bTheta_t^{(\ell)} \\
\sum_{\ell=0}^\infty c_{1,\ell}\bPhi_t^\ell
& \sum_{\ell=0}^\infty c_{1,\ell+1} \bTheta_t^{(\ell)} \end{pmatrix}
\end{align*}
Then (\ref{eq:upsiloninv2}) follows from multiplying on the right
by $\bUpsilon_t$, and applying (\ref{eq:upsiloninv}) to the right side
with $k$ and also with $0,1,k+1$ in place of $k$.
\end{proof}

\subsection{Conditioning argument}

We now prove Theorem \ref{thm:main}, applying the conditioning argument
described in Section \ref{sec:proof}. Theorem \ref{thm:main} follows directly
from the following extended lemma, where part (b) identifies the
limits (\ref{eq:keylimits}) with the limit of $\L_t^{(k)}$.

\begin{lemma}\label{lemma:main}
Suppose Assumption \ref{assump:main} holds. Almost surely
for each $t=1,2,3,\ldots$:
\begin{enumerate}[(a)]
\item There exist deterministic matrices
$(\bDelta_t^\infty,\bPhi_t^\infty,\bTheta_t^{(j,\infty)},\B_t^\infty,\bSigma_t^\infty,\L_t^{(k,\infty)})$
for all fixed $j,k \geq 0$ such that
\[(\bDelta_t^\infty,\bPhi_t^\infty,\bTheta_t^{(j,\infty)},\B_t^\infty,\bSigma_t^\infty,\L_t^{(k,\infty)})=\lim_{n
\to \infty} (\bDelta_t,\bPhi_t,\bTheta_t^{(j)},\B_t,\bSigma_t,\L_t^{(k)}).\]
\item For some random variables $R_1,\ldots,R_t$ having finite moments of all
orders,
\[(\r_1,\ldots,\r_t,\blambda) \toW (R_1,\ldots,R_t,\Lambda).\]
Furthermore, for each $k \geq 0$,
\[\EE[(R_1,\ldots,R_t)^\top \Lambda^k(R_1,\ldots,R_t)]
\equiv \lim_{n \to \infty} n^{-1}
(\r_1,\ldots,\r_t)^\top \bLambda^k(\r_1,\ldots,\r_t)=\L_t^{(k,\infty)}.\]
\item We have
\[(\u_1,\ldots,\u_{t+1},\z_1,\ldots,\z_t,\E) \toW
(U_1,\ldots,U_{t+1},Z_1,\ldots,Z_t,E)\]
as described in Theorem \ref{thm:main}.
\item The matrix
\[\begin{pmatrix} \bDelta_t^\infty & \bPhi_t^\infty\bSigma_t^\infty \\
\bSigma_t^\infty(\bPhi_t^\infty)^\top & \bSigma_t^\infty \end{pmatrix}\]
is non-singular.
\end{enumerate}
\end{lemma}
\begin{proof}
Denote by $t^{(a)},t^{(b)},t^{(c)},t^{(d)}$ the claims of parts (a--d) up to and
including iteration $t$. We induct on $t$.
Note that since $\blambda \toW \Lambda$ by Assumption
\ref{assump:main}(b), the empirical
moments $m_k$ of $\blambda$ satisfy $m_k \to m_k^\infty \equiv
\EE[\Lambda^k]$ for each $k \geq 0$. Then also
\[\kappa_k \to \kappa_k^\infty, \qquad c_{k,j} \to c_{k,j}^\infty\]
for all $j,k \geq 0$, where $\kappa_k^\infty$ and $c_{k,j}^\infty$ are the free
cumulants and partial moment coefficients of $\Lambda$.\\

{\bf Step 1: $t=1$.} We have
$\bDelta_1=\langle \u_1^2 \rangle \to \bDelta_1^\infty \equiv
\EE[U_1^2]$ by Assumption \ref{assump:main}(c),
$\kappa_k \to \kappa_k^\infty$ and $c_{k,j} \to
c_{k,j}^\infty$ by the above, and $\bPhi_1=0$. Then $1^{(a)}$ follows from the
definitions.

Noting that $\r_1=\O\u_1$ and
applying Proposition \ref{prop:orthognormal} with $\proj=\Id$, we have
\[(\blambda,\r_1) \toW (\Lambda,R_1)\]
where $R_1 \sim \N(0,\EE[U_1^2])$ is independent of $\Lambda$. Then
for any $k \geq 0$,
\[n^{-1}\r_1^\top \bLambda^k \r_1
=n^{-1}\sum_{i=1}^n \lambda_i^k r_{i1}^2
\to \EE[\Lambda^kR_1^2]=m_k^\infty\EE[U_1^2].\]
Note that $m_k^\infty=c_{k,0}^\infty$ by Lemma
\ref{lemma:ckj}. Furthermore, $\bPhi_1=0$ so that
$\bTheta_1^{(0)}=\bDelta_1=\langle \u_1^2 \rangle$ and
$\bTheta_1^{(j)}=0$ for all $j \geq 1$. Hence
$\L_1^{(k,\infty)}=m_k^\infty\EE[U_1^2]$ for each $k \geq 0$.
This shows $1^{(b)}$.

For $1^{(c)}$, conditioning on $\u_1,\r_1,\blambda,\E$, 
the conditional law of $\O$ is that of $\O$ conditioned on the event
\[\r_1=\O\u_1.\]
Since $n^{-1}\|\r_1\|^2 \to \EE[R_1^2]=\EE[U_1^2]$, and this is non-zero by
Assumption \ref{assump:main}(e), we must have $\r_1 \neq 0$
for all large $n$.
Then by Proposition \ref{prop:orthogconditioning}, this conditional law of
$\O$ is equal to
\[\r_1(\r_1^\top\r_1)^{-1}\u_1^\top+\proj_{\r_1^\perp}\tilde{\O}\proj_{\u_1^\perp}^\top\]
where $\tilde{\O} \in \RR^{(n-1) \times (n-1)}$ is Haar-uniform and independent of
$(\u_1,\r_1,\blambda,\E)$, and $\Pi_{\r_1^\perp},\Pi_{\u_1^\perp} \in
\RR^{n \times (n-1)}$ have orthonormal columns spanning the orthogonal
complements of $\r_1,\u_1$. Thus, to analyze the joint behavior of
$(\u_1,\u_2,\z_1,\E)$, we may replace the update $\s_1=\O^\top \bLambda \r_1$
in this first iteration $t=1$ by the update
\begin{align*}
\s_1&=\s_\parallel+\s_\perp\\
\s_\parallel&=\u_1(\r_1^\top\r_1)^{-1} \r_1^\top \bLambda \r_1\\
\s_\perp&=\proj_{\u_1^\perp} \tilde{\O}^\top \proj_{\r_1^\perp}^\top \bLambda \r_1
\end{align*}
as this will not change the joint law of $(\u_1,\u_2,\z_1,\E)$.

For $\s_\parallel$, applying $1^{(b)}$,
we have $n^{-1}\r_1^\top \bLambda \r_1/(n^{-1}\r_1^\top \r_1)
\to m_1^\infty\EE[R_1^2]/\EE[R_1^2]=\kappa_1^\infty$. Then applying
$(\u_1,\E) \toW (U_1,E)$ and Proposition \ref{prop:scalarprod},
\[(\u_1,\E,\s_\parallel) \toW (U_1,E,S_\parallel), \qquad
S_\parallel=\kappa_1^\infty U_1.\]
For $\s_\perp$, applying $1^{(b)}$ again and identifying
$\kappa_2^\infty=m_2^\infty-(m_1^\infty)^2$, observe that
\[n^{-1}\|\proj_{\r_1^\perp}^\top \bLambda \r_1\|^2
=n^{-1}\|\bLambda \r_1\|^2-\frac{(n^{-1} \r_1^\top \bLambda
\r_1)^2}{n^{-1}\|\r_1\|^2} \to m_2^\infty\EE[U_1^2]
-\frac{(m_1^\infty \EE[U_1^2])^2}{\EE[U_1^2]}
=\kappa_2^\infty \EE[U_1^2].\]
Then applying Proposition \ref{prop:orthognormal},
\[\s_\perp \toW S_\perp \sim \N(0,\kappa_2^\infty \EE[U_1^2]),\]
where this limit $S_\perp$ is independent of
$(U_1,E)$. Observe that $\B_1=\kappa_1$, so
\[\z_1=\s_1-\kappa_1 \u_1=(\s_\parallel-\kappa_1 \u_1)+\s_\perp.\]
Applying $\kappa_1 \to \kappa_1^\infty$, $S_\parallel=\kappa_1^\infty U_1$,
and Propositions \ref{prop:composition} and \ref{prop:scalarprod}, we obtain
\[(\z_1,\u_1,\E) \toW (Z_1,U_1,E), \qquad Z_1=S_\perp.\]
Then also $(\z_1,\u_1,\u_2,\E) \toW (Z_1,U_1,U_2,E)$
where $U_2=u_2(Z_1,E)$, by Proposition \ref{prop:composition} and the polynomial
growth condition for $u_2(\cdot)$ in Assumption \ref{assump:main}(d).
Identifying $\bSigma_1^\infty=\kappa_2^\infty \EE[U_1^2]$ as the variance of
$S_\perp$, this shows $1^{(c)}$.

Finally, we have $\bPhi_1^\infty=0$, and $\kappa_2^\infty=\Var[\Lambda]>0$
and $\EE[U_1^2]>0$ by Assumption \ref{assump:main}(e). This implies $1^{(d)}$.\\

{\bf Step 2: Analysis of $\r_{t+1}$}. Suppose that
$t^{(a)},t^{(b)},t^{(c)},t^{(d)}$ all hold,
and consider iteration $t+1$. Note that $t^{(c)}$ implies $\langle \u_s\u_{s'}
\rangle \to \EE[U_sU_{s'}]$ for all $s,s' \leq t+1$,
so $\bDelta_{t+1} \to \bDelta_{t+1}^\infty$.
By Assumption \ref{assump:main}(d), for all $s'<s \leq t+1$, each
derivative $\partial_{s'} u_s$
satisfies the growth condition (\ref{eq:growth}) and is also continuous on a set
of probability 1 under $(Z_1,\ldots,Z_{s-1},E)$, since $\bSigma_{s-1}^\infty$ is
non-singular by $t^{(d)}$. Then by $t^{(c)}$ and Proposition
\ref{prop:discontinuous}, we also have $\langle \partial_{s'} \u_s \rangle
\to \EE[\partial_{s'} u_s(Z_1,\ldots,Z_{s-1},E)]$, so $\bPhi_{t+1} \to
\bPhi_{t+1}^\infty$. Combining with the
convergence $\kappa_k \to \kappa_k^\infty$ and $c_{k,j} \to c_{k,j}^\infty$ and
the definitions, this yields $t+1^{(a)}$.

Let us now show $t+1^{(b)}$ by analyzing the iterate $\r_{t+1}$.
We define the $n \times t$ matrices
\[\U_t=(\u_1,\ldots,\u_t), \qquad \R_t=(\r_1,\ldots,\r_t),
\qquad \Z_t=(\z_1,\ldots,\z_t).\]
Then the updates (\ref{eq:AMPextr}--\ref{eq:AMPextz}) up to iteration $t$ may be
written as
\[\R_t=\O\U_t, \qquad \Z_t=\O^\top \bLambda \R_t-\U_t\B_t,\]
or equivalently,
\[\R_t=\O\U_t, \qquad \O\Z_t=\bLambda \R_t-\R_t\B_t.\]
Thus, conditioning on $\U_t,\R_t,\Z_t,\u_{t+1},\blambda,\E$,
the law of $\O$ is conditioned on the event
\[\begin{pmatrix} \R_t & \bLambda \R_t \end{pmatrix}
\begin{pmatrix} \Id & -\B_t \\ 0 & \Id \end{pmatrix}
=\O\begin{pmatrix} \U_t & \Z_t \end{pmatrix}.\]

Let us introduce
\[\M_t=n^{-1}\begin{pmatrix} \U_t^\top \U_t & \U_t^\top \Z_t \\
\Z_t^\top \U_t & \Z_t^\top \Z_t \end{pmatrix}.\]
By $t^{(c)}$, we have $n^{-1}\U_t^\top \U_t \to \bDelta_t^\infty$
and $n^{-1}\Z_t^\top \Z_t \to \bSigma_t^\infty$.
Applying Proposition \ref{prop:stein} (derived from Stein's lemma) entrywise to
$n^{-1}\U_t^\top \Z_t$, and recalling the definition of $\bPhi_t$ in
(\ref{eq:DeltaPhi}), we also have
$n^{-1}\U_t^\top \Z_t \to \bPhi_t^\infty \bSigma_t^\infty$. So
\begin{equation}\label{eq:Mt}
\M_t \to \M_t^\infty
=\begin{pmatrix} \bDelta_t^\infty & \bPhi_t^\infty\bSigma_t^\infty \\
\bSigma_t^\infty(\bPhi_t^\infty)^\top & \bSigma_t^\infty \end{pmatrix}.
\end{equation}
This limit $\M_t^\infty$ is invertible by $t^{(d)}$. Then $\M_t$
must have full rank $2t$ for all large $n$, so $(\U_t,\Z_t)$ also has full
column rank $2t$ for all large $n$. Then by Proposition
\ref{prop:orthogconditioning}, the above conditional law of $\O$ is given by
\[\begin{pmatrix} \R_t & \bLambda \R_t \end{pmatrix}
\begin{pmatrix} \Id & -\B_t \\ 0 & \Id \end{pmatrix}
\M_t^{-1} \cdot n^{-1}\begin{pmatrix} \U_t^\top \\ \Z_t^\top \end{pmatrix}
+\proj_{(\R_t,\bLambda \R_t)^\perp}
\tilde{\O}\proj_{(\U_t,\Z_t)^\perp}^\top\]
where $\tilde{\O} \in \RR^{(n-2t) \times (n-2t)}$ is again an independent Haar-orthogonal matrix.
To analyze $(\r_1,\ldots,\r_{t+1},\blambda)$, we may then replace the update
$\r_{t+1}=\O\u_{t+1}$ by
\begin{align*}
\r_{t+1}&=\r_\parallel+\r_\perp\\
\r_\parallel&=\begin{pmatrix} \R_t & \bLambda \R_t \end{pmatrix}
\begin{pmatrix} \Id & -\B_t \\ 0 & \Id \end{pmatrix}
\M_t^{-1} \cdot n^{-1}
\begin{pmatrix} \U_t^\top \\ \Z_t^\top \end{pmatrix}\u_{t+1}\\
\r_\perp&=\proj_{(\R_t,\bLambda \R_t)^\perp}
\tilde{\O}\proj_{(\U_t,\Z_t)^\perp}^\top \u_{t+1},
\end{align*}
as this does not change the joint law of $(\r_1,\ldots,\r_{t+1},\blambda)$.

To analyze $\r_\parallel$, let us define
\[\bdelta_t^\infty=\begin{pmatrix} \EE[U_1U_{t+1}] \\
\vdots \\ \EE[U_tU_{t+1}] \end{pmatrix},
\qquad \bphi_t^\infty=\begin{pmatrix} \EE[\partial_1 u_{t+1}(Z_1,\ldots,Z_t,E)]
\\ \vdots \\ \EE[\partial_t u_{t+1}(Z_1,\ldots,Z_t,E)] \end{pmatrix}.\]
These are the last columns of $\bDelta_{t+1}^\infty$ and
$(\bPhi_{t+1}^\infty)^\top$ with their last entries removed.
Then, applying again $t^{(c)}$ and Proposition \ref{prop:stein},
\[n^{-1}\U_t^\top \u_{t+1} \to \bdelta_t^\infty, \qquad n^{-1} \Z_t^\top \u_{t+1}
\to \bSigma_t^\infty \bphi_t^\infty.\]
Noting that $\bSigma_t^\infty$ is invertible by $t^{(d)}$, this yields
\begin{align*}
&\begin{pmatrix} \Id & -\B_t \\ 0 & \Id \end{pmatrix}\M_t^{-1}
\cdot n^{-1}\begin{pmatrix} \U_t^\top \\ \Z_t^\top \end{pmatrix}\u_{t+1}\\
&\to \left(\begin{pmatrix} \Id & 0 \\ 0 & (\bSigma_t^\infty)^{-1} \end{pmatrix}
\M_t^\infty
\begin{pmatrix} \Id & \B_t^\infty \\ 0 & \Id \end{pmatrix} \right)^{-1}
\begin{pmatrix} \bdelta_t^\infty \\ \bphi_t^\infty \end{pmatrix}
=(\bUpsilon_t^\infty)^{-1}
\begin{pmatrix} \bdelta_t^\infty \\ \bphi_t^\infty \end{pmatrix}
\end{align*}
where $\bUpsilon_t^\infty$ is the limit of $\bUpsilon_t$
defined in (\ref{eq:upsilondef}). This shows also that
$\bUpsilon_t^\infty$ is invertible.
Then applying Proposition \ref{prop:scalarprod} and
$(\r_1,\ldots,\r_t,\blambda) \to (R_1,\ldots,R_t,\Lambda)$
by $t^{(b)}$, we have
\[\r_\parallel \toW R_\parallel
=\begin{pmatrix} R_1 & \cdots & R_t & \Lambda R_1 & \cdots &
\Lambda R_t \end{pmatrix} (\bUpsilon_t^\infty)^{-1}
\begin{pmatrix} \bdelta_t^\infty \\ \bphi_t^\infty \end{pmatrix}.\]

For $\r_\perp$, observe that
\begin{align}
&n^{-1}\|\proj_{(\U_t,\Z_t)^\perp}^\top \u_{t+1}\|^2\nonumber\\
&=n^{-1}\|\u_{t+1}\|^2-n^{-1}\u_{t+1}^\top \begin{pmatrix} \U_t & \Z_t
\end{pmatrix} \cdot \M_t^{-1} \cdot n^{-1}\begin{pmatrix} \U_t^\top \\ \Z_t^\top
\end{pmatrix} \u_{t+1}\nonumber\\
& \to \EE[U_{t+1}^2]
-\left(\begin{pmatrix} \bdelta_t \\ \bSigma_t\bphi_t
\end{pmatrix}^\top
\begin{pmatrix} \bDelta_t & \bPhi_t \bSigma_t \\
\bSigma_t\bPhi_t^\top & \bSigma_t \end{pmatrix}^{-1}
\begin{pmatrix} \bdelta_t \\ \bSigma_t\bphi_t \end{pmatrix}\right)^\infty\label{eq:schur}
\end{align}
where we use the condensed notation $(\cdots)^\infty$ to indicate that all
quantities in the parentheses are evaluated at their
$n \to \infty$ limits. Then by Proposition \ref{prop:orthognormal},
\[\r_\perp \toW R_\perp \sim \N\left(0,\;\EE[U_{t+1}^2]
-\left(\begin{pmatrix} \bdelta_t \\ \bSigma_t\bphi_t
\end{pmatrix}^\top
\begin{pmatrix} \bDelta_t & \bPhi_t \bSigma_t \\
\bSigma_t\bPhi_t^\top & \bSigma_t \end{pmatrix}^{-1}
\begin{pmatrix} \bdelta_t \\ \bSigma_t\bphi_t
\end{pmatrix}\right)^\infty\right),\]
where this limit $R_\perp$ is independent of $(R_1,\ldots,R_t,\Lambda)$.
Combining these, we have
$(\r_1,\ldots,\r_{t+1},\blambda) \toW (R_1,\ldots,R_{t+1},\Lambda)$ where
\begin{equation}\label{eq:RTplus1}
R_{t+1}=\begin{pmatrix} R_1 & \cdots & R_t & \Lambda R_1 & \cdots &\Lambda R_t
\end{pmatrix} (\bUpsilon_t^\infty)^{-1}
\begin{pmatrix} \bdelta_t^\infty \\ \bphi_t^\infty \end{pmatrix}+R_\perp.
\end{equation}

We will require later in the argument that $\Var[R_\perp]$ given by
(\ref{eq:schur}) is strictly positive. Let us verify this here:
Identifying the entries of
\[\bdelta_t^\infty, \qquad \bSigma_t^\infty\bphi_t^\infty, \qquad
\text{ and } \qquad \M_t^\infty\]
as the quantities $\EE[U_sU_{s'}]$, $\EE[Z_sU_{s'}]$, and $\EE[Z_sZ_{s'}]$ for
indices $1 \leq s,s' \leq t+1$, observe that this variance of $R_\perp$ given by
(\ref{eq:schur}) is the variance of the residual of the
projection of $U_{t+1}$ onto the
linear span of the random variables $(Z_1,\ldots,Z_t,U_1,\ldots,U_t)$ with
respect to the $L_2$-inner-product $(X,Y) \mapsto \EE[XY]$. Thus if
$\Var[R_\perp]=0$, then there would exist scalar constants
$\alpha_1,\ldots,\alpha_t,\beta_1,\ldots,\beta_t$ such that
\[U_{t+1}=\alpha_1Z_1+\ldots+\alpha_tZ_t+\beta_1U_1+\ldots+\beta_tU_t\]
almost surely, but this contradicts Assumption \ref{assump:main}(e). So
\begin{equation}\label{eq:Rperpnondegenerate}
\Var[R_\perp]>0.
\end{equation}

Let us now introduce a block notation for $\L_{t+1}^{(k,\infty)}$ (with blocks
of sizes $t$ and 1) given by
\begin{equation}\label{eq:blockL}
\L_{t+1}^{(k,\infty)}=\begin{pmatrix} \L_t^{(k,\infty)} & \l_t^{(k,\infty)} \\
(\l_t^{(k,\infty)})^\top & l_{t+1,t+1}^{(k,\infty)} \end{pmatrix}.
\end{equation}
To conclude the proof of $t+1^{(b)}$,
it remains to compute the two quantities
\[\EE\left[\begin{pmatrix} R_1 & \cdots & R_t \end{pmatrix}^\top \Lambda^k
R_{t+1} \right] \qquad \text{ and } \qquad \EE[\Lambda^kR_{t+1}^2]\]
and show that they are given by $\l_t^{(k,\infty)}$ and
$l_{t+1,t+1}^{(k,\infty)}$.

For the first quantity, observe that $\EE[R_s \Lambda^k R_{\perp}]=0$ for all $s
\leq t$, because $R_\perp$ has mean 0 and is independent of $(R_s,\Lambda)$.
Then applying (\ref{eq:RTplus1}) and $t^{(b)}$,
\[\EE\left[\begin{pmatrix} R_1 & \cdots & R_t \end{pmatrix}^\top \Lambda^k
R_{t+1} \right]
=\left(\begin{pmatrix} \L_t^{(k)} & \L_t^{(k+1)} \end{pmatrix}
\bUpsilon_t^{-1}
\begin{pmatrix} \bdelta_t \\ \bphi_t \end{pmatrix}\right)^\infty.\]
Applying the identity (\ref{eq:upsiloninv}), we get
\begin{align*}
\EE\left[\begin{pmatrix} R_1 & \cdots & R_t \end{pmatrix}^\top \Lambda^k
R_{t+1} \right]
&=\left(\sum_{j=0}^\infty c_{k,j}\bPhi_t^j \bdelta_t
+\sum_{j=0}^\infty c_{k,j+1}\bTheta_t^{(j)}\bphi_t\right)^\infty\\
&=\left(\sum_{j=0}^\infty c_{k,j} \bPhi_{t+1}^j\bDelta_{t+1}
+\sum_{j=0}^\infty c_{k,j+1}\bTheta_{t+1}^{(j)}\bPhi_{t+1}^\top
\right)_{1:t,\,t+1}^\infty.
\end{align*}
Here, we use the notation $(\cdot)_{1:t,\,t+1}$ to indicate the entries of rows
1 to $t$ of column $t+1$. This last equality holds by writing
$\bPhi_{t+1}^j\bDelta_{t+1}$ and $\bTheta_{t+1}^{(j)}\bPhi_{t+1}^\top$ in block
form, and noting that we have the blocks
\[\bPhi_{t+1}^j=\begin{pmatrix} \bPhi_t^j & 0 \\ * & * \end{pmatrix},
\qquad \bDelta_{t+1}=\begin{pmatrix} * & \bdelta \\ * & * \end{pmatrix},
\qquad \bTheta_{t+1}^{(j)}=\begin{pmatrix} \bTheta_t^{(j)} & * \\ * & *
\end{pmatrix}, \qquad \bPhi_{t+1}^\top
=\begin{pmatrix} * & \bphi_t \\ * & 0 \end{pmatrix}.\]
Finally, from the definitions
of $\bTheta_{t+1}^{(j)}$ and $\L_{t+1}^{(k)}$, the above is simply
\[\EE\left[\begin{pmatrix} R_1 & \cdots & R_t \end{pmatrix}^\top \Lambda^k
R_{t+1} \right]=(\L_{t+1}^{(k,\infty)})_{1:t,\,t+1}=\l_t^{(k,\infty)}.\]

For $\EE[\Lambda^k R_{t+1}^2]$, we again apply (\ref{eq:RTplus1}) and the
independence of $R_\perp$ and $(R_1,\ldots,R_t,\Lambda)$ to obtain
similarly
\begin{align}
\EE[\Lambda^k R_{t+1}^2]
&=\left(\begin{pmatrix} \bdelta_t \\ \bphi_t \end{pmatrix}^\top
(\bUpsilon_t^{-1})^\top
\begin{pmatrix} \L_t^{(k)} & \L_t^{(k+1)} \\
\L_t^{(k+1)} & \L_t^{(k+2)} \end{pmatrix}
\bUpsilon_t^{-1}\begin{pmatrix} \bdelta_t \\ \bphi_t \end{pmatrix}
\right)^\infty+\EE[\Lambda^kR_\perp^2].\label{eq:rTplus1rTplus1}
\end{align}
Applying independence of $\Lambda$ and $R_\perp$ and
taking the expected square on both sides of (\ref{eq:RTplus1}), we also have
\begin{align*}
\EE[\Lambda^k R_\perp^2]
&=m_k^\infty \EE[R_\perp^2]\\
&=m_k^\infty \left(\EE[R_{t+1}^2]-\EE\left[\left(
\begin{pmatrix} R_1 & \cdots & R_t & \Lambda R_1 & \cdots &\Lambda R_t
\end{pmatrix} (\bUpsilon_t^\infty)^{-1}
\begin{pmatrix} \bdelta_t^\infty \\ \bphi_t^\infty
\end{pmatrix}\right)^2\right]\right)\\
&=c_{k,0}^\infty\left(\EE[U_{t+1}^2]
-\left(\begin{pmatrix} \bdelta_t \\ \bphi_t \end{pmatrix}^\top
(\bUpsilon_t^{-1})^\top
\begin{pmatrix} \L_t^{(0)} & \L_t^{(1)} \\ \L_t^{(1)} &
\L_t^{(2)} \end{pmatrix}
\bUpsilon_t^{-1}\begin{pmatrix} \bdelta_t \\ \bphi_t \end{pmatrix}\right)^\infty
\right),
\end{align*}
the last line identifying $m_k^\infty=c_{k,0}^\infty$ by Lemma \ref{lemma:ckj}
and using
\[\EE[R_{t+1}^2]=\lim_{n \to \infty} n^{-1}\|\r_{t+1}\|^2=
\lim_{n \to \infty} n^{-1}\|\u_{t+1}\|^2=\EE[U_{t+1}^2].\]
Applying this to (\ref{eq:rTplus1rTplus1}),
and then applying the identity (\ref{eq:upsiloninv2}), we get
\begin{align*}
\EE[\Lambda^k R_{t+1}^2]
&=c_{k,0}^\infty \EE[U_{t+1}^2]
+\left(\begin{pmatrix}\bdelta_t \\ \bphi_t \end{pmatrix}^\top
\begin{pmatrix} 0 & \sum_{j=0}^\infty c_{k,j+1} (\bPhi_t^j)^\top\\
\sum_{j=0}^\infty c_{k,j+1} \bPhi_t^j & \sum_{j=0}^\infty
c_{k,j+2} \bTheta_t^{(j)} \end{pmatrix}
\begin{pmatrix}\bdelta_t \\ \bphi_t \end{pmatrix}\right)^\infty\\
&=c_{k,0}^\infty \EE[U_{t+1}^2]
+\left(\sum_{j=0}^\infty c_{k,j+1} \bdelta_t^\top (\bPhi_t^j)^\top \bphi_t
+c_{k,j+1}\bphi_t^\top \bPhi_t^j \bdelta_t+
c_{k,j+2} \bphi_t^\top \bTheta_t^{(j)}\bphi_t\right)^\infty\\
&=\left(c_{k,0} \bDelta_{t+1}
+\sum_{j=0}^\infty c_{k,j+1} \bDelta_{t+1} (\bPhi_{t+1}^{j+1})^\top
+c_{k,j+1}\bPhi_{t+1}^{j+1}\bDelta_{t+1}+c_{k,j+2}\bPhi_{t+1}
\bTheta_{t+1}^{(j)}\bPhi_{t+1}^\top\right)^\infty_{t+1,t+1}.
\end{align*}
Here, we use $(\cdot)_{t+1,t+1}$ to denote the lower-right entry, and this last
equality follows again from writing the matrix products in block form and
observing that
\[\bDelta_{t+1}=\begin{pmatrix} * & \bdelta_t
\\ \bdelta_t^\top & \EE[U_{t+1}^2] \end{pmatrix},
\quad (\bPhi_{t+1}^{j+1})^\top=\begin{pmatrix} * & (\bPhi_t^j)^\top\bphi_t \\
* & 0 \end{pmatrix},
\quad \bPhi_{t+1}^{j+1}=\begin{pmatrix} * & * \\ \bphi_t^\top \bPhi_t^j & 0
\end{pmatrix}, 
\quad \bTheta_{t+1}^{(j)}=\begin{pmatrix} \bTheta_t^{(j)} & * \\ * & *
\end{pmatrix}.\]
Applying the definitions of $\bTheta_{t+1}^{(j)}$ and $\L_{t+1}^{(k)}$,
this is just
\[\EE[\Lambda^k R_{t+1}^2]=\left(\sum_{j=0}^\infty c_{k,j}\bTheta_{t+1}^{(j)}
\right)_{t+1,t+1}^\infty=l_{t+1,t+1}^{(k,\infty)}.\]
This concludes the proof of $t+1^{(b)}$.\\

{\bf Step 3: Analysis of $\z_{t+1}$.}
Assuming $t+1^{(a)},t+1^{(b)},t^{(c)},t^{(d)}$,
we now show $t+1^{(c)}$ and $t+1^{(d)}$. Define the $(t+1) \times t$ matrices
\begin{equation}\label{eq:tildeB}
\tilde{\bPhi}_t=\begin{pmatrix} \bPhi_t \\ \bphi_t^\top \end{pmatrix},
\qquad \tilde{\B}_t=\begin{pmatrix} \B_t \\ \0 \end{pmatrix}.
\end{equation}
These are the first $t$ columns of $\bPhi_{t+1}$ and
$\B_{t+1}$, and we have $\R_t\B_t=\R_{t+1}\tilde{\B}_t$.
Conditional on $\U_t,\R_t,\Z_t,\u_{t+1},\r_{t+1},\blambda,\E$, the law of $\O$
is conditioned on the event
\begin{equation}\label{eq:sconditioning}
\begin{pmatrix} \R_{t+1} & \bLambda \R_t \end{pmatrix}
\begin{pmatrix} \Id & -\tilde{\B}_t \\ 0 & \Id \end{pmatrix}
=\O\begin{pmatrix} \U_{t+1} & \Z_t \end{pmatrix}.
\end{equation}
Let $\tilde{\bPhi}_t^\infty$ and $\tilde{\B}_t^\infty$ be the $n \to \infty$
limits of $\tilde{\bPhi}_t$ and $\tilde{\B}_t$, and let us introduce
\[\tilde{\M}_t=n^{-1}\begin{pmatrix} \U_{t+1}^\top \U_{t+1} & \U_{t+1}^\top \Z_t
\\ \Z_t^\top \U_{t+1} & \Z_t^\top \Z_t \end{pmatrix}.\]
Then by $t^{(c)}$ and Proposition \ref{prop:stein},
\begin{equation}\label{eq:tildeM}
\tilde{\M}_t \to \tilde{\M}_t^\infty
=\begin{pmatrix} \bDelta_{t+1}^\infty & \tilde{\bPhi}_t^\infty
\bSigma_t^\infty \\ \bSigma_t^\infty (\tilde{\bPhi}_t^\infty)^\top &
\bSigma_t^\infty \end{pmatrix}.
\end{equation}

To check that this limit $\tilde{\M}_t^\infty$
is invertible, observe that its $2t \times 2t$
submatrix removing row and column $t+1$ is just $\M_t^\infty$ from
(\ref{eq:Mt}), which is invertible by $t^{(d)}$. The
Schur-complement of the $(t+1,t+1)$ entry is exactly (\ref{eq:schur}), which we
have shown is positive in (\ref{eq:Rperpnondegenerate}). Thus
$\tilde{\M}_t^\infty$ is invertible,
so $(\U_{t+1},\Z_t)$ has full column rank $2t+1$
for all large $n$. Then by Proposition
\ref{prop:orthogconditioning}, the conditional law of $\O$ is
\[\begin{pmatrix} \R_{t+1} & \bLambda \R_t \end{pmatrix}
\begin{pmatrix} \Id & -\tilde{\B}_t \\ \0 & \Id \end{pmatrix}
\tilde{\M}_t^{-1} \cdot n^{-1}\begin{pmatrix} \U_{t+1} & \Z_t \end{pmatrix}^\top
+\proj_{(\R_{t+1},\bLambda \R_t)^\perp}\tilde{\O}
\proj_{(\U_{t+1},\Z_t)^\perp}^\top\]
where $\tilde{\O}$ is an independent Haar-orthogonal matrix. Thus, to analyze
the joint behavior of $(\u_1,\ldots,\u_{t+2},\z_1,\ldots,\z_{t+1},\E)$,
we may replace the update $\s_{t+1}=\O^\top \bLambda \r_{t+1}$ by
\begin{align*}
\s_{t+1}&=\s_\parallel+\s_\perp\\
\s_\parallel&=\begin{pmatrix} \U_{t+1} & \Z_t \end{pmatrix} \tilde{\M}_t^{-1}
\begin{pmatrix} \Id & 0 \\ -\tilde{\B}_t^\top & \Id \end{pmatrix}
\cdot n^{-1}
\begin{pmatrix} \R_{t+1}^\top \\ \R_t^\top \bLambda \end{pmatrix}
\bLambda \r_{t+1}\\
\s_\perp&=\proj_{(\U_{t+1},\Z_t)^\perp} \tilde{\O}^\top \proj_{(\R_{t+1},
\bLambda \R_t)^\perp}^\top \bLambda \r_{t+1}.
\end{align*}

To analyze $\s_\parallel$, recall the notation $\l_t^{(k,\infty)}$
from (\ref{eq:blockL}) and set
\[\tilde{\L}_t^{(k,\infty)}=\begin{pmatrix} \L_t^{(k,\infty)} \\
(\l_t^{(k,\infty)})^\top \end{pmatrix}, \qquad
\tilde{\l}_t^{(k,\infty)}=\begin{pmatrix} \l_t^{(k,\infty)} \\
l_{t+1,t+1}^{(k,\infty)}
\end{pmatrix}.\]
Then applying $t+1^{(b)}$,
\[n^{-1}\begin{pmatrix} \R_{t+1}^\top \bLambda \r_{t+1}\\
\R_t^\top \bLambda^2 \r_{t+1} \end{pmatrix}
\to \begin{pmatrix} \tilde{\l}_t^{(1,\infty)} \\ \l_t^{(2,\infty)} \end{pmatrix}.\]
Introducing the matrices
\[\tilde{\bUpsilon}_t=\begin{pmatrix} \bDelta_{t+1} &
\bDelta_{t+1}\tilde{\B}_t+\tilde{\bPhi}_t\bSigma_t \\
\tilde{\bPhi}_t^\top & \tilde{\bPhi}_t^\top \tilde{\B}_t+\Id_{t \times t}
\end{pmatrix}, \qquad \tilde{\bUpsilon}_t^\infty=\lim_{n \to \infty}
\tilde{\bUpsilon}_t,\]
we have
\begin{align}
\tilde{\M}_t^{-1}
\begin{pmatrix} \Id & 0 \\ -\tilde{\B}_t^\top & \Id \end{pmatrix}
\cdot n^{-1}
\begin{pmatrix} \R_{t+1}^\top \\ \R_t^\top \bLambda \end{pmatrix}
\bLambda \r_{t+1}
&\to
\left(\begin{pmatrix} \Id & 0 \\ (\tilde{\B}_t^\infty)^\top & \Id \end{pmatrix}
\tilde{\M}_t\right)^{-1}
\begin{pmatrix} \tilde{\l}_t^{(1,\infty)} \\ \l_t^{(2,\infty)}
\end{pmatrix}\nonumber\\
&=\begin{pmatrix} \Id & 0 \\ 0 & (\bSigma_t^\infty)^{-1}\end{pmatrix}
((\tilde{\bUpsilon}_t^\infty)^{-1})^\top
\begin{pmatrix} \tilde{\l}_t^{(1,\infty)} \\ \l_t^{(2,\infty)} \end{pmatrix}.
\label{eq:tildeupsiloninv}
\end{align}
This also shows that $\tilde{\bUpsilon}_t^\infty$ is invertible.

Let us introduce the block notations
\[\B_{t+1}=\begin{pmatrix} \B_t & \b_t \\ 0 &
b_{t+1,t+1} \end{pmatrix},
\qquad \tilde{\b}_t=\begin{pmatrix} \b_t \\ b_{t+1,t+1}
\end{pmatrix},
\qquad \bSigma_{t+1}=\begin{pmatrix} \bSigma_t & \bsigma_t \\
\bsigma_t^\top & \sigma_{t+1,t+1} \end{pmatrix}\]
and denote with $^\infty$ their $n \to \infty$ limits.
Defining $\bUpsilon_{t+1}$ by (\ref{eq:upsilondef}) and writing this in block
form, it may be checked that
\[\bUpsilon_{t+1}=\begin{pmatrix} \tilde{\bUpsilon}_t & * \\
0 & * \end{pmatrix}\]
where $\tilde{\bUpsilon}_t$ constitutes the first $2t+1$ rows and columns.
Applying the identity (\ref{eq:upsiloninv}) with $t+1$ and $k=1$ yields
\[\begin{pmatrix} \L_{t+1}^{(1)} \\ \L_{t+1}^{(2)} \end{pmatrix}
=\bUpsilon_{t+1}^\top\begin{pmatrix}
\sum_{j=0}^\infty c_{1,j}(\bPhi_{t+1}^j)^\top \\
\sum_{j=0}^\infty c_{1,j+1}\bTheta_{t+1}^{(j)} \end{pmatrix}
=\bUpsilon_{t+1}^\top \begin{pmatrix} \B_{t+1} \\ \bSigma_{t+1} \end{pmatrix},\]
the second equality identifying $c_{1,j}=\kappa_{j+1}$ and applying
the definitions of $\B_{t+1}$ and $\bSigma_{t+1}$ in (\ref{eq:BSigma}).
Then equating the first $2t+1$ entries of the last column on both sides, and
taking the limit $n \to \infty$, we get
\[\begin{pmatrix} \tilde{\l}_t^{(1,\infty)} \\ \l_t^{(2,\infty)} \end{pmatrix}
=(\tilde{\bUpsilon}_t^\infty)^\top\begin{pmatrix}
\tilde{\b}_t^\infty \\ \bsigma_t^\infty \end{pmatrix}.\]
Inverting $(\tilde{\bUpsilon}_t^\infty)^\top$ and applying this to
(\ref{eq:tildeupsiloninv}),
\[\s_\parallel \toW \begin{pmatrix} U_1 & \cdots & U_{t+1} \end{pmatrix}
\tilde{\b}_t^\infty+\begin{pmatrix} Z_1 & \cdots & Z_t \end{pmatrix}
(\bSigma_t^\infty)^{-1}\bsigma_t^\infty.\]

For $\s_\perp$, note that we have shown $\tilde{\M}_t^\infty$ in
(\ref{eq:tildeM}) is invertible.
Applying the definition of $\tilde{\M}_t$ and the
identity (\ref{eq:sconditioning}), we also have
\begin{align*}
\tilde{\M}_t^\infty&=\lim_{n \to \infty} n^{-1}
\begin{pmatrix} \Id & 0 \\
-\tilde{\B}_t^\top & 0 \end{pmatrix} 
\begin{pmatrix} \R_{t+1}^\top \R_{t+1} & \R_{t+1}^\top \bLambda
\R_t \\ \R_t^\top \bLambda \R_{t+1} & \R_t^\top \bLambda^2 \R_t
\end{pmatrix} \begin{pmatrix} \Id & -\tilde{\B}_t \\ 0 & \Id
\end{pmatrix}\\
&=\begin{pmatrix} \Id & 0 \\ -(\tilde{\B}_t^\infty)^\top & 0 \end{pmatrix}
\begin{pmatrix} \L_{t+1}^{(0,\infty)} & \tilde{\L}_t^{(1,\infty)} \\
(\tilde{\L}_t^{(1,\infty)})^\top & \L_t^{(2,\infty)} \end{pmatrix}
\begin{pmatrix} \Id & -\tilde{\B}_t^\infty \\ 0 & \Id \end{pmatrix}.
\end{align*}
Thus the matrices
\begin{equation}\label{eq:partialL}
n^{-1}\begin{pmatrix} \R_{t+1}^\top \R_{t+1} & \R_{t+1}^\top \bLambda
\R_t \\ \R_t^\top \bLambda \R_{t+1} & \R_t^\top \bLambda^2 \R_t
\end{pmatrix}, \qquad
\begin{pmatrix} \L_{t+1}^{(0,\infty)} & \tilde{\L}_t^{(1,\infty)} \\
(\tilde{\L}_t^{(1,\infty)})^\top & \L_t^{(2,\infty)} \end{pmatrix}
\end{equation}
are also invertible (the former almost surely
for all large $n$). Observe then that
\begin{align*}
&n^{-1}\|\proj_{(\R_{t+1},\bLambda\R_t)^\perp}^\top \bLambda \r_{t+1}\|^2\\
&\qquad=n^{-1}\|\bLambda \r_{t+1}\|^2-n^{-1}
\begin{pmatrix} \R_{t+1}^\top\bLambda \r_{t+1} \\ \R_t^\top \bLambda^2 \r_{t+1}
\end{pmatrix}^\top
\begin{pmatrix} \R_{t+1}^\top \R_{t+1} & \R_{t+1}^\top \bLambda
\R_t \\ \R_t^\top \bLambda \R_{t+1} & \R_t^\top \bLambda^2 \R_t
\end{pmatrix}^{-1} \begin{pmatrix} \R_{t+1}^\top\bLambda \r_{t+1} \\
\R_t^\top \bLambda^2 \r_{t+1}\end{pmatrix}\\
&\qquad\to \left(l_{t+1,t+1}^{(2)}-
\begin{pmatrix} \tilde{\l}_t^{(1)} \\ \l_t^{(2)} \end{pmatrix}^\top
\begin{pmatrix} \L_{t+1}^{(0)} & \tilde{\L}_t^{(1)} \\
(\tilde{\L}_t^{(1)})^\top & \L_t^{(2)} \end{pmatrix}^{-1}
\begin{pmatrix} \tilde{\l}_t^{(1)} \\ \l_t^{(2)} \end{pmatrix}\right)^\infty
\end{align*}
Then by Proposition \ref{prop:orthognormal},
\begin{equation}\label{eq:sperp}
\s_\perp \toW S_\perp \sim \N\left(0,\;
\left(l_{t+1,t+1}^{(2)}-
\begin{pmatrix} \tilde{\l}_t^{(1)} \\ \l_t^{(2)} \end{pmatrix}^\top
\begin{pmatrix} \L_{t+1}^{(0)} & \tilde{\L}_t^{(1)} \\
(\tilde{\L}_t^{(1)})^\top & \L_t^{(2)} \end{pmatrix}^{-1}
\begin{pmatrix} \tilde{\l}_t^{(1)} \\ \l_t^{(2)} \end{pmatrix}\right)^\infty
\right)
\end{equation}
where this limit $S_\perp$ is independent of
$(U_1,\ldots,U_{t+1},Z_1,\ldots,Z_t,E)$.
Combining the above, we obtain
\begin{equation}\label{eq:Stplus1}
\s_{t+1} \toW S_{t+1}=
\begin{pmatrix} U_1 & \cdots & U_{t+1} \end{pmatrix}
\tilde{\b}_t^\infty+\begin{pmatrix} Z_1 & \cdots & Z_t \end{pmatrix}
(\bSigma_t^\infty)^{-1}\bsigma_t^\infty+S_\perp.
\end{equation}
Then, since
\[\z_{t+1}=\s_{t+1}-(\u_1,\ldots,\u_{t+1})\begin{pmatrix} \tilde{\b}_t
\end{pmatrix},\]
applying Propositions \ref{prop:composition} and \ref{prop:scalarprod},
this shows
\[(\u_1,\ldots,\u_{t+2},\z_1,\ldots,\z_{t+1},\E)
\toW (U_1,\ldots,U_{t+2},Z_1,\ldots,Z_{t+1},E)\]
where $U_{t+2}=u_{t+2}(Z_1,\ldots,Z_{t+1},E)$ and
\[Z_{t+1}=\begin{pmatrix} Z_1 & \cdots & Z_t \end{pmatrix}
(\bSigma_t^\infty)^{-1}\bsigma_t^\infty+S_\perp.\]
In particular, $(Z_1,\ldots,Z_{t+1})$ has a multivariate normal limit 
independent of $(U_1,E)$.

To conclude the proof of $t+1^{(c)}$, it remains to compute
\[\EE[\begin{pmatrix}
Z_1 & \cdots & Z_t \end{pmatrix}^\top Z_{t+1}], \qquad \EE[Z_{t+1}^2]\]
and show that these are given by $\bsigma_t^\infty$ and $\sigma_{t+1,t+1}^\infty$. Observe
that $\EE[Z_s S_\perp]=0$ for all $s \leq t$, since $S_\perp$ has mean 0 and is
independent of $Z_s$. Then
\[\EE[\begin{pmatrix} Z_1 & \cdots & Z_t \end{pmatrix}^\top Z_{t+1}]
=\bSigma_t^\infty (\bSigma_t^\infty)^{-1}\bsigma_t^\infty=\bsigma_t^\infty.\]

To compute $\EE[Z_{t+1}^2]$, note that (\ref{eq:Stplus1}) may be written as
\[S_{t+1}=\begin{pmatrix} U_1 & \cdots & U_{t+1} \end{pmatrix}
\tilde{\b}_t^\infty+Z_{t+1}.\]
Taking the expected square on both sides,
\[\EE[S_{t+1}^2]=\EE\Big[\Big(\begin{pmatrix} U_1 & \cdots & U_{t+1} \end{pmatrix}
\tilde{\b}_t^\infty\Big)^2\Big]+2\EE\Big[\Big(\begin{pmatrix} U_1 & \cdots & U_{t+1} \end{pmatrix}
\tilde{\b}_t^\infty\Big)Z_{t+1}\Big]+\EE[Z_{t+1}^2].\]
Since $\s_{t+1}=\O^\top \bLambda \r_{t+1}$, we have
\[\EE[S_{t+1}^2]=\lim_{n \to \infty}
n^{-1}\|\s_{t+1}\|^2=\lim_{n \to \infty}
n^{-1} \r_{t+1}^\top \bLambda^2 \r_{t+1}=(\L_{t+1}^{(2,\infty)})_{t+1,t+1}.\]
Identifying $\tilde{\b}_t$ as the last column of $\B_{t+1}$, we have also
\[\EE\big[\big(\begin{pmatrix} U_1 & \cdots & U_{t+1} \end{pmatrix}
\tilde{\b}_t^\infty\big)^2\big]
=(\tilde{\b}_t^\infty)^\top \bDelta_{t+1}^\infty\tilde{\b}_t^\infty
=((\B_{t+1}^\infty)^\top \bDelta_{t+1}^\infty\B_{t+1}^\infty)_{t+1,t+1}.\]
Applying
\[\EE[(U_1,\ldots,U_{t+1})^\top Z_{t+1}]
=\lim_{n \to \infty} n^{-1}(\u_1,\ldots,\u_{t+1})^\top \z_{t+1}
=\tilde{\bPhi}_t^\infty\bsigma_t^\infty,\]
we get
\begin{align*}
\EE\big[\big(\begin{pmatrix} U_1 & \cdots & U_{t+1} \end{pmatrix}
\tilde{\b}_t^\infty\big)Z_{t+1}\big]
&=(\tilde{\b}_t^\infty)^\top \tilde{\bPhi}_t^\infty\bsigma_t^\infty
=((\B_{t+1}^\infty)^\top \bPhi_{t+1}^\infty\bSigma_{t+1}^\infty)_{t+1,t+1}\\
&=(\bsigma_t^\infty)^\top(\tilde{\bPhi}_t^\infty)^\top\tilde{\b}_t^\infty
=(\bSigma_{t+1}^\infty(\bPhi_{t+1}^\infty)^\top \B_{t+1}^\infty)_{t+1,t+1}
\end{align*}
Thus
\begin{align*}
\EE[Z_{t+1}^2]&=\Bigg(\Big(\L^{(2)}_{t+1}-\B_{t+1}^\top \bDelta_{t+1} \B_{t+1}
-\B_{t+1}^\top \bPhi_{t+1}\bSigma_{t+1}
-\bSigma_{t+1}\bPhi_{t+1}^\top \B_{t+1}\Big)^\infty\Bigg)_{t+1,t+1}\\
&=(\bSigma_{t+1}^\infty)_{t+1,t+1}=\sigma_{t+1,t+1}^\infty.
\end{align*}
where the second equality applies the identity (\ref{eq:L2}).
This concludes the proof of $t+1^{(c)}$.

Finally, to show $t+1^{(d)}$, observe that
\begin{equation}\label{eq:fullL}
\left(\begin{pmatrix} \Id & 0 \\ \B_{t+1}^\top & \Id \end{pmatrix}
\begin{pmatrix} \bDelta_{t+1} & \bPhi_{t+1}\bSigma_{t+1} \\
\bSigma_{t+1}\bPhi_{t+1}^\top & \bSigma_{t+1} \end{pmatrix}
\begin{pmatrix} \Id & \B_{t+1} \\ 0 & \Id \end{pmatrix}\right)^\infty
=\begin{pmatrix} \L_{t+1}^{(0,\infty)} & \L_{t+1}^{(1,\infty)} \\
\L_{t+1}^{(1,\infty)} & \L_{t+1}^{(2,\infty)}
\end{pmatrix}
\end{equation}
by Lemma \ref{lemma:Lidentities}. The upper-left $(2t+1) \times (2t+1)$
submatrix of (\ref{eq:fullL}) is exactly the second matrix of
(\ref{eq:partialL}), which we have already shown is invertible. So to check
invertibility of (\ref{eq:fullL}), it suffices to show that the Schur
complement of the lower-right entry is non-zero. By (\ref{eq:sperp}), this Schur
complement is equal to $\Var[S_\perp]$. Thus, we must show that
$\Var[S_\perp]>0$.

Interpreting the $(s,s')$
entry of $\L_t^{(k,\infty)}$ as $\EE[\Lambda^k R_s R_{s'}]$, note that
$\Var[S_\perp]$ in (\ref{eq:sperp}) is the variance of the residual of
the projection of $\Lambda R_{t+1}$ onto the linear span of
$(R_1,\ldots,R_{t+1},\Lambda R_1,\ldots,\Lambda R_t)$ with respect to the
$L_2$-inner-product $(X,Y) \mapsto \EE[XY]$.  
Thus, if $\Var[S_\perp]=0$, then
\[\Lambda R_{t+1}=\alpha_1 R_1+\ldots+\alpha_{t+1}R_{t+1}
+\beta_1 \Lambda R_1+\ldots+\beta_t \Lambda R_t\]
for some scalar constants $\alpha_1,\ldots,\alpha_{t+1},\beta_1,\ldots,\beta_t$
almost surely. Substituting (\ref{eq:RTplus1}) and rearranging to isolate
$R_\perp$, we get
\[(\Lambda-\alpha_{t+1})R_\perp=f(R_1,\dots,R_t,\Lambda)\]
for some quantity $f(R_1,\dots,R_t,\Lambda)$ that does not depend on $R_\perp$.
By Assumption \ref{assump:main}(e), $\Lambda$ is not a constant random
variable, so on an event of positive probability, we have $\Lambda \neq
\alpha_{t+1}$. Then conditioning on $(R_1,\ldots,R_t,\Lambda)$ and on this
event, we have
$R_\perp=f(R_1,\ldots,R_t,\Lambda)/(\Lambda-\alpha_{t+1})$, implying that
the conditional law of $R_\perp$ is constant. Recall
that $R_\perp$ is independent of $(R_1,\ldots,R_t,\Lambda)$---thus $R_\perp$
must be a constant random variable unconditionally. However, $R_\perp$ is a
mean-zero normal variable with positive variance by
(\ref{eq:Rperpnondegenerate}). This is a contradiction, so
$\Var[S_\perp]>0$. This shows $t+1^{(d)}$, concluding the induction.
\end{proof}

\section{Proof for rectangular matrices}\label{appendix:rectproof}

In this appendix, we prove Theorem \ref{thm:rect} using similar ideas.
Let us write the iterations (\ref{eq:AMPrectz}--\ref{eq:AMPrectu}) as
\begin{align*}
\r_t&=\O \u_t\\
\s_t&=\Q^\top \bLambda^\top \r_t\\
\z_t&=\s_t-b_{t1}\v_1-\ldots-b_{t,t-1}\v_{t-1}\\
\v_t&=v_t(\z_1,\ldots,\z_t,\F)\\
\p_t&=\Q \v_t\\
\q_t&=\O^\top \bLambda \p_t\\
\y_t&=\q_t-a_{t1}\u_1-\ldots-a_{tt}\u_t\\
\u_{t+1}&=u_{t+1}(\y_1,\ldots,\y_t,\E)
\end{align*}
Note that $\u_t,\r_t,\q_t,\y_t \in \RR^m$ while $\v_t,\p_t,\s_t,\z_t \in \RR^n$.

In the proof, we will identify the limits of the quantities
\begin{align}
m^{-1} \r_s^\top (\bLambda \bLambda^\top)^k \r_{s'}
&\equiv m^{-1} \u_s^\top (\W \W^\top)^k \u_{s'}\label{eq:keylimitrect1}\\
m^{-1} \p_s^\top \bLambda^\top (\bLambda \bLambda^\top)^k \r_{s'}
&\equiv m^{-1} \v_s^\top \W^\top (\W \W^\top)^k \u_{s'}\\
n^{-1} \r_s^\top \bLambda (\bLambda^\top \bLambda)^k \p_{s'}
&\equiv n^{-1} \u_s^\top\W(\W^\top \W)^k \v_{s'}\\
n^{-1} \p_s^\top (\bLambda^\top \bLambda)^k \p_{s'}
&\equiv n^{-1} \v_s^\top (\W^\top \W)^k \v_{s'}.\label{eq:keylimitrect4}
\end{align}
In addition to the matrices $\bTheta_t^{(j)}$ and $\bXi_t^{(j)}$ in
(\ref{eq:Thetarect}--\ref{eq:Xirect}), let us define
\begin{equation}\label{eq:Xrect}
\X_t^{(j)}=\sum_{i=0}^j (\bPsi_t\bPhi_t)^i \bPsi_t \bDelta_t
(\bPsi_t^\top \bPhi_t^\top)^{j-i}
+\sum_{i=0}^j (\bPsi_t\bPhi_t)^i\bGamma_t \bPhi_t^\top
(\bPsi_t^\top \bPhi_t^\top)^{j-i}.
\end{equation}
For example,
\begin{align*}
\X_t^{(0)}&=\bPsi_t\bDelta_t+\bGamma_t\bPhi_t^\top\\
\X_t^{(1)}&=\bPsi_t\bPhi_t\bPsi_t\bDelta_t+\bPsi_t\bPhi_t\bGamma_t\bPhi_t^\top
+\bPsi_t\bDelta_t\bPsi_t^\top\bPhi_t^\top
+\bGamma_t\bPhi_t^\top\bPsi_t^\top\bPhi_t^\top.
\end{align*}
Corresponding to (\ref{eq:keylimitrect1}--\ref{eq:keylimitrect4}),
we then define four families of matrices
\begin{align}
\H_t^{(2k)}=\sum_{j=0}^\infty c_{2k,j}\bTheta_t^{(j)}, \qquad
\I_t^{(2k+1)}=\sum_{j=0}^\infty c_{2k+1,j}\X_t^{(j)},\nonumber\\
\J_t^{(2k+1)}=\sum_{j=0}^\infty \bar{c}_{2k+1,j}(\X_t^{(j)})^\top, \qquad
\L_t^{(2k)}=\sum_{j=0}^\infty \bar{c}_{2k,j}\bXi_t^{(j)}
\label{eq:HIJLrect}
\end{align}
where $c_{2k,j},c_{2k+1,j},\bar{c}_{2k,j},\bar{c}_{2k+1,j}$ are certain
rectangular partial moment coefficients, defined in
Section \ref{sec:partialmomentrect} below. We show in (\ref{eq:cbarc}) below
that $\bar{c}_{2k+1,j}=\gamma \cdot c_{2k+1,j}$, so that
\begin{equation}\label{eq:IJ}
\J_t^{(2k+1)}=\gamma \cdot (\I_t^{(2k+1)})^\top.
\end{equation}
The limits of (\ref{eq:keylimitrect1}--\ref{eq:keylimitrect4})
will be identified as the entries of
$\lim_{m,n \to \infty} \H_t^{(2k)},\I_t^{(2k+1)},\J_t^{(2k+1)},\L_t^{(2k)}$.

\subsection{Coefficients for ``partial moments''}\label{sec:partialmomentrect}

Let $\{\kappa_{2k}\}_{k \geq 1}$ be the rectangular free cumulants for
the moment sequence (\ref{eq:momentsrect}) with aspect ratio $\gamma=m/n$.
Recall from Section \ref{sec:cumulantsrect}
the second cumulant sequence $\bar{\kappa}_{2k}=\gamma \cdot \kappa_{2k}$ for
all $k \geq 1$. For notational convenience, we set
\[\kappa_0=1, \qquad \bar{\kappa}_0=1.\]
We define four sequences of combinatorial coefficients, denoted by
\[c_{2k,j},\qquad \bar{c}_{2k,j},
\qquad c_{2k+1,j}, \qquad \bar{c}_{2k+1,j}\]
for integers $k,j \geq 0$. These sequences are defined by the initializations
\begin{equation}\label{eq:ckjrectinit}
c_{0,0}=\bar{c}_{0,0}=1,
\qquad c_{0,j}=\bar{c}_{0,j}=0 \quad \text{ for } j \geq 1
\end{equation}
and by the recursions, for all $j,k \geq 0$,
\begin{align}
c_{2k+1,j}&=\sum_{m=0}^{j+1} c_{2k,m}\kappa_{2(j+1-m)}\label{eq:ckjrectodd}\\
\bar{c}_{2k+1,j}&=\sum_{m=0}^{j+1} \bar{c}_{2k,m}\bar{\kappa}_{2(j+1-m)}\label{eq:barckjrectodd}\\
c_{2k+2,j}&=\sum_{m=0}^j c_{2k+1,m}\bar{\kappa}_{2(j-m)}\label{eq:ckjrecteven}\\
\bar{c}_{2k+2,j}&=\sum_{m=0}^j
\bar{c}_{2k+1,m}\kappa_{2(j-m)}.\label{eq:barckjrecteven}
\end{align}

Let
\[\NC'(2k,\ell)=\Big\{\pi \in \NC'(2k):S \cap \{1,\ldots,\ell\} \neq S
\text{ for all } S \in \pi\Big\}\]
be the subset of non-crossing partitions $\pi \in \NC'(2k)$ where no
set $S \in \pi$ is contained in $\{1,\ldots,\ell\}$. We set
$\NC'(2k,0)=\NC'(2k)$. Recall the moment-cumulant
relations (\ref{eq:momentcumulantrect}),
where $e(\pi)$ and $o(\pi)$ count the number of sets $S \in \pi$ whose smallest
element is even and odd, and $\bar{m}_{2k}$ is defined from $m_{2k}$ by
(\ref{eq:barmu}). Then these coefficients admit the following interpretations.

\begin{lemma}\label{lemma:ckjrect}
For each $k \geq 0$,
\begin{equation}\label{eq:ckjrectinterpodd}
c_{2k+1,j}=\sum_{\pi \in \NC'(2k+2j+2,2j+1)} \gamma^{e(\pi)}
\prod_{S \in \pi} \kappa_{|S|},
\qquad \bar{c}_{2k+1,j}=\sum_{\pi \in \NC'(2k+2j+2,2j+1)} \gamma^{o(\pi)}
\prod_{S \in \pi} \kappa_{|S|},
\end{equation}
and for each $k \geq 1$,
\begin{equation}\label{eq:ckjrectinterpeven}
c_{2k,j}=\sum_{\pi \in \NC'(2k+2j,2j)} \gamma^{e(\pi)}
\prod_{S \in \pi} \kappa_{|S|}, \qquad
\bar{c}_{2k,j}=\sum_{\pi \in \NC'(2k+2j,2j)} \gamma^{o(\pi)}
\prod_{S \in \pi} \kappa_{|S|}.
\end{equation}
In particular, for all $j,k \geq 0$, we have
\[c_{1,j}=\kappa_{2(j+1)}, \qquad \bar{c}_{1,j}=\bar{\kappa}_{2(j+1)},
\qquad c_{2k,0}=m_{2k}, \qquad \bar{c}_{2k,0}=\bar{m}_{2k}.\]
Finally, for all $j,k \geq 0$, we have
\begin{equation}\label{eq:cbarc}
\bar{c}_{2k+1,j}=\gamma \cdot c_{2k+1,j}.
\end{equation}
\end{lemma}
\begin{proof}
Let us show (\ref{eq:ckjrectinterpodd}--\ref{eq:ckjrectinterpeven})
by induction on $k$. By the initialization (\ref{eq:ckjrectinit})
and the recursions (\ref{eq:ckjrectodd}) and (\ref{eq:barckjrectodd}), we have
$c_{1,j}=\kappa_{2(j+1)}$ and $\bar{c}_{1,j}=\bar{\kappa}_{2(j+1)}$
for all $j \geq 0$. Since the sets of each partition in $\NC'(2j+2)$ must have
even cardinality, $\NC'(2j+2,2j+1)$ consists of only the
partition $\pi$ with the single set $\{1,\ldots,2j+2\}$ and this partition has
$e(\pi)=0$ and $o(\pi)=1$. Applying $\bar{\kappa}_{2j+2}=\gamma \cdot
\kappa_{2j+2}$, this shows both identities of
(\ref{eq:ckjrectinterpodd}) for $k=0$.

Assuming that (\ref{eq:ckjrectinterpodd}) holds for some $k \geq 0$, we now
check (\ref{eq:ckjrectinterpeven}) for $k+1$.
If $\pi \in \NC'(2k+2j+2,2j) \setminus \NC'(2k+2j+2,2j+1)$, then there is a set
$S \in \pi$ containing $2j+1$ that is a subset of $\{1,\ldots,2j+1\}$. Since
$\pi$ is non-crossing and $S$ has even cardinality, this set must be of the form
$S=\{2m+2,\ldots,2j+1\}$ for some $m \in \{0,\ldots,j-1\}$. 
This set $S$ has cardinality $2(j-m)$, and its smallest element is even.
Removing $S$ from $\pi$ yields a bijection between all such partitions $\pi$
and the partitions $\pi' \in \NC(2k+2m+2,2m+1)$. Thus,
applying the induction hypothesis (\ref{eq:ckjrectinterpodd}) with $m$ in
place of $j$,
\begin{align*}
c_{2k+1,m} \cdot \gamma \kappa_{2(j-m)}&=
\mathop{\sum_{\pi \in \NC'(2k+2j+2,2j) \setminus \NC'(2k+2j+2,2j+1)
}}_{\{2m+2,\ldots,2j+1\} \in \pi}
\gamma^{e(\pi)} \prod_{S \in \pi} \kappa_{|S|},\\
\bar{c}_{2k+1,m} \cdot \kappa_{2(j-m)}&=
\mathop{\sum_{\pi \in \NC'(2k+2j+2,2j) \setminus \NC'(2k+2j+2,2j+1)
}}_{\{2m+2,\ldots,2j+1\} \in \pi}
\gamma^{o(\pi)} \prod_{S \in \pi} \kappa_{|S|}.
\end{align*}
Summing over $m=0,\ldots,j-1$, combining with the induction hypothesis
(\ref{eq:ckjrectinterpodd}) applied for $j$,
and recalling that $\kappa_0=\bar{\kappa}_0=1$
while $\bar{\kappa}_{2j}=\gamma \cdot \kappa_{2j}$ for $j \geq 1$, we obtain
\begin{align*}
\sum_{m=0}^j c_{2k+1,m}\bar{\kappa}_{2(j-m)}
&=\sum_{\pi \in \NC'(2k+2j+2,2j)}
\gamma^{e(\pi)} \prod_{S \in \pi} \kappa_{|S'|},\\
\sum_{m=0}^j \bar{c}_{2k+1,m}\kappa_{2(j-m)}
&=\sum_{\pi \in \NC'(2k+2j+2,2j)}
\gamma^{o(\pi)} \prod_{S \in \pi} \kappa_{|S'|}.
\end{align*}
Recognizing the left sides as $c_{2(k+1),j}$ and $\bar{c}_{2(k+1),j}$
by (\ref{eq:ckjrecteven}--\ref{eq:barckjrecteven}),
this shows (\ref{eq:ckjrectinterpeven}) for $k+1$.

Now assuming that (\ref{eq:ckjrectinterpeven}) holds for some $k \geq 1$, we
check (\ref{eq:ckjrectinterpodd}) for $k$.
If $\pi \in \NC'(2k+2j+2,2j+1) \setminus \NC'(2k+2j+2,2j+2)$, then similar to
the above, there is some
set $S=\{2m+1,\ldots,2j+2\} \in \pi$ for some $m \in \{0,\ldots,j\}$,
with cardinality $2(j-m)+2$ and whose smallest element is odd. Removing $S$
from $\pi$ yields a bijection between such partitions $\pi$
and the partitions $\pi' \in \NC(2k+2m,2m)$. Then applying the induction
hypothesis (\ref{eq:ckjrectinterpeven}) with $m$ in place of $k$,
\begin{align*}
c_{2k,m} \cdot \kappa_{2(j-m)+2}
&=\mathop{\sum_{\pi \in \NC'(2k+2j+2,2j+2) \setminus \NC'(2k+2j+2,2j+1)
}}_{\{2m+1,\ldots,2j+2\} \in \pi}
\gamma^{e(\pi)} \prod_{S \in \pi} \kappa_{|S|},\\
\bar{c}_{2k,m} \cdot \gamma\kappa_{2(j-m)+2}
&=\mathop{\sum_{\pi \in \NC'(2k+2j+2,2j+2) \setminus \NC'(2k+2j+2,2j+1)
}}_{\{2m+1,\ldots,2j+2\} \in \pi}
\gamma^{o(\pi)} \prod_{S \in \pi} \kappa_{|S|}.
\end{align*}
Summing over $m=0,\ldots,j$, combining with the induction hypothesis
(\ref{eq:ckjrectinterpeven}) applied for $j+1$, and
applying again $\kappa_0=\bar{\kappa}_0=1$ and $\bar{\kappa}_{2j}=\gamma
\cdot \kappa_{2j}$ for $j \geq 1$, we obtain
(\ref{eq:ckjrectinterpodd}) for $k$. This concludes the induction,
showing (\ref{eq:ckjrectinterpodd}) for all $k \geq 0$
and (\ref{eq:ckjrectinterpeven}) for all $k \geq 1$.

The statements $c_{1,j}=\kappa_{2(j+1)}$ and
$\bar{c}_{1,j}=\bar{\kappa}_{2(j+1)}$ are already shown.
The statements $c_{2k,0}=m_{2k}$ and $\bar{c}_{2k,0}=\bar{m}_{2k}$
follow from $\NC'(2k,0)=\NC'(2k)$, together with the
moment-cumulant relations (\ref{eq:momentcumulantrect}). Finally, for the
identity (\ref{eq:cbarc}), note that this holds for
$k=0$ because $\bar{\kappa}_{2(j+1)}=\gamma \cdot \kappa_{2(j+1)}$. Supposing
that it holds for $k-1$, we may compose (\ref{eq:ckjrectodd}) and
(\ref{eq:ckjrecteven}) to get
\begin{align*}
c_{2k+1,j}&=\sum_{m=0}^{j+1} \sum_{p=0}^m c_{2k-1,p}\bar{\kappa}_{2(m-p)}
\kappa_{2(j+1-m)}\\
&=\sum_{p=0}^{j+1} c_{2k-1,p} \sum_{m=p}^{j+1} \bar{\kappa}_{2(m-p)}
\kappa_{2(j+1-m)}
=\sum_{p=0}^{j+1} c_{2k-1,p} \sum_{m=0}^{j-p+1} \bar{\kappa}_{2m}
\kappa_{2(j-p+1-m)}.
\end{align*}
Similarly
\[\bar{c}_{2k+1,j}=\sum_{p=0}^{j+1} \bar{c}_{2k-1,p} \sum_{m=0}^{j-p+1}
\kappa_{2m}\bar{\kappa}_{2(j-p+1-m)}.\]
Comparing these two expressions and applying $\bar{c}_{2k-1,p}=\gamma \cdot
c_{2k-1,p}$ for each $p=0,\ldots,j+1$, we get $\bar{c}_{2k+1,j}=\gamma \cdot
c_{2k+1,j}$. This shows (\ref{eq:cbarc}) for all $k$.
\end{proof}

\subsection{Partial moment identities}

Recalling the definitions of
$\H_t^{(2k)},\I_t^{(2k+1)},\J_t^{(2k+1)},\L_t^{(2k)}$ from (\ref{eq:HIJLrect}),
we now collect several identities derived from the recursions for
$c_{2k,j},c_{2k+1,j},\bar{c}_{2k,j},\bar{c}_{2k+1,j}$.

\begin{lemma}
For every $t \geq 1$,
\begin{align}
\H_t^{(0)}&=\bDelta_t\label{eq:H0rect}\\
\L_t^{(0)}&=\bGamma_t\label{eq:L0rect}\\
\I_t^{(1)}&=\A_t^\top \bDelta_t+\bSigma_t\bPhi_t^\top\label{eq:I1rectA}\\
&=\gamma^{-1} \cdot (\bGamma_t\B_t+\bPsi_t\bOmega_t)\label{eq:I1rectB}\\
\J_t^{(1)}&=\B_t^\top \bGamma_t+\bOmega_t\bPsi_t^\top\label{eq:J1rectA}\\
&=\gamma \cdot (\bDelta_t\A_t+\bPhi_t\bSigma_t)\label{eq:J1rectB}\\
\L_t^{(2)}&=\gamma \cdot (\A_t^\top \bDelta_t\A_t+\A_t^\top\bPhi_t\bSigma_t
+\bSigma_t\bPhi_t^\top\A_t+\bSigma_t)\label{eq:L2rect}\\
\H_t^{(2)}&=\gamma^{-1} \cdot
(\B_t^\top \bGamma_t\B_t+\B_t^\top \bPsi_t\bOmega_t
+\bOmega_t\bPsi_t^\top \B_t+\bOmega_t)\label{eq:H2rect}
\end{align}
\end{lemma}
\begin{proof}
The identities (\ref{eq:H0rect}) and (\ref{eq:L0rect}) follow immediately from
the initializations $c_{0,0}=\bar{c}_{0,0}=1$ and
$c_{0,j}=\bar{c}_{0,j}=0$ for all $j \geq 1$, and the
observations $\bTheta_t^{(0)}=\bDelta_t$ and $\bXi_t^{(0)}=\bGamma_t$.

For (\ref{eq:I1rectA}), let us separate the terms of $\I_t^{(1)}$ ending with
$\bDelta_t$ from those ending with $\bPhi_t^\top$. Applying
$c_{1,j}=\kappa_{2(j+1)}$, this yields
\[\I_t^{(1)}=\sum_{j=0}^\infty \kappa_{2(j+1)}
\left((\bPsi_t\bPhi_t)^j \bPsi_t\bDelta_t
+\sum_{i=0}^{j-1} (\bPsi_t\bPhi_t)^i \bPsi_t\bDelta_t
(\bPsi_t^\top \bPhi_t^\top)^{j-i}
+\sum_{i=0}^j (\bPsi_t\bPhi_t)^i\bGamma_t\bPhi_t^\top
(\bPsi_t^\top \bPhi_t^\top)^{j-i}\right).\]
Observe that
\[\sum_{j=0}^\infty \kappa_{2(j+1)} (\bPsi_t\bPhi_t)^j \bPsi_t\bDelta_t
=\A_t^\top \bDelta_t,\]
while
\begin{align*}
&\sum_{j=0}^\infty \kappa_{2(j+1)}
\left(\sum_{i=0}^{j-1} (\bPsi_t\bPhi_t)^i \bPsi_t\bDelta_t
(\bPsi_t^\top \bPhi_t^\top)^{j-i}
+\sum_{i=0}^j (\bPsi_t\bPhi_t)^i\bGamma_t\bPhi_t^\top
(\bPsi_t^\top \bPhi_t^\top)^{j-i}\right)\\
&=\sum_{j=0}^\infty \kappa_{2(j+1)}
\left(\sum_{i=0}^{j-1}
(\bPsi_t\bPhi_t)^i \bPsi_t\bDelta_t\bPsi_t^\top(\bPhi_t^\top
\bPsi_t^\top)^{j-1-i}+\sum_{i=0}^j (\bPsi_t\bPhi_t)^i
\bGamma_t(\bPhi_t^\top \bPsi_t^\top)^{j-i}\right)\bPhi_t^\top
=\bSigma_t\bPhi_t^\top.
\end{align*}
Thus we obtain (\ref{eq:I1rectA}). The identity (\ref{eq:I1rectB}) follows
analogously by separating the terms of $\I_t^{(1)}$ starting with $\bGamma_t$
from those starting with $\bPsi_t$. (The factor $\gamma^{-1}$ cancels the factor
of $\gamma$ in the definitions of $\B_t$ and $\bOmega_t$.) The identities
(\ref{eq:J1rectA}--\ref{eq:J1rectB}) follow from
(\ref{eq:I1rectA}--\ref{eq:I1rectB}) and the relation
$\J_t^{(1)}=\gamma \cdot (\I_t^{(1)})^\top$ from (\ref{eq:IJ}).

For (\ref{eq:L2rect}), applying (\ref{eq:barckjrecteven}) and the identity
$\bar{c}_{1,m}=\bar{\kappa}_{2(m+1)}=\gamma \cdot \kappa_{2(m+1)}$, we have
\begin{align*}
\L_t^{(2)}=\sum_{j=0}^\infty \bar{c}_{2,j}\bXi_t^{(j)}
&=\gamma \cdot \sum_{j=0}^\infty \left(\sum_{m=0}^j
\kappa_{2(m+1)}\kappa_{2(j-m)} \right)\\
&\qquad\cdot \left(\sum_{i=0}^j (\bPsi_t\bPhi_t)^i
\bGamma_t(\bPhi_t^\top \bPsi_t^\top)^{j-i}
+\sum_{i=0}^{j-1} (\bPsi_t\bPhi_t)^i
\bPsi_t\bDelta_t\bPsi_t^\top(\bPhi_t^\top \bPsi_t^\top)^{j-1-i}\right).
\end{align*}
Collecting terms by powers of $\bPsi_t\bPhi_t$ and its transpose, this is
\begin{align}
\L_t^{(2)}&=\gamma \cdot \sum_{i=0}^\infty \sum_{p=0}^\infty \Bigg[
\left(\sum_{m=0}^{i+p} \kappa_{2(m+1)}\kappa_{2(i+p-m)}\right)
(\bPsi_t\bPhi_t)^i\bGamma_t(\bPhi_t^\top \bPsi_t^\top)^p\nonumber\\
&\hspace{1in}
+\left(\sum_{m=0}^{i+p+1} \kappa_{2(m+1)}\kappa_{2(i+p+1-m)}\right)
(\bPsi_t\bPhi_t)^i\bPsi_t\bDelta_t\bPsi_t^\top(\bPhi_t^\top \bPsi_t^\top)^p
\Bigg].\label{eq:L2tmp}
\end{align}
From the definitions and the notation $\kappa_0=1$, we now identify
\begin{align*}
\A_t^\top \bDelta_t\A_t&=\sum_{i=0}^\infty \sum_{p=0}^\infty
\kappa_{2(i+1)}\kappa_{2(p+1)}
(\bPsi_t\bPhi_t)^i\bPsi_t\bDelta_t\bPsi_t^\top(\bPhi_t^\top \bPsi_t^\top)^p\\
\bSigma_t&=\sum_{i=0}^\infty \sum_{p=0}^\infty \kappa_{2(i+p+1)}\kappa_0
(\bPsi_t\bPhi_t)^i\bGamma_t(\bPhi_t^\top \bPsi_t^\top)^p
+\sum_{i=0}^\infty \sum_{p=0}^\infty \kappa_{2(i+p+2)}\kappa_0
(\bPsi_t\bPhi_t)^i\bPsi_t\bDelta_t\bPsi_t^\top(\bPhi_t^\top \bPsi_t^\top)^p\\
\A_t^\top \bPhi_t\bSigma_t&=\sum_{m=0}^\infty \kappa_{2(m+1)}
(\bPsi_t\bPhi_t)^{m+1} \\
&\hspace{0.5in}\cdot \sum_{q=0}^\infty \kappa_{2(q+1)}\left(\sum_{p=0}^q
(\bPsi_t\bPhi_t)^{q-p}\bGamma_t(\bPhi_t^\top \bPsi_t^\top)^p
+\sum_{p=0}^{q-1}
(\bPsi_t\bPhi_t)^{q-1-p}\bPsi_t\bDelta_t\bPsi_t^\top
(\bPhi_t^\top \bPsi_t^\top)^p\right)\\
&=\sum_{i=0}^\infty \sum_{p=0}^\infty
\left(\sum_{m=0}^{i-1} \kappa_{2(m+1)}\kappa_{2(i+p-m)}\right)
(\bPsi_t\bPhi_t)^i\bGamma_t(\bPhi_t^\top \bPsi_t^\top)^p\\
&\hspace{0.5in}+\sum_{i=0}^\infty \sum_{p=0}^\infty 
\left(\sum_{m=0}^{i-1} \kappa_{2(m+1)}\kappa_{2(i+p+1-m)}\right)
(\bPsi_t\bPhi_t)^i\bPsi_t\bDelta_t\bPsi_t^\top(\bPhi_t^\top \bPsi_t^\top)^p\\
\bSigma_t \bPhi_t^\top \A_t&=(\A_t^\top \bPhi_t\bSigma_t)^\top\\
&=\sum_{i=0}^\infty \sum_{p=0}^\infty
\left(\sum_{m=i}^{i+p-1} \kappa_{2(m+1)}\kappa_{2(i+p-m)}\right)
(\bPsi_t\bPhi_t)^i\bGamma_t(\bPhi_t^\top \bPsi_t^\top)^p\\
&\hspace{0.5in}+\sum_{i=0}^\infty \sum_{p=0}^\infty 
\left(\sum_{m=i+1}^{i+p} \kappa_{2(m+1)}\kappa_{2(i+p+1-m)}\right)
(\bPsi_t\bPhi_t)^i\bPsi_t\bDelta_t\bPsi_t^\top(\bPhi_t^\top \bPsi_t^\top)^p
\end{align*}
Summing these four expressions and comparing with (\ref{eq:L2tmp}) 
yields the identity (\ref{eq:L2rect}).

For (\ref{eq:H2rect}), we may write similarly
\begin{align*}
\H_t^{(2)}&=\gamma \cdot \sum_{j=0}^\infty \left(\sum_{m=0}^j
\kappa_{2(m+1)}\kappa_{2(j-m)}\gamma^{-\1\{m=j\}}\right)\\
&\qquad\cdot \left(\sum_{i=0}^j (\bPhi_t\bPsi_t)^i
\bDelta_t(\bPsi_t^\top \bPhi_t^\top)^{j-i}
+\sum_{i=0}^{j-1} (\bPhi_t\bPsi_t)^i
\bPhi_t\bGamma_t\bPhi_t^\top(\bPsi_t^\top \bPhi_t^\top)^{j-1-i}\right),
\end{align*}
where the factor $\gamma^{-\1\{m=j\}}$ comes from the fact that
$\bar{\kappa}_{2(j-m)}=\gamma \kappa_{2(j-m)}$ if $m<j$, but
$\bar{\kappa}_{2(j-m)}=\kappa_{2(j-m)}=1$ for $m=j$.
This is matched by the observation that $\B_t^\top \bGamma_t\B_t$,
$\B_t^\top\bPsi_t\bOmega_t$, and $\bOmega_t\bPsi_t^\top \B_t$ on the right of
(\ref{eq:H2rect}) have a factor of $\gamma^2$, whereas the last term $\bOmega_t$
has a factor of only $\gamma$. Then (\ref{eq:H2rect}) follows from an argument
analogous to the above, and we omit this for brevity.
\end{proof}

\begin{lemma}\label{lemma:upsilonrect}
Define
\begin{equation}\label{eq:Upsilonrect}
\bUpsilon_t=\begin{pmatrix} \bDelta_t & \bDelta_t \A_t+\bPhi_t\bSigma_t \\
\bPhi_t^\top & \bPhi_t^\top \A_t+\Id \end{pmatrix},
\qquad
\T_t=\begin{pmatrix} \bGamma_t & \bGamma_t \B_t+\bPsi_t\bOmega_t \\
\bPsi_t^\top & \bPsi_t^\top \B_t+\Id \end{pmatrix}.
\end{equation}
Then for every $t \geq 1$ and $k \geq 0$,
\begin{align}
&\begin{pmatrix} \H_t^{(2k)} & (\I_t^{(2k+1)})^\top \\
\I_t^{(2k+1)} & \gamma^{-1} \cdot \L_t^{(2k+2)} \end{pmatrix}\nonumber\\
&=\begin{pmatrix} \sum_{j=0}^\infty c_{2k,j} (\bPhi_t\bPsi_t)^j &
\sum_{j=0}^\infty c_{2k,j+1} (\X_t^{(j)})^\top \\
\sum_{j=0}^\infty c_{2k+1,j} (\bPsi_t\bPhi_t)^j\bPsi_t
& \sum_{j=0}^\infty c_{2k+1,j} \bXi_t^{(j)} \end{pmatrix} \bUpsilon_t
\label{eq:Upsiloninvrect}\\
&=c_{2k,0}\begin{pmatrix} \H_t^{(0)} & (\I_t^{(1)})^\top \\
\I_t^{(1)} & \gamma^{-1} \cdot \L_t^{(2)} \end{pmatrix}
+\bUpsilon_t^\top \begin{pmatrix} 0 &
\sum_{j=0}^\infty c_{2k,j+1} \bPsi_t^\top (\bPhi_t^\top \bPsi_t^\top)^j \\
\sum_{j=0}^\infty c_{2k,j+1} (\bPsi_t\bPhi_t)^j\bPsi_t &
\sum_{j=0}^\infty c_{2k,j+1} \bXi_t^{(j)} \end{pmatrix} \bUpsilon_t,
\label{eq:Upsiloninvrect2}
\end{align}
and
\begin{align}
&\begin{pmatrix} \L_t^{(2k)} & (\J_t^{(2k+1)})^\top \\
\J_t^{(2k+1)} & \gamma \cdot \H_t^{(2k+2)} \end{pmatrix}\nonumber\\
&=\begin{pmatrix} \sum_{j=0}^\infty \bar{c}_{2k,j} (\bPsi_t\bPhi_t)^j &
\sum_{j=0}^\infty \bar{c}_{2k,j+1} \X_t^{(j)} \\
\sum_{j=0}^\infty \bar{c}_{2k+1,j} (\bPhi_t\bPsi_t)^j\bPhi_t
& \sum_{j=0}^\infty \bar{c}_{2k+1,j} \bTheta_t^{(j)} \end{pmatrix} \T_t
\label{eq:Tinvrect}\\
&=\bar{c}_{2k,0}\begin{pmatrix} \L_t^{(0)} & (\J_t^{(1)})^\top \\
\J_t^{(1)} & \gamma \cdot \H_t^{(2)} \end{pmatrix}
+\T_t^\top \begin{pmatrix} 0 &
\sum_{j=0}^\infty \bar{c}_{2k,j+1} \bPhi_t^\top (\bPsi_t^\top \bPhi_t^\top)^j \\
\sum_{j=0}^\infty \bar{c}_{2k,j+1} (\bPhi_t\bPsi_t)^j\bPhi_t &
\sum_{j=0}^\infty \bar{c}_{2k,j+1} \bTheta_t^{(j)} \end{pmatrix} \T_t.
\label{eq:Tinvrect2}
\end{align}
\end{lemma}
\begin{proof}
The arguments are similar to those of Lemma \ref{lemma:upsilon}.
Applying (\ref{eq:I1rectA}), the definitions of $\I_t^{(1)}$ and $\A_t$ from
(\ref{eq:HIJLrect}) and (\ref{eq:ABrect}), and the
notation $\kappa_0=1$, we have
\begin{equation}\label{eq:Upsilonrectform}
\bUpsilon_t^\top=\begin{pmatrix} \bDelta_t & \bPhi_t \\
\I_t^{(1)} & \A_t^\top \bPhi_t+\Id \end{pmatrix}
=\begin{pmatrix} \bDelta_t & \bPhi_t \\
\sum_{j=0}^\infty \kappa_{2(j+1)} \X_t^{(j)} & \sum_{j=0}^\infty
\kappa_{2j} (\bPsi_t\bPhi_t)^j \end{pmatrix}.
\end{equation}
Then applying the definitions of $\bTheta_t^{(j)}$, $\bXi_t^{(j)}$, and
$\X_t^{(j)}$ from (\ref{eq:Thetarect}--\ref{eq:Xirect}) and (\ref{eq:Xrect}),
we may compute
\begin{align}
&\begin{pmatrix} \bDelta_t & \bPhi_t \end{pmatrix}
\begin{pmatrix} \sum_{j=0}^\infty c_{2k,j}(\bPsi_t^\top\bPhi_t^\top)^j \\
\sum_{j=0}^\infty c_{2k,j+1}\X_t^{(j)} \end{pmatrix}\nonumber\\
&=\sum_{j=0}^\infty c_{2k,j} \bDelta_t (\bPsi_t^\top\bPhi_t^\top)^j
+\sum_{j=0}^\infty c_{2k,j+1}
\left(\sum_{i=0}^j
(\bPhi_t\bPsi_t)^{i+1}\bDelta_t(\bPsi_t^\top\bPhi_t^\top)^{j-i}
+\sum_{i=0}^j (\bPhi_t\bPsi_t)^i\bPhi_t\bGamma_t\bPhi_t^\top
(\bPsi_t^\top\bPhi_t^\top)^{j-i}\right)\nonumber\\
&=\sum_{j=0}^\infty c_{2k,j}\bTheta_t^{(j)}=\H_t^{(2k)},\label{eq:Upsinvrect1}\\
&\begin{pmatrix} \bDelta_t & \bPhi_t \end{pmatrix}
\begin{pmatrix} \sum_{j=0}^\infty c_{2k+1,j}
\bPsi_t^\top(\bPhi_t^\top \bPsi_t^\top)^j
\\ \sum_{j=0}^\infty c_{2k+1,j}\bXi_t^{(j)} \end{pmatrix}\nonumber\\
&=\sum_{j=0}^\infty c_{2k+1,j}\left(
\bDelta_t\bPsi_t^\top(\bPhi_t^\top \bPsi_t^\top)^j
+\sum_{i=0}^j (\bPhi_t\bPsi_t)^i
\bPhi_t\bGamma_t(\bPhi_t^\top\bPsi_t^\top)^{j-i}
+\sum_{i=0}^{j-1} (\bPhi_t\bPsi_i)^{i+1}\bDelta_t\bPsi_i^\top
(\bPhi_t^\top\bPsi_t^\top)^{j-1-i}\right)\nonumber\\
&=\sum_{j=0}^\infty
c_{2k+1,j}(\X_t^{(j)})^\top=(\I_t^{(2k+1)})^\top.\label{eq:Upsinvrect2}
\end{align}
We may also compute, analogously to (\ref{eq:Lkplus1prod}),
\begin{align}
&\begin{pmatrix}
\sum_{j=0}^\infty \kappa_{2(j+1)} \X_t^{(j)} & \sum_{j=0}^\infty
\kappa_{2j} (\bPsi_t\bPhi_t)^j \end{pmatrix}
\begin{pmatrix} \sum_{j=0}^\infty c_{2k,j}(\bPsi_t^\top\bPhi_t^\top)^j \\
\sum_{j=0}^\infty c_{2k,j+1}\X_t^{(j)} \end{pmatrix}\nonumber\\
&\quad=\left(\sum_{j=0}^\infty \kappa_{2(j+1)} \sum_{i=0}^j (\bPsi_t\bPhi_t)^i
(\bPsi_t\bDelta_t+\bGamma_t\bPhi_t^\top)(\bPsi_t^\top \bPhi_t^\top)^{j-i}\right)
\cdot \left(\sum_{p=0}^\infty c_{2k,p}(\bPsi_t^\top
\bPhi_t^\top)^p\right)\nonumber\\
&\qquad\qquad+\left(\sum_{j=0}^\infty \kappa_{2j} (\bPsi_t\bPhi_t)^j\right)
\left(\sum_{p=0}^\infty c_{2k,p+1}\sum_{q=0}^p (\bPsi_t\bPhi_t)^q
(\bPsi_t\bDelta_t+\bGamma_t\bPhi_t^\top)(\bPsi_t^\top
\bPhi_t^\top)^{p-q}\right)\nonumber\\
&\quad=\sum_{i=0}^\infty \sum_{r=0}^\infty \left(\sum_{m=0}^{i+r+1}
\kappa_{2m}c_{2k,i+r+1-m}\right)
(\bPsi_t\bPhi_t)^i (\bPsi_t\bDelta_t+\bGamma_t\bPhi_t^\top)
(\bPsi_t^\top\bPhi_t^\top)^r\nonumber\\
&\quad=\sum_{i=0}^\infty \sum_{r=0}^\infty c_{2k+1,i+r}
(\bPsi_t\bPhi_t)^i (\bPsi_t\bDelta_t+\bGamma_t\bPhi_t^\top)
(\bPsi_t^\top\bPhi_t^\top)^r=\I_t^{(2k+1)}.\label{eq:Upsinvrect3}
\end{align}
In the last two equalities above, we used the recursion (\ref{eq:ckjrectodd})
and the definitions of $\X_t^{(j)}$ and $\I_t^{(2k+1)}$. Similarly,
\begin{align}
&\begin{pmatrix}
\sum_{j=0}^\infty \kappa_{2(j+1)} \X_t^{(j)} & \sum_{j=0}^\infty
\kappa_{2j} (\bPsi_t\bPhi_t)^j \end{pmatrix}
\begin{pmatrix} \sum_{j=0}^\infty c_{2k+1,j}
\bPsi_t^\top(\bPhi_t^\top \bPsi_t^\top)^j
\\ \sum_{j=0}^\infty c_{2k+1,j}\bXi_t^{(j)} \end{pmatrix}\nonumber\\
&\quad=\left(\sum_{j=0}^\infty \kappa_{2(j+1)} \sum_{i=0}^j (\bPsi_t\bPhi_t)^i
(\bPsi_t\bDelta_t+\bGamma_t\bPhi_t^\top)(\bPsi_t^\top \bPhi_t^\top)^{j-i}\right)
\cdot \left(\sum_{p=0}^\infty c_{2k+1,p} \bPsi_t^\top(\bPhi_t^\top
\bPsi_t^\top)^p\right)\nonumber\\
&\quad\quad+\left(\sum_{j=0}^\infty \kappa_{2j} (\bPsi_t\bPhi_t)^j\right)
\cdot \sum_{p=0}^\infty c_{2k+1,p}\Bigg(\sum_{q=0}^p
(\bPsi_t\bPhi_t)^q\bGamma_t(\bPhi_t^\top\bPsi_t^\top)^{p-q}\nonumber\\
&\hspace{3in}+\sum_{q=0}^{p-1} (\bPsi_t\bPhi_t)^q\bPsi_t\bDelta_t\bPsi_t^\top
(\bPhi_t^\top \bPsi_t^\top)^{p-1-q}\Bigg)\nonumber\\
&\quad=\sum_{i=0}^\infty \sum_{r=0}^\infty \Bigg(
\left(\sum_{m=0}^{i+r+1}\kappa_{2m}
c_{2k+1,i+r+1-m}\right)(\bPsi_t\bPhi_t)^i\bPsi_t\bDelta_t\bPsi_t^\top
(\bPhi_t^\top \bPsi_t^\top)^r\nonumber\\
&\hspace{2in}+\left(\sum_{m=0}^{i+r}\kappa_{2m}
c_{2k+1,i+r-m}\right)(\bPsi_t\bPhi_t)^i\bGamma_t(\bPhi_t^\top \bPsi_t^\top)^r
\Bigg)\nonumber\\
&\quad=\gamma^{-1} \sum_{i=0}^\infty \sum_{r=0}^\infty
\left(\bar{c}_{2k+2,i+r+1}(\bPsi_t\bPhi_t)^i\bPsi_t\bDelta_t\bPsi_t^\top
(\bPhi_t^\top \bPsi_t^\top)^r
+\bar{c}_{2k+2,i+r}(\bPsi_t\bPhi_t)^i\bGamma_t(\bPhi_t^\top
\bPsi_t^\top)^r\right)\nonumber\\
&\quad=\gamma^{-1} \L_t^{(2k+2)}.\label{eq:Upsinvrect4}
\end{align}
In the last two equalities above, we used the identity
$c_{2k+1,j}=\gamma^{-1}\bar{c}_{2k+1,j}$, the recursion
(\ref{eq:barckjrecteven}), and the definitions of $\bXi_t^{(j)}$ and
$\L_t^{(2k+2)}$. Recalling (\ref{eq:Upsilonrectform}),
stacking (\ref{eq:Upsinvrect1}), (\ref{eq:Upsinvrect2}),
(\ref{eq:Upsinvrect3}), and (\ref{eq:Upsinvrect4}),
and taking the transpose yields (\ref{eq:Upsiloninvrect}).

For (\ref{eq:Upsiloninvrect2}), applying again
the form (\ref{eq:Upsilonrectform}) for $\bUpsilon_t^\top$, we may compute
\[\bUpsilon_t^\top \begin{pmatrix} 0 \\
\sum_{j=0}^\infty c_{2k,j+1}(\bPsi_t\bPhi_t)^j\bPsi_t \end{pmatrix}
=\begin{pmatrix}
\sum_{j=0}^\infty c_{2k,j+1}(\bPhi_t\bPsi_t)^{j+1} \\
\sum_{\ell=0}^\infty \left(\sum_{j=0}^\ell \kappa_{2j}c_{2k,\ell+1-j}\right)
(\bPsi_t\bPhi_t)^\ell\bPsi_t\end{pmatrix}.\]
Let us apply $c_{0,0}=1$, $c_{0,\ell}=0$ for $\ell \geq 1$,
$c_{1,\ell}=\kappa_{2(\ell+1)}$, and the recursion (\ref{eq:ckjrectodd}) in the
form
\[c_{2k+1,\ell}-c_{2k,0}\kappa_{2(\ell+1)}=\sum_{j=0}^\ell
c_{2k,\ell+1-j}\kappa_{2j}.\]
This gives
\begin{align}
\bUpsilon_t^\top \begin{pmatrix} 0 \\
\sum_{j=0}^\infty c_{2k,j+1}(\bPsi_t\bPhi_t)^j\bPsi_t \end{pmatrix}
&=\begin{pmatrix} \sum_{\ell=0}^\infty (c_{2k,\ell}-c_{2k,0}c_{0,\ell})
(\bPhi_t\bPsi_t)^\ell \\
\sum_{\ell=0}^\infty (c_{2k+1,\ell}-c_{2k,0}c_{1,\ell})
(\bPsi_t\bPhi_t)^\ell\bPsi_t \end{pmatrix}\label{eq:Upsinvrect5}
\end{align}
Applying the same computations as leading to (\ref{eq:Upsinvrect2}) and
(\ref{eq:Upsinvrect4}), with the recursion (\ref{eq:ckjrectodd}) in the forms
\[\sum_{m=0}^{i+r+1} \kappa_{2m}c_{2k,i+r+2-m}
=c_{2k+1,i+r+1}-\kappa_{2(i+r+2)}c_{2k,0},
\quad \sum_{m=0}^{i+r} \kappa_{2m}c_{2k,i+r+1-m}
=c_{2k+1,i+r}-\kappa_{2(i+r+1)}c_{2k,0}\]
replacing the final two steps of (\ref{eq:Upsinvrect4}),
we have also
\begin{align}
\bUpsilon_t^\top \begin{pmatrix} \sum_{j=0}^\infty c_{2k,j+1}\bPsi_t^\top
(\bPhi_t^\top \bPsi_t^\top)^j \\
\sum_{j=0}^\infty c_{2k,j+1}\bXi_t^{(j)} \end{pmatrix}
&=\begin{pmatrix} \sum_{j=0}^\infty c_{2k,j+1}(\X_t^{(j)})^\top \\
\sum_{j=0}^\infty
(c_{2k+1,j}-\kappa_{2(j+1)}c_{2k,0})\bXi_t^{(j)}
\end{pmatrix}\nonumber\\
&=\begin{pmatrix} \sum_{\ell=0}^\infty (c_{2k,\ell+1}-c_{2k,0}c_{0,\ell+1})
(\X_t^{(\ell)})^\top \\
\sum_{\ell=0}^\infty (c_{2k+1,\ell}-c_{2k,0}c_{1,\ell})\bXi_t^{(\ell)}
\end{pmatrix}.\label{eq:Upsinvrect6}
\end{align}
Stacking (\ref{eq:Upsinvrect5}) and (\ref{eq:Upsinvrect6}), multiplying on the
right by $\bUpsilon_t$, and then applying
(\ref{eq:Upsiloninvrect}) for $k$ and also for $k=0$, yields
\begin{align*}
&\bUpsilon_t^\top \begin{pmatrix} 0 & \sum_{j=0}^\infty c_{2k,j+1}\bPsi_t^\top
(\bPhi_t^\top \bPsi_t^\top)^j \\
\sum_{j=0}^\infty c_{2k,j+1}(\bPsi_t\bPhi_t)^j\bPsi_t &
\sum_{j=0}^\infty c_{2k,j+1}\bXi_t^{(j)} \end{pmatrix}\bUpsilon_t\\
&\qquad\qquad=\begin{pmatrix} \H_t^{(2k)} & (\I_t^{(2k+1)})^\top \\
\I_t^{(2k+1)} & \gamma^{-1} \cdot \L_t^{(2k+2)} \end{pmatrix}
-c_{2k,0}\begin{pmatrix} \H_t^{(0)} & (\I_t^{(1)})^\top \\
\I_t^{(1)} & \gamma^{-1} \cdot \L_t^{(2)} \end{pmatrix}.
\end{align*}
Rearranging this yields (\ref{eq:Upsiloninvrect2}).

The identities (\ref{eq:Tinvrect}) and (\ref{eq:Tinvrect2}) follow from writing
\[\T_t^\top=\begin{pmatrix} \bGamma_t & \bPsi_t \\
\J_t^{(1)} & \B_t^\top \bPsi_t+\Id \end{pmatrix}
=\begin{pmatrix} \bGamma_t & \bPsi_t \\
\sum_{j=0}^\infty \bar{\kappa}_{2(j+1)}(\X_t^{(j)})^\top &
\sum_{j=0}^\infty \bar{\kappa}_{2j}(\bPhi_t\bPsi_t)^j \end{pmatrix}\]
and applying the same arguments, which we omit for brevity.
\end{proof}

\subsection{Conditioning argument}

We now prove Theorem \ref{thm:rect} using a conditioning argument similar to the
symmetric square setting.
Recall the definition of $\blambda \in \RR^{\min(m,n)}$ from
(\ref{eq:lambdarect}). Let us define
\[\blambda_m \in \RR^m, \qquad \blambda_n \in \RR^n\]
to be this vector extended by $m-n$ and $n-m$ additional 0's,
respectively. Thus $\blambda=\blambda_m$ if $m \leq n$, and
$\blambda=\blambda_n$ if $n \leq m$. By Assumption \ref{assump:rect}(c), we then
have
\[\blambda_m \toW \Lambda_m, \qquad \blambda_n \toW \Lambda_n\]
where $\Lambda_m$ denotes a mixture of $\Lambda$ and the point mass at 0
when $\gamma>1$, and $\Lambda_n$ denotes such a mixture when $\gamma<1$.

Let $\tilde{\I}_{t-1}^{(2k+1)} \in \RR^{(t-1) \times t}$ denote the first $t-1$
rows of $\I_t^{(2k+1)} \in \RR^{t \times t}$. The following extended
lemma implies Theorem \ref{thm:rect}, where parts (b) and (e) identify the
almost sure limits of (\ref{eq:keylimitrect1}--\ref{eq:keylimitrect4}).

\begin{lemma}
Suppose Assumption \ref{assump:rect} holds. Almost surely for each
$t=1,2,3,\ldots$:
\begin{enumerate}[(a)]
\item For all fixed $j,k \geq 0$, there exist deterministic limit matrices
\begin{align*}
(\bDelta_t^\infty,\bPhi_t^\infty,\bTheta_t^{(j,\infty)},\B_t^\infty,\bOmega_t^\infty,\H_t^{(2k,\infty)},\tilde{\I}_{t-1}^{(2k+1,\infty)})
&=\lim_{m,n \to \infty}
(\bDelta_t,\bPhi_t,\bTheta_t^{(j)},\B_t,\bOmega_t,\H_t^{(2k)},\tilde{\I}_{t-1}^{(2k+1)})
\end{align*}
\item For some random variables $R_1,\ldots,R_t,\bar{P}_1,\ldots,\bar{P}_{t-1}$
with finite moments of all orders,
\[(\r_1,\ldots,\r_t,\bLambda\p_1,\ldots,\bLambda\p_{t-1},\blambda_m) \toW
(R_1,\ldots,R_t,\bar{P}_1,\ldots,\bar{P}_{t-1},\Lambda_m).\]
For each $k \geq 0$,
\begin{align*}
\EE[(R_1,\ldots,R_t)^\top \Lambda_m^{2k}(R_1,\ldots,R_t)]&=\H_t^{(2k,\infty)}\\
\EE[(\bar{P}_1,\ldots,\bar{P}_{t-1})^\top \Lambda_m^{2k}
(R_1,\ldots,R_t)]&=\tilde{\I}_{t-1}^{(2k+1,\infty)}.
\end{align*}
\item $(\v_1,\ldots,\v_t,\z_1,\ldots,\z_t,\F) \toW (V_1,\ldots,V_t,Z_1,\ldots,Z_t,F)$
as described in Theorem \ref{thm:rect}.
\item For all fixed $j,k \geq 0$, there exist deterministic limit matrices
\begin{align*}
(\bGamma_t^\infty,\bPsi_t^\infty,\bXi_t^{(j,\infty)},\A_t^\infty,\bSigma_t^\infty,\L_t^{(2k,\infty)},\J_t^{(2k+1,\infty)})
&=\lim_{m,n \to \infty}
(\bGamma_t,\bPsi_t,\bXi_t^{(j)},\A_t,\bSigma_t,\L_t^{(2k)},\J_t^{(2k+1)})
\end{align*}
\item For some random variables $P_1,\ldots,P_t,\bar{R}_1,\ldots,\bar{R}_t$ with
finite moments of all orders,
\[(\p_1,\ldots,\p_t,\bLambda^\top \r_1,\ldots,\bLambda^\top \r_t,\blambda_n)
\toW (P_1,\ldots,P_t,\bar{R}_1,\ldots,\bar{R}_t,\Lambda_n).\]
For each $k \geq 0$,
\begin{align*}
\EE[(\bar{R}_1,\ldots,\bar{R}_t)^\top \Lambda_n^{2k}(P_1,\ldots,P_t)]
&=\J_t^{(2k+1,\infty)},\\
\EE[(P_1,\ldots,P_t)^\top \Lambda_n^{2k}(P_1,\ldots,P_t)]
&=\L_t^{(2k,\infty)}.
\end{align*}
\item $(\u_1,\ldots,\u_{t+1},\y_1,\ldots,\y_t,\E)
\toW (U_1,\ldots,U_{t+1},Y_1,\ldots,Y_t,E)$
as described in Theorem \ref{thm:rect}.
\item The matrices
\[\begin{pmatrix} \bDelta_t^\infty & \bPhi_t^\infty \bSigma_t^\infty \\
\bSigma_t^\infty (\bPhi_t^\infty)^\top & \bSigma_t^\infty \end{pmatrix},
\qquad
\begin{pmatrix} \bGamma_t^\infty & \bPsi_t^\infty \bOmega_t^\infty \\
\bOmega_t^\infty (\bPsi_t^\infty)^\top & \bOmega_t^\infty \end{pmatrix}\]
are both non-singular.
\end{enumerate}
\end{lemma}

\begin{proof}
Denote by $t^{(a)},t^{(b)},\ldots$ the claims of part (a), part (b), etc.\ up
to and including iteration $t$. We induct on $t$.
We will omit details of the argument that are similar to the proof
of Theorem \ref{thm:main}.\\

{\bf Step 1: $t=1$.} We have
$\bDelta_1 \to \bDelta_1^\infty=\EE[U_1^2]$, $\bPhi_1=0$, $\B_1=0$,
and $\kappa_{2k} \to \kappa_{2k}^\infty$ where
$\kappa_{2k}^\infty$ are the rectangular free cumulants defined by the limiting
moments
\begin{equation}\label{eq:mu2kinfty}
m_{2k}^\infty=\lim_{m,n \to \infty} \frac{1}{m}\sum_{i=1}^{\min(m,n)}
\lambda_i^{2k}=\EE[\Lambda_m^{2k}].
\end{equation}
Then also $(c_{2k,j},c_{2k+1,j},\bar{c}_{2k,j},\bar{c}_{2k+1,j}) \to
(c_{2k,j}^\infty,c_{2k+1,j}^\infty,\bar{c}_{2k,j}^\infty,
\bar{c}_{2k+1,j}^\infty)$,
and claim $1^{(a)}$ follows from the definitions.

Since $\r_1=\O\u_1$, by Proposition \ref{prop:orthognormal},
\begin{equation}\label{eq:r1rect}
(\r_1,\blambda_m) \toW (R_1,\Lambda_m)
\end{equation}
where $R_1 \sim \N(0,\EE[U_1^2])$ is independent of $\Lambda_m$. Then
$\EE[\Lambda_m^{2k}R_1^2]=\EE[\Lambda_m^{2k}]\EE[R_1^2]
=m_{2k}^\infty\EE[U_1^2]$. Note that $\bTheta_1^{(j)}=0$ for $j \geq 1$, so
$\H_1^{(2k)}=c_{2k,0}\bTheta_1^{(0)}=m_{2k}\langle \u_1^2 \rangle$,
the last equality using $m_{2k}=c_{2k,0}$ by Lemma \ref{lemma:ckjrect}.
Thus $\H_1^{(2k,\infty)}=m_{2k}^\infty \EE[U_1^2]$, and 
this shows $1^{(b)}$.

For $1^{(c)}$, note that $c_{2,0}=m_2=\kappa_2$. Then $1^{(b)}$ implies
\[\frac{1}{n}\|\bLambda^\top \r_1\|^2
=\frac{m}{n} \cdot \frac{1}{m}\sum_{i=1}^{\min(m,n)}
\lambda_i^2 r_{i1}^2 \to \gamma \cdot \EE[\Lambda_m^2R_1^2]
=\gamma \kappa_2^\infty\bDelta_1^{\infty}.\]
Since $\z_1=\Q^\top \bLambda^\top \r_1$,
Proposition \ref{prop:orthognormal} shows
\begin{equation}\label{eq:z1rect}
(\z_1,\F) \toW (Z_1,F)
\end{equation}
where $Z_1 \sim \N(0,\gamma \kappa_2^\infty \bDelta_1^\infty)$ is
independent of
$F$. Identifying $\bOmega_1^\infty=\gamma \kappa_2^\infty
\bDelta_1^\infty$ and applying Proposition \ref{prop:composition} for the joint 
convergence with $\v_1=v_1(\z_1,\F)$, this shows $1^{(c)}$.

Observe that $1^{(c)}$ implies
$\bGamma_1 \to \bGamma_1^\infty=\EE[V_1^2]$. As $\partial_1 v_1$
satisfies (\ref{eq:growth}) and is continuous on a set of probability 1 under
the limit law $(Z_1,F)$, we have also
$\bPsi_1 \to \bPsi_1^\infty=\EE[\partial_1 v_1(Z_1,F)]$
by Proposition \ref{prop:discontinuous}. Then $1^{(d)}$ follows from the
definitions.

For $1^{(e)}$, define $\check{\r}_1 \in \RR^n$ as the first $n$ entries of
$\r_1$ if $n \leq m$, or $\r_1$ extended by $n-m$ additional i.i.d.\
$\N(0,\EE[U_1^2])$ random variables if $n>m$. By Proposition
\ref{prop:orthognormal}(b),
\[(\check{\r}_1,\blambda_n) \toW (\check{R}_1,\Lambda_n)\]
where $\check{R}_1 \sim \N(0,\EE[U_1^2])$ is independent of $\Lambda_n$. Note
that $\bLambda^\top \r_1 \in \RR^n$ may be written as the entrywise product of
$\blambda_n$ with $\check{\r}_1$, in both cases $n \leq m$ and $n>m$. Thus
\[(\bLambda^\top \r_1,\blambda_n) \toW (\bar{R}_1,\Lambda_n),
\qquad \bar{R}_1=\Lambda_n \check{R}_1.\]
To analyze the joint convergence with $\p_1$, we now condition on
$\u_1,\r_1,\z_1,\v_1,\blambda,\E,\F$.
The law of $\Q$ is then conditioned on the event
$\bLambda^\top \r_1=\Q\z_1$. As $\Var[Z_1]=\gamma \kappa_2^\infty
\EE[U_1^2]>0$ by the given assumptions, we have
$n^{-1}\z_1^\top \z_1 \neq 0$ for all large $n$.
Then by Proposition \ref{prop:orthogconditioning},
the conditional law of $\Q$ is
\[\bLambda^\top \r_1(\z_1^\top \z_1)^{-1}\z_1^\top+\proj_{(\bLambda^\top
\r_1)^\perp} \tilde{\Q}\proj_{\z_1^\perp}^\top\]
where $\tilde{\Q} \in \RR^{(n-1) \times (n-1)}$ is an independent
Haar-orthogonal matrix. So we may replace the update
$\p_1=\Q\v_1$ by
\[\p_1=\p_\parallel+\p_\perp, \qquad
\p_\parallel=\bLambda^\top \r_1(\z_1^\top \z_1)^{-1}\z_1^\top \v_1, \qquad
\p_\perp=\proj_{(\bLambda^\top \r_1)^\perp} \tilde{\Q}
\proj_{\z_1^\perp}^\top\v_1.\]
By $1^{(c)}$ and Proposition \ref{prop:stein},
$n^{-1} \z_1^\top \z_1 \to \bOmega_1^\infty$ and $n^{-1} \z_1^\top \v_1 \to
\bPsi_1^\infty\bOmega_1^\infty$. Then by Proposition \ref{prop:scalarprod},
\[(\p_\parallel,\bLambda^\top \r_1,\blambda_n) \toW
(P_\parallel,\bar{R}_1,\Lambda_n), \qquad P_\parallel=\bPsi_1^\infty
\cdot \bar{R}_1.\]
For $\p_\perp$, observe that
\[n^{-1}\|\proj_{\z_1^\perp}^\top\v_1\|^2
=n^{-1}\left(\|\v_1\|^2-\frac{(\v_1^\top \z_1)^2}{\|\z_1\|^2}\right)
\to \bGamma_1^\infty-(\bPsi_1^\infty)^2\bOmega_1^\infty,\]
so Proposition \ref{prop:orthognormal} shows
\begin{equation}\label{eq:pperpiter1}
\p_\perp \toW P_\perp \sim
\N\Big(0,\,\bGamma_1^\infty-(\bPsi_1^\infty)^2\bOmega_1^\infty\Big)
\end{equation}
where $P_\perp$ is independent of $(\bar{R}_1,\Lambda_n)$. Then
\[(\p_1,\bLambda^\top \r_1,\blambda_n) \toW (P_1,\bar{R}_1,\Lambda_n), \qquad
P_1=\bPsi_1^\infty \cdot \bar{R}_1+P_\perp.\]
Since $P_\perp$ has mean 0 and is independent of $(\bar{R}_1,\Lambda_n)$, we
have
\begin{align*}
\EE[\Lambda_n^{2k}P_1\bar{R}_1]
&=\bPsi_1^\infty \cdot \EE[\Lambda_n^{2k}\bar{R}_1^2]\\
&=\bPsi_1^\infty \cdot
\lim_{n \to \infty} n^{-1} \r_1^\top \bLambda(\bLambda^\top
\bLambda)^k\bLambda^\top \r_1\\
&=\gamma \bPsi_1^\infty \cdot \lim_{n \to \infty} m^{-1}
\r_1^\top (\bLambda\bLambda^\top)^{k+1}\r_1\\
&=\gamma\bPsi_1^\infty  \cdot \EE[\Lambda_m^{2k+2}R_1^2]
=\gamma c_{2k+2,0}^\infty \bPsi_1^\infty \bDelta_1^\infty
\end{align*}
where the last equality applies $1^{(b)}$.
By the recursion (\ref{eq:ckjrecteven}) and identity (\ref{eq:cbarc}), we have
$c_{2k+2,0}=c_{2k+1,0}=\gamma^{-1} \bar{c}_{2k+1,0}$. We have also
$\X_1^{(0)}=\bPsi_1\bDelta_1$ and $\X_1^{(j)}=0$ for $j \geq 1$, because
$\bPhi_1=0$. Then $\J_1^{(2k+1)}=\bar{c}_{2k+1,0}\bPsi_1\bDelta_1$,
so the above is simply
\[\EE[\Lambda_n^{2k}P_1\bar{R}_1]=\J_1^{(2k+1,\infty)}.\]
Similarly, applying the above identity and also
$\EE[\Lambda_n^{2k}]=\bar{m}_{2k}^\infty=\bar{c}_{2k,0}^\infty$ by Lemma
\ref{lemma:ckjrect}, we have
\begin{align*}
\EE[\Lambda_n^{2k}P_1^2]&=(\bPsi_1^\infty)^2 \cdot
\EE[\Lambda_n^{2k}\bar{R}_1^2]+\EE[\Lambda_n^{2k}P_\perp^2]\\
&=\bar{c}_{2k+1,0}^\infty (\bPsi_1^\infty)^2\bDelta_1^\infty
+\bar{c}_{2k,0}^\infty \cdot
(\bGamma_1^\infty-(\bPsi_1^\infty)^2\bOmega_1^\infty).
\end{align*}
Identifying $\bOmega_1=\gamma \kappa_2 \bDelta_1=\bar{\kappa}_2\bDelta_1$,
applying $\bar{c}_{2k+1,0}=\bar{c}_{2k,0}\bar{\kappa}_2
+\bar{c}_{2k,1}$ by (\ref{eq:barckjrectodd}), and then identifying
$\bXi_1^{(0)}=\bGamma_1$, $\bXi_1^{(1)}=\bPsi_1^2\bDelta_1$,
and $\bXi_1^{(j)}=0$ for $j \geq 2$, this is
\[\EE[\Lambda_n^{2k}P_1^2]=
(\bar{c}_{2k+1,0}^\infty-\bar{c}_{2k,0}^\infty
\bar{\kappa}_2^\infty)(\bPsi_1^\infty)^2\bDelta_1^\infty
+\bar{c}_{2k,0}^\infty\bGamma_1^\infty
=\bar{c}_{2k,1}^\infty \bXi_1^{(1,\infty)}+\bar{c}_{2k,0}^\infty
\bXi_1^{(0,\infty)}=\L_1^{(2k,\infty)}.\]
This shows $1^{(e)}$.

For $1^{(f)}$, we now condition on
$\u_1,\r_1,\z_1,\v_1,\p_1,\blambda,\E,\F$. Then the law of $\O$ is
conditioned on the event $\r_1=\O\u_1$. By assumption,
$\EE[U_1^2]>0$, so $m^{-1}\u_1^\top \u_1 \neq 0$ for all large $m$.
Then the conditional law of $\O$ is
\[\r_1(\u_1^\top
\u_1)^{-1}\u_1^\top+\proj_{\r_1^\perp}\tilde{\O}\proj_{\u_1^\perp}^\top.\]
So the update for $\q_1=\O^\top \bLambda \p_1$ may be replaced by
\[\q_1=\q_\parallel+\q_\perp, \qquad 
\q_\parallel=\u_1(\u_1^\top \u_1)^{-1}\r_1^\top \bLambda \p_1, \qquad
\q_\perp=\proj_{\u_1^\perp}\tilde{\O}^\top\proj_{\r_1^\perp}^\top\bLambda \p_1.\]
Applying $1^{(e)}$ and the above observation
$\J_1^{(1)}=\bar{c}_{1,0}\bPsi_1\bDelta_1=\gamma \kappa_2 \bPsi_1\bDelta_1$,
we have
\begin{align*}
\frac{1}{m}\r_1^\top \bLambda \p_1
&=\frac{n}{m} \cdot \frac{1}{n} \r_1^\top \bLambda \p_1
\to \gamma^{-1} \J_1^{(1,\infty)}=
\kappa_2^\infty \bPsi_1^\infty \bDelta_1^\infty
=a_{11}^\infty \bDelta_1^\infty,\\
\frac{1}{m}\|\proj_{\r_1^\perp}^\top \bLambda \p_1\|^2
&=\frac{n}{m} \cdot \frac{1}{n}\left(\|\bLambda \p_1\|^2
-\frac{(\r_1^\top \bLambda \p_1)^2}{\|\r_1\|^2} \right)\\
&\to \gamma^{-1}
\cdot\left(\L_1^{(2,\infty)}-\frac{(\J_1^{(1,\infty)})^2}{\gamma \cdot
\bDelta_1^\infty}\right)
=\left(\gamma^{-1} \cdot
\left(\bar{c}_{2,1}\bXi_1^{(1)}+\bar{c}_{2,0} \bXi_1^{(0)}
-\gamma \kappa_2^2 \bPsi_1^2\bDelta_1\right)\right)^{\infty}.
\end{align*}
Recalling from the above that $\bXi_1^{(1)}=\bPsi_1^2\bDelta_1$, and applying
$\gamma^{-1}\bar{c}_{2,1}=\gamma^{-1}(\bar{c}_{1,0}\kappa_2+\bar{c}_{1,1})
=\kappa_2^2+\kappa_4$ and
$\gamma^{-1}\bar{c}_{2,0}=\gamma^{-1}\bar{c}_{1,0}=\kappa_2$, this yields
\begin{equation}\label{eq:Sigma1iter1}
\frac{1}{m}\|\proj_{\r_1^\perp}^\top \bLambda \p_1\|^2
\to \kappa_4^\infty\bXi_1^{(1,\infty)}+\kappa_2^\infty
\bXi_1^{(0,\infty)}=\bSigma_1^\infty.
\end{equation}
So $\q_\perp \to Q_\perp \sim \N(0,\bSigma_1^\infty)$ where this is independent
of $(U_1,E)$, and
\[(\q_1,\u_1,\E) \toW (a_{11}^\infty U_1+Q_\perp,U_1,E).\]
Then applying $\y_1=\q_1-a_{11}\u_1$, $a_{11} \to a_{11}^\infty$ by $1^{(d)}$,
and Propositions \ref{prop:scalarprod} and \ref{prop:composition},
\[(\u_1,\u_2,\y_1,\E) \toW (U_1,U_2,Y_1,E)\]
where $Y_1=Q_\perp \sim \N(0,\bSigma_1^\infty)$ and $U_2=u_2(Y_1,E)$. This
yields $1^{(f)}$.

For $1^{(g)}$, observe first that $\bDelta_1^\infty=\EE[U_1^2]>0$, $\bPhi_1=0$,
and $\bOmega_1^\infty=\gamma \kappa_2^\infty \EE[U_1^2]>0$ by the given
assumptions. The Schur-complement
$\bGamma_1^\infty-(\bPsi_1^\infty)^2\bOmega_1^\infty$ in the second matrix of
$1^{(g)}$
is the residual variance of projecting of $V_1$ onto the span of $Z_1$, which is
positive by Assumption \ref{assump:rect}(f), so the second matrix of $1^{(g)}$
is invertible. By (\ref{eq:pperpiter1}), this shows also that
\[\Var[P_\perp]>0.\]
For the first matrix of $1^{(g)}$, it remains to show that $\bSigma_1^\infty>0$.
Note that by (\ref{eq:Sigma1iter1}), $\bSigma_1^\infty$ is the residual variance
of projecting $\bar{P}_1$ onto the span of $R_1$. If this were 0, then
$\bar{P}_1=\alpha R_1$ for some constant $\alpha \in \RR$ with probability 1.
Applying $1^{(b)}$ and $1^{(e)}$, we then have
\[0=\EE[\Lambda_m^2(\bar{P}_1-\alpha R_1)^2]
=\lim_{m,n \to \infty} m^{-1}\|\bLambda^\top \bLambda \p_1-\alpha \bLambda^\top \r_1\|^2
=\gamma^{-1} \cdot \EE[(\Lambda_n^2 P_1-\alpha \bar{R}_1)^2],\]
so also $\Lambda_n^2P_1=\alpha \bar{R}_1$ with probability 1.
Recalling $P_1=\bPsi_1^\infty \cdot \bar{R}_1+P_\perp$, this shows
$\Lambda_n^2 P_\perp=(\alpha-\bPsi_1^\infty \cdot \Lambda_n^2)\bar{R}_1$.
Since $\Lambda_n^2$ is not identically 0, and $P_\perp$ is
independent of $(\Lambda_n,\bar{R}_1)$, we must then have that $P_\perp$ is
constant with probability 1, but this contradicts that $P_\perp \sim
\N(0,\Var[P_\perp])$ where $\Var[P_\perp]>0$ as argued above. So
$\bSigma_1^\infty>0$, and the first matrix of $1^{(g)}$ is also invertible.\\

{\bf Step 2: Analysis of $\r_{t+1}$.} Suppose that $t^{(a-g)}$ hold. To show
$t+1^{(a)}$, observe that the limits $\bDelta_{t+1}^\infty$ and
$\bPhi_{t+1}^\infty$ exist by $t^{(f)}$,
Proposition \ref{prop:discontinuous}, and the given
conditions for the functions $\partial_{s'} u_s$. Furthermore,
by $t^{(d)}$ and $t^{(f)}$, the limits
$\bTheta_{t+1}^{(j,\infty)}$, $\B_{t+1}^\infty$, $\bOmega_{t+1}^\infty$, and
$\H_{t+1}^{(2k,\infty)}$ also exist because these matrices do not depend on
$\v_{t+1}$ or its derivatives. Each term constituting $\X_{t+1}^{(j)}$
in (\ref{eq:Xrect})
may be written as either $\bPsi_{t+1}$ or $\bGamma_{t+1}\bPhi_{t+1}^\top$ times
a matrix that depends only on $\bDelta_{t+1}$, $\bPhi_{t+1}\bPsi_{t+1}$,
and $\bPhi_{t+1}\bGamma_{t+1}\bPhi_{t+1}^\top$.
Then the first $t$ rows of $\X_{t+1}^{(j)}$ also do not depend on $\v_{t+1}$ or
its derivatives, so $\tilde{\I}_t^{(2k+1,\infty)}$ exists. This establishes
$t+1^{(a)}$.

Let us now show $t+1^{(b)}$. Define the matrices
\[\U_t=\begin{pmatrix} \u_1 & \cdots & \u_t \end{pmatrix},
\quad \R_t=\begin{pmatrix} \r_1 & \cdots & \r_t \end{pmatrix}, 
\quad \Z_t=\begin{pmatrix} \z_1 & \cdots & \z_t \end{pmatrix},\]
\[\V_t=\begin{pmatrix} \v_1 & \cdots & \v_t \end{pmatrix},
\quad \P_t=\begin{pmatrix} \p_1 & \cdots & \p_t \end{pmatrix}, 
\quad \Y_t=\begin{pmatrix} \y_1 & \cdots & \y_t \end{pmatrix}.\]
Conditional on the AMP iterates up to $\u_{t+1}$,
the law of $\O$ is conditioned on
\[\begin{pmatrix} \R_t & \bLambda \P_t \end{pmatrix}
\begin{pmatrix} \Id & -\A_t \\ \0 & \Id \end{pmatrix}
=\O\begin{pmatrix} \U_t & \Y_t \end{pmatrix}.\]
Let us introduce
\[\M_t=m^{-1}\begin{pmatrix} \U_t^\top \U_t & \U_t^\top \Y_t \\
\Y_t^\top \U_t & \Y_t^\top \Y_t \end{pmatrix},\]
noting that by $t^{(f)}$, $t^{(g)}$, and Proposition \ref{prop:stein},
\[\M_t \to \M_t^\infty
=\begin{pmatrix} \bDelta_t^\infty & \bPhi_t^\infty \bSigma_t^\infty \\
\bSigma_t^\infty (\bPhi_t^\infty)^\top & \bSigma_t^\infty \end{pmatrix}\]
where $\M_t^\infty$ is invertible. Then by
Proposition \ref{prop:orthogconditioning}, for all large $n$,
the conditional law of $\O$ is
\[\begin{pmatrix} \R_t & \bLambda \P_t \end{pmatrix}
\begin{pmatrix} \Id & -\A_t \\ \0 & \Id \end{pmatrix}
\M_t^{-1} \cdot m^{-1}\begin{pmatrix} \U_t^\top \\ \Y_t^\top \end{pmatrix}
+\proj_{(\R_t,\bLambda \P_t)^\perp}
\tilde{\O}\proj_{(\U_t,\Y_t)^\perp}^\top\]
where $\tilde{\O} \in \RR^{(m-2t) \times (m-2t)}$ is an independent
Haar-orthogonal matrix. We may thus replace the
update for $\r_{t+1}$ by
\begin{align*}
\r_{t+1}&=\r_\parallel+\r_\perp\\
\r_\parallel&=\begin{pmatrix} \R_t & \bLambda \P_t \end{pmatrix}
\begin{pmatrix} \Id & -\A_t \\ \0 & \Id \end{pmatrix} \M_t^{-1} \cdot m^{-1}
\begin{pmatrix} \U_t^\top \\ \Y_t^\top \end{pmatrix}\u_{t+1}\\
\r_\perp&=\proj_{(\R_t,\bLambda \P_t)^\perp}
\tilde{\O}\proj_{(\U_t,\Y_t)^\perp}^\top \u_{t+1}.
\end{align*}

For $\r_\parallel$, define
\[\bdelta_t^\infty=\begin{pmatrix} \EE[U_1U_{t+1}] \\ \vdots \\ \EE[U_tU_{t+1}]
\end{pmatrix} \in \RR^t, \qquad
\bphi_t^\infty=\begin{pmatrix} \EE[\partial_1 u_{t+1}(Y_1,\ldots,Y_t,E)] \\
\vdots \\ \EE[\partial_t u_{t+1}(Y_1,\ldots,Y_t,E)] \end{pmatrix} \in \RR^t,\]
which are the last columns of $\bDelta_{t+1}^\infty$ and
$(\bPhi_{t+1}^\infty)^\top$ with their last entries removed. Observe that
\[m^{-1}\U_t^\top \u_{t+1} \to \bdelta_t^\infty,
\qquad m^{-1} \Y_t^\top \u_{t+1}
\to \bSigma_t^\infty \bphi_t^\infty.\]
Then by arguments similar to the proof of Theorem \ref{thm:main},
\begin{equation}\label{eq:Rparallelrect}
\r_\parallel \toW R_\parallel=\begin{pmatrix} R_1 & \cdots & R_t &
\bar{P}_1 & \cdots & \bar{P}_t \end{pmatrix}
(\bUpsilon_t^\infty)^{-1}\begin{pmatrix} \bdelta_t^\infty \\ \bphi_t^\infty
\end{pmatrix}
\end{equation}
and $\bUpsilon_t^\infty$ is the limit of $\bUpsilon_t$ defined in
(\ref{eq:Upsilonrect}). Also,
\begin{equation}\label{eq:Rperprect}
\r_\perp \toW R_\perp \sim \N\left(0,\;\EE[U_{t+1}^2]
-\left(\begin{pmatrix} \bdelta_t \\ \bSigma_t\bphi_t \end{pmatrix}^\top
\begin{pmatrix} \bDelta_t & \bPhi_t \bSigma_t \\
\bSigma_t\bPhi_t^\top & \bSigma_t \end{pmatrix}^{-1}
\begin{pmatrix} \bdelta_t \\ \bSigma_t\bphi_t
\end{pmatrix}\right)^\infty\right)
\end{equation}
where this limit $R_\perp$ is independent of
$(R_1,\ldots,R_t,\bar{P}_1,\ldots,\bar{P}_t,\Lambda_m)$. So
\begin{equation}\label{eq:Rconvergencerect}
(\r_1,\ldots,\r_{t+1},\bLambda \p_1,\ldots,\bLambda \p_t,\blambda_m)
\toW (R_1,\ldots,R_{t+1},\bar{P}_1,\ldots,\bar{P}_t,\Lambda_m), 
\qquad R_{t+1}=R_\parallel+R_\perp.
\end{equation}
We have
\begin{equation}\label{eq:varRperprect}
\Var[R_\perp]>0
\end{equation}
because this is the residual variance of projecting
$U_{t+1}$ onto the span of $(U_1,\ldots,U_t,Y_1,\ldots,Y_t)$, which is positive
by Assumption \ref{assump:rect}(f).

Let us now introduce the block notation
\begin{equation}\label{eq:HIblockrect}
\H_{t+1}^{(2k,\infty)}=\begin{pmatrix} \H_t^{(2k,\infty)} &
\h_t^{(2k,\infty)} \\ (\h_t^{(2k,\infty)})^\top & h_{t+1,t+1}^{(2k,\infty)}
\end{pmatrix}, \qquad
\tilde{\I}_t^{(2k+1,\infty)}=\begin{pmatrix} \I_t^{(2k+1,\infty)} &
\i_t^{(2k+1,\infty)} \end{pmatrix}.
\end{equation}
To conclude the proof of $t+1^{(b)}$, it remains to show that
\begin{align}
\EE[(R_1,\ldots,R_t)^\top \Lambda_m^{2k}R_{t+1}]
&=\h_t^{(2k,\infty)}\label{eq:RtRtplus1rect}\\
\EE[(\bar{P}_1,\ldots,\bar{P}_t)^\top \Lambda_m^{2k}R_{t+1}]
&=\i_t^{(2k+1,\infty)}\label{eq:PtRtplus1rect}\\
\EE[\Lambda_m^{2k}R_{t+1}^2]&=h_{t+1,t+1}^{(2k,\infty)}.\label{eq:Rtplus1rect}
\end{align}

For (\ref{eq:RtRtplus1rect}), observe that by $t^{(b)}$ and $t^{(e)}$, we have
\[\EE[(R_1,\ldots,R_t)^\top \Lambda_m^{2k} (R_1,\ldots,R_t)]
=\H_t^{(2k,\infty)}\]
and
\begin{align*}
\EE[(R_1,\ldots,R_t)^\top \Lambda_m^{2k} (\bar{P}_1,\ldots,\bar{P}_t)]
&=\lim_{m,n \to \infty} \frac{1}{m}\R_t^\top (\bLambda\bLambda^\top)^{2k}
\bLambda \P_t\\
&=\gamma^{-1} \lim_{m,n \to \infty}
\frac{1}{n}\R_t^\top\bLambda (\bLambda^\top\bLambda)^{2k} \P_t\\
&=\gamma^{-1}
\EE[(\bar{R}_1,\ldots,\bar{R}_t)^\top \Lambda_n^{2k} (P_1,\ldots,P_t)]\\
&=\gamma^{-1} \J_t^{(2k+1,\infty)}=(\I_t^{(2k+1,\infty)})^\top,
\end{align*}
the last equality applying (\ref{eq:IJ}).
Then applying (\ref{eq:Rconvergencerect}),
(\ref{eq:Rparallelrect}), and the independence of $R_\perp$ from
$(R_1,\ldots,R_t,\bar{P}_1,\ldots,\bar{P}_t,\Lambda_m)$, we have
\[\EE[(R_1,\ldots,R_t)^\top \Lambda_m^{2k}R_{t+1}]
=\left(\begin{pmatrix} \H_t^{2k} & (\I_t^{2k+1})^\top
\end{pmatrix} \bUpsilon_t^{-1} \begin{pmatrix}
\bdelta_t \\ \bphi_t \end{pmatrix}\right)^\infty.\]
Applying the first row of the identity (\ref{eq:Upsiloninvrect}),
and the definitions of $\bTheta_{t+1}^{(j)}$, $\X_{t+1}^{(j)}$,
and $\H_{t+1}^{(2k)}$ from (\ref{eq:Thetarect}), (\ref{eq:Xrect}),
and (\ref{eq:HIJLrect}),
\begin{align*}
\EE[(R_1,\ldots,R_t)^\top \Lambda_m^{2k}R_{t+1}]
&=\left(\sum_{j=0}^\infty c_{2k,j} (\bPhi_t\bPsi_t)^j \bdelta_t
+\sum_{j=0}^\infty c_{2k,j+1} (\X_t^{(j)})^\top\bphi_t\right)^\infty\\
&=\left(\sum_{j=0}^\infty c_{2k,j} (\bPhi_{t+1}\bPsi_{t+1})^j \bDelta_{t+1}
+\sum_{j=0}^\infty c_{2k,j+1} (\X_{t+1}^{(j)})^\top\bPhi_{t+1}^\top
\right)_{1:t,\,t+1}^\infty\\
&=\left(\sum_{j=0}^\infty c_{2k,j}\bTheta_{t+1}^{(j)}\right)_{1:t,\,t+1}^\infty
=\h_t^{(2k,\infty)}.
\end{align*}

For (\ref{eq:PtRtplus1rect}), observe that also by $t^{(e)}$,
\begin{align*}
\EE[(\bar{P}_1,\ldots,\bar{P}_t)^\top \Lambda_m^{2k} (\bar{P}_1,\ldots,\bar{P}_t)]
&=\lim_{m,n \to \infty} \frac{1}{m}\P_t^\top \bLambda^\top
(\bLambda\bLambda^\top)^{2k} \bLambda \P_t\\
&=\gamma^{-1} \lim_{m,n \to \infty}
\frac{1}{n}\P_t^\top(\bLambda^\top\bLambda)^{2k+2} \P_t\\
&=\gamma^{-1} \cdot \EE[(P_1,\ldots,P_t)^\top \Lambda_n^{2k+2}(P_1,\ldots,P_t)]
=\gamma^{-1} \cdot \L_t^{(2k+2,\infty)}.
\end{align*}
So
\[\EE[(\bar{P}_1,\ldots,\bar{P}_t)^\top \Lambda_m^{2k}R_{t+1}]
=\left(\begin{pmatrix} \I_t^{(2k+1)} & \gamma^{-1} \cdot
\L_t^{(2k+2)} \end{pmatrix}
\bUpsilon_t^{-1} \begin{pmatrix} \bdelta_t \\ \bphi_t
\end{pmatrix}\right)^\infty.\]
Applying the second row of the identity (\ref{eq:Upsiloninvrect}),
\begin{align*}
\EE[(\bar{P}_1,\ldots,\bar{P}_t)^\top \Lambda_m^{2k}R_{t+1}]
&=\left(\sum_{j=0}^\infty c_{2k+1,j} (\bPsi_t\bPhi_t)^j \bPsi_t \bdelta_t
+\sum_{j=0}^\infty c_{2k+1,j} \bXi_t^{(j)}\bphi_t\right)^\infty\\
&=\left(\sum_{j=0}^\infty c_{2k+1,j} (\bPsi_{t+1}\bPhi_{t+1})^j
\bPsi_{t+1} \bDelta_{t+1}+\sum_{j=0}^\infty c_{2k+1,j}
\bXi_{t+1}^{(j)}\bPhi_{t+1}^\top\right)_{1:t,\,t+1}^\infty\\
&=\left(\sum_{j=0}^\infty c_{2k+1,j}\X_{t+1}^{(j)}\right)_{1:t,\,t+1}^\infty
=\i_t^{(2k+1,\infty)}.
\end{align*}

For (\ref{eq:Rtplus1rect}), applying again (\ref{eq:Rconvergencerect}) and a 
computation similar to the proof of Theorem \ref{thm:main}, we have
\[\EE[\Lambda_m^{2k}R_{t+1}^2]
=\left(\begin{pmatrix} \bdelta_t \\ \bphi_t \end{pmatrix}^\top
(\bUpsilon_t^{-1})^\top \begin{pmatrix} \H_t^{(2k)} & (\I_t^{(2k+1)})^\top \\
\I_t^{(2k+1)} & \gamma^{-1} \cdot \L_t^{(2k+2)} \end{pmatrix} \bUpsilon_t^{-1}
\begin{pmatrix} \bdelta_t \\ \bphi_t \end{pmatrix}\right)^\infty
+\EE[\Lambda_m^{2k}R_\perp^2]\]
where, by independence of $R_\perp$ and $\Lambda_m$,
\[\EE[\Lambda_m^{2k}R_\perp^2]
=c_{2k,0}^\infty \left(\EE[U_{t+1}^2]-
\left(\begin{pmatrix} \bdelta_t \\ \bphi_t \end{pmatrix}^\top
(\bUpsilon_t^{-1})^\top \begin{pmatrix} \H_t^{(0)} & (\I_t^{(1)})^\top \\
\I_t^{(1)} & \gamma^{-1} \cdot \L_t^{(2)} \end{pmatrix} \bUpsilon_t^{-1}
\begin{pmatrix} \bdelta_t \\ \bphi_t \end{pmatrix}\right)^\infty\right).\]
Combining these and applying the identity (\ref{eq:Upsiloninvrect2}),
\begin{align*}
&\EE[\Lambda_m^{2k}R_{t+1}^2]\\
&=c_{2k,0}^\infty\EE[U_{t+1}^2]+\left(\sum_{j=0}^\infty c_{2k,j+1}
\left(\bdelta_t^\top \bPsi_t^\top (\bPhi_t^\top \bPsi_t^\top)^j \bphi_t
+\bphi_t^\top (\bPsi_t\bPhi_t)^j\bPsi_t\bdelta_t
+\bphi_t^\top \bXi_t^{(j)}\bphi_t\right)\right)^\infty\\
&=\left(c_{2k,0}\bDelta_{t+1}
+\sum_{j=1}^\infty c_{2k,j}
\Big(\bDelta_{t+1}(\bPsi_{t+1}^\top\bPhi_{t+1}^\top)^j
+(\bPhi_{t+1}\bPsi_{t+1})^j \bDelta_{t+1}
+\bPhi_{t+1} \bXi_{t+1}^{(j-1)}\bPhi_{t+1}^\top\Big)\right)_{t+1,t+1}^\infty\\
&=\left(\sum_{j=0}^\infty c_{2k,j} \bTheta_{t+1}^{(j)}\right)_{t+1,t+1}^\infty
=h_{t+1,t+1}^{(2k,\infty)}.
\end{align*}
This completes the proof of $t+1^{(b)}$.

Let us make here the following additional observation: This also shows
\begin{equation}\label{eq:PbarRpartial}
(\p_1,\ldots,\p_t,\bLambda^\top \r_1,\ldots,\bLambda^\top \r_{t+1},
\blambda_n) \toW (P_1,\ldots,P_t,\bar{R}_1,\ldots,\bar{R}_{t+1},\Lambda_n)
\end{equation}
for a certain limit $\bar{R}_{t+1}$, which is part of the claim in $t+1^{(e)}$.
This is because from the decomposition
$\r_{t+1}=\r_\parallel+\r_\perp$, we have
$\bLambda^\top \r_{t+1}=\bLambda^\top \r_\parallel+\bLambda^\top \r_\perp$.
From the form of $\r_\parallel$ and claim $t^{(e)}$, we have
\[\bLambda^\top \r_\parallel \toW
\bar{R}_\parallel=\begin{pmatrix} \bar{R}_1 & \cdots & \bar{R}_t &
\Lambda_n^2 P_1 & \cdots & \Lambda_n^2 P_t \end{pmatrix}
(\bUpsilon_t^\infty)^{-1} \begin{pmatrix} \bdelta_t^\infty \\ \bphi_t^\infty
\end{pmatrix}.\]
For $\bLambda^\top \r_\perp$, let us define $\check{\r}_\perp \in \RR^n$
to be the first $n$ entries of $\r_\perp$ if $n \leq m$, or
$\r_\perp$ extended by an additional $n-m$
i.i.d.\ $\N(0,\Var[R_\perp])$ variables if $n>m$.
By Proposition \ref{prop:orthognormal}(b), $\check{\r}_\perp \toW
\check{R}_\perp$
in both cases, where this limit $\check{R}_\perp$ has the same normal law as
$R_\perp$ above,
and is independent of $(P_1,\ldots,P_t,\bar{R}_1,\ldots,\bar{R}_t,\Lambda_n)$. 
Since $\bLambda^\top \r_\perp$ is the entrywise product of $\blambda_n$
with $\check{\r}_\perp$, this shows that (\ref{eq:PbarRpartial}) holds where
\[\bar{R}_{t+1}=\bar{R}_\parallel+\Lambda_n \check{R}_\perp.\]
Furthermore,
\begin{align}
\EE[(P_1,\ldots,P_t)^\top \Lambda_n^{2k}(\bar{R}_1,\ldots,\bar{R}_{t+1})]
&=\lim_{m,n \to \infty} n^{-1} \P_t^\top (\bLambda^\top \bLambda)^{2k}
\bLambda^\top \R_{t+1}\nonumber\\
&=\gamma \cdot \EE[(\bar{P}_1,\ldots,\bar{P}_t)^\top
\Lambda_m^{2k}(R_1,\ldots,R_{t+1})]\nonumber\\
&=\gamma \cdot \tilde{\I}_t^{(2k+1,\infty)},\label{eq:PbarRpartialmoment}\\
\EE[(\bar{R}_1,\ldots,\bar{R}_{t+1})^\top \Lambda_n^{2k}
(\bar{R}_1,\ldots,\bar{R}_{t+1})]
&=\lim_{m,n \to \infty} n^{-1}
\R_{t+1}^\top \bLambda(\bLambda^\top\bLambda)^k \bLambda^\top \R_{t+1}\nonumber\\
&=\gamma \cdot
\EE[(R_1,\ldots,R_{t+1})^\top\Lambda_m^{2k+2}(R_1,\ldots,R_{t+1})]\nonumber\\
&=\gamma \cdot \H_{t+1}^{(2k+2,\infty)}.\label{eq:RbarRpartialmoment}
\end{align}

{\bf Step 3: Analysis of $\z_{t+1}$.} We now show $t+1^{(c)}$.
Conditioning on the iterates up to $\r_{t+1}$, the law of $\Q$ is conditioned on
\[\begin{pmatrix} \P_t & \bLambda^\top \R_t \end{pmatrix}
\begin{pmatrix} \Id & -\B_t \\ \0 & \Id \end{pmatrix}
=\Q\begin{pmatrix} \V_t & \Z_t \end{pmatrix}\]
Set
\[\bN_t=n^{-1}\begin{pmatrix} \V_t^\top \V_t & \V_t^\top \Z_t \\
\Z_t^\top \V_t & \Z_t^\top \Z_t \end{pmatrix},\]
and note that by $t^{(c)}$, $t^{(g)}$, and Proposition \ref{prop:stein},
\[\bN_t \to \bN_t^\infty
=\begin{pmatrix} \bGamma_t^\infty & \bPsi_t^\infty \bOmega_t^\infty
\\ \bOmega_t^\infty (\bPsi_t^\infty)^\top & \bOmega_t^\infty \end{pmatrix}\]
where $\bN_t^\infty$ is invertible. Thus, for all large $n$,
the conditional law of $\Q$ is
\[\begin{pmatrix} \P_t & \bLambda^\top \R_t \end{pmatrix}
\begin{pmatrix} \Id & -\B_t \\ \0 & \Id \end{pmatrix}
\bN_t^{-1} \cdot n^{-1}\begin{pmatrix} \V_t^\top \\ \Z_t^\top \end{pmatrix}
+\proj_{(\P_t,\bLambda^\top \R_t)^\perp}
\tilde{\Q}\proj_{(\V_t,\Z_t)^\perp}^\top\]
where $\tilde{\Q} \in \RR^{(n-2t) \times (n-2t)}$ is an independent
Haar-orthogonal matrix. So we may replace the update
for $\s_{t+1}=\Q^\top \bLambda^\top \r_{t+1}$ by
\begin{align*}
\s_{t+1}&=\s_\parallel+\s_\perp\\
\s_\parallel&=\begin{pmatrix} \V_t & \Z_t \end{pmatrix} \bN_t^{-1}
\begin{pmatrix} \Id & \0 \\ -\B_t^\top & \Id \end{pmatrix}
\cdot n^{-1} \begin{pmatrix} \P_t^\top \\ \R_t^\top \bLambda
\end{pmatrix} \bLambda^\top \r_{t+1},\\
\s_\perp&=\proj_{(\V_t,\Z_t)^\perp} \tilde{\Q}^\top
\proj_{(\P_t,\bLambda^\top \R_t)^\perp}^\top \bLambda^\top \r_{t+1}
\end{align*}
Applying $t+1^{(b)}$ shown above, and recalling the block notation
(\ref{eq:HIblockrect}), observe that
\begin{align*}
n^{-1}\P_t^\top \bLambda^\top \r_{t+1} &\to
\gamma \cdot \EE[(\bar{P}_1,\ldots,\bar{P}_t)^\top R_{t+1}]
=\gamma \cdot \i_t^{(1,\infty)}\\
n^{-1} \R_t^\top \bLambda\bLambda^\top \r_{t+1}
&\to \gamma \cdot \EE[(R_1,\ldots,R_t)^\top \Lambda_m^2 R_{t+1}]
=\gamma \cdot \h_t^{(2,\infty)}.
\end{align*}
Then by a computation analogous to the proof of Theorem \ref{thm:main},
\[\bN_t^{-1}
\begin{pmatrix} \Id & \0 \\ -\B_t^\top & \Id \end{pmatrix}
\cdot n^{-1} \begin{pmatrix} \P_t^\top \\ \R_t^\top \bLambda
\end{pmatrix} \bLambda^\top \r_{t+1}
\to \left(\begin{pmatrix} \Id & 0 \\ 0 & \bOmega_t^{-1}
\end{pmatrix} (\T_t^{-1})^\top \begin{pmatrix} \gamma \cdot \i_t^{(1)} \\ \gamma \cdot \h_t^{(2)} \end{pmatrix}\right)^\infty\]
where $\T_t$ is as defined in (\ref{eq:Upsilonrect}). Applying the second row of
the identity (\ref{eq:Tinvrect}) with $t+1$ and with $k=0$,
and recalling $\bar{c}_{1,j}=\bar{\kappa}_{2(j+1)}=\gamma \cdot \kappa_{2(j+1)}$, we have
\[\begin{pmatrix} (\J_{t+1}^{(1)})^\top
\\ \gamma \cdot \H_{t+1}^{(2)} \end{pmatrix}
=\T_{t+1}^\top \begin{pmatrix} \sum_{j=0}^\infty \gamma \cdot \kappa_{2(j+1)}
\bPhi_{t+1}^\top (\bPsi_{t+1}^\top \bPhi_{t+1}^\top)^j \\
\sum_{j=0}^\infty \gamma \cdot \kappa_{2(j+1)} \bTheta_{t+1}^{(j)} \end{pmatrix}
=\T_{t+1}^\top \begin{pmatrix} \B_{t+1} \\ \bOmega_{t+1} \end{pmatrix}.\]
Writing the block forms
\[\B_{t+1}=\begin{pmatrix} \B_t & \b_t \\ 0 & 0 \end{pmatrix},
\qquad \bOmega_{t+1}=\begin{pmatrix} \bOmega_t & \bomega_t \\ \bomega_t^\top &
\omega_{t+1,t+1} \end{pmatrix},\]
and applying $(\J_{t+1}^{(1)})^\top=\gamma \cdot \I_{t+1}^{(1)}$,
this yields
\[\gamma \cdot \begin{pmatrix} \i_t^{(1)} \\ \h_t^{(2)} \end{pmatrix}
=\left(\T_{t+1}^\top \begin{pmatrix} \B_{t+1} \\ \bOmega_{t+1} \end{pmatrix}
\right)_{(1:t)\cup (t+2:2t+1),\,t+1}
=\T_t^\top \begin{pmatrix} \b_t \\ \bomega_t \end{pmatrix},\]
where the second equality follows because $\B_{t+1}$ is 0 in its lower-right
entry while $\T_{t+1}^\top$ is 0 in rows $1:t$ and $t+2:2t+1$ of its last
column. Inverting $\T_t^\top$ and applying this above,
\[\s_\parallel \toW S_\parallel=\begin{pmatrix} V_1 & \cdots & V_t \end{pmatrix}
\b_t^\infty+\begin{pmatrix} Z_1 & \cdots & Z_t \end{pmatrix}
(\bOmega_t^\infty)^{-1}\bomega_t^\infty.\]
Similar to the proof of Theorem \ref{thm:main}, we have also
\[\s_\perp \toW S_\perp \sim \N\left(0,\;
\left(\gamma \cdot h_{t+1,t+1}^{(2)}-
\begin{pmatrix} \gamma \cdot \i_t^{(1)} \\ \gamma \cdot \h_t^{(2)}
\end{pmatrix}^\top
\begin{pmatrix} \L_t^{(0)} & (\J_t^{(1)})^\top \\ \J_t^{(1)} &
\gamma \cdot \H_t^{(2)} \end{pmatrix}^{-1}\begin{pmatrix}
\gamma \cdot \i_t^{(1)} \\ \gamma \cdot \h_t^{(2)} \end{pmatrix}
\right)^\infty\right)\]
where $S_\perp$ is independent of $(V_1,\ldots,V_t,Z_1,\ldots,Z_t)$. So
\begin{equation}\label{eq:stplus1rect}
\s_{t+1} \toW S_{t+1}=\begin{pmatrix} V_1 & \cdots & V_t \end{pmatrix}
\b_t^\infty+\begin{pmatrix} Z_1 & \cdots & Z_t \end{pmatrix}
(\bOmega_t^\infty)^{-1}\bomega_t^\infty+S_\perp.
\end{equation}
Since $\z_{t+1}=\s_{t+1}-\V_t\b_t$, this yields
\[(\v_1,\ldots,\v_{t+1},\z_1,\ldots,\z_{t+1},\F)
\toW (V_1,\ldots,V_{t+1},Z_1,\ldots,Z_{t+1},F)\]
where $V_{t+1}=v_{t+1}(Z_1,\ldots,Z_{t+1},F)$ and
\[Z_{t+1}=\begin{pmatrix} Z_1 & \cdots & Z_t \end{pmatrix}
(\bOmega_t^\infty)^{-1}\bomega_t^\infty+S_\perp.\]

Thus $(Z_1,\ldots,Z_t,Z_{t+1})$ has a multivariate normal distribution.
To compute the covariance, observe that
since $S_\perp$ is independent of $(Z_1,\ldots,Z_t)$, we have
\[\EE[(Z_1,\ldots,Z_t)^\top
Z_{t+1}]=\bOmega_t^\infty(\bOmega_t^\infty)^{-1}\bomega_t^\infty=\bomega_t^\infty.\]
For $\EE[Z_{t+1}^2]$, squaring both sides of (\ref{eq:stplus1rect}),
applying
\[\EE[S_{t+1}^2]=\lim_{m,n \to \infty} n^{-1}\|\s_{t+1}\|^2
=\lim_{m,n \to \infty} n^{-1}\|\bLambda^\top\r_{t+1}\|^2
=\gamma \cdot h_{t+1,t+1}^{(2,\infty)},\]
and rearranging yields
\[\EE[Z_{t+1}^2]=\left(\Big(\gamma \cdot \H_{t+1}^{(2)}
-\B_{t+1}^\top \bGamma_{t+1} \B_{t+1}
-\B_{t+1}^\top \bPsi_{t+1}\bOmega_{t+1}
-\bOmega_{t+1}\bPsi_{t+1}^\top \B_{t+1}\Big)^\infty\right)_{t+1,t+1}.\]
Applying the identity (\ref{eq:H2rect}), this gives
$\EE[Z_{t+1}^2]=\omega_{t+1,t+1}^\infty$, and this concludes the proof of
$t+1^{(c)}$.

Let us make here the following additional observation: We have
\begin{equation}\label{eq:varSperprect}
\Var[S_\perp]>0
\end{equation}
above. This is because by (\ref{eq:PbarRpartialmoment}) and
(\ref{eq:RbarRpartialmoment}) for $k=0$,
the quantity $\Var[S_\perp]$ above may be seen to be the residual variance of
projecting $\bar{R}_{t+1}$ onto the linear span of
$(P_1,\ldots,P_t,\bar{R}_1,\ldots,\bar{R}_t)$. If this residual variance were 0,
then for some constants $\alpha_1,\ldots,\alpha_t,\beta_1,\ldots,\beta_t$
we would have $\bar{R}_{t+1}=\alpha_1 P_1+\ldots+\alpha_t P_t
+\beta_1 \bar{R}_1+\ldots+\beta_t \bar{R}_t$ with probability 1, so that
\begin{align*}
0&=\EE[\Lambda_n^2 \cdot (\bar{R}_{t+1}-\alpha_1 P_1-\ldots-\alpha_t P_t
-\beta_1 \bar{R}_1-\ldots-\beta_t \bar{R}_t)^2]\\
&=\lim_{m,n \to \infty}
n^{-1} \|\bLambda \bLambda^\top \r_{t+1}-\alpha_1 \bLambda \p_1-\ldots-\alpha_t
\bLambda\p_t-\beta_1 \bLambda \bLambda^\top \r_1-\ldots-\beta_t\bLambda
\bLambda^\top \r_t\|^2\\
&=\gamma \cdot \EE[(\Lambda_m^2 R_{t+1}-\alpha_1\bar{P}_1-\ldots
-\alpha_t\bar{P}_t-\beta_1\Lambda_m^2 R_1-\ldots-\beta_t \Lambda_m^2 R_t)^2].
\end{align*}
Thus also
\[\Lambda_m^2 R_{t+1}=\alpha_1\bar{P}_1+\ldots
+\alpha_t\bar{P}_t+\beta_1\Lambda_m^2 R_1+\ldots+\beta_t \Lambda_m^2 R_t\]
with probability 1. However, recall the decomposition
$R_{t+1}=R_\parallel+R_\perp$ where
$R_\perp$ is independent of
$(R_1,\ldots,R_t,\bar{P}_1,\ldots,\bar{P}_t,\Lambda_m)$. Thus we have
\[\Lambda_m^2 R_\perp=f(R_1,\ldots,R_t,\bar{P}_1,\ldots,\bar{P}_t,\Lambda_m)\]
where the quantity on the right does not depend on $R_\perp$. Since $\Lambda_m$
is not identically 0 by the condition $\Var[\Lambda]>0$ in
Assumption \ref{assump:rect}(f), this implies that $R_\perp$ must be constant
almost surely, contradicting (\ref{eq:varRperprect}) already shown. Thus,
(\ref{eq:varSperprect}) holds.\\

{\bf Step 4: Analysis of $\p_{t+1}$.} Note that $t^{(f)}$, $t+1^{(c)}$, and
the given conditions for the functions $\partial_{s'} v_s$
imply the existence of all limits in $t+1^{(d)}$. Let us now show
$t+1^{(e)}$. The joint convergence with $\bLambda^\top \r_{t+1}$ has been
established already in (\ref{eq:PbarRpartial}), so we proceed to analyze
$\p_{t+1}$.

For this, let
\[\tilde{\B}_t=\begin{pmatrix} \B_t & \b_t \end{pmatrix} \in \RR^{t \times
(t+1)},
\qquad \tilde{\bPsi}_t=\begin{pmatrix} \bPsi_t & \0 \end{pmatrix} \in \RR^{t
\times (t+1)}\]
be the first $t$ rows of $\B_{t+1}$ and $\bPsi_{t+1}$, and let
\[\tilde{\A}_t=\begin{pmatrix} \A_t \\ 0 \end{pmatrix} \in \RR^{(t+1) \times t},
\qquad \tilde{\bPhi}_t=\begin{pmatrix} \bPhi_t \\ (\bphi_t)^\top \end{pmatrix}
\in \RR^{(t+1) \times t}\]
be the first $t$ columns of $\A_{t+1}$ and $\bPhi_{t+1}$. Conditional on
the iterates up to $\v_{t+1}$, the law of $\Q$ is now conditioned on
\[\begin{pmatrix} \P_t & \bLambda^\top \R_{t+1} \end{pmatrix}
\begin{pmatrix} \Id & -\tilde{\B}_t \\ \0 & \Id \end{pmatrix}
=\Q \begin{pmatrix} \V_t & \Z_{t+1} \end{pmatrix}.\]
Let us denote
\[\tilde{\bN}_t=n^{-1}\begin{pmatrix} \V_t^\top \V_t & \V_t^\top \Z_{t+1} \\
\Z_{t+1}^\top \V_t & \Z_{t+1}^\top \Z_{t+1} \end{pmatrix},\]
noting that by $t+1^{(c)}$ and Proposition \ref{prop:stein},
\begin{equation}\label{eq:tildeNrect}
\tilde{\bN}_t \to \tilde{\bN}_t^\infty=\begin{pmatrix} \bGamma_t^\infty &
\tilde{\bPsi}_t^\infty\bOmega_{t+1}^\infty \\ \bOmega_{t+1}^\infty
(\tilde{\bPsi}_t^\infty)^\top & \bOmega_{t+1}^\infty \end{pmatrix}.
\end{equation}
The upper-left $2t \times 2t$ submatrix of $\tilde{\bN}_t^\infty$ is
$\bN_t^\infty$, which is
invertible by $t^{(g)}$. The Schur-complement of its lower-right entry is
the residual variance of projecting $Z_{t+1}$ onto the span of
$(V_1,\ldots,V_t,Z_1,\ldots,Z_t)$. Since $\z_{t+1}=\s_{t+1}-\V_t\b_t$, this is
the same as the residual variance of projecting $S_{t+1}$ onto the 
span of $(V_1,\ldots,V_t,Z_1,\ldots,Z_t)$, which is exactly $\Var[S_\perp]$ by
the convergence (\ref{eq:stplus1rect}) and the fact that $S_\perp$ is a
mean-zero variable independent of $(V_1,\ldots,V_t,Z_1,\ldots,Z_t)$. This was
shown to be non-zero in (\ref{eq:varSperprect}), so $\tilde{\bN}_t^\infty$
is invertible. Thus for all large $n$, the conditional law of $\Q$ is
\[\begin{pmatrix} \P_t & \bLambda^\top \R_{t+1} \end{pmatrix}
\begin{pmatrix} \Id & -\tilde{\B}_t \\ \0 & \Id \end{pmatrix}
\tilde{\bN}_t^{-1} \cdot n^{-1}
\begin{pmatrix} \V_t^\top \\ \Z_{t+1}^\top \end{pmatrix}
+\proj_{(\P_t,\bLambda^\top \R_{t+1})^\perp}
\tilde{\Q}\proj_{(\V_t,\Z_{t+1})^\perp}^\top\]
So we may replace the update for $\p_{t+1}=\Q\v_{t+1}$ by
\begin{align*}
\p_{t+1}&=\p_\parallel+\p_\perp\\
\p_\parallel&=\begin{pmatrix} \P_t & \bLambda^\top \R_{t+1} \end{pmatrix}
\begin{pmatrix} \Id & -\tilde{\B}_t \\ \0 & \Id \end{pmatrix}
\tilde{\bN}_t^{-1} \cdot n^{-1} \begin{pmatrix} \V_t^\top \\
\Z_{t+1}^\top \end{pmatrix} \v_{t+1}\\
\p_\perp&=\proj_{(\P_t,\bLambda^\top \R_{t+1})^\perp}
\tilde{\Q}\proj_{(\V_t,\Z_{t+1})^\perp}^\top\v_{t+1}
\end{align*}

For $\p_\parallel$, define
\[\bgamma_t^\infty=\begin{pmatrix} \EE[V_1V_{t+1}] \\ \vdots \\
\EE[V_tV_{t+1}] \end{pmatrix} \in \RR^t, \qquad
\tilde{\bpsi}_t^\infty=\begin{pmatrix}
\EE[\partial_1 v_{t+1}(Z_1,\ldots,Z_{t+1},F)] \\ \vdots \\
\EE[\partial_{t+1} v_{t+1}(Z_1,\ldots,Z_{t+1},F)] \end{pmatrix} \in \RR^{t+1}\]
and observe that
\[n^{-1} \V_t^\top \v_{t+1} \to \bgamma_t^\infty, \qquad
n^{-1}\Z_{t+1}^\top \v_{t+1} \to \bOmega_{t+1}^\infty \tilde{\bpsi}_t^\infty.\]
Then by a computation similar to the proof of $t+1^{(b)}$ above,
\[\p_\parallel \toW \begin{pmatrix}
P_1 & \ldots & P_t & \bar{R}_1 & \ldots & \bar{R}_{t+1} \end{pmatrix}
(\tilde{\T}_t^\infty)^{-1} \begin{pmatrix} \bgamma_t^\infty \\ \bpsi_t^\infty
\end{pmatrix}\]
where
\begin{equation}\label{eq:tildeTrect}
\tilde{\T}_t=\begin{pmatrix} \bGamma_t & \bGamma_t
\tilde{\B}_t+\tilde{\bPsi}_t\bOmega_{t+1} \\
\tilde{\bPsi}_t^\top & \tilde{\bPsi}_t^\top \tilde{\B}_t\end{pmatrix}.
\end{equation}
Also,
\[\p_\perp \toW P_\perp \sim \N\left(0,\; \EE[V_{t+1}^2]
-\left(\begin{pmatrix} \bgamma_t \\ \bOmega_{t+1} \tilde{\bpsi}_t
\end{pmatrix}^\top \begin{pmatrix} \bGamma_t & \tilde{\bPsi}_t\bOmega_{t+1} \\
\bOmega_{t+1}\tilde{\bPsi}_t^\top & \bOmega_{t+1} \end{pmatrix}^{-1}
\begin{pmatrix} \bgamma_t \\ \bOmega_{t+1} \tilde{\bpsi}_t \end{pmatrix}
\right)^\infty\right)\]
where $P_\perp$ is independent of
$(P_1,\ldots,P_t,\bar{R}_1,\ldots,\bar{R}_{t+1},\Lambda_n)$. So
\begin{equation}\label{eq:Pconvergencerect}
(\p_1,\ldots,\p_{t+1},\bLambda^\top \r_1,\ldots,\bLambda^\top
\r_{t+1},\blambda_n) \toW (P_1,\ldots,P_{t+1},\bar{R}_1,\ldots,\bar{R}_{t+1},
\Lambda_n), \qquad P_{t+1}=P_\parallel+P_\perp.
\end{equation}
We have
\begin{equation}\label{eq:varPperprect}
\Var[P_\perp]>0
\end{equation}
because this is the residual variance of projecting $V_{t+1}$ onto the span of
$(V_1,\ldots,V_t,Z_1,\ldots,Z_{t+1})$, which is positive by Assumption
\ref{assump:rect}(f).

Let us introduce the block notation
\begin{equation}\label{eq:LJblockrect}
\L_{t+1}^{(2k)}=\begin{pmatrix} \L_t^{(2k)} &
\l_t^{(2k)} \\ (\l_t^{(2k)})^\top & l_{t+1,t+1}^{(2k)}
\end{pmatrix}, \qquad 
\J_{t+1}^{(2k+1)}=\begin{pmatrix} \gamma \cdot (\tilde{\I}_t^{(2k+1)})^\top
& \tilde{\j}_t^{(2k+1)} \end{pmatrix}
\end{equation}
where $\tilde{\I}_t^{(2k+1)} \in \RR^{t \times (t+1)}$ forms the first $t$
rows of $\I_{t+1}^{(2k+1)}$ as previously defined, and thus
$\gamma \cdot (\tilde{\I}_t^{(2k+1)})^\top$ forms the first $t$ columns
of $\J_{t+1}^{(2k+1)}$ by the identity (\ref{eq:IJ}). To conclude the proof of
$t+1^{(e)}$, it remains to show that
\begin{align}
\EE[(\bar{R}_1,\ldots,\bar{R}_{t+1})^\top \Lambda_n^{2k}P_{t+1}]
&=\tilde{\j}_t^{(2k+1,\infty)}\label{eq:barRPrect}\\
\EE[(P_1,\ldots,P_t)^\top \Lambda_n^{2k} P_{t+1}]&=\l_t^{(2k,\infty)}
\label{eq:PtPtplus1rect}\\
\EE[\Lambda_n^{2k}P_{t+1}^2]&=l_{t+1,t+1}^{(2k,\infty)}.
\label{eq:Ptplus1rect}
\end{align}

The arguments are similar to those for $t+1^{(b)}$: For
(\ref{eq:barRPrect}), by the convergence (\ref{eq:Pconvergencerect}), the form
of $P_\parallel$, and the
identities (\ref{eq:PbarRpartialmoment}--\ref{eq:RbarRpartialmoment}), we have
\[\EE[(\bar{R}_1,\ldots,\bar{R}_{t+1})^\top \Lambda_n^{2k}P_{t+1}]
=\left(\begin{pmatrix} \gamma \cdot 
(\tilde{\I}_t^{(2k+1)})^\top & \gamma \cdot \H_{t+1}^{(2k+2)}
\end{pmatrix} \tilde{\T}_t^{-1}\begin{pmatrix} \bgamma_t \\
\tilde{\bpsi}_t \end{pmatrix}\right)^\infty.\]
Applying the second row of the identity (\ref{eq:Tinvrect}),
\[\begin{pmatrix} \J_{t+1}^{(2k+1)} & \gamma \cdot \H_{t+1}^{(2k+2)} \end{pmatrix}
=\begin{pmatrix}
\sum_{j=0}^\infty \bar{c}_{2k+1,j}(\bPhi_{t+1}\bPsi_{t+1})^j\bPhi_{t+1}
& \sum_{j=0}^\infty \bar{c}_{2k+1,j}\bTheta_{t+1}^{(j)} \end{pmatrix} \T_{t+1}.\]
Note that $\tilde{\T}_t$ defined in (\ref{eq:tildeTrect}) is the submatrix of
$\T_{t+1}$ with row and column $t+1$ removed. 
Furthermore, column $t+1$ of $\bPhi_{t+1}$ is 0. Thus, removing column $t+1$
from both sides of this identity yields
\[\begin{pmatrix}
\gamma \cdot (\tilde{\I}_t^{(2k+1)})^\top & \gamma \cdot
\H_{t+1}^{(2k+2)} \end{pmatrix}
=\begin{pmatrix} \sum_{j=0}^\infty \bar{c}_{2k+1,j} (\bPhi_{t+1}\bPsi_{t+1})^j
\tilde{\bPhi}_t & \sum_{j=0}^\infty
\bar{c}_{2k+1,j} \bTheta_{t+1}^{(j)} \end{pmatrix} \tilde{\T}_t.\]
Taking the limit $m,n \to \infty$, 
inverting $\tilde{\T}_t^\infty$, and applying this above,
\begin{align*}
\EE[(\bar{R}_1,\ldots,\bar{R}_{t+1})^\top \Lambda_n^{2k}P_{t+1}]
&=\left(\sum_{j=0}^\infty \bar{c}_{2k+1,j}\Big(
(\bPhi_{t+1}\bPsi_{t+1})^j\tilde{\bPhi}_t\bgamma_t
+\bTheta_{t+1}^{(j)}\tilde{\bpsi}_t\Big)\right)^\infty\\
&=\left(\sum_{j=0}^\infty \bar{c}_{2k+1,j} (\X_{t+1}^{(j)})^\top
\right)_{1:t+1,\,t+1}^\infty=\tilde{\j}_t^{(2k+1,\infty)}.
\end{align*}
Here, the last two equalities apply again the fact that the last column of
$\bPhi_{t+1}$ is 0, and the definitions of $\X_{t+1}^{(j)}$ and
$\J_{t+1}^{(2k+1)}$ in (\ref{eq:Xrect}) and (\ref{eq:HIJLrect}).

For (\ref{eq:PtPtplus1rect}), we have
\[\EE[(P_1,\ldots,P_t)^\top \Lambda_n^{2k} P_{t+1}]
=\left(\begin{pmatrix} \L_t^{(2k)} &
\gamma \cdot \tilde{\I}_t^{(2k+1)} \end{pmatrix}
\tilde{\T}_t^{-1} \begin{pmatrix} \bgamma_t \\
\tilde{\bpsi}_t \end{pmatrix}\right)^\infty.\]
Applying the first row of the identity (\ref{eq:Tinvrect}) with $t+1$,
and a similar argument of removing the $t+1^\text{th}$ rows and columns from
both sides, we obtain
\[\begin{pmatrix} \L_t^{(2k)} & \gamma \cdot \tilde{\I}_t^{(2k+1)}
\end{pmatrix}=\begin{pmatrix} \sum_{j=0}^\infty \bar{c}_{2k,j}
(\bPsi_t\bPhi_t)^j & \sum_{j=0}^\infty \bar{c}_{2k,j+1}\tilde{\X}_t^{(j)}
\end{pmatrix} \tilde{\T}_t\]
where $\tilde{\X}_t^{(j)}$ are the first $t$ rows of $\X_{t+1}^{(j)}$. Then
\[\EE[(P_1,\ldots,P_t)^\top \Lambda_n^{2k} P_{t+1}]
=\left(\sum_{j=0}^\infty (\bar{c}_{2k,j}
(\bPsi_t\bPhi_t)^j\bgamma_t+\sum_{j=0}^\infty \bar{c}_{2k,j+1}\tilde{\X}_t^{(j)}
\tilde{\bpsi}_t)\right)^\infty
=\l_{t+1}^{(2k,\infty)}.\]

For (\ref{eq:Ptplus1rect}), squaring both sides of (\ref{eq:Pconvergencerect})
and recalling $\EE[\Lambda_n^{2k}]=\bar{m}_{2k}^\infty=\bar{c}_{2k,0}^\infty$
from Lemma \ref{lemma:ckjrect}, we have
\begin{align*}
\EE[\Lambda_n^{2k}P_{t+1}^2]
&=\left(\begin{pmatrix} \bgamma_t \\ \tilde{\bpsi}_t \end{pmatrix}^\top
(\tilde{\T}_t^{-1})^\top
\begin{pmatrix} \L_t^{(2k)} & \gamma \cdot \tilde{\I}_t^{(2k+1)} \\
\gamma \cdot (\tilde{\I}_t^{(2k+1)})^\top & \gamma \cdot
\H_t^{(2k+2)} \end{pmatrix} \tilde{\T}_t^{-1}
\begin{pmatrix} \bgamma_t \\ \tilde{\bpsi}_t \end{pmatrix}\right)^\infty
+\EE[\Lambda_n^{2k}P_\perp^2],\\
\EE[\Lambda_n^{2k}P_\perp^2]
&=\bar{c}_{2k,0}^\infty
\left(\EE[V_{t+1}^2]
-\begin{pmatrix} \bgamma_t \\ \tilde{\bpsi}_t \end{pmatrix}^\top
(\tilde{\T}_t^{-1})^\top
\begin{pmatrix} \L_t^{(0)} & \gamma \cdot \tilde{\I}_t^{(1)} \\
\gamma \cdot (\tilde{\I}_t^{(1)})^\top & \gamma \cdot
\H_t^{(2)} \end{pmatrix} \tilde{\T}_t^{-1}
\begin{pmatrix} \bgamma_t \\ \tilde{\bpsi}_t \end{pmatrix}\right)^\infty.
\end{align*}
Applying (\ref{eq:Tinvrect2}) with $t+1$, and removing the $t+1^\text{th}$ rows
and columns from both sides, we have
\begin{align*}
&\begin{pmatrix} \L_t^{(2k)} & \gamma \cdot \tilde{\I}_t^{(2k+1)} \\
\gamma \cdot (\tilde{\I}_t^{(2k+1)})^\top & \gamma \cdot \H_{t+1}^{(2k+2)}
\end{pmatrix}-\bar{c}_{2k,0}\begin{pmatrix} \L_t^{(0)} & \gamma \cdot \tilde{\I}_t^{(1)} \\
\gamma \cdot (\tilde{\I}_t^{(1)})^\top & \gamma \cdot \H_{t+1}^{(2)}
\end{pmatrix}\\
&=\tilde{\T}_t^\top
\begin{pmatrix} 0 & \sum_{j=0}^\infty \bar{c}_{2k,j+1}\tilde{\bPhi}_t^\top
(\bPsi_{t+1}^\top\bPhi_{t+1}^\top)^j \\
\sum_{j=0}^\infty \bar{c}_{2k,j+1}(\bPhi_{t+1}\bPsi_{t+1})^j\tilde{\bPhi}_t
& \sum_{j=0}^\infty \bar{c}_{2k,j+1}\bTheta_{t+1}^{(j)} \end{pmatrix}
\tilde{\T}_t.
\end{align*}
Then combining the above,
\begin{align*}
&\EE[\Lambda_n^{2k}P_{t+1}^2]\\
&=\left(\bar{c}_{2k,0} \EE[V_{t+1}^2]
+\sum_{j=0}^\infty \bar{c}_{2k,j+1}(\bgamma_t^\top \tilde{\bPhi}_t^\top
(\bPsi_{t+1}^\top\bPhi_{t+1}^\top)^j \tilde{\bpsi}_t
+\tilde{\bpsi}_t^\top (\bPhi_{t+1}\bPsi_{t+1})^j\tilde{\bPhi}_t\bgamma_t
+\tilde{\bpsi}_t^\top \bTheta_{t+1}^{(j)}\tilde{\bpsi}_t)\right)^\infty\\
&=\left(\sum_{j=0}^\infty \bar{c}_{2k,j}\bXi_{t+1}^{(j)}\right)_{t+1,t+1}^\infty
=l_{t+1,t+1}^{(2k,\infty)}.
\end{align*}
This concludes the proof of $t+1^{(e)}$.\\

{\bf Step 5: Analysis of $\y_{t+1}$.}
Finally, let us show $t+1^{(f)}$ and $t+1^{(g)}$. Conditional on
the iterates up to $\p_{t+1}$, the law of $\O$ is now conditioned on
\[\begin{pmatrix} \R_{t+1} & \bLambda \P_t \end{pmatrix}
\begin{pmatrix} \Id & -\tilde{\A}_t \\ \0 & \Id \end{pmatrix}
=\O\begin{pmatrix} \U_{t+1} & \Y_t \end{pmatrix}.\]
Let us set
\[\tilde{\M}_t=m^{-1}
\begin{pmatrix} \U_{t+1}^\top \U_{t+1} & \U_{t+1}^\top \Y_t \\
\Y_t^\top \U_{t+1} & \Y_t^\top \Y_t \end{pmatrix},\]
noting that by $t+1^{(c)}$,
\[\tilde{\M}_t \to \tilde{\M}_t^\infty
=\begin{pmatrix} \bDelta_{t+1} & \tilde{\bPhi}_t \bSigma_t \\
\bSigma_t \tilde{\bPhi}_t^\top & \bSigma_t \end{pmatrix}.\]
This limit is invertible because its submatrix removing row and column
$t+1$ is $\tilde{\M}_t^\infty$ which is
invertible by $t^{(g)}$, and the Schur complement of the $(t+1,t+1)$
entry is exactly $\Var[R_\perp]$ from (\ref{eq:Rperprect}), which we have shown
is non-zero in (\ref{eq:varRperprect}). Then for all large $n$,
the conditional law of $\O$ is
\[\begin{pmatrix} \R_{t+1} & \bLambda \P_t \end{pmatrix}
\begin{pmatrix} \Id & -\tilde{\A}_t \\ \0 & \Id \end{pmatrix}
\tilde{\M}_t^{-1} \cdot m^{-1}
\begin{pmatrix} \U_{t+1}^\top \\ \Y_t^\top \end{pmatrix}
+\proj_{(\R_{t+1},\bLambda\P_t)^\perp}\tilde{\O}\proj_{(\U_{t+1},\Y_t)^\perp}^\top\]
So we may replace the update for $\q_{t+1}=\O^\top \bLambda \p_{t+1}$ by
\begin{align*}
\q_{t+1}&=\q_\parallel+\q_\perp\\
\q_\parallel&=\begin{pmatrix} \U_{t+1} & \Y_t \end{pmatrix}
\tilde{\M}_t^{-1}
\begin{pmatrix} \Id & \0 \\ -\tilde{\A}_t^\top & \Id \end{pmatrix}
\cdot m^{-1}
\begin{pmatrix} \R_{t+1}^\top \\ \P_t^\top \bLambda^\top \end{pmatrix}
\bLambda \p_{t+1}\\
\q_\perp&=\proj_{(\U_{t+1},\Y_t)^\perp}\tilde{\O}^\top
\proj_{(\R_{t+1},\bLambda\P_t)^\perp}^\top\bLambda \p_{t+1}
\end{align*}
Setting
\[\tilde{\bUpsilon}_t=\begin{pmatrix} \bDelta_{t+1} &
\bDelta_{t+1}\tilde{\A}_t+\tilde{\bPhi}_t\bSigma_t \\
\tilde{\bPhi}_t^\top & \tilde{\bPhi}_t^\top \tilde{\A}_t \end{pmatrix}\]
and recalling $\tilde{\j}_1^{(1)}$ from (\ref{eq:LJblockrect}),
by a computation similar to the proof of $t+1^{(c)}$ above, we have
\[\tilde{\M}_t^{-1}
\begin{pmatrix} \Id & \0 \\ -\tilde{\A}_t^\top & \Id \end{pmatrix}
\cdot m^{-1}
\begin{pmatrix} \R_{t+1}^\top \\ \P_t^\top \bLambda^\top \end{pmatrix}
\bLambda \p_{t+1}
\to \left(\gamma^{-1} \cdot
\begin{pmatrix} \Id & 0 \\ 0 & \bSigma_t^{-1} \end{pmatrix}
(\tilde{\bUpsilon}_t^{-1})^\top \begin{pmatrix}
\tilde{\j}_t^{(1)} \\ \l_t^{(2)} \end{pmatrix}\right)^\infty.\]
Applying the second row of (\ref{eq:Upsiloninvrect}) with $t+1$ and with $k=0$,
and recalling
$\I_{t+1}^{(2k+1)}=\gamma^{-1} \cdot (\J_{t+1}^{(2k+1)})^\top$ and
$c_{1,j}=\kappa_{2(j+1)}$, we get
\[\gamma^{-1} \cdot
\begin{pmatrix} \J_{t+1}^{(1)} \\ \L_{t+1}^{(2)} \end{pmatrix}
=\bUpsilon_{t+1}^\top \begin{pmatrix} \sum_{j=0}^\infty 
\kappa_{2(j+1)} \bPsi_{t+1}^\top (\bPhi_{t+1}^\top\bPsi_{t+1}^\top)^j \\
\sum_{j=0}^\infty \kappa_{2(j+1)} \bXi_{t+1}^{(j)} \end{pmatrix}
=\bUpsilon_{t+1}^\top \begin{pmatrix} \A_{t+1} \\ \bSigma_{t+1} \end{pmatrix}.\]
Hence, writing
\[\A_{t+1}=\begin{pmatrix} \tilde{\A}_t & \tilde{\a}_t \end{pmatrix},
\qquad \bSigma_{t+1}=\begin{pmatrix} \bSigma_t & \bsigma_t \\
\bsigma_t^\top & \sigma_{t+1,t+1}\end{pmatrix},\]
noting that $\tilde{\bUpsilon}_t$ is the matrix $\bUpsilon_{t+1}$ with the
last row and column removed, and that $\bUpsilon_{t+1}^\top$ is
0 in entries $1:2t+1$ of its last column, this yields
\[\gamma^{-1} \cdot \begin{pmatrix} \tilde{\j}_t^{(1)} \\
\l_t^{(2)} \end{pmatrix}=\tilde{\bUpsilon}_t^\top \begin{pmatrix} \tilde{\a}_t
\\ \bsigma_t \end{pmatrix}.\]
Taking the limit $m,n \to \infty$, inverting
$(\tilde{\bUpsilon}_t^\infty)^\top$, and applying this above,
\[\q_\parallel \toW \begin{pmatrix} U_1 & \cdots & U_{t+1} \end{pmatrix}
\tilde{\a}_t^\infty+\begin{pmatrix} Y_1 & \cdots & Y_t \end{pmatrix}
(\bSigma_t^\infty)^\infty \bsigma_t^\infty.\]
We have also
\[\q_\perp \toW Q_\perp \sim \N\left(0,\;
\left(\gamma^{-1} \cdot l_{t+1,t+1}^{(2)}-\begin{pmatrix}
\gamma^{-1} \cdot \tilde{\j}_t^{(1)} \\ \gamma^{-1} \cdot
\l_t^{(2)} \end{pmatrix}^\top
\begin{pmatrix} \H_{t+1}^{(0)} & (\tilde{\I}_t^{(1)})^\top \\
\tilde{\I}_t^{(1)} & \gamma^{-1} \L_t^{(2)} \end{pmatrix}^{-1}
\begin{pmatrix} \gamma^{-1} \cdot \tilde{\j}_t^{(1)} \\ \gamma^{-1} \cdot
\l_t^{(2)} \end{pmatrix}\right)^\infty\right)\]
where $\q_\perp$ is independent of
$(U_1,\ldots,U_{t+1},Y_1,\ldots,Y_t)$. Then
\begin{equation}\label{eq:Qlimitrect}
\q_{t+1} \toW Q_{t+1}=\begin{pmatrix} U_1 & \cdots & U_{t+1} \end{pmatrix}
\tilde{\a}_t^\infty+\begin{pmatrix} Y_1 & \cdots & Y_t \end{pmatrix}
(\bSigma_t^\infty)^{-1}\bsigma_t^\infty+Q_\perp.
\end{equation}
Recalling $\y_{t+1}=\q_{t+1}-\U_{t+1}\tilde{\a}_t$, this yields
\[(\u_1,\ldots,\u_{t+2},\y_1,\ldots,\y_{t+1},\E) \toW
(U_1,\ldots,U_{t+2},Y_1,\ldots,Y_{t+1},E)\]
where $U_{t+2}=u_{t+2}(Y_1,\ldots,Y_{t+1},E)$ and
\[Y_{t+1}=\begin{pmatrix} Y_1 & \cdots & Y_t \end{pmatrix}
(\bSigma_t^\infty)^{-1} \bsigma_t^\infty+Q_\perp.\]
So $(Y_1,\ldots,Y_t,Y_{t+1})$ has a multivariate normal limit.
To compute the covariance, note that
\[\EE[(Y_1,\ldots,Y_t)^\top Y_{t+1}]=\bSigma_t^\infty(\bSigma_t^\infty)^{-1}
\bsigma_t^\infty=\bsigma_t^\infty.\]
Squaring both sides of (\ref{eq:Qlimitrect}), applying $\EE[Q_{t+1}^2]=\lim_{m,n
\to \infty} m^{-1}\|\bLambda \p_{t+1}\|^2=\gamma^{-1}
l_{t+1,t+1}^{(2,\infty)}$,
and rearranging,
\[\EE[Y_{t+1}^2]=\left(\gamma^{-1} \cdot
\L_{t+1}^{(2)}-\A_{t+1}^\top \bDelta_{t+1}
\A_{t+1}-\A_{t+1}^\top \bPhi_{t+1}\bSigma_{t+1}
-\bSigma_{t+1}\bPhi_{t+1}^\top \A_{t+1}\right)^\infty_{t+1,t+1}.\]
Applying (\ref{eq:L2rect}), this is $\sigma_{t+1,t+1}^\infty$.
This concludes the proof of $t+1^{(f)}$.

Finally, for the invertibility claim of $t+1^{(g)}$, let us first observe that
\begin{equation}\label{eq:varQperprect}
\Var[Q_\perp]>0
\end{equation}
above. This is because $\Var[Q_\perp]$ is the residual
variance of projecting $\bar{P}_{t+1}$ onto the span of
$(R_1,\ldots,R_{t+1},\bar{P}_1,\ldots,\bar{P}_t)$. If this were 0, then for some
constants $\alpha_1,\ldots,\alpha_{t+1},\beta_1,\ldots,\beta_t$, we would have
\begin{align*}
0&=\EE[\Lambda_m^2 \cdot (\bar{P}_{t+1}-\alpha_1 R_1-\ldots-\alpha_{t+1}
R_{t+1}-\beta_1\bar{P}_1-\ldots-\beta_t\bar{P}_t)^2]\\
&=\lim_{m,n \to \infty} m^{-1}\|\bLambda^\top \bLambda \p_{t+1}
-\alpha_1\bLambda^\top \r_1-\ldots-\alpha_{t+1}\bLambda^\top\r_{t+1}
-\beta_1\bLambda^\top\bLambda\p_1-\ldots-\beta_t\bLambda^\top\bLambda\p_t\|^2\\
&=\gamma^{-1} \cdot \EE[(\Lambda_n^2 P_{t+1}-\alpha_1
\bar{R}_1-\ldots-\alpha_{t+1}\bar{R}_{t+1}-\beta_1\Lambda_n^2 P_1-\ldots
-\beta_n\Lambda_n^2P_t)^2].
\end{align*}
So
\[\Lambda_n^2P_\perp=f(P_1,\ldots,P_t,\bar{R}_1,\ldots,\bar{R}_{t+1},\Lambda_n)\]
for some quantity on the right not depending on $P_\perp$. This contradicts the
independence of $P_\perp$ from
$P_1,\ldots,P_t,\bar{R}_1,\ldots,\bar{R}_{t+1},\Lambda_n$, the
assumption $\Var[\Lambda_n]>0$, and the condition $\Var[P_\perp]>0$ already
shown in (\ref{eq:varPperprect}). So (\ref{eq:varQperprect}) holds.

To show the invertibility of
\begin{equation}\label{eq:invert1rect}
\begin{pmatrix} \bDelta_{t+1}^\infty & \bPhi_{t+1}^\infty \bSigma_{t+1}^\infty
\\ \bSigma_{t+1}^\infty (\bPhi_{t+1}^\infty)^\top & \bSigma_{t+1}^\infty
\end{pmatrix},
\end{equation}
note that its upper-left $(2t+1) \times (2t+1)$ submatrix
is $\tilde{\M}_t^\infty$, which we have shown is invertible.
The Schur-complement of its lower-right entry is the residual variance of
projecting $Y_{t+1}$ onto $(U_1,\ldots,U_t,Y_1,\ldots,Y_{t+1})$. As
$\y_{t+1}=\q_{t+1}-\U_{t+1}\tilde{\a}_t$, this is equivalently the residual
variance of projecting $Q_{t+1}$ onto $(U_1,\ldots,U_t,Y_1,\ldots,Y_{t+1})$,
which is exactly $\Var[Q_\perp]$ by (\ref{eq:Qlimitrect}) and the fact that
$Q_\perp$ is a mean-zero variable independent of
$(U_1,\ldots,U_t,Y_1,\ldots,Y_{t+1})$. Since $\Var[Q_\perp]>0$, this shows that
(\ref{eq:invert1rect}) is invertible.

To show the invertibility of
\begin{equation}\label{eq:invert2rect}
\begin{pmatrix} \bGamma_{t+1}^\infty & \bPsi_{t+1}^\infty \bOmega_{t+1}^\infty
\\ \bOmega_{t+1}^\infty (\bPsi_{t+1}^\infty)^\top & \bOmega_{t+1}^\infty
\end{pmatrix},
\end{equation}
note that its submatrix removing row and column $t+1$ is 
$\tilde{\bN}_t^\infty$, which we have also shown is invertible. The
Schur-complement of the $(t+1,t+1)$ entry is the residual variance of projecting
$V_{t+1}$ onto the span of $(V_1,\ldots,V_t,Z_1,\ldots,Z_{t+1})$, which is
non-zero by Assumption \ref{assump:rect}(f). Thus (\ref{eq:invert2rect}) is
invertible. This shows $t+1^{(g)}$, and concludes the induction and the proof.
\end{proof}

\section{Analysis of AMP for PCA}\label{appendix:PCA}

In this appendix, we prove Theorems \ref{thm:PCA} and \ref{thm:PCArect}.
We also complete the verification of Eq.\ (\ref{eq:rectPCAimproves}), showing
that the rectangular AMP algorithm achieves lower matrix 
mean-squared error than the sample PCs.

\subsection{State evolution for PCA}\label{subsec:PCAparta}

We prove Theorems \ref{thm:PCA}(a) and \ref{thm:PCArect}(a), using the general
results of Corollaries \ref{cor:degenerate} and \ref{cor:degeneraterect}.

\begin{proof}[Proof of Theorem \ref{thm:PCA}(a)]
We may write the AMP iterations (\ref{eq:AMPPCA1}--\ref{eq:AMPPCA2}) as
\[\f_t=\u_* \cdot (\alpha/n)\u_*^\top \u_t+\W \u_t
-b_{t1}\u_1-\ldots-b_{tt}\u_t, \qquad \u_{t+1}=u_{t+1}(\f_t).\]
Approximating $(\alpha/n)\u_*^\top \u_t \approx
\alpha \cdot \EE[U_*U_t]=\mu_t^\infty$, we consider the auxiliary AMP sequence
initialized at $\tilde{\u}_1=\u_1$ and defined by
\[\tilde{\z}_t=\W \tilde{\u}_t-\tilde{b}_{t1}\tilde{\u}_1-\ldots
-\tilde{b}_{tt}\tilde{\u}_t, \qquad
\tilde{\u}_{t+1}=\tilde{u}_{t+1}(\tilde{\z}_t,\u_*)
\equiv u_{t+1}(\tilde{\z}_t+\mu_t^\infty\u_*).\]
Here, the debiasing coefficients $\tilde{b}_{t1},\ldots,\tilde{b}_{tt}$ are
the values of the last column of $\tilde{\B}_t$, defined by (\ref{eq:DeltaPhi})
and (\ref{eq:BSigma}) with the iterates $\tilde{\u}_t$ and the
free cumulants $\tilde{\kappa}_k$ of $\W$. The partial derivatives that define
(\ref{eq:DeltaPhi}) are given by $\partial_s
\tilde{u}_{t+1}(\cdot)=\tilde{u}_{t+1}'(\cdot)$ if $s=t$ and
$\partial_s \tilde{u}_{t+1}(\cdot)=0$ otherwise, where $\tilde{u}_{t+1}'(\cdot)$
denotes the derivative in its first argument $\tilde{z}_t$.

This auxiliary AMP sequence is of the general form
(\ref{eq:AMPz}--\ref{eq:AMPu}) with side information $\E=\u_*$.
By the given differentiability and Lipschitz assumption for $u_{t+1}(\cdot)$,
the conditions of Corollary \ref{cor:degenerate} are satisfied for
$\tilde{u}_{t+1}(\cdot)$, so we have for each fixed $T \geq 1$ that
\[(\tilde{\u}_1,\ldots,\tilde{\u}_{T+1},\tilde{\z}_1,\ldots,\tilde{\z}_T,\u_*)
\toWtwo
(\tilde{U}_1,\ldots,\tilde{U}_{T+1},\tilde{Z}_1,\ldots,\tilde{Z}_T,U_*).\]
Here, $(\tilde{Z}_1,\ldots,\tilde{Z}_T) \sim \N(0,\tilde{\bSigma}_T^\infty)$
where $\tilde{\bSigma}_T^\infty$ is defined by (\ref{eq:Theta}) and
(\ref{eq:BSigma}) for this auxiliary AMP sequence, $\tilde{U}_1=U_1$,
$\tilde{U}_{s+1}=\tilde{u}_{s+1}(\tilde{Z}_s,U_*)$ for $s \geq 1$,
and $(\tilde{Z}_1,\ldots,\tilde{Z}_T)$ is independent of $(\tilde{U}_1,U_*)$. 
Defining
\[\tilde{\f}_t=\tilde{\z}_t+\mu_t^\infty \u_*,
\qquad \tilde{F}_t=\tilde{Z}_t+\mu_t^\infty U_*,\]
this implies
\begin{equation}\label{eq:PCAtildeconvergence}
(\tilde{\u}_1,\ldots,\tilde{\u}_{T+1},\tilde{\f}_1,\ldots,\tilde{\f}_T,\u_*)
\toWtwo
(\tilde{U}_1,\ldots,\tilde{U}_{T+1},\tilde{F}_1,\ldots,\tilde{F}_T,U_*).
\end{equation}
Since each derivative $\partial_s \tilde{u}_{t+1}$ is non-zero only for $s=t$,
the covariance matrix $\tilde{\bSigma}_T^\infty$ has the entries
\begin{equation}\label{eq:tildeSigmaPCA}
\tilde{\sigma}_{st}^\infty=\sum_{j=0}^{s-1}\sum_{k=0}^{t-1}
\tilde{\kappa}_{j+k+2}^\infty \left(\prod_{i=s-j+1}^s
\EE[\tilde{u}_i'(\tilde{Z}_{i-1},U_*)]\right)
\left(\prod_{i=t-k+1}^t \EE[\tilde{u}_i'(\tilde{Z}_{i-1},U_*)]\right)
\EE[\tilde{U}_{s-j}\tilde{U}_{t-k}]
\end{equation}
where the summand for $(j,k)$ corresponds to
$\bPhi_T^j\bDelta_T(\bPhi_T^k)^\top$ in the definitions (\ref{eq:Theta}) and
(\ref{eq:BSigma}).

We conclude the proof by showing that the joint law of this limit in
(\ref{eq:PCAtildeconvergence}) coincides with the limit described in Theorem
\ref{thm:PCA}(a), and that $(\u_1,\ldots,\u_{T+1},\f_1,\ldots,\f_T,\u_*)$ for
the original AMP algorithm converges to the same limit.
Observe first that the $n-1$ smallest eigenvalues of $\X$ are interlaced with
the $n$ eigenvalues of $\W$. Letting $m_k$ be as defined in
(\ref{eq:PCAmk}), and denoting the moments of the empirical spectral
distribution of $\W$ by
\[\tilde{m}_k=\frac{1}{n}\sum_{i=1}^n \lambda_i^k,\]
this interlacing and the
condition $\|\W\| \leq C_0$ imply $|m_k-\tilde{m}_k| \to 0$ and $m_k,\tilde{m}_k
\to m_k^\infty=\EE[\Lambda^k]$ for each fixed $k \geq 1$ as $n \to \infty$.
Hence also for each fixed $k \geq 1$,
\begin{equation}\label{eq:PCAkappadiff}
|\kappa_k-\tilde{\kappa}_k| \to 0 \quad \text{ and } \quad
\kappa_k,\tilde{\kappa}_k \to \kappa_k^\infty
\end{equation}
where $\{\kappa_k^\infty\}_{k \geq 1}$ are the free cumulants of $\Lambda$.

We now check inductively that, almost surely for each fixed
$T=0,1,2,\ldots$ as $n \to \infty$,
\begin{equation}\label{eq:PCAclaim1}
n^{-1}\|\u_s-\tilde{\u}_s\|^2 \to 0 \text{ for all } s \leq T+1, \qquad
n^{-1}\|\f_s-\tilde{\f}_s\|^2 \to 0 \text{ for all } s \leq T,
\end{equation}
and
\begin{equation}\label{eq:PCAclaim2}
(\u_1,\ldots,\u_{T+1},\f_1,\ldots,\f_T,\u_*)
\toWtwo (U_1,\ldots,U_{T+1},F_1,\ldots,F_T,U_*)
\end{equation}
where the joint law of this limit is as in Theorem \ref{thm:PCA} and
coincides with the limit in (\ref{eq:PCAtildeconvergence})

For the base case $T=0$, we have $\|\u_1-\tilde{\u}_1\|=0$,
$(\u_1,\u_*) \toW (U_1,U_*)=(\tilde{U}_1,U_*)$
by assumption, and the remaining claims are
vacuous. Assume inductively that these claims hold for $T-1$. Then for all $s,s'
\leq T$,
\[\lim_{n \to \infty} \langle \u_s \u_{s'} \rangle=\EE[U_sU_{s'}]
=\EE[\tilde{U}_s\tilde{U}_{s'}]=\lim_{n \to \infty} \langle \tilde{\u}_s
\tilde{\u}_{s'} \rangle,\]
and similarly for all $s \leq T$,
\[\lim_{n \to \infty} \langle u_s'(\f_{s-1}) \rangle=\EE[u_s'(F_{s-1})]
=\EE[\tilde{u}_s'(\tilde{Z}_{s-1},U_*)]
=\lim_{n \to \infty} \langle \partial_{s-1} \tilde{\u}_s \rangle.\]
Combining with (\ref{eq:PCAkappadiff}) and comparing (\ref{eq:PCAsigma})
with (\ref{eq:tildeSigmaPCA}),
this shows that $\tilde{\bSigma}_T^\infty$ coincides with $\bSigma_T^\infty$,
and hence that the limit laws in (\ref{eq:PCAtildeconvergence})
and (\ref{eq:PCAclaim2}) coincide for $T$.

Comparing (\ref{eq:PCAb}) with the general definition of $\tilde{\B}_T$ from
(\ref{eq:BSigma}), this also shows that $|b_{st}-\tilde{b}_{st}| \to 0$
as $n \to \infty$, for all $s,t \leq T$. Denoting
\[\z_t=\W \u_t-b_{t1}\u_1-\ldots-b_{tt}\u_t,\]
and applying also $\|\W\| \leq C_0$ and
$n^{-1}\|\u_s-\tilde{\u}_s\|^2 \to 0$ for all $s \leq T$ by the
induction hypothesis, we obtain $n^{-1}\|\z_T-\tilde{\z}_T\|^2 \to 0$.
Since $\f_T=\u_* \cdot (\alpha/n)\u_*^\top \u_T+\z_T$,
$\tilde{\f}_T=\u_* \cdot \mu_T^\infty+\tilde{\z}_T$, and
$(\alpha/n)\u_*^\top \u_T \to \alpha \cdot \EE[U_*U_T]=\mu_T^\infty$
by the induction hypothesis (\ref{eq:PCAclaim2}), this shows
\begin{equation}\label{eq:gcompare}
n^{-1}\|\f_T-\tilde{\f}_T\|^2 \to 0.
\end{equation}
Then, as the function $u_{T+1}$ is Lipschitz,
\begin{equation}\label{eq:vcompare}
n^{-1}\|\u_{T+1}-\tilde{\u}_{T+1}\|^2
=n^{-1}\|u_{T+1}(\f_T)-u_{T+1}(\tilde{\f}_T)\|^2 \to 0.
\end{equation}
This shows (\ref{eq:PCAclaim1}) for $T$.
Applying Proposition \ref{prop:scalarprod}, this implies that
$(\u_1,\ldots,\u_{T+1},\f_1,\ldots,\f_T,\u_*)$ must have the same empirical
limit in $W_2$ as
$(\tilde{\u}_1,\ldots,\tilde{\u}_{T+1},\tilde{\f}_1,\ldots,\tilde{\f}_T,\u_*)$.
Together with (\ref{eq:PCAtildeconvergence}) and the coincidence of the two
joint limit laws in (\ref{eq:PCAtildeconvergence}) and (\ref{eq:PCAclaim2})
that was already established, this shows (\ref{eq:PCAclaim2})
for $T$, concluding the induction and the proof.
\end{proof}

\begin{proof}[Proof of Theorem \ref{thm:PCArect}(a)]
We may write the AMP iterations (\ref{eq:AMPPCArect1}--\ref{eq:AMPPCArect4}) as
\[\g_t=\v_* \cdot (\alpha/m)\u_*^\top \u_t+\W^\top
\u_t-b_{t1}\v_1-\ldots-b_{t,t-1}\v_{t-1}, \qquad\qquad \v_t=v_t(\g_t)\]
\[\f_t=\u_* \cdot (\alpha/m)\v_*^\top \v_t+\W\v_t-a_{t1}\u_1-\ldots-a_{tt}\u_t,
\qquad\qquad \u_{t+1}=u_{t+1}(\f_t).\]
Approximating $(\alpha/m)\u_*^\top \u_t \approx \alpha \cdot
\EE[U_*U_t]=\nu_t^\infty$ and $(\alpha/m)\v_*^\top \v_t \approx
\alpha/\gamma \cdot \EE[V_*V_t]=\mu_t^\infty$, we consider the
auxiliary AMP sequence initialized at $\tilde{\u}_1=\u_1$ and defined by
\[\tilde{\z}_t=\W^\top \tilde{\u}_t-\tilde{b}_{t1}\tilde{\v}_1-\ldots
-\tilde{b}_{t,t-1}\tilde{\v}_{t-1}, \qquad\qquad
\tilde{\v}_t=\tilde{v}_t(\tilde{\z}_t,\v_*)
\equiv v_t(\tilde{\z}_t+\nu_t^\infty \v_*),\]
\[\tilde{\y}_t=\W
\tilde{\v}_t-\tilde{a}_{t1}\tilde{\u}_t-\ldots-\tilde{a}_{tt}\tilde{\u}_t,
\qquad\qquad
\tilde{\u}_{t+1}=\tilde{u}_{t+1}(\tilde{\y}_t,\u_*)
\equiv u_{t+1}(\tilde{\y}_t+\mu_t^\infty\u_*).\]
Here, the debiasing coefficients are the last columns of $\tilde{\A}_t$ and
$\tilde{\B}_t$ defined by (\ref{eq:ABrect})
with the iterates $\tilde{\u}_t,\tilde{\v}_t$ and with
rectangular free cumulants $\tilde{\kappa}_{2k}$ of $\W$. The partial
derivatives in (\ref{eq:DeltaPhirect}) and (\ref{eq:GammaPsirect}) are
given by $\partial_s \tilde{u}_{t+1}(\cdot)=\tilde{u}_{t+1}'(\cdot)$
if $s=t$ and 0 otherwise, and $\partial_s
\tilde{v}_t(\cdot)=\tilde{v}_t'(\cdot)$ if $s=t$ and 0 otherwise,
where $\tilde{u}_{t+1}'(\cdot)$ and $\tilde{v}_t'(\cdot)$ denote their
derivatives with respect to the first arguments $\tilde{y}_t$ and $\tilde{z}_t$.

This auxiliary AMP sequence
is of the form (\ref{eq:AMPrectz}--\ref{eq:AMPrectu}) with side information
$\E=\u_*$ and $\F=\v_*$. Setting
\[\tilde{\f}_t=\tilde{\y}_t+\mu_t^\infty \u_*, \qquad
\tilde{\g}_t=\tilde{\z}_t+\nu_t^\infty \v_*,\]
Corollary \ref{cor:degeneraterect} then implies for each fixed $T \geq 1$ that
\begin{align*}
(\tilde{\v}_1,\ldots,\tilde{\v}_T,\tilde{\g}_1,\ldots,\tilde{\g}_T,\v_*)
&\toWtwo (\tilde{V}_1,\ldots,\tilde{V}_T,\tilde{G}_1,\ldots,\tilde{G}_T,V_*)\\
(\tilde{\u}_1,\ldots,\tilde{\u}_{T+1},\tilde{\f}_1,\ldots,\tilde{\f}_T,\u_*)
&\toWtwo (\tilde{U}_1,\ldots,\tilde{U}_{T+1},\tilde{F}_1,\ldots,\tilde{F}_T,U_*)
\end{align*}
where these limits are described by
$\tilde{F}_t=\tilde{Y}_t+\mu_t^\infty U_*$,
$\tilde{G}_t=\tilde{Z}_t+\nu_t^\infty V_*$,
$(\tilde{Y}_1,\ldots,\tilde{Y}_T) \sim \N(0,\tilde{\bSigma}_T^\infty)$
and $(\tilde{Z}_1,\ldots,\tilde{Z}_T) \sim \N(0,\tilde{\bOmega}_T^\infty)$.
Here, the forms for $\tilde{\bSigma}_T^\infty$ and $\tilde{\bOmega}_T^\infty$
are given by (\ref{eq:PCAsigmarect}) and (\ref{eq:PCAomegarect}) defined for
this auxiliary sequence: In the definitions (\ref{eq:SigmaOmegarect}),
summing the first terms of $\bTheta_T^{(j)}$ and
$\bXi_T^{(j)}$ in (\ref{eq:Thetarect}--\ref{eq:Xirect}) yields
the terms with coefficient $\kappa_{2(j+k+1)}$ in
(\ref{eq:PCAsigmarect}) and (\ref{eq:PCAomegarect}), while
summing the second terms of (\ref{eq:Thetarect}--\ref{eq:Xirect})
yields the terms with coefficient $\kappa_{2(j+k+2)}$.

The proof is concluded by a similar comparison argument as in the preceding
proof of Theorem \ref{thm:PCA}(a), showing that these joint limit laws
coincide with those of Theorem \ref{thm:PCArect} and that the original AMP
sequence converges also to these joint laws. We omit the details for brevity.
\end{proof}

\subsection{Analysis of state evolutions}\label{subsec:PCApartb}

We now prove Theorems \ref{thm:PCA}(b) and \ref{thm:PCArect}(b).
For notational simplicity, we will drop all
superscripts $^\infty$ in this section, so that
$\kappa_k,\bSigma_T,\bDelta_T$ etc.\ are all understood as their
deterministic $n \to \infty$ limits. We use the entrywise notation
\[\bDelta_T=(\delta_{st})_{s,t=1}^T,
\; \bSigma_T=(\sigma_{st})_{s,t=1}^T, \; \bmu_T=(\mu_t)_{t=1}^T,
\; \bGamma_T=(\gamma_{st})_{s,t=1}^T,
\; \bOmega_T=(\omega_{st})_{s,t=1}^T, \; \bnu_T=(\nu_t)_{t=1}^T.\]

The proofs will apply a contractive mapping argument to show that
the matrices $\bDelta_T$, $\bSigma_T$, $\bGamma_T$, and $\bOmega_T$ all
converge in a certain normed space. Fix an arbitrary constant $\zeta \in (0,1)$,
say
\[\zeta=1/2.\]
We consider the space of ``infinite matrices''
$\x=(x_{st}:s,t \leq 0)$, indexed by the non-positive integers. The index
$(0,0)$ should be interpreted as the lower-right corner of $\x$. We equip this
space with a weighted $\ell_\infty$-norm
\[\|\x\|_\zeta=\sup_{s,t \leq 0} \zeta^{|s| \vee |t|} |x_{st}|,
\qquad |s| \vee |t|=\max(|s|,|t|).\]
Thus the weight is $\zeta^k$ for the $2k+1$ coordinate pairs
\[(s,t)=(-k,0),(-k,-1),\ldots,(-k,-k+1),(-k,-k),(-k+1,-k),\ldots,
(-1,-k),(0,-k).\]
Define $\cX=\{\x:\|\x\|_\zeta<\infty\}$,
and observe that $\cX$ is complete under $\|\cdot\|_\zeta$.
For any compact interval $I \subset \RR$, denote
\begin{equation}\label{eq:Xsubset}
\cX_I=\{\x:\x_{st} \in I \text{ for all } s,t \leq 0\} \subset \cX.
\end{equation}
Then $\cX_I$ is closed in $\cX$, and hence $\cX_I$ is also
complete under the norm $\|\cdot\|_\zeta$.

We will embed the matrices $\bDelta_T,\bSigma_T,\bGamma_T,\bOmega_T$ as
elements $\x,\y,\z,\w \in \cX$, with the coordinate identifications
\[\delta_{st}=x_{s-T,t-T}, \qquad \sigma_{st}=y_{s-T,t-T},
\qquad \gamma_{st}=z_{s-T,t-T}, \qquad \omega_{st}=w_{s-T,t-T}\]
\begin{equation}\label{eq:Xembedding}
x_{st}=y_{st}=z_{st}=w_{st}=0 \quad \text{ if } \quad
s \leq -T \text{ or } t \leq -T.
\end{equation}
Thus $\bDelta_T$, $\bSigma_T$, $\bGamma_T$, and $\bOmega_T$ fill out the
lower-right $T \times T$ corners of the corresponding sequences in $\cX$,
with their lower-right $(T,T)$ entries identified with the coordinate $(0,0)$
of $\cX$. Zero-padding is applied for the remaining entries of $(\x,\y,\z,\w)$
not belonging to this corner. The proofs will then have two main steps:
\begin{enumerate}
\item For large $T$, the state evolution that maps these matrices from
iterate $T$ to iterate $T+1$ will be approximated by a fixed map that is
independent of $T$, where the approximation is in the norm $\|\cdot\|_\zeta$.
\item This map will be shown to be contractive over certain sub-domains of
$\cX$ with respect to $\|\cdot\|_\zeta$, and hence these matrices will
converge to a fixed point of this map.
\end{enumerate}

\subsubsection{Symmetric square matrices}

We first show Theorem \ref{thm:PCA}(b). Recall that the AMP algorithm is given
by (\ref{eq:AMPPCA1}--\ref{eq:AMPPCA2}), where we take $u_{t+1}(\cdot)$
to be the single-iterate posterior mean denoiser in (\ref{eq:denoiserexplicit}).
Differentiating (\ref{eq:denoiserexplicit}) in $f$, we obtain
\[\frac{\partial}{\partial f}
\eta(f \mid \mu,\sigma^2)=\Cov\left[U_*,\;\frac{\partial}{\partial f}
\left(-\frac{(f-\mu U_*)^2}{2\sigma^2}\right) \bigg|F=f\right]
=\frac{\mu}{\sigma^2}\Var[U_* \mid F=f].\]
Then
\[u_{t+1}'(f_t)=\frac{\mu_t}{\sigma_{tt}}\Var[U_* \mid F_t=f_t].\]
Observe that for all $t \geq 1$,
\[\EE[\Var[U_* \mid F_t]]=\EE[(U_*-U_{t+1})^2]
=\EE[U_*^2]-\EE[U_{t+1}^2]=1-\delta_{t+1,t+1}.\]
So (\ref{eq:PCAsigma}) may be written more explicitly as
\begin{equation}\label{eq:PCAsigmaexplicit}
\sigma_{st}=\sum_{j=0}^{s-1} \sum_{k=0}^{t-1}
\kappa_{j+k+2} \left(\prod_{i=s-j+1}^s \frac{\mu_{i-1}}
{\sigma_{i-1,i-1}}(1-\delta_{ii})\right)
\left(\prod_{i=t-k+1}^t \frac{\mu_{i-1}}
{\sigma_{i-1,i-1}}(1-\delta_{ii})\right)\delta_{s-j,t-k}.
\end{equation}
In this expression, we have
\begin{equation}\label{eq:PCAmu}
\mu_1=\alpha \cdot \EE[U_1U_*]=\alpha \eps, \qquad
\mu_i=\alpha \cdot \EE[U_iU_*]=\alpha\cdot\EE[U_i^2]=\alpha \delta_{ii}
\text{ for } i \geq 2.
\end{equation}

For a sufficiently large constant $C>0$ depending on $C_0$ and $\eps$,
we define the intervals
\[I_\Delta=\left[1-\frac{C}{\alpha^2},\;1\right],
\qquad I_\Sigma=\left[\frac{1}{2}\kappa_2,\frac{3}{2}\kappa_2\right],\]
and the corresponding domains
$\cX_{I_\Delta},\cX_{I_\Sigma} \subset \cX$ by (\ref{eq:Xsubset}).
Motivated by the forms (\ref{eq:PCAsigmaexplicit}) and (\ref{eq:PCAmu}),
we will approximate the map $(\bDelta_T,\bSigma_{T-1})
\mapsto \bSigma_T$ by a fixed map $h^\Sigma:\cX_{I_\Delta} \times \cX_{I_\Sigma}
\to \cX$, defined entrywise by
\begin{equation}\label{eq:PCASigmaapprox}
h^\Sigma_{st}(\x,\y)=\sum_{j=0}^\infty \sum_{k=0}^\infty \kappa_{j+k+2}
\left(\prod_{i=s-j+1}^s \frac{\alpha x_{i-1,i-1}}{y_{ii}}(1-x_{ii})\right)
\left(\prod_{i=t-k+1}^t \frac{\alpha x_{i-1,i-1}}
{y_{ii}}(1-x_{ii})\right)x_{s-j,t-k}.
\end{equation}
(Note that the embedding of $\bSigma_{T-1}$ in $\cX$ has indices that are
offset from those of $\bSigma_T$ and $\bDelta_T$ by 1, so $y_{ii}$ appears
instead of $y_{i-1,i-1}$.)
We will approximate the map $(\bDelta_T,\bSigma_T) \mapsto \bDelta_{T+1}$ by a
fixed map
$h^\Delta:\cX_{I_\Delta} \times \cX_{I_\Sigma} \to \cX$, defined as
\[h^\Delta_{st}(\x,\y)=\EE_{\x,\y}\Big[\EE_{\x,\y}[U_* \mid F_s]
\EE_{\x,\y}[U_* \mid F_t]\Big]\]
where the expectations are with respect to the $(\x,\y)$-dependent joint law
\[(F_s,F_t)=(\alpha x_{ss},\alpha x_{tt})U_*
+(Z_s,Z_t), \qquad (Z_s,Z_t) \sim \N\left(0,\;\Pi\begin{pmatrix} y_{ss} &
y_{st} \\ y_{ts} & y_{tt} \end{pmatrix}\right) \text{ independent of } U_*.\]
Here, we denote by $\Pi:I_\Sigma^{2 \times 2} \to I_\Sigma^{2 \times 2}$ the
map
\begin{equation}\label{eq:Pi}
\Pi\begin{pmatrix} a & b \\ c & d \end{pmatrix}
=\begin{pmatrix} a & \min(\sqrt{bc},\sqrt{ad}) \\
\min(\sqrt{bc},\sqrt{ad}) & d \end{pmatrix}
\end{equation}
whose image is always symmetric positive-semidefinite, so that the above
bivariate normal law is always well-defined. (If $M$ is already
symmetric positive-semidefinite, then $\Pi(M)=M$.)

The following lemma establishes the Lipschitz bounds for
$h^\Sigma$ and $h^\Delta$.

\begin{lemma}\label{lemma:PCAcontraction}
In the setting of Theorem \ref{thm:PCA}(b), there exist constants $C,\alpha_0>0$
such that for all $\alpha>\alpha_0$ and
$(\x,\y),(\x',\y') \in \cX_{I_\Delta} \times \cX_{I_\Sigma}$:
\begin{enumerate}[(a)]
\item $h^\Sigma(\x,\y) \in \cX_{I_\Sigma}$ and
$\|h^\Sigma(\x,\y)-h^\Sigma(\x',\y')\|_\zeta \leq C\alpha\|\x-\x'\|_\zeta
+(C/\alpha)\|\y-\y'\|_\zeta$.
\item $h^\Delta(\x,\y) \in \cX_{I_\Delta}$ and
$\|h^\Delta(\x,\y)-h^\Delta(\x',\y')\|_\zeta \leq (C/\alpha)\|\x-\x'\|_\zeta
+(C/\alpha^2)\|\y-\y'\|_\zeta$.
\end{enumerate}
\end{lemma}
\begin{proof}
Let $C,C',c,\ldots$ denote constants depending only on $C_0$ and $\eps$ and
changing from instance to instance.
For part (a), let us write (\ref{eq:PCASigmaapprox}) as
\[h^\Sigma_{st}(\x,\y)=\sum_{j,k=0}^\infty \kappa_{j+k+2}
h_{st}^{(j,k)}(\x,\y).\]
Observe that for $(j,k)=(0,0)$, we have simply
$h_{st}^{(0,0)}(\x,\y)=x_{st}$.
By the given domains of $\x$ and $\y$, for all other $(j,k)$, we have the bounds
$x_{i-1,i-1} \leq 1$, $y_{ii} \geq \kappa_2/2$, and $1-x_{ii} \leq C/\alpha^2$
in the products defining (\ref{eq:PCASigmaapprox}). There are $j+k$ factors of
the form $(\alpha x_{i-1,i-1}/y_{ii})(1-x_{ii})$, yielding
\[|h_{st}^{(j,k)}(\x,\y)| \leq (C/\alpha)^{j+k}.\]
Applying $|\kappa_{j+k+2}| \leq (16C_0)^{j+k+2}$ by
Proposition \ref{prop:cumulantbound}, for $\alpha>\alpha_0$ sufficiently large,
this implies
\[|h^\Sigma_{st}(\x,\y)-\kappa_2 x_{st}|
=\left|\sum_{(j,k) \neq (0,0)} \kappa_{j+k+2}
h_{st}^{(j,k)}(\x,\y)\right| \leq \kappa_2/3.\]
Applying $x_{st} \in [1-C/\alpha^2,1]$, this yields
$h^\Sigma_{st}(\x,\y) \in I_\Sigma=[\kappa_2/2,3\kappa_2/2]$ for
$\alpha>\alpha_0$ sufficiently large. Then $h^\Sigma(\x,\y) \in \cX_{I_\Sigma}$.

To show the Lipschitz bound for $h^\Sigma$,
for $\lambda \in [0,1]$ we set
\[\x^\lambda=\lambda \x+(1-\lambda)\x', \qquad \y^\lambda
=\lambda \y+(1-\lambda)\y'.\]
Then
\begin{equation}\label{eq:sigmadiffPCA}
|h^\Sigma_{st}(\x,\y)-h^\Sigma_{st}(\x',\y')|
=\left|\int_0^1 \frac{d}{d\lambda}
h^\Sigma_{st}(\x^\lambda,\y^\lambda)d\lambda\right|
\leq \sup_{\lambda \in [0,1]} \left|\frac{d}{d\lambda}
h^\Sigma_{st}(\x^\lambda,\y^\lambda)\right|.
\end{equation}
By the chain rule,
\begin{equation}\label{eq:dsigmaPCA}
\frac{d}{d\lambda} h^\Sigma_{st}(\x^\lambda,\y^\lambda)
=\sum_{p,q \leq 0} (x_{pq}-x_{pq}')\frac{\partial h^\Sigma_{st}}
{\partial x_{pq}} (\x^\lambda,\y^\lambda)
+\sum_{p \leq 0} (y_{pp}-y_{pp}')\frac{\partial h^\Sigma_{st}}{\partial y_{pp}}
(\x^\lambda,\y^\lambda).
\end{equation}

We establish a uniform bound for these partial derivatives.
For any $\x,\y \in \cX_{I_\Delta} \times \cX_{I_\Sigma}$,
applying the above bound for $(\alpha x_{i-1,i-1}/y_{ii})(1-x_{ii})$,
we have
\begin{align*}
\left|\frac{\partial h_{st}^{(j,k)}}{\partial x_{pp}}\right|
&\leq \begin{cases}
C\alpha^2 (C/\alpha)^{j+k} & \text{ if } p \in \{s-j+1,\ldots,s\} \cup
\{t-k+1,\ldots,t\} \\
(C/\alpha)^{j+k} & \text{ if } p=s-j \text{ or } p=t-k
\end{cases}\\
\left|\frac{\partial h_{st}^{(j,k)}}{\partial y_{pp}}\right|
&\leq (C/\alpha)^{j+k} \quad \text{ if } p \in \{s-j+1,\ldots,s\} \cup
\{t-k+1,\ldots,t\}\\
\left|\frac{\partial h_{st}^{(j,k)}}{\partial x_{s-j,t-k}}\right|
&\leq (C/\alpha)^{j+k},
\end{align*}
and all other partial derivatives of $h_{st}^{(j,k)}$ are 0.
Multiplying by $\kappa_{j+k+2}$, applying $|\kappa_{j+k+2}| \leq
(16C_0)^{j+k+2}$, and summing over $j,k \geq 0$, this implies
\begin{align*}
\left|\frac{\partial h^\Sigma_{st}}{\partial x_{pp}}\right|
&\leq \sum_{j,k \geq 0} |\kappa_{j+k+2}| \cdot
\left|\frac{\partial h_{st}^{(j,k)}}{\partial x_{pp}}\right|\\
&\leq \1\{p \leq s\} \sum_{k \geq
0}C\left(\frac{C'}{\alpha}\right)^{(s-p)+k}+
\1\{p \leq t\} \sum_{j \geq 0} C\left(\frac{C'}{\alpha}\right)^{j+(t-p)}\\
&\hspace{1in}+\1\{p \leq s\} \sum_{j \geq s+1-p} \sum_{k \geq 0}
C\alpha^2\left(\frac{C'}{\alpha}\right)^{j+k}
+\1\{p \leq t\} \sum_{k \geq t+1-p} \sum_{j \geq 0}
C\alpha^2\left(\frac{C'}{\alpha}\right)^{j+k}
\end{align*}
where the first two terms are the contributions from
$p=s-j$ and $p=t-k$, and the latter two terms are the contributions
from $p \in \{s-j+1,\ldots,s\}$ and $p \in \{t-k+1,\ldots,t\}$.
For $\alpha>\alpha_0$ sufficiently large, this simplifies to the bound
\[\left|\frac{\partial h^\Sigma_{st}}{\partial x_{pp}}\right|
\leq C\alpha \left(\1\{p \leq s\} \cdot \left(\frac{C'}{\alpha}\right)^{s-p}
+\1\{p \leq t\} \cdot \left(\frac{C'}{\alpha}\right)^{t-p}\right).\]
We have similarly
\begin{align*}
\left|\frac{\partial h^\Sigma_{st}}{\partial y_{pp}}\right|
&\leq \sum_{j,k \geq 0} |\kappa_{j+k+2}| \cdot
\left|\frac{\partial h_{st}^{(j,k)}}{\partial y_{pp}}\right|\\
&\leq \1\{p \leq s\} \sum_{j \geq s+1-p} \sum_{k \geq 0}
C\left(\frac{C'}{\alpha}\right)^{j+k}
+\1\{p \leq t\} \sum_{k \geq t+1-p} \sum_{j \geq 0}
C\left(\frac{C'}{\alpha}\right)^{j+k}\\
&\leq \frac{C}{\alpha}\left(\1\{p \leq s\} \cdot \left(\frac{C'}{\alpha}
\right)^{s-p}
+\1\{p \leq t\} \cdot \left(\frac{C'}{\alpha}\right)^{t-p}\right)
\end{align*}
and, for $p \neq q$,
\[\left|\frac{\partial h^\Sigma_{st}}{\partial x_{pq}}\right|
\leq \sum_{j,k \geq 0} |\kappa_{j+k+2}| \cdot
\left|\frac{\partial h_{st}^{(j,k)}}{\partial x_{pq}}\right|
\leq C \cdot \1\{p \leq s \text{ and } q \leq t\}
\cdot \left(\frac{C'}{\alpha}\right)^{s-p+t-q}.\]

Applying these bounds to (\ref{eq:sigmadiffPCA}) and (\ref{eq:dsigmaPCA}),
\begin{align*}
|h^\Sigma_{st}(\x,\y)-h^\Sigma_{st}(\x',\y')|
&\leq C\alpha \sum_{p \leq s}|x_{pp}-x_{pp}'|
\left(\frac{C'}{\alpha}\right)^{s-p}
+C\alpha\sum_{p \leq t}|x_{pp}-x_{pp}'|\left(\frac{C'}{\alpha}\right)^{t-p}\\
&\hspace{0.2in}
+\frac{C}{\alpha}
\sum_{p \leq s}|y_{pp}-y_{pp}'|\left(\frac{C'}{\alpha}\right)^{s-p}
+\frac{C}{\alpha}
\sum_{p \leq t}|y_{pp}-y_{pp}'|\left(\frac{C'}{\alpha}\right)^{t-p}\\
&\hspace{0.2in}+C\sum_{p \leq s}\sum_{q \leq t} |x_{pq}-x_{pq}'|
\left(\frac{C'}{\alpha}\right)^{s-p+t-q}.
\end{align*}
For $\alpha>\alpha_0$ large enough, we may bound the terms above using
\begin{align*}
\sum_{p \leq s} |x_{pp}-x_{pp}'| (C'/\alpha)^{s-p}
&\leq \sup_{p \leq s} |x_{pp}-x_{pp}'|\zeta^{|p|}
\cdot \sum_{p \leq s} \zeta^{-|p|} (C'/\alpha)^{s-p}
\leq C\|\x-\x'\|_\zeta \cdot \zeta^{-|s|},\\
\sum_{p \leq s} |y_{pp}-y_{pp}'| (C'/\alpha)^{s-p}
&\leq \sup_{p \leq s} |y_{pp}-y_{pp}'|\zeta^{|p|}
\cdot \sum_{p \leq s} \zeta^{-|p|} (C'/\alpha)^{s-p}
\leq C\|\y-\y'\|_\zeta \cdot \zeta^{-|s|},\\
\sum_{p \leq s}\sum_{q \leq t}
|x_{pq}-x_{pq}'|(C'/\alpha)^{s-p+t-q}
&\leq \sup_{p \leq s \text{ and } q \leq t}
|x_{pq}-x_{pq}'|\zeta^{|p| \vee |q|}
\cdot \sum_{p \leq s}\sum_{q \leq t}
\zeta^{-(|p| \vee |q|)}(C'/\alpha)^{s-p+t-q}\\
&\leq C\|\x-\x'\|_\zeta \cdot \zeta^{-(|s| \vee |t|)}.
\end{align*}
Then
\[\|h^\Sigma(\x,\y)-h^\Sigma(\x',\y')\|_\zeta
=\sup_{s,t \leq 0} |h^\Sigma_{st}(\x,\y)-h^\Sigma_{st}(\x',\y')|\zeta^{|s| \vee |t|}
\leq C\alpha\|\x-\x'\|_\zeta+\frac{C}{\alpha}\|\y-\y'\|_\zeta,\]
yielding the Lipschitz bound in part (a).

For part (b), let us denote
\begin{equation}\label{eq:posterioruvfuncs}
\eta(f \mid \mu,\sigma^2)=\EE[U_* \mid F=f], \qquad
v(f \mid \mu,\sigma^2)=\Var[U_* \mid F=f]
\end{equation}
in the scalar model $F=\mu \cdot U_*+Z$, where $Z \sim \N(0,\sigma^2)$ is
independent of $U_*$. Observe that
\[\Var[U_* \mid F]=\Var[\mu^{-1}(F-Z) \mid F]=(1/\mu)^2 \Var[Z \mid F],\]
so that
\begin{equation}\label{eq:posteriorvariance}
\EE[v(F \mid \mu,\sigma^2)]
=\EE[\Var[U_* \mid F]]=(1/\mu)^2 \EE[\Var[Z \mid F]]
\leq (1/\mu)^2 \Var[Z]=(\sigma/\mu)^2.
\end{equation}

For ease of notation, let us write $\EE$ for $\EE_{\x,\y}$ in the
definition of the function $h^\Delta$. We denote
\[U_t=\EE[U_* \mid F_t]=\eta(F_t \mid \alpha x_{tt},y_{tt}).\]
Observe that the above then implies
\[\EE[(U_*-U_t)^2]=\EE[v(F_t \mid \alpha x_{tt},y_{tt})]
\leq y_{tt}/(\alpha x_{tt})^2.\]
For $\x \in \cX_{I_\Delta}$ and $\y \in \cX_{I_\Sigma}$,
applying $y_{tt} \leq (3/2)\kappa_2$ and $x_{tt}^2 \geq 3/4$ for
all $\alpha>\alpha_0$ sufficiently large, this shows
\[\EE[(U_*-U_t)^2] \leq 2\kappa_2/\alpha^2.\]
Now applying $\EE[U_*U_t]=\EE[U_t^2]$ and
$\EE[U_*^2]=1$, we may write
\begin{align}
\EE[U_sU_t]-1&=\EE[(U_s-U_*)(U_t-U_*)]+(\EE[U_s^2]-1)+
(\EE[U_t^2]-1)\nonumber\\
&=\EE[(U_s-U_*)(U_t-U_*)]-\EE[(U_s-U_*)^2]
-\EE[(U_t-U_*)^2]\label{eq:UsUtidentity}\\
&\geq -\frac{3}{2}\Big(\EE[(U_s-U_*)^2]+\EE[(U_t-U_*)^2]\Big).
\nonumber
\end{align}
So $\EE[U_sU_t] \geq 1-6\kappa_2/\alpha^2$. By Cauchy-Schwarz, also
$\EE[U_sU_t] \leq \EE[U_s^2]^{1/2}\EE[U_t^2]^{1/2} \leq \EE[U_*^2]=1$,
so $h^\Delta(\x,\y) \in \cX_{I_\Delta}$.

To show the Lipschitz bound for $h^\Delta$, from (\ref{eq:UsUtidentity}) we have
\begin{align}
|h_{st}^\Delta(\x,\y)-h_{st}^\Delta(\x',\y')|
&=\big|\EE[U_sU_t]-\EE[U_s'U_t']\big|\nonumber\\
&\leq \big|\EE[(U_s-U_*)(U_t-U_*)]-\EE[(U_s'-U_*)(U_t'-U_*)]\big|\nonumber\\
&\hspace{0.5in}+\big|\EE[(U_s-U_*)^2]-\EE[(U_s'-U_*)^2]\big|
+\big|\EE[(U_t-U_*)^2]-\EE[(U_t'-U_*)^2]\big|\label{eq:UUdiff}
\end{align}
so it suffices to bound these three terms individually. We demonstrate the
bound for the first term: Let us write
\[\begin{pmatrix} F_s \\ F_t \end{pmatrix}=\begin{pmatrix} \mu_s \\ \mu_t
\end{pmatrix}U_*+\begin{pmatrix}
\beta_{ss} & \beta_{st} \\ \beta_{st} & \beta_{tt} \end{pmatrix}
\begin{pmatrix} W_s \\ W_t \end{pmatrix}\]
where
\[(\mu_s,\mu_t)=(\alpha x_{ss},\alpha x_{tt}), \qquad
\begin{pmatrix} \beta_{ss} & \beta_{st} \\ \beta_{st} & \beta_{tt}
\end{pmatrix}=\left(\Pi\begin{pmatrix} y_{ss} & y_{st} \\ y_{ts} & y_{tt}
\end{pmatrix}\right)^{1/2}, \qquad (W_s,W_t) \sim \N(0,\Id).\]
Here $(\cdot)^{1/2}$ denotes the positive-semidefinite matrix square-root,
given explicitly for $2 \times 2$ matrices by
\begin{equation}\label{eq:matrixsquareroot}
M^{1/2}=\frac{1}{\sqrt{\Tr M+2\sqrt{\det M}}}(M+\sqrt{\det M} \cdot
\Id_{2 \times 2}).
\end{equation}
Then
\[U_s=\eta(F_s \mid \alpha x_{ss},y_{ss})
=\eta(F_s \mid \mu_s,\beta_{ss}^2+\beta_{st}^2),
\qquad U_t=\eta(F_t \mid \alpha x_{tt},y_{tt})
=\eta(F_t \mid \mu_t,\beta_{st}^2+\beta_{tt}^2).\]
For $\lambda \in [0,1]$, writing the same forms for $F_s',F_t',U_s',U_t'$,
we define the linear interpolations
\[\mu_s^\lambda=\lambda \mu_s+(1-\lambda)\mu_s',
\qquad \beta_{ss}^\lambda=\lambda \beta_{ss}+(1-\lambda) \beta_{ss}',
\qquad F_s^\lambda=\lambda F_s+(1-\lambda) F_s',\]
and similarly for $\mu_t,\beta_{st},\beta_{tt},F_t$. Finally, we define
\[\sigma_{ss}^\lambda=(\beta_{ss}^\lambda)^2+(\beta_{st}^\lambda)^2,
\qquad \sigma_{tt}^\lambda=(\beta_{st}^\lambda)^2+(\beta_{tt}^\lambda)^2,\]
\begin{equation}\label{eq:Ulambda}
U_s^\lambda=\eta(F_s^\lambda \mid \mu_s^\lambda,\sigma_{ss}^\lambda),
\qquad
U_t^\lambda=\eta(F_t^\lambda \mid \mu_t^\lambda,\sigma_{tt}^\lambda).
\end{equation}
Denoting $\partial_\lambda$ as the derivative in $\lambda$, we then have
\[\EE[(U_s-U_*)(U_t-U_*)]-\EE[(U_s'-U_*)(U_t'-U_*)]
=\int_0^1 \partial_\lambda \EE[(U_s^\lambda-U_*)(U_t^\lambda-U_*)]d\lambda\]
where this latter expectation is over the underlying random variables
$(U_*,W_s,W_t)$. The law of
$(U_*,W_s,W_t)$ does not depend on $\lambda$, so we may take the derivative
inside this expectation, yielding
\begin{align}
&\Big|\EE[(U_s-U_*)(U_t-U_*)]-\EE[(U_s'-U_*)(U_t'-U_*)]\Big|\nonumber\\
&\leq \int_0^1
\EE\left[\big|\partial_\lambda U_s^\lambda \cdot
(U_t^\lambda-U_*)\big|+\big|(U_s^\lambda-U_*) \cdot \partial_\lambda U_t^\lambda
\big|\right]\;d\lambda\label{eq:UUdifftmp}
\end{align}
Observe from (\ref{eq:matrixsquareroot}) and the condition
$y_{ss},y_{st},y_{ts},y_{tt}>0$ that $\beta_{ss},\beta_{st},\beta_{tt}>0$. Then
\begin{align*}
\sigma_{ss}^\lambda &\leq 2\lambda^2(\beta_{ss}^2+\beta_{st}^2)
+2(1-\lambda)^2({\beta_{ss}'}^2+{\beta_{st}'}^2)
=2\lambda^2 y_{ss}+2(1-\lambda)^2y_{ss}'\\
\sigma_{ss}^\lambda &\geq \lambda^2(\beta_{ss}^2+\beta_{st}^2)
+(1-\lambda)^2({\beta_{ss}'}^2+{\beta_{st}'}^2)
=\lambda^2 y_{ss}+(1-\lambda)^2y_{ss}'.
\end{align*}
Recalling that $\x \in \cX_{I_\Delta}$ and $\y \in \cX_{I_\Sigma}$,
we get the bounds
\begin{equation}\label{eq:musigmabounds}
\mu_s^\lambda \leq \alpha, \qquad \mu_s^\lambda \geq \alpha-C/\alpha,
\qquad \sigma_{ss}^\lambda \leq C, \qquad \sigma_{ss}^\lambda \geq c.
\end{equation}
for some constants $C,c>0$.
Applying these bounds together with (\ref{eq:posteriorvariance}) yields
$\EE[(U_s^\lambda-U_*)^2] \leq C/\alpha^2$. The same argument holds for
$\EE[(U_t^\lambda-U_*)^2]$. Then applying Cauchy-Schwarz to
(\ref{eq:UUdifftmp}),
\begin{equation}\label{eq:interp}
\Big|\EE[(U_s-U_*)(U_t-U_*)]-\EE[(U_s'-U_*)(U_t'-U_*)]\Big|
\leq \frac{C}{\alpha} \int_0^1
\left(\EE\left[(\partial_\lambda U_s^\lambda)^2
\right]^{1/2}+\EE\left[(\partial_\lambda
U_t^\lambda)^2\right]^{1/2}\right)d\lambda.
\end{equation}

We proceed to bound $\EE[(\partial_\lambda U_s^\lambda)^2]$:
Recall $\eta(f \mid \mu,\sigma^2)$ and $v(f \mid \mu,\sigma^2)$
from (\ref{eq:posterioruvfuncs}), and define in addition
\[k(f \mid \mu,\sigma^2)=\Cov[U_*,U_*^2 \mid F=f].\]
Differentiating the explicit form for $\eta(f \mid \mu,\sigma^2)$
in (\ref{eq:denoiserexplicit}) yields
\begin{align*}
\frac{\partial}{\partial f} \eta(f \mid \mu,\sigma^2)
&=\Cov\left[U_*,\;\frac{\partial}{\partial f} \left(-\frac{(f-\mu
U_*)^2}{2\sigma^2}\right) \bigg| F=f\right]
=\frac{\mu}{\sigma^2} \cdot v(f \mid \mu,\sigma^2)\\
\frac{\partial}{\partial \mu} \eta(f \mid \mu,\sigma^2)
&=\Cov\left[U_*,\;\frac{\partial}{\partial \mu}\left(
-\frac{(f-\mu U_*)^2}{2\sigma^2}\right)\bigg|F=f\right]
=\frac{f}{\sigma^2}v(f \mid \mu,\sigma^2)-\frac{\mu}{\sigma^2}k(f \mid
\mu,\sigma^2)\\
\frac{\partial}{\partial \sigma^2} \eta(f \mid \mu,\sigma^2)
&=\Cov\left[U_*,\;\frac{\partial}{\partial \sigma^2}\left(
-\frac{(f-\mu U_*)^2}{2\sigma^2}\right)\bigg|F=f\right]
=-\frac{f\mu}{\sigma^4}v(f \mid \mu,\sigma^2)
+\frac{\mu^2}{2\sigma^4}k(f \mid \mu,\sigma^2).
\end{align*}
Let us write as shorthand
\[V_s^\lambda=v(F_s^\lambda \mid \mu_s^\lambda,\sigma_{ss}^\lambda),
\qquad K_s^\lambda=k(F_s^\lambda \mid \mu_s^\lambda,\sigma_{ss}^\lambda).\]
Then applying the chain rule to differentiate (\ref{eq:Ulambda}),
\begin{align}
\partial_\lambda U_s^\lambda
&=\frac{\partial \eta}{\partial f} \cdot \left(\partial_\lambda
\mu_s^\lambda \cdot U_*
+\partial_\lambda \beta_{ss}^\lambda \cdot W_s
+\partial_\lambda \beta_{st}^\lambda \cdot W_t\right)
+\frac{\partial \eta}{\partial \mu} \cdot \partial_\lambda \mu_s^\lambda
+\frac{\partial \eta}{\partial \sigma^2} \cdot
\left(2\beta_{ss}^\lambda \cdot \partial_\lambda \beta_{ss}^\lambda
+2\beta_{st}^\lambda \cdot \partial_\lambda \beta_{st}^\lambda\right)\nonumber\\
&=\left(\frac{\mu_s^\lambda}{\sigma_{ss}^\lambda}
U_*V_s^\lambda+\frac{1}{\sigma_{ss}^\lambda}F_s^\lambda V_s^\lambda
-\frac{\mu_s^\lambda}{\sigma_{ss}^\lambda}K_s^\lambda\right)
\partial_\lambda \mu_s^\lambda
+\left(\frac{\mu_s^\lambda}{\sigma_{ss}^\lambda}W_sV_s^\lambda
-\frac{2\beta_{ss}^\lambda \mu_s^\lambda}{(\sigma_{ss}^\lambda)^2}
F_s^\lambda V_s^\lambda+\frac{(\mu_s^\lambda)^2
\beta_{ss}^\lambda}{(\sigma_{ss}^\lambda)^2}K_s^\lambda \right)
\partial_\lambda \beta_{ss}^\lambda\nonumber\\
&\hspace{0.5in}+\left(\frac{\mu_s^\lambda}{\sigma_{ss}^\lambda}W_tV_s^\lambda
-\frac{2\beta_{st}^\lambda \mu_s^\lambda}{(\sigma_{ss}^\lambda)^2}
F_s^\lambda V_s^\lambda+\frac{(\mu_s^\lambda)^2
\beta_{st}^\lambda}{(\sigma_{ss}^\lambda)^2}K_s^\lambda \right)
\partial_\lambda \beta_{st}^\lambda\nonumber\\
&\equiv \mathrm{I} \cdot \partial_\lambda \mu_s^\lambda
+\mathrm{II} \cdot \partial_\lambda \beta_{ss}^\lambda
+\mathrm{III} \cdot \partial_\lambda \beta_{st}^\lambda.\label{eq:dpostmeanPCA}
\end{align}

We bound the expected squares of these coefficients
$\mathrm{I},\mathrm{II},\mathrm{III}$:
Applying (\ref{eq:musigmabounds}) and Cauchy-Schwarz,
\begin{align*}
\EE[(U_*V_s^\lambda)^2] &\leq \EE[U_*^4]^{1/2}\EE[(V_s^\lambda)^4]^{1/2}
\leq C \cdot \EE[(V_s^\lambda)^4]^{1/2}\\
\EE[(F_s^\lambda V_s^\lambda)^2]
&\leq \EE[(F_s^\lambda)^4]^{1/2}\EE[(V_s^\lambda)^4]^{1/2}
\leq C\alpha^2 \EE[(V_t^\lambda)^4]^{1/2}.
\end{align*}
Then the coefficient for $\partial_\lambda \mu_s^\lambda$ in
(\ref{eq:dpostmeanPCA}) has expected square bounded as
\[\EE[\mathrm{I}^2]=\EE\left[\left(\frac{\mu_s^\lambda}{\sigma_{ss}^\lambda}
U_*V_s^\lambda+\frac{1}{\sigma_{ss}^\lambda}F_s^\lambda V_s^\lambda
-\frac{\mu_s^\lambda}{\sigma_{ss}^\lambda}K_s^\lambda\right)^2\right]
\leq C\alpha^2\left(\EE[(V_s^\lambda)^4]^{1/2}+\EE[(K_s^\lambda)^2]\right).\]
Recalling the identity
\begin{equation}\label{eq:PCAvidentity}
v(F \mid \mu,\sigma^2)
=\Var[U_* \mid F]=\Var[\mu^{-1}(F-\sigma W) \mid F]=(\sigma/\mu)^2
\Var[W \mid F],
\end{equation}
we obtain
\begin{align}
\EE\big[v(F \mid \mu,\sigma^2)^4\big]
&=(\sigma/\mu)^8 \EE\big[\Var[W \mid F]^4\big]\nonumber\\
&=(\sigma/\mu)^8 \EE\Big[\EE[(W-\EE[W \mid F])^2 \mid F]^4\Big]\nonumber\\
&\leq (\sigma/\mu)^8 \EE[(W-\EE[W \mid F])^8] \leq
C(\sigma/\mu)^8,\label{eq:v4bound}
\end{align}
the last inequality applying $W \sim \N(0,1)$ and
Proposition \ref{prop:conditionalcentralmoment}. So $\EE[(V_s^\lambda)^4]
\leq C\alpha^{-8}$. We also have the identity
\begin{align}
k(F \mid \mu,\sigma^2)
=\Cov[U_*,U_*^2 \mid F]&=\mu^{-3}\Cov[F-\sigma W,(F-\sigma W)^2 \mid F]\nonumber\\
&=(2\sigma^2F/\mu^3)\Var[W \mid F]-(\sigma^3/\mu^3)\Cov[W,W^2 \mid
F],\label{eq:PCAkidentity}
\end{align}
so
\begin{align*}
&\EE[k(F \mid \mu,\sigma^2)^2]\\
&\leq (8\sigma^4/\mu^6)\EE[F^2\Var[W \mid F]^2]
+(2\sigma^6/\mu^6)\EE[\Cov[W,W^2 \mid F]^2]\\
&\leq (8\sigma^4/\mu^6)\EE[F^4]^{1/2}\EE[\Var[W \mid F]^4]^{1/2}
+(2\sigma^6/\mu^6)\EE[\Var[W \mid F]^2]^{1/2}\EE[\Var[W^2 \mid F]^2]^{1/2}\\
&\leq C(\sigma^4/\mu^6)\EE[F^4]^{1/2}+C(\sigma/\mu)^6
\end{align*}
by similar arguments. Then $\EE[(K_s^\lambda)^2] \leq C\alpha^{-4}$, and
we obtain
\begin{equation}\label{eq:xdiffcoef}
\EE[\mathrm{I}^2] \leq C/\alpha^2.
\end{equation}

For the coefficients of $\partial_\lambda \beta_{ss}^\lambda$ 
and $\partial_\lambda \beta_{st}^\lambda$ in (\ref{eq:dpostmeanPCA}), we first
apply a cancellation of the leading-order term: Comparing the identities
(\ref{eq:PCAvidentity}) and (\ref{eq:PCAkidentity}), we have
\[K_s^\lambda=\frac{2}{\mu_s^\lambda}F_s^\lambda V_s^\lambda
-\frac{(\sigma_{ss}^\lambda)^{3/2}}{(\mu_s^\lambda)^3}
\Cov[W,W^2 \mid F_s^\lambda]\]
where $W \sim \N(0,1)$ is the Gaussian variable such that
$F_s^\lambda=\mu_s^\lambda \cdot U_*+\sqrt{\sigma_{ss}^\lambda} \cdot W$.
Then the coefficient of $\partial_\lambda \beta_{ss}^\lambda$ in
(\ref{eq:dpostmeanPCA}) is
\[\mathrm{II}=\frac{\mu_s^\lambda}{\sigma_{ss}^\lambda}W_sV_s^\lambda
-\frac{2\beta_{ss}^\lambda \mu_s^\lambda}{(\sigma_{ss}^\lambda)^2}
F_s^\lambda V_s^\lambda+\frac{(\mu_s^\lambda)^2
\beta_{ss}^\lambda}{(\sigma_{ss}^\lambda)^2}K_s^\lambda 
=\frac{\mu_s^\lambda}{\sigma_{ss}^\lambda}W_sV_s^\lambda
-\frac{\beta_{ss}^\lambda}{\mu_s^\lambda(\sigma_{ss}^\lambda)^{1/2}}
\Cov[W,W^2 \mid F_s^\lambda].\]
Similar arguments as above yield $\EE[(W_s^2V_s^\lambda)^2]
\leq C\alpha^{-4}$ and $\EE[\Cov[W,W^2 \mid F_s^\lambda]^2] \leq C$.
Thus we obtain the bound
\begin{equation}\label{eq:ydiffcoef1}
\EE[\mathrm{II}^2] \leq C/\alpha^2.
\end{equation}
For the coefficient of $\partial_\lambda \beta_{st}^\lambda$, the same argument
shows
\begin{equation}\label{eq:ydiffcoef2}
\EE[\mathrm{III}^2] \leq C/\alpha^2.
\end{equation}
Applying (\ref{eq:xdiffcoef}), (\ref{eq:ydiffcoef1}), and (\ref{eq:ydiffcoef2})
to (\ref{eq:dpostmeanPCA}) yields
\[\EE\left[(\partial_\lambda U_s^\lambda)^2\right]
\leq \frac{C}{\alpha^2}\left((\partial_\lambda \mu_s^\lambda)^2
+(\partial_\lambda \beta_{ss}^\lambda)^2+(\partial_\lambda
\beta_{st}^\lambda)^2\right).\]
The same argument applies for $U_t^\lambda$ to show
\[\EE\left[(\partial_\lambda U_t^\lambda)^2\right]
\leq \frac{C}{\alpha^2}\left((\partial_\lambda \mu_t^\lambda)^2
+(\partial_\lambda \beta_{st}^\lambda)^2+(\partial_\lambda
\beta_{tt}^\lambda)^2\right).\]
By the definition of our linear interpolation,
$\partial_\lambda \mu_s^\lambda=\mu_s-\mu_s'$ which does not depend on
$\lambda$, and similarly for the other derivatives above. Then applying this to
(\ref{eq:interp}) yields a bound of
\[\frac{C}{\alpha^2}\Big(|\mu_s-\mu_s'|+|\mu_t-\mu_t'|
+|\beta_{ss}-\beta_{ss}'|+|\beta_{st}-\beta_{st}'|
+|\beta_{tt}-\beta_{tt}'|\Big)\]
for the first term in (\ref{eq:UUdiff}).

The same argument applied with $s=t$ bounds the other two terms
of (\ref{eq:UUdiff}), and we obtain
\begin{align*}
|h^\Delta_{st}(\x,\y)-h^\Delta_{st}(\x',\y')| &\leq
\frac{C}{\alpha^2}\Big(|\mu_s-\mu_s'|+|\mu_t-\mu_t'|
+|\beta_{ss}-\beta_{ss}'|+|\beta_{st}-\beta_{st}'|
+|\beta_{tt}-\beta_{tt}'|\Big)
\end{align*}
Observe that $|\mu_s-\mu_s'|=\alpha|x_{ss}-x_{ss}'|$. Furthermore, it may be
verified from (\ref{eq:Pi}) and (\ref{eq:matrixsquareroot}) that the map
$(y_{ss},y_{st},y_{ts},y_{tt}) \mapsto (\beta_{ss},\beta_{st},\beta_{tt})$
is Lipschitz over $y_{ss},y_{st},y_{ts},y_{tt} \in I_\Sigma$,
since $I_\Sigma$ is bounded away from 0. Then
\begin{align*}
|h^\Delta_{st}(\x,\y)-h^\Delta_{st}(\x',\y')| &\leq \frac{C}{\alpha}
\left(|x_{ss}-x_{ss}'|+|x_{tt}-x_{tt}'|\right)\\
&\hspace{0.5in}
+\frac{C}{\alpha^2}\left(|y_{ss}-y_{ss}'|+|y_{st}-y_{st}'|+|y_{ts}-y_{ts}'|
+|y_{tt}-y_{tt}'|\right).
\end{align*}
This implies
\[\|h^\Delta(\x,\y)-h^\Delta(\x',\y')\|_\zeta
=\sup_{s,t \leq 0} |h^\Delta(\x,\y)_{s,t}-h^\Delta(\x',\y')_{s,t}|\zeta^{|s| \vee |t|}
\leq \frac{C}{\alpha}\|\x-\x'\|_\zeta+\frac{C}{\alpha^2}\|\y-\y'\|_\zeta,\]
which shows part (b).
\end{proof}

The next lemma establishes the approximation of the state evolution for
$\bSigma_T$ by the fixed map $h^\Sigma$, and the state evolution for $\bDelta_T$
by the fixed map $h^\Delta$.

\begin{lemma}\label{lemma:PCAapproxmap}
In the setting of Theorem \ref{thm:PCA}(b), there exist constants $C,\alpha_0>0$
such that for all $\alpha>\alpha_0$ and $T \geq 1$:
\begin{enumerate}[(a)]
\item $\delta_{st} \in I_\Delta$ and $\sigma_{st} \in I_\Sigma$ for all $s,t
\in \{2,\ldots,T\}$ whereas $\delta_{st} \in [-1,1]$ and $\sigma_{st} \in
[-3\kappa_2/2,\;3\kappa_2/2]$ if $s=1$ and $t \in \{1,\ldots,T\}$ or $t=1$ and
$s \in \{1,\ldots,T\}$.
\item Consider $\bDelta_T,\bSigma_T$ as elements of
$\cX$, with the coordinate identifications and zero-padding
of (\ref{eq:Xembedding}). Then for any
$(\x,\y) \in \cX_{I_\Delta} \times \cX_{I_\Sigma}$,
\[\|\bSigma_T-h^\Sigma(\x,\y)\|_\zeta \leq C\alpha\|\x-\bDelta_T\|_\zeta
+(C/\alpha)\|\y-\bSigma_{T-1}\|_\zeta+C\zeta^{T/2}.\]
\item Similarly, for any
$(\x,\y) \in \cX_{I_\Delta} \times \cX_{I_\Sigma}$,
\[\|\bDelta_{T+1}-h^\Delta(\x,\y)\|_\zeta \leq (C/\alpha)\|\x-\bDelta_T\|_\zeta
+(C/\alpha^2)\|\y-\bSigma_T\|_\zeta+C\zeta^T.\]
\end{enumerate}
\end{lemma}
\begin{proof}
For part (a), the arguments are similar to those in the proof of
Lemma \ref{lemma:PCAcontraction}: We induct on $T$.
Note that $\delta_{11}=\EE[U_1^2] \leq 1$ so the claim holds for $\bDelta_1$.
Suppose that the claims hold for $\bDelta_T$ and $\bSigma_{T-1}$.
To establish the claim for $\bSigma_T$,
for any $s,t \geq 1$ we may write (\ref{eq:PCAsigma}) as
\[\sigma_{st}=\sum_{j=0}^{s-1}\sum_{k=0}^{t-1} \kappa_{j+k+2}
\sigma_{st}^{(j,k)}.\]
Observe that $\sigma_{st}^{(0,0)}=\delta_{st}$, which belongs to $I_\Delta$ if
$s,t \geq 2$ and to $[-1,1]$ otherwise. Observe also that
$|\sigma_{st}^{(j,k)}| \leq (C/\alpha)^{j+k}$ for all $(j,k) \neq (0,0)$
by the same argument as in Lemma \ref{lemma:PCAcontraction}(a).
(Here, each factor $1-\delta_{ii}$ has an index $i \geq 2$, so this is at most
$C/\alpha^2$.)
For sufficiently large $\alpha$, applying $|\kappa_{j+k+2}| \leq
(16C_0)^{j+k+2}$ and summing over $(j,k) \neq (0,0)$ shows the claim
for $\bSigma_T$.
Now suppose that the claims of part (a) hold for $\bDelta_T$
and $\bSigma_T$, and consider $\bDelta_{T+1}$. For $s=1$ or $t=1$, we apply
$|\delta_{s,T+1}|=|\EE[U_1U_{T+1}]| \leq
\EE[U_1^2]^{1/2}\EE[U_{T+1}^2]^{1/2} \leq 1$.
For $s,t \geq 2$, the argument of (\ref{eq:UsUtidentity}) in
Lemma \ref{lemma:PCAcontraction}(b) shows
\[\delta_{st}-1 \geq -\frac{3}{2}\left(\frac{\sigma_{s-1,s-1}}
{\mu_{s-1}^2}+\frac{\sigma_{t-1,t-1}}{\mu_{t-1}^2}\right).\]
We recall that $\mu_j=\alpha \delta_{jj}$ if $j \geq 2$ and
$\mu_1=\alpha \eps$. For sufficiently large $\alpha$, this implies
$\delta_{st} \in I_\Delta$ in both cases (where the constant $C$ defining
$I_\Delta$ depends on $\eps$), so the claim holds for
$\bDelta_{T+1}$. This concludes the induction and establishes part (a).

For part (b), let us now index the entries of
$\bDelta_T,\bSigma_{T-1},\bSigma_T$ by
$s,t \leq 0$, to coincide with the indices of $\cX$. Let
$\x' \in \cX_{I_\Delta}$ be $\bDelta_T$ with each coordinate projected onto
the interval $I_\Delta$, and let $\y' \in \cX_{I_\Sigma}$
be the analogous projection of $\bSigma_{T-1}$ onto
$I_\Sigma$. We first bound $\|\bSigma_T-h^\Sigma(\x',\y')\|_\zeta$.
Observe that by part (a) already shown, $\x'$ must coincide
with $\bDelta_T$ in the lower-right $(T-1) \times (T-1)$ corner,
and $\y'$ must coincide with $\bSigma_{T-1}$ in the lower-right
$(T-2) \times (T-2)$ corner.
Applying again $|\kappa_{j+k+2}| \leq (16C_0)^{j+k+2}$, we may write
\[(\bSigma_T)_{st}=\sum_{0 \leq j,k<T/4}
\kappa_{j+k+2}(\bSigma_T)_{st}^{(j,k)}+R_{st}\]
where this remainder satisfies $|R_{st}| \leq C(C'/\alpha)^{T/4}$ for
$\alpha>\alpha_0$ sufficiently large. Similarly, we may write
\[h^\Sigma_{st}(\x',\y')=\sum_{0 \leq j,k<T/4}
\kappa_{j+k+2}h_{st}^{(j,k)}(\x',\y')+R_{st}(\x',\y')\]
where $|R_{st}(\x',\y')| \leq C(C'/\alpha)^{T/4}$. Comparing the forms of
(\ref{eq:PCAsigma}) and (\ref{eq:PCASigmaapprox}), observe that for $s,t>-T/2$
and $j,k<T/4$, we have $(\bSigma_T)_{st}^{(j,k)}=h_{st}^{(j,k)}(\x',\y')$
because these are identical functions of the entries of the lower-right
$(3T/4) \times (3T/4)$ sub-matrices of $\bDelta_T$ and $\bSigma_{T-1}$. So
\[|(\bSigma_T)_{st}-h_{st}^\Sigma(\x',\y')| \leq
C(C'/\alpha)^{T/4} \text{ for all } s,t>-T/2.\] 
Applying the trivial bound
\[|(\bSigma_T)_{st}-h^\Sigma_{st}(\x',\y')|
\leq |(\bSigma_T)_{st}|+|h^\Sigma_{st}(\x',\y')| \leq 3\kappa_2
\text{ for } s \leq -T/2 \text{ or } t \leq -T/2,\]
we obtain
\[\|\bSigma_T-h^\Sigma(\x',\y')\|_\zeta \leq C(C'/\alpha)^{T/4}+C\zeta^{T/2}
\leq C\zeta^{T/2}\]
for $\alpha>\alpha_0$ large enough. By Lemma \ref{lemma:PCAcontraction}(a) and
the definitions of $\x',\y'$, we have also
\begin{align*}
\|h^\Sigma(\x',\y')-h^\Sigma(\x,\y)\|_\zeta &\leq
C\alpha\|\x-\x'\|_\zeta+(C/\alpha)\|\y-\y'\|_\zeta\\
&\leq C\alpha\|\x-\bDelta_T\|_\zeta+(C/\alpha)\|\y-\bSigma_{T-1}\|_\zeta,
\end{align*}
and combining these shows part (b).

For part (c), now let $\x'$ and $\y'$ be the coordinate-wise projections of
$\bDelta_T$ and $\bSigma_T$ onto $I_\Delta$ and $I_\Sigma$. We bound
$\|\bDelta_{T+1}-h^\Delta(\x',\y')\|_\zeta$. Observe that $\x'$ and $\y'$
coincide with $\bDelta_T$ and $\bSigma_T$ in their lower-right
$(T-1) \times (T-1)$ corners, and each $2 \times 2$ principal minor of
$\bSigma_T$ must be positive-semidefinite because $\bSigma_T$ is a covariance
matrix. Thus
\[(\bDelta_{T+1})_{st}=h^\Delta_{st}(\x',\y') \text{ for all } s,t \in
\{-T+2,\ldots,0\}.\]
Applying the trivial bound $|(\bDelta_{T+1})_{st}-h^\Delta(\x',\y')_{st}| \leq 2$ for the remaining $s,t$, and
\begin{align*}
\|h^\Delta(\x,\y)-h^\Delta(\x',\y')\|_\zeta &\leq (C/\alpha)\|\x-\x'\|_\zeta
+(C/\alpha^2)\|\y-\y'\|_\zeta \\
&\leq (C/\alpha)\|\x-\bDelta_T\|_\zeta+(C/\alpha^2)\|\y-\bSigma_T\|_\zeta
\end{align*}
similar to the above, we obtain part (c).
\end{proof}

\begin{proof}[Proof of Theorem \ref{thm:PCA}(b)]
Theorem \ref{thm:PCA}(a) shows
\[\lim_{n \to \infty} n^{-1}\|\u_T\|^2=\EE[U_T^2]=\delta_{TT},\qquad
\lim_{n \to \infty} n^{-1}\u_T^\top \u_*=\EE[U_TU_*]=\EE[U_T^2]=\delta_{TT}.\]
Thus, it suffices to show that $\delta_{TT} \to \Delta_*$ as $T \to \infty$,
where $(\Delta_*,\Sigma_*) \in I_\Delta \times I_\Sigma$ is the unique fixed
point of (\ref{eq:PCAfixedpoint}).

Consider the map $G:\cX_{I_\Delta} \times \cX_{I_\Sigma} \to
\cX_{I_\Delta} \times \cX_{I_\Sigma}$ that is the successive composition of
\[(\x,\y) \mapsto (\x,h^\Sigma(\x,\y)), \qquad
(\x,\y) \mapsto (h^\Delta(\x,\y),\y)\]
which approximates $(\bDelta_T,\bSigma_{T-1}) \mapsto
(\bDelta_{T+1},\bSigma_T)$. Writing its components as $G=(G_x,G_y)$,
Lemma \ref{lemma:PCAcontraction} implies
\begin{align*}
\|G_y(\x,\y)-G_y(\x',\y')\|_\zeta &\leq
C\alpha\|\x-\x'\|_\zeta+(C/\alpha)\|\y-\y'\|_\zeta\\
\|G_x(\x,\y)-G_x(\x',\y')\|_\zeta &\leq
(C/\alpha)\|\x-\x'\|_\zeta+(C/\alpha^2)\|G_y(\x,\y)-G_y(\x',\y')\|_\zeta\\
&\leq (C'/\alpha)\|\x-\x'\|_\zeta+(C'/\alpha^3)\|\y-\y'\|_\zeta.
\end{align*}
Then defining the norm $\|\cdot\|_{\zeta,\alpha}$ on the product space
$\cX_{I_\Delta} \times \cX_{I_\Sigma}$ by
\[\|(\x,\y)\|_{\zeta,\alpha}=\|\x\|_\zeta+(1/\alpha^2)\|\y\|_\zeta,\]
this shows
\[\|G(\x,\y)-G(\x',\y')\|_{\zeta,\alpha}
\leq (C/\alpha)\|(\x,\y)-(\x',\y')\|_{\zeta,\alpha}
\leq \tau\|(\x,\y)-(\x',\y')\|_{\zeta,\alpha}\]
for some constant $\tau \in (0,1)$ and all $\alpha>\alpha_0$ sufficiently large.
Thus $G$ is a contraction on $\cX_{I_\Delta} \times \cX_{I_\Sigma}$ in this
norm, and admits a unique fixed point $(\x_*,\y_*)
\in \cX_{I_\Delta} \times \cX_{I_\Sigma}$ by the Banach fixed point theorem.

We claim that this fixed point is such that $\x_*$ equals a constant $\Delta_*
\in I_\Delta$ and $\y_*$ equals a constant $\Sigma_* \in I_\Sigma$
in every coordinate. By the definitions of the functions $h^\Delta$ and
$h^\Sigma$, such a pair is a fixed point if and only if
\[\Sigma_*=\sum_{j,k=0}^\infty
\kappa_{j+k+2}\left(\frac{\alpha \Delta_*}{\Sigma_*}(1-\Delta_*)\right)^{j+k}
\Delta_*, \qquad \Delta_*=\EE[\EE[U_* \mid F]^2]\]
in the model $F=\alpha \Delta_* U_*+Z$ where $Z \sim \N(0,\Sigma_*)$.
These equations may be rewritten as
\begin{align*}
\Sigma_*&=\sum_{k=0}^\infty (k+1)\kappa_{k+2}\left(\frac{\alpha
\Delta_*}{\Sigma_*}(1-\Delta_*)\right)^k \Delta_*
=\Delta_* R'\left(\frac{\alpha \Delta_*(1-\Delta_*)}{\Sigma_*}\right)\\
1-\Delta_*&=\EE[(U_*-\EE[U_* \mid F])^2]
=\mmse\left(\frac{\alpha^2\Delta_*^2}{\Sigma_*}\right)
\end{align*}
which is exactly the pair of fixed point equations (\ref{eq:PCAfixedpoint}).

To argue that such a fixed point exists and is unique in $I_\Delta \times
I_\Sigma$, consider the pair of scalar maps
\[h^\Sigma_\text{sc}(\Delta_*,\Sigma_*)=\Delta_* R'\left(\frac{\alpha
\Delta_*(1-\Delta_*)}{\Sigma_*}\right),
\qquad h^\Delta_\text{sc}(\Delta_*,\Sigma_*)=1-\mmse\left(\frac{\alpha^2\Delta_*^2}
{\Sigma_*}\right)=\EE[\EE[U_* \mid F]^2].\]
Denote their composition as $G_\text{sc}(\Delta_*,\Sigma_*)$.
Specializing Lemma \ref{lemma:PCAcontraction} to pairs $(\x,\y)$ and
$(\x',\y')$ where $\x,\x',\y,\y'$ are each equal to a constant in every
coordinate, our preceding arguments
imply that $G_{\text{sc}}:I_\Delta \times I_\Sigma
\to I_\Delta \times I_\Sigma$ is a contraction with respect to the norm
$\|(\Delta,\Sigma)\|=|\Delta|+(1/\alpha^2)|\Sigma|$. Then there exists a unique
fixed point $(\Delta_*,\Sigma_*) \in I_\Delta \times I_\Sigma$ for
$G_{\text{sc}}$, by the Banach fixed point theorem applied to this scalar
setting. So the fixed point $(\x_*,\y_*)$ for $G$ must be such that
$\x_*$ is constant and equal to $\Delta_*$, and $\y_*$ is constant and equal to
$\Sigma_*$, by uniqueness of $(\x_*,\y_*)$.

Finally, to conclude the proof, fix any $\eps>0$. Let $G^{T_0}=G \circ \ldots
\circ G$ denote the $T_0$-fold composition of $G$. For the above contraction
rate $\tau$ of this function $G$, the Banach fixed point theorem implies
quantitatively, for any $\x,\y \in \cX_{I_\Delta} \times \cX_{I_\Sigma}$,
\[\|G^{T_0}(\x,\y)-(\x_*,\y_*)\|_{\zeta,\alpha} \leq 
\tau^{T_0}\|(\x,\y)-(\x_*,\y_*)\|_{\zeta,\alpha} \leq C\tau^{T_0}\]
where the second inequality holds because
$\cX_{I_\Delta} \times \cX_{I_\Sigma}$ is bounded under
$\|\cdot\|_{\zeta,\alpha}$. Then for all large enough $T_0$, we have
\[\|G^{T_0}(\x,\y)-(\x_*,\y_*)\|_{\zeta,\alpha}<\eps/2.\]
By Lemma \ref{lemma:PCAapproxmap}(b) and (c), for any $(\x,\y) \in
\cX_{I_\Delta} \times \cX_{I_\Sigma}$ and any $T \geq 1$, also
\begin{align*}
&\|(\bDelta_{T+1},\bSigma_T)-G(\x,\y)\|_{\zeta,\alpha}\\
&=\|\bDelta_{T+1}-h^\Delta(\x,h^\Sigma(\x,\y))\|_\zeta+(1/\alpha^2)
\|\bSigma_T-h^\Sigma(\x,\y)\|_\zeta\\
&\leq (C/\alpha)\|\x-\bDelta_T\|_\zeta+C\zeta^T
+[(C/\alpha^2)+(1/\alpha^2)]
\Big(C\alpha \|\x-\bDelta_T\|_\zeta+(C/\alpha)\|\y-\bSigma_{T-1}\|_\zeta
+C\zeta^{T/2}\Big)\\
&\leq \tau\|(\bDelta_T,\bSigma_{T-1})-(\x,\y)\|_{\zeta,\alpha}
+C'\zeta^{T/2}.
\end{align*}
Iterating this bound,
\[\|(\bDelta_{T+T_0},\bSigma_{T+T_0-1})-G^{T_0}(\x,\y)\|_{\zeta,\alpha}
\leq \tau^{T_0}\|(\bDelta_T,\bSigma_{T-1})-(\x,\y)\|_{\zeta,\alpha}
+C'\zeta^{T/2}(1-\tau)^{-1}.\]
For all large enough $T_0$ and $T$, this is also at most $\eps/2$. Thus,
combining with the above,
\[\limsup_{T \to \infty} |\delta_{TT}-\Delta_*| \leq
\limsup_{T \to \infty} \|(\bDelta_{T+1},\bSigma_T)-(\x_*,\y_*)\|_{\zeta,\alpha}
\leq \eps.\]
Here $\eps>0$ is arbitrary, so $|\delta_{TT}-\Delta_*| \to 0$ as desired.
\end{proof}

\subsubsection{Rectangular matrices}

We now prove Theorem \ref{thm:PCArect}(b) using a similar argument.

Recall that the AMP algorithm is
given by (\ref{eq:AMPPCArect1}--\ref{eq:AMPPCArect4}), where $v_t(\cdot)$ and
$u_{t+1}(\cdot)$ are the single-iterate posterior-mean denoisers
in (\ref{eq:vtutBayes}). As in the symmetric square setting, we have
\begin{align*}
u_{t+1}'(f_t)&=\frac{\mu_t}{\sigma_{tt}}\Var[U_* \mid F_t=f_t]
=\frac{\mu_t}{\sigma_{tt}}(1-\delta_{t+1,t+1})\\
v_t'(g_t)&=\frac{\nu_t}{\omega_{tt}}\Var[V_* \mid G_t=g_t]
=\frac{\nu_t}{\omega_{tt}}(1-\gamma_{tt}).
\end{align*}
Here,
\[\mu_t=(\alpha/\gamma) \gamma_{tt}, \qquad
\nu_t=\begin{cases} \alpha \eps & \text{ if } t=1 \\
\alpha \delta_{tt} & \text{ if } t \geq 2. \end{cases}\]

For a sufficiently large constant $C>0$, we define the intervals
\[I_\Delta=I_\Gamma=\left[1-\frac{C}{\alpha^2},\;1\right], \quad
I_\Sigma=\left[\frac{1}{2}\kappa_2,\frac{3}{2}\kappa_2\right], \quad
I_\Omega=\left[\frac{1}{2}\gamma \kappa_2,\frac{3}{2}\gamma\kappa_2 \right]\]
and the corresponding domains
$\cX_{I_\Delta},\cX_{I_\Sigma},\cX_{I_\Gamma},\cX_{I_\Omega}$.
We then define four maps $h^\Omega,h^\Gamma,h^\Sigma,h^\Delta:
\cX_{I_\Delta} \times \cX_{I_\Sigma}
\times \cX_{I_\Gamma} \times \cX_{I_\Omega} \to \cX$ that respectively
approximate the state evolution functions
\begin{align*}
(\bDelta_T,\bSigma_{T-1},\bGamma_{T-1},\bOmega_{T-1}) &\mapsto \bOmega_T\\
(\bDelta_T,\bSigma_{T-1},\bGamma_{T-1},\bOmega_T) &\mapsto \bGamma_T\\
(\bDelta_T,\bSigma_{T-1},\bGamma_T,\bOmega_T) &\mapsto \bSigma_T\\
(\bDelta_T,\bSigma_T,\bGamma_T,\bOmega_T) &\mapsto \bDelta_{T+1}.
\end{align*}
Substituting the above forms of the derivatives into
(\ref{eq:PCAsigmarect}) and (\ref{eq:PCAomegarect}), and
identifying $(\bDelta,\bSigma,\bGamma,\bOmega) \leftrightarrow
(\x,\y,\z,\w)$ with the appropriate offsets of indices,
we may define these maps to have the entries
\begin{align*}
h^\Omega_{st}(\x,\y,\z,\w)
&=\gamma \sum_{j=0}^\infty \sum_{k=0}^\infty 
\left(\left(\prod_{i=s-j+1}^s \times \prod_{i=t-k+1}^t\right)
\frac{\alpha z_{ii}}{\gamma y_{ii}}
(1-x_{ii})\frac{\alpha x_{i-1,i-1}}{w_{ii}}(1-z_{ii})\right)
\bigg(\kappa_{2(j+k+1)}x_{s-j,t-k}\\
&\hspace{0.5in}+\kappa_{2(j+k+2)}
\frac{\alpha z_{s-j,s-j}}{\gamma y_{s-j,s-j}} (1-x_{s-j,s-j})
\frac{\alpha z_{t-k,t-k}}{\gamma y_{t-k,t-k}} (1-x_{t-k,t-k})z_{s-j,t-k}\bigg),\\
h^\Gamma_{st}(\x,\y,\z,\w)&=\EE_{\x,\w}\Big[\EE_{\x,\w}[V_* \mid G_s]
\EE_{\x,\w}[V_* \mid G_t]\Big]\\
h^\Sigma_{st}(\x,\y,\z,\w)
&=\sum_{j=0}^\infty \sum_{k=0}^\infty
\left(\left(\prod_{i=s-j+1}^s \times \prod_{i=t-k+1}^t\right)
 \frac{\alpha x_{ii}}{w_{ii}}(1-z_{ii})\frac{\alpha z_{i-1,i-1}}
{\gamma y_{ii}}(1-x_{ii}) \right)
\bigg(\kappa_{2(j+k+1)} z_{s-j,t-k}\\
&\hspace{0.5in}+\kappa_{2(j+k+2)}
\frac{\alpha x_{s-j,s-j}}{w_{s-j,s-j}}(1-z_{s-j,s-j})
\frac{\alpha x_{t-k,t-k}}{w_{t-k,t-k}}(1-z_{t-k,t-k})x_{s-j,t-k}\bigg),\\
h^\Delta_{st}(\x,\y,\z,\w)&=\EE_{\z,\y}\Big[\EE_{\z,\y}[U_* \mid F_s]
\EE_{\z,\y}[U_* \mid F_t]\Big],
\end{align*}
where these expectations are taken with respect to the $(\x,\y,\z,\w)$-dependent
joint laws
\begin{align*}
(G_s,G_t)&=(\alpha x_{ss},\alpha x_{tt})V_*+
\N\left(0,\;\Pi\begin{pmatrix} w_{ss} & w_{st} \\ w_{ts} & w_{tt}
\end{pmatrix}\right)\\
(F_s,F_t)&=((\alpha/\gamma) z_{ss},(\alpha/\gamma)
z_{tt})U_*+\N\left(0,\;\Pi\begin{pmatrix}
y_{ss} & y_{st} \\ y_{ts} & y_{tt}\end{pmatrix}\right)
\end{align*}
and $\Pi(\cdot)$ is as defined in (\ref{eq:Pi}).
Note that $h^\Gamma$ depends only on $(\x,\w)$, while $h^\Delta$ depends only on
$(\z,\y)$.

The following establishes Lipschitz bounds for these functions,
and is analogous to Lemma \ref{lemma:PCAcontraction}.

\begin{lemma}\label{lemma:PCAcontractionrect}
In the setting of Theorem \ref{thm:PCArect}(b),
there exist constants $C,\alpha_0>0$
such that for all $\alpha>\alpha_0$ and $(\x,\y,\z,\w),(\x',\y',\z',\w') \in
\cX_{I_\Delta} \times \cX_{I_\Sigma} \times \cX_{I_\Gamma} \times
\cX_{I_\Omega}$:
\begin{enumerate}[(a)]
\item $h^\Sigma(\x,\y,\z,\w) \in \cX_{I_\Sigma}$, $h^\Omega(\x,\y,\z,\w) \in \cX_{I_\Omega}$,
and
\begin{align*}
\|h^\Sigma(\x,\y,\z,\w)-h^\Sigma(\x',\y',\z',\w')\|_\zeta
&\leq C(\|\x-\x'\|_\zeta+\|\z-\z'\|)_\zeta+(C/\alpha^2)(\|\y-\y'\|_\zeta
+\|\w-\w'\|_\zeta),\\
\|h^\Omega(\x,\y,\z,\w)-h^\Omega(\x',\y',\z',\w')\|_\zeta
&\leq C(\|\x-\x'\|_\zeta+\|\z-\z'\|)_\zeta+(C/\alpha^2)(\|\y-\y'\|_\zeta
+\|\w-\w'\|_\zeta).
\end{align*}
\item $h^\Delta(\x,\y,\z,\w) \in \cX_{I_\Delta}$, $h^\Gamma(\x,\y,\z,\w) \in \cX_{I_\Gamma}$,
and
\begin{align*}
\|h^\Delta(\x,\y,\z,\w)-h^\Delta(\x',\y',\z',\w')\|_\zeta &\leq
(C/\alpha)\|\z-\z'\|_\zeta+(C/\alpha^2)\|\y-\y'\|_\zeta,\\
\|h^\Gamma(\x,\y,\z,\w)-h^\Gamma(\x',\y',\z',\w')\|_\zeta &\leq
(C/\alpha)\|\x-\x'\|_\zeta+(C/\alpha^2)\|\w-\w'\|_\zeta.
\end{align*}
\end{enumerate}
\end{lemma}
\begin{proof}
For part (a), the argument is similar to Lemma \ref{lemma:PCAcontraction}(a).
We denote by $C,C',c,\ldots$ constants that depend only on $C_0,\eps,\gamma$.
Let us write
\[h_{st}^\Sigma(\x,\y,\z,\w)=\sum_{j,k=0}^\infty
\kappa_{2(j+k+1)}h_{st,0}^{(j,k)}(\x,\y,\z,\w)
+\kappa_{2(j+k+2)}h_{st,1}^{(j,k)}(\x,\y,\z,\w).\]
For both $a=0$ and $a=1$, we have
\[|h_{st,a}^{(j,k)}| \leq C(C'/\alpha)^{2(j+k+a)}.\]
Applying $|\kappa_{2j}| \leq C^{2j}$ from
Proposition \ref{prop:cumulantbound}(b), for $\alpha>\alpha_0$ large enough,
we obtain $h_{st}^\Sigma(\x,\y,\z,\w) \in I_\Sigma$. We may also verify the 
bounds, for both $a=0$ and $a=1$,
\begin{align*}
\left|\frac{\partial h_{st,a}^{(j,k)}}{\partial y_{pp}}\right|
&\leq C(C'/\alpha)^{2(j+k+a)} \quad
\text{ if } p \in \{s-j+1,\ldots,s\} \text{ or } p \in \{t-k+1,\ldots,t\}\\
\left|\frac{\partial h_{st,a}^{(j,k)}}{\partial w_{pp}}\right|
&\leq C(C'/\alpha)^{2(j+k+a)} \quad
\text{ if } p \in \{s-j+1-a,\ldots,s\} \text{ or } p \in \{t-k+1-a,\ldots,t\}\\
\left|\frac{\partial h_{st,0}^{(j,k)}}{\partial x_{pp}}\right|
&\leq C\alpha^2(C'/\alpha)^{2(j+k)} \quad
\text{ if } p \in \{s-j+1,\ldots,s\} \text{ or } p \in \{t-k+1,\ldots,t\}\\
\left|\frac{\partial h_{st,1}^{(j,k)}}{\partial x_{pp}}\right|
&\leq \begin{cases} C\alpha^2(C'/\alpha)^{2(j+k+1)} &
\text{ if } p \in \{s-j+1,\ldots,s\} \text{ or } p \in \{t-k+1,\ldots,t\}\\
C(C'/\alpha)^{2(j+k+1)} & \text{ if } p=s-j \text{ or } p=t-k
\end{cases}\\
\left|\frac{\partial h_{st,0}^{(j,k)}}{\partial z_{pp}}\right|
&\leq \begin{cases} C\alpha^2(C'/\alpha)^{2(j+k)} &
\text{ if } p \in \{s-j+1,\ldots,s\} \text{ or } p \in \{t-k+1,\ldots,t\} \\
C(C'/\alpha)^{2(j+k)} & \text{ if } p=s-j \text{ or } t-k
\end{cases}\\
\left|\frac{\partial h_{st,1}^{(j,k)}}{\partial z_{pp}}\right|
&\leq C\alpha^2(C'/\alpha)^{2(j+k+1)} \quad
\text{ if } p \in \{s-j,\ldots,s\} \text{ or } p \in \{t-k,\ldots,t\}\\
\left|\frac{\partial h_{st,0}^{(j,k)}}{\partial z_{s-j,t-k}}\right|
&\leq C(C'/\alpha)^{2(j+k)}\\
\left|\frac{\partial h_{st,1}^{(j,k)}}{\partial x_{s-j,t-k}}\right|
&\leq C(C'/\alpha)^{2(j+k+1)}
\end{align*}
and all other partial derivatives are 0. Multiplying by $\kappa_{2(j+k+1)}$ and
$\kappa_{2(j+k+2)}$, applying the bound $|\kappa_{2j}| \leq C^{2j}$ from
Proposition \ref{prop:cumulantbound}(b), and summing over $j,k \geq 0$, we
obtain
\begin{align*}
\left|\frac{\partial h_{st}^\Sigma}{\partial x_{pp}}\right|,
\left|\frac{\partial h_{st}^\Sigma}{\partial z_{pp}}\right|
&\leq C\left(\1\{p \leq s\} \left(\frac{C'}{\alpha}\right)^{2(s-p)}
+\1\{p \leq t\} \left(\frac{C'}{\alpha}\right)^{2(t-p)}\right)\\
\left|\frac{\partial h_{st}^\Sigma}{\partial y_{pp}}\right|,
\left|\frac{\partial h_{st}^\Sigma}{\partial w_{pp}}\right|
&\leq C\left(\1\{p \leq s\} \left(\frac{C'}{\alpha}\right)^{2(s-p+1)}
+\1\{p \leq t\} \left(\frac{C'}{\alpha}\right)^{2(s-p+1)}\right)\\
\left|\frac{\partial h_{st}^\Sigma}{\partial z_{pq}}\right|
&\leq C \cdot \1\{p \leq s \text{ and } q \leq
t\}\left(\frac{C'}{\alpha}\right)^{2(s-p+t-q)}\\
\left|\frac{\partial h_{st}^\Sigma}{\partial x_{pq}}\right|
&\leq C \cdot \1\{p \leq s \text{ and } q \leq
t\}\left(\frac{C'}{\alpha}\right)^{2(s-p+t-q+1)}.
\end{align*}
Then applying the same argument as in Lemma \ref{lemma:PCAcontraction}(a),
we obtain
\[\|h^\Sigma(\x,\y,\z,\w)-h^\Sigma(\x',\y',\z',\w')\|_\zeta
\leq C(\|\x-\x'\|_\zeta+\|\z-\z'\|)_\zeta+(C/\alpha^2)(\|\y-\y'\|_\zeta
+\|\w-\w'\|_\zeta).\]
The proof for $h^\Omega$ is analogous, and part (a) follows.

Part (b) is a direct consequence of Lemma \ref{lemma:PCAcontraction}(b), since
$h^\Delta(\x,\y,\z,\w)$ is a function only of $(\z,\y)$ that has the same form
as $h^\Delta(\x,\y)$ in Lemma \ref{lemma:PCAcontraction}(b), and similarly for
$h^\Gamma$.
\end{proof}

The next lemma now follows from Lemma \ref{lemma:PCAcontractionrect} via
the same argument as Lemma \ref{lemma:PCAapproxmap}, and we omit the proof for
brevity.

\begin{lemma}\label{lemma:PCAapproxmaprect}
In the setting of Theorem \ref{thm:PCArect}(b), there exist constants 
$C,\alpha_0>0$ such that for all $\alpha>\alpha_0$ and $T \geq 1$:
\begin{enumerate}[(a)]
\item Each entry of $\bDelta_T$, $\bSigma_T$, $\bGamma_T$, and $\bOmega_T$
belongs respectively to $I_\Delta$, $I_\Sigma$, $I_\Gamma$, and $I_\Omega$,
except for entries in the first row or column of $\bDelta_T$ and $\bOmega_T$
which belong to $[-1,1]$ and $[-3\gamma \kappa_2/2,3\gamma \kappa_2/2]$.
\item For any $(\x,\y,\z,\w) \in \cX_{I_\Delta} \times \cX_{I_\Sigma} \times 
\cX_{I_\Gamma} \times \cX_{I_\Omega}$, we have
\begin{align*}
\|\bOmega_T-h^\Omega(\x,\y,\z,\w)\|_\zeta
&\leq C\Big(\|\x-\bDelta_T\|_\zeta+\|\z-\bGamma_{T-1}\|_\zeta\Big)\\
&\hspace{1in}+(C/\alpha^2)\Big(\|\y-\bSigma_{T-1}\|_\zeta
+\|\w-\bOmega_{T-1}\|_\zeta\Big)+C\zeta^{T/2}\\
\|\bGamma_T-h^\Gamma(\x,\y,\z,\w)\|_\zeta
&\leq (C/\alpha)\|\x-\bDelta_T\|_\zeta
+(C/\alpha^2)\|\w-\bOmega_T\|_\zeta+C\zeta^T\\
\|\bSigma_T-h^\Sigma(\x,\y,\z,\w)\|_\zeta
&\leq C\Big(\|\x-\bDelta_T\|_\zeta+\|\z-\bGamma_T\|_\zeta\Big)\\
&\hspace{1in}+(C/\alpha^2)\Big(\|\y-\bSigma_{T-1}\|_\zeta
+\|\w-\bOmega_T\|_\zeta\Big)+C\zeta^{T/2}\\
\|\bDelta_{T+1}-h^\Delta(\x,\y,\z,\w)\|_\zeta
&\leq (C/\alpha)\|\z-\bGamma_T\|_\zeta
+(C/\alpha^2)\|\y-\bSigma_T\|_\zeta+C\zeta^T
\end{align*}
\end{enumerate}
\end{lemma}

\begin{proof}[Proof of Theorem \ref{thm:PCArect}(b)]
Given Theorem \ref{thm:PCArect}(a), it suffices to show that
$\delta_{TT} \to \Delta_*$ and $\gamma_{TT} \to
\Gamma_*$ as $T \to \infty$, where $(\Delta_*,\Sigma_*,\Gamma_*,\Omega_*,X_*)
\in I_\Delta \times I_\Sigma \times I_\Gamma \times I_\Omega \times \RR$ is
the unique fixed point of (\ref{eq:PCAfixedpointrect}).

We define the map $G:\cX_{I_\Delta} \times
\cX_{I_\Sigma} \times \cX_{I_\Gamma} \times \cX_{I_\Omega}
\to \cX_{I_\Delta} \times
\cX_{I_\Sigma} \times \cX_{I_\Gamma} \times \cX_{I_\Omega}$ as the successive
composition of the four maps
\begin{align*}
(\x,\y,\z,\w) &\mapsto (\x,\y,\z,h^\Omega(\x,\y,\z,\w))\\
(\x,\y,\z,\w) &\mapsto (\x,\y,h^\Gamma(\x,\y,\z,\w),\w)\\
(\x,\y,\z,\w) &\mapsto (\x,h^\Sigma(\x,\y,\z,\w),\z,\w)\\
(\x,\y,\z,\w) &\mapsto (h^\Delta(\x,\y,\z,\w),\y,\z,\w)
\end{align*}
Writing its components as $G=(G_x,G_y,G_z,G_w)$,
Lemma \ref{lemma:PCAcontractionrect} may be applied to show that
\begin{align*}
\|G_w(\x,\y,\z,\w)-G_w(\x',\y',\z',\w')\|_\zeta
&\leq C\|\x-\x'\|_\zeta+C\|\z-\z'\|_\zeta
+\frac{C}{\alpha^2}\|\y-\y'\|_\zeta+\frac{C}{\alpha^2}\|\w-\w'\|_\zeta\\
\|G_z(\x,\y,\z,\w)-G_z(\x',\y',\z',\w')\|_\zeta
&\leq \frac{C}{\alpha}\|\x-\x'\|_\zeta+\frac{C}{\alpha^2}\|\z-\z'\|_\zeta
+\frac{C}{\alpha^4}\|\y-\y'\|_\zeta+\frac{C}{\alpha^4}\|\w-\w'\|_\zeta\\
\|G_y(\x,\y,\z,\w)-G_y(\x',\y',\z',\w')\|_\zeta
&\leq C\|\x-\x'\|_\zeta+\frac{C}{\alpha^2}\|\z-\z'\|_\zeta
+\frac{C}{\alpha^2}\|\y-\y'\|_\zeta+\frac{C}{\alpha^4}\|\w-\w'\|_\zeta\\
\|G_x(\x,\y,\z,\w)-G_x(\x',\y',\z',\w')\|_\zeta
&\leq \frac{C}{\alpha^2}\|\x-\x'\|_\zeta+\frac{C}{\alpha^3}\|\z-\z'\|_\zeta
+\frac{C}{\alpha^4}\|\y-\y'\|_\zeta+\frac{C}{\alpha^5}\|\w-\w'\|_\zeta.
\end{align*}
Then defining the norm $\|\cdot\|_{\zeta,\alpha}$ on $\cX_{I_\Delta} \times
\cX_{I_\Sigma} \times \cX_{I_\Gamma} \times \cX_{I_\Omega}$ by
\[\|(\x,\y,\z,\w)\|_{\zeta,\alpha}=\|\x\|_\zeta+(1/\alpha)\|\z\|_\zeta
+(1/\alpha^2)\|\y\|_\zeta+(1/\alpha^3)\|\w\|_\zeta,\]
we obtain
\begin{align*}
\|G(\x,\y,\z,\w)-G(\x',\y',\z',\w')\|_{\zeta,\alpha}
&\leq (C/\alpha^2)\|(\x,\y,\z,\w)-(\x',\y',\z',\w')\|_{\zeta,\alpha}\\
&\leq \tau\|(\x,\y,\z,\w)-(\x',\y',\z',\w')\|_{\zeta,\alpha}
\end{align*}
for some $\tau \in (0,1)$ and $\alpha>\alpha_0$ sufficiently large. So $G$
admits a unique fixed point $(\x_*,\y_*,\z_*,\w_*) \in \cX_{I_\Delta} \times
\cX_{I_\Sigma} \times \cX_{I_\Gamma} \times \cX_{I_\Omega}$.

By the same argument as in the proof of Theorem \ref{thm:PCA}(b) for
the symmetric square setting, this fixed point
$(\x_*,\y_*,\z_*,\w_*)$ must be equal to scalar constants
$(\Delta_*,\Sigma_*,\Gamma_*,\Omega_*)$ in every coordinate,
where these constants satisfy
\begin{align*}
\Sigma_*&=\sum_{j,k=0}^\infty 
\left(\frac{\alpha^2 \Delta_*\Gamma_*(1-\Delta_*)(1-\Gamma_*)}
{\gamma \Sigma_*\Omega_*}\right)^{j+k}
\left(\kappa_{2(j+k+1)}\Gamma_*+\kappa_{2(j+k+2)}
\frac{\alpha^2 \Delta_*^3(1-\Gamma_*)^2}{\Omega_*^2}\right)\\
\Omega_*&=\sum_{j,k=0}^\infty 
\left(\frac{\alpha^2 \Delta_*\Gamma_*(1-\Delta_*)(1-\Gamma_*)}
{\gamma \Sigma_*\Omega_*}\right)^{j+k}
\left(\gamma \kappa_{2(j+k+1)}\Delta_*+\gamma\kappa_{2(j+k+2)}
\frac{\alpha^2 \Gamma_*^3(1-\Delta_*)^2}{\gamma \Sigma_*^2}\right)\\
\Delta_*&=1-\mmse\left(\frac{\alpha^2\Gamma_*^2}{\gamma^2\Sigma_*}\right),
\qquad \Gamma_*=1-\mmse\left(\frac{\alpha^2\Delta_*^2}{\Omega_*}\right).
\end{align*}
(The fixed point $(\Delta_*,\Sigma_*,\Gamma_*,\Omega_*)$ to these equations
exists by the Banach fixed point theorem specialized to the scalar setting.)
Writing
\begin{align*}
R'(x)&=\sum_{k=1}^\infty \kappa_{2k} \cdot kx^{k-1}
=\sum_{k=0}^\infty \kappa_{2(k+1)} \cdot (k+1)x^k\\
S(x)&=\left(\frac{R(x)}{x}\right)'
=\sum_{k=2}^\infty \kappa_{2k} \cdot (k-1)x^{k-2}
=\sum_{k=0}^\infty \kappa_{2(k+2)} \cdot (k+1)x^k,
\end{align*}
we see that the above equations are equivalent to the fixed point equations
(\ref{eq:PCAfixedpointrect}). The proof is concluded using the same
contractive mapping argument as in Theorem \ref{thm:PCA}(b).
\end{proof}

\subsection{Verification of
Eq.\ (\ref{eq:rectPCAimproves})}\label{subsec:rectPCAimproves}

Denote $T(z)=(1+z)(1+\gamma z)$.
Let us first show that the values $\Delta_\PCA$ and $\Gamma_\PCA$ in
(\ref{eq:DeltaPCA}--\ref{eq:GammaPCA}) may be written equivalently as
\begin{equation}\label{eq:DeltaGammaPCAequiv}
\Delta_\PCA=\frac{T(R(x))-xT'(R(x))R'(x)}{1+\gamma R(x)},
\qquad \Gamma_\PCA=\frac{T(R(x))-xT'(R(x))R'(x)}{1+R(x)}.
\end{equation}
To see this, let us define $\varphi(z)$, $\bar{\varphi}(z)$, and $D(z)$ as in
(\ref{eq:BGNnotation}), and define also
\[M(z)=\sum_{k=1}^\infty m_{2k}^\infty z^k.\]
From \cite[Eq.\ (8)]{benaych2012singular},
the rectangular R-transform is given by $R(z)=U(z(D^{-1}(z))^2-1)$,
where $U(z)$ is a function defined such that $T(U(z-1))=z$. Thus
$T(R(z))=z(D^{-1}(z))^2$, and differentiating on both sides yields
\begin{equation}\label{eq:PCArectid1}
T(R(z))-zT'(R(z))R'(z)=\frac{-2z^2D^{-1}(z)}{D'(D^{-1}(z))}.
\end{equation}
Next, applying series expansions for $\varphi(z)$ and $\bar{\varphi}(z)$ 
and substituting into $D(z)$, we obtain
\[\varphi(z)=z^{-1}\left(1+M(z^{-2})\right),
\qquad \bar{\varphi}(z)=z^{-1}\left(1+\gamma M(z^{-2})\right),
\qquad D(z)=z^{-2}T(M(z^{-2})).\]
By (\ref{eq:rectseries}), the rectangular
R-transform satisfies the identity $M(z)=R(zT(M(z)))$.
Then $M(z^{-2})=R(D(z))$, so
\[D(z)=z^{-1}\Big(1+R(D(z))\Big) \cdot \bar{\varphi}(z)
=\varphi(z) \cdot z^{-1}\Big(1+\gamma R(D(z))\Big).\]
Hence, applying this with $D^{-1}(z)$ in place of $z$ and rearranging,
\begin{equation}\label{eq:PCArectid2}
\frac{z}{1+R(z)}=\frac{\bar{\varphi}(D^{-1}(z))}{D^{-1}(z))},
\qquad \frac{z}{1+\gamma R(z)}=\frac{\varphi(D^{-1}(z))}{D^{-1}(z)}.
\end{equation}
Applying these identities (\ref{eq:PCArectid1}) and (\ref{eq:PCArectid2}) for
$z=x$, we see that (\ref{eq:DeltaGammaPCAequiv}) coincides with the definitions
(\ref{eq:DeltaPCA}--\ref{eq:GammaPCA}), as desired. 

Now we proceed to verify (\ref{eq:rectPCAimproves}).
As in Remark \ref{remark:squarePCA}, applying the mmse
inequality (\ref{eq:mmsebound}) to the second and third
fixed point equations of (\ref{eq:PCAfixedpointrect}) and rearranging, we obtain
\begin{equation}\label{eq:mmseboundrect}
\Sigma_* \geq \Sigma_\lb
\equiv \frac{\alpha^2\Gamma_*^2(1-\Delta_*)}{\gamma^2\Delta_*}, \qquad
\Omega_* \geq \Omega_\lb
\equiv \frac{\alpha^2\Delta_*^2(1-\Gamma_*)}{\Gamma_*}.
\end{equation}
We apply the following argument to ``substitute'' these inequalities into the
remaining fixed-point equations: Fixing $\Delta_* \in I_\Delta$ and $\Gamma_*
\in I_\Gamma$, denote
\begin{align*}
X(\Sigma,\Omega)&=\frac{\alpha^2 \Delta_*\Gamma_*(1-\Delta_*)(1-\Gamma_*)}
{\gamma \Sigma\Omega}\\
f(\Sigma,\Omega)&=\Delta_* R'(X(\Sigma,\Omega))
+\frac{\alpha^2\Delta_*^4(1-\Gamma_*)^2}{\Gamma_*\Omega^2}S(X(\Sigma,\Omega)),\\
g(\Sigma,\Omega)&=\gamma \Gamma_*R'(X(\Sigma,\Omega))
+\frac{\alpha^2\Gamma_*^4(1-\Delta_*)^2}{\gamma \Delta_*\Sigma^2}
S(X(\Sigma,\Omega)).
\end{align*}
The fourth and fifth fixed point equations of (\ref{eq:PCAfixedpointrect}) may
be written as $(\Delta_*/\Gamma_*)\Sigma_*=f(\Sigma_*,\Omega_*)$ and
$(\Gamma_*/\Delta_*)\Omega_*=g(\Sigma_*,\Omega_*)$. So for any constant
$\eta \in \RR$, $(\Sigma_*,\Omega_*)$ solves the equation
\begin{equation}\label{eq:SigmaOmegaeq}
0=f(\Sigma,\Omega)+\eta \cdot g(\Sigma,\Omega)
-\frac{\Delta_*}{\Gamma_*}\Sigma-\frac{\eta\cdot \Gamma_*}{\Delta_*}\Omega.
\end{equation}

Let us denote
\[x=X(\Sigma_\lb,\Omega_\lb)=\gamma/\alpha^2\]
and pick this constant $\eta$ to solve the linear equation
\begin{equation}\label{eq:etacondition}
\Big(\eta \alpha^2-R'(x)-\frac{\eta \gamma^3}{\alpha^2}S(x)\Big)
\Big(1+R(x)\Big)=\Big(\frac{\alpha^2}{\gamma^2}-\eta \gamma R'(x)
-\frac{1}{\alpha^2}S(x)\Big)\Big(1+\gamma R(x)\Big).
\end{equation}
Note that for all $\alpha>\alpha_0$ sufficiently large, we have $\eta \approx
[(1+\gamma R(x))/\gamma^2]/(1+R(x)) \approx 1/\gamma^2$, which is of constant
order. We claim that for any $\Delta_* \in I_\Delta$ and
$\Gamma_* \in I_\Gamma$, the right side of (\ref{eq:SigmaOmegaeq}) is decreasing
as a function of $\Sigma \in [\Sigma_\lb,\infty)$ and $\Omega \in
[\Omega_\lb,\infty)$.
To see this, observe first that since $1-\Delta_* \leq C/\alpha^2$ and
$1-\Gamma_* \leq C/\alpha^2$, we have $|X(\Sigma,\Omega)| \leq C/\alpha^2$ for
parameters in these domains. Then to compute the derivatives of
$f(\Sigma,\Omega)$ and $g(\Sigma,\Omega)$, we may apply the series expansions
\begin{align*}
R'(X(\Sigma,\Omega))&=\kappa_2^\infty+2\kappa_4^\infty
\cdot \frac{\alpha^2 \Delta_*\Gamma_*(1-\Delta_*)(1-\Gamma_*)}
{\gamma \Sigma\Omega}+\ldots\\
S(X(\Sigma,\Omega))&=\kappa_4^\infty+2\kappa_6^\infty
\cdot \frac{\alpha^2 \Delta_*\Gamma_*(1-\Delta_*)(1-\Gamma_*)}
{\gamma \Sigma\Omega}+\ldots
\end{align*}
which are convergent for $\alpha>\alpha_0$ sufficiently large, and
differentiate these term-by-term. We may thus verify the bounds
\[|\partial_\Sigma R'(X(\Sigma,\Omega))|,
|\partial_\Omega R'(X(\Sigma,\Omega))|,
|\partial_\Sigma S(X(\Sigma,\Omega))|,
|\partial_\Omega S(X(\Sigma,\Omega))| \leq \frac{C}{\alpha^2}, \qquad
|S(X(\Sigma,\Omega))| \leq C,\]
which imply
\[|\partial_\Sigma f(\Sigma,\Omega)|,|\partial_\Omega f(\Sigma,\Omega)|,
|\partial_\Sigma g(\Sigma,\Omega)|,|\partial_\Omega g(\Sigma,\Omega)|
\leq \frac{C}{\alpha^2}.\]
Then the derivatives in $(\Sigma,\Omega)$ of the right
side of (\ref{eq:SigmaOmegaeq}) are negative for all $\alpha>\alpha_0$
sufficiently large, yielding the desired monotonicity.

Since $(\Sigma_*,\Omega_*)$ satisfies (\ref{eq:SigmaOmegaeq}) with equality, we
may then substitute (\ref{eq:mmseboundrect}) to obtain
\begin{equation}\label{eq:SigmaOmegalbineq}
0 \leq f(\Sigma_\lb,\Omega_\lb)+\eta \cdot g(\Sigma_\lb,\Omega_\lb)
-\frac{\Delta_*}{\Gamma_*}\Sigma_\lb-\frac{\eta \cdot \Gamma_*}{\Delta_*}
\Omega_\lb.
\end{equation}
Applying the forms of $\Sigma_\lb$, $\Omega_{\lb}$, $f$, and $g$ and
rearranging, we arrive at
\[\left(\eta \alpha^2-R'(x)-\frac{\eta \gamma^3}{\alpha^2}S(x)\right)\Delta_*
+\left(\frac{\alpha^2}{\gamma^2}-\eta\gamma R'(x)-\frac{1}{\alpha^2}S(x)
\right)\Gamma_*
\leq \left(\frac{\alpha^2}{\gamma^2}+\eta\alpha^2\right)\Delta_*\Gamma_*.\]
Now applying the identity (\ref{eq:etacondition}), we may write this as
\begin{equation}\label{eq:rectimprovesABeq}
A(x)\Big((1+\gamma R(x))\Delta_*+(1+R(x))\Gamma_*\Big) \leq 
B(x)\Delta_*\Gamma_*.
\end{equation}
Here, solving explicitly the equation (\ref{eq:etacondition}) for $\eta$
and applying also $S(x)=R'(x)/x-R(x)/x^2$, these values $A(x)$ and $B(x)$
may be computed after some algebraic simplification to be
\begin{align*}
A(x)&=\frac{(1+R(x))(1+\gamma R(x))-x(1+\gamma+2\gamma R(x))R'(x)}
{\gamma x[(1+R(x))(1+\gamma R(x))+xR'(x)(1-\gamma)]}\\
B(x)&=\frac{2(1+R(x))(1+\gamma R(x))}{\gamma x[(1+R(x))(1+\gamma
R(x))+xR'(x)(1-\gamma)]}.
\end{align*}
Note that for $\alpha>\alpha_0$ sufficiently large (and hence small
$x=\gamma/\alpha^2$), the numerators and denominators of $A(x)$ and $B(x)$ are
all positive. Then clearing the denominators of $A(x)$ and $B(x)$ in
(\ref{eq:rectimprovesABeq}) and applying to the left side
\begin{equation}\label{eq:AMGM}
(1+\gamma R(x))\Delta_*+(1+R(x))\Gamma_*
\geq 2\sqrt{(1+\gamma R(x))(1+R(x))\Delta_*\Gamma_*},
\end{equation}
we obtain
\[\sqrt{(1+R(x))(1+\gamma R(x))\Delta_*\Gamma_*} \geq
(1+R(x))(1+\gamma R(x))-x(1+\gamma+2\gamma R(x))R'(x).\]
Recalling the notation $T(z)=(1+z)(1+\gamma z)$, this may be rewritten as
\[\Delta_*\Gamma_* \geq \frac{[T(R(x))-xT'(R(x))R'(x)]^2}{T(R(x))},\]
where the right side coincides with $\Delta_\PCA \Gamma_\PCA$ by
(\ref{eq:DeltaGammaPCAequiv}). This establishes (\ref{eq:rectPCAimproves}).

The inequalities in the preceding argument stem from (\ref{eq:mmseboundrect})
and (\ref{eq:AMGM}). If both $U_* \sim \N(0,1)$ and $V_* \sim \N(0,1)$,
then (\ref{eq:mmseboundrect}) holds with equality. In this case, we have
$(\Delta_*/\Gamma_*)\Sigma_\lb=f(\Sigma_\lb,\Omega_\lb)$ and
$(\Gamma_*/\Delta_*)\Omega_\lb=g(\Sigma_\lb,\Omega_\lb)$ in the preceding
argument, and these two equations may be solved to yield $\Delta_*=\Delta_\PCA$
and $\Gamma_*=\Gamma_\PCA$. Then equality holds in (\ref{eq:rectPCAimproves}).
(Note that equality also holds in (\ref{eq:AMGM}) because
$(1+\gamma R(x))\Delta_\PCA=(1+R(x))\Gamma_\PCA$ by
(\ref{eq:DeltaGammaPCAequiv}).) Conversely, if either $U_*$ or $V_*$ is not
distributed as $\N(0,1)$, then at least one of the inequalities in
(\ref{eq:mmseboundrect}) is strict. Then the inequality
(\ref{eq:SigmaOmegalbineq}) is also strict, implying that
(\ref{eq:rectPCAimproves}) holds with strict inequality as well.

\section{Removing the non-degeneracy assumption}\label{appendix:degenerate}

In this appendix, we prove Corollary \ref{cor:degenerate}.
We follow a similar approach to \cite{berthier2020state} and
construct a perturbed version of the AMP sequence:
Let $\bgamma,\w_1,\w_2,\ldots \in \RR^n$ be random vectors independent of each
other and of all other quantities, where $\bgamma$ has i.i.d.\
$\operatorname{Uniform}(-1,1)$ entries and each $\w_t$ has
i.i.d.\ $\N(0,1)$ entries. For a fixed small parameter $\eps>0$, consider the
perturbed noise matrix
\[\W^\eps=\O^\top \diag(\blambda+\eps\bgamma)\O,\]
the perturbed initialization
\begin{equation}\label{eq:u1eps}
\u_1^\eps=\u_1+\eps \w_1,
\end{equation}
and the perturbed AMP iterations
\begin{align}
\z_t^\eps&=\W^\eps\u_t^\eps-b_{t1}^\eps \u_1^\eps-\ldots
-b_{tt}^\eps\u_t^\eps\label{eq:AMPzeps}\\
\u_{t+1}^\eps&=u_{t+1}(\z_1^\eps,\ldots,\z_t^\eps,\E)+\eps
\w_{t+1}.\label{eq:AMPueps}
\end{align}
We define
$\bDelta_t^\eps,\bPhi_t^\eps,\B_t^\eps,\bSigma_t^\eps$ by (\ref{eq:DeltaPhi})
and (\ref{eq:BSigma}) using this perturbed sequence,
and the above coefficients
$(b_{t1}^\eps,\ldots,b_{tt}^\eps)$ are the last column of $\B_t^\eps$.

Note that for any fixed $\eps>0$ and up to any fixed iteration $T$, these
perturbed iterations are an example of the general iterations
(\ref{eq:AMPz}--\ref{eq:AMPu}) applied with noise matrix $\W^\eps$, by defining
the augmented side-information matrix $\E^\eps=(\E,\w_2,\ldots,\w_{T+1})$ and
considering the functions
\[u_{t+1}^\eps(\z_1,\ldots,\z_t,\E^\eps)
=u_{t+1}(\z_1,\ldots,\z_t,\E)+\eps \w_{t+1}.\]
By Propositions \ref{prop:iid} and \ref{prop:scalarprod}, we have
\[\blambda+\eps \bgamma \toW \Lambda^\eps \equiv \Lambda+\eps\Gamma, \quad
(\u_1^\eps,\E,\w_2,\ldots,\w_{T+1}) \toW (U_1+\eps W_1,E,W_2,\ldots,W_{T+1}),\]
where $\Gamma \sim \operatorname{Uniform}(-1,1)$ is independent of $\Lambda$ and
$(W_1,\ldots,W_{T+1}) \sim \N(0,\Id)$ is independent of $(U_1,E)$.
It is then clear that Assumption \ref{assump:main} including part (e) holds
for this perturbed sequence, so Lemma \ref{lemma:main} applies.

Define the almost-sure limits
\[(\bDelta_t^{\eps,\infty},\bPhi_t^{\eps,\infty},\B_t^{\eps,\infty},\bSigma_t^{\eps,\infty})=\lim_{n
\to \infty} (\bDelta_t^\eps,\bPhi_t^\eps,\B_t^\eps,\bSigma_t^\eps),\]
as guaranteed by Lemma \ref{lemma:main}.
We let $\u_1,\z_1,\u_2,\z_2,\ldots$ continue to denote the original
AMP sequence. We now establish inductively the following
two claims, almost surely for each $t=1,2,3,\ldots$, where the second claim
implies the corollary:
\begin{enumerate}[(a)]
\item 
\[\limsup_{n \to \infty} n^{-1} \|\z_{t-1}\|^2<\infty,
\qquad \lim_{\eps \to 0}\;
\limsup_{n \to \infty} n^{-1} \|\z_{t-1}^\eps-\z_{t-1}\|^2=0,\]
\[\limsup_{n \to \infty} n^{-1} \|\u_t\|^2<\infty,
\qquad \lim_{\eps \to 0}\;\limsup_{n \to \infty} n^{-1}
\|\u_t^\eps-\u_t\|^2=0.\]
\item $(\u_1,\ldots,\u_t,\z_1,\ldots,\z_{t-1},\E) \toWtwo
(U_1,\ldots,U_t,Z_1,\ldots,Z_{t-1},E)$. The deterministic limits
\[(\bDelta_t^\infty,\bPhi_t^\infty,\B_t^\infty,\bSigma_t^\infty)=\lim_{n
\to \infty} (\bDelta_t,\bPhi_t,\B_t,\bSigma_t)\]
all exist, where
$(\bDelta_t^\infty,\bPhi_t^\infty,\B_t^\infty,\bSigma_t^\infty)
=\lim_{\eps \to 0}
(\bDelta_t^{\eps,\infty},\bPhi_t^{\eps,\infty},\B_t^{\eps,\infty},\bSigma_t^{\eps,\infty})$.
\end{enumerate}

Let $t^{(a)},t^{(b)}$ denote these claims up to and including iteration $t$.
We induct on $t$.

For $1^{(a)}$, we have $n^{-1}\|\u_1\|^2 \to \EE[U_1^2]<\infty$ by Assumption
\ref{assump:main}(c). We also have $\u_1^\eps-\u_1=\eps \w_1$
and $n^{-1} \|\w_1\|^2 \to 1$. Thus
$\lim_{\eps \to 0} \limsup_{n \to \infty} n^{-1}\|\u_1^\eps-\u_1\|^2=0$.

For $1^{(b)}$, we have $(\u_1,\E) \toWtwo (U_1,E)$
and $\bDelta_1 \to \EE[U_1^2]$ also by Assumption \ref{assump:main}(c).
Since $\bPhi_1=0$ and $\kappa_k \to \kappa_k^\infty$ (the $k^\text{th}$ free
cumulant of $\Lambda$) for each $k \geq 1$, this shows the existence of all four
limits $\bDelta_1^\infty,\bPhi_1^\infty,\B_1^\infty,\bSigma_1^\infty$.
Note that $\bDelta_1^{\eps,\infty}=\EE[U_1^2]+\eps^2$, so that
$\bDelta_1^{\eps,\infty} \to \bDelta_1=\EE[U_1^2]$ as $\eps \to 0$.
Letting $\kappa_k^{\eps,\infty}$ be the
free cumulants of $\Lambda^\eps$, note that the moments of $\Lambda^\eps$
converge to those of $\Lambda$ as $\eps \to 0$, so also
$\kappa_k^{\eps,\infty} \to \kappa_k^\infty$. Since
$\bPhi_1^{\eps,\infty}=0=\bPhi_1$, this shows the last statement of $1^{(b)}$.

Suppose now that $t^{(a)}$ and $t^{(b)}$ hold. To show $t+1^{(a)}$, observe that
\[\|\z_t\| \leq \|\W\|\|\u_t\|+\sum_{s=1}^t |b_{ts}| \|\u_s\|.\]
Applying $\|\W\|=\|\blambda\|_\infty$,
$\limsup_{n \to \infty} \|\blambda\|_\infty<\infty$,
and $\limsup_{n \to \infty} n^{-1} \|\u_s\|^2<\infty$
and $\lim_{n \to \infty} |b_{ts}|=|b_{ts}^\infty|$ by $t^{(a)}$ and $t^{(b)}$,
this shows
\[\limsup_{n \to \infty} n^{-1}\|\z_t\|^2<\infty.\]
Now comparing (\ref{eq:AMPz}) with (\ref{eq:AMPzeps}),
\[\|\z_t^\eps-\z_t\|\leq
\|\W^\eps-\W\|\|\u_t^\eps\|+\|\W\|\|\u_t^\eps-\u_t\|
+\sum_{s=1}^t |b_{ts}^\eps-b_{ts}|\|\u_s^\eps\|+|b_{ts}|
\|\u_s^\eps-\u_s\|.\]
Applying also $\|\W^\eps-\W\| \leq \eps$,
$\lim_{\eps \to 0} |b_{ts}^\eps|=|b_{ts}|$, and
$\lim_{\eps \to 0} \limsup_{n \to \infty} n^{-1} \|\u_s^\eps-\u_s\|^2=0$
by $t^{(a)}$ and $t^{(b)}$, this shows
\begin{equation}\label{eq:zdiff}
\lim_{\eps \to 0}\;\limsup_{n \to \infty}
n^{-1}\|\z_t^\eps-\z_t\|^2=0.
\end{equation}
For $\u_{t+1}$, we have
\[n^{-1}\sum_{i=1}^n \Big(u_{t+1}(z_{i1},\ldots,z_{it},E)
-u_{t+1}(0,\ldots,0,E)\Big)^2 \leq Cn^{-1}(\|\z_1\|^2+\ldots+\|\z_t\|^2)\]
by the Lipschitz assumption for $u_{t+1}$. Then applying $t^{(a)}$ to bound the
right side,
\[\limsup_{n \to \infty} n^{-1}\|\u_{t+1}\|^2<\infty.\]
Now comparing (\ref{eq:AMPu}) and (\ref{eq:AMPueps}),
\begin{align*}
n^{-1}\|\u_{t+1}^\eps-\u_{t+1}\|^2
&\leq 2\eps^2 \cdot n^{-1}\|\w_{t+1}\|^2+2n^{-1}
\|u_{t+1}(\z_1,\ldots,\z_t,\E)-u_{t+1}(\z_1^\eps,\ldots,\z_t^\eps,\E)\|^2\\
&\leq 2\eps^2 \cdot n^{-1}\|\w_{t+1}\|^2
+2C\sum_{s=1}^t n^{-1}\|\z_s^\eps-\z_s\|^2.
\end{align*}
Then applying $t^{(a)}$ and (\ref{eq:zdiff}) to bound the right side, we obtain
\[\lim_{\eps \to 0}\;\limsup_{n \to \infty}
n^{-1}\|\u_{t+1}^\eps-\u_{t+1}\|^2=0.\]
This shows $t+1^{(a)}$.

For $t+1^{(b)}$, let
\begin{align}
x_i&=(\u_1,\ldots,\u_{t+1},\z_1,\ldots,\z_t,\E)_i\label{eq:xi}\\
x_i^\eps&=(\u_1^\eps,\ldots,\u_{t+1}^\eps,\z_1^\eps,\ldots,\z_t^\eps,\E)_i\label{eq:xieps}
\end{align}
be the $i^\text{th}$ rows of these matrices.
Let $X=(U_1,\ldots,U_{t+1},Z_1,\ldots,Z_t,E)$ be the limit to be shown, and
let $X^\eps$ be the limit of the perturbed sequence.
To show the desired $W_2$ convergence, it suffices to check the convergence
\begin{equation}\label{eq:Xconvergence}
\frac{1}{n} \sum_{i=1}^n f(x_i) \to \EE[f(X)]
\end{equation}
for all Lipschitz functions $f(x)$ and for $f(x)=\|x\|^2$. Let us write
\begin{align}
&\limsup_{n \to \infty}
\left|\frac{1}{n}\sum_{i=1}^n f(x_i)-\EE[f(X)]\right|\nonumber\\
&\leq \lim_{\eps \to 0} \limsup_{n \to \infty}
\left(\left|\frac{1}{n}\sum_{i=1}^n f(x_i)-f(x_i^\eps)\right|
+\left|\frac{1}{n}\sum_{i=1}^n f(x_i^\eps)-\EE[f(X^\eps)]\right|
+\Big|\EE[f(X^\eps)]-\EE[f(X)]\Big|\right).\label{eq:triangle}
\end{align}
For the first term of (\ref{eq:triangle}), note that any such 
function $f$ satisfies the pseudo-Lipschitz condition
\[|f(x)-f(x')| \leq C(1+\|x\|+\|x'\|)\|x-x'\|\]
for some constant $C>0$. Then by this and Cauchy-Schwarz,
\begin{align*}
\left|\frac{1}{n}\sum_{i=1}^n f(x_i)-f(x_i^\eps)\right|
&\leq \frac{C}{n}\sum_{i=1}^n (1+\|x_i\|+\|x_i^\eps\|)\|x_i-x_i^\eps\|\\
&\leq C'\left(\frac{1}{n}\sum_{i=1}^n (1+\|x_i\|^2+\|x_i^\eps\|^2)\right)^{1/2}
\left(\frac{1}{n}\sum_{i=1}^n \|x_i-x_i^\eps\|^2\right)^{1/2}.
\end{align*}
Recalling the definitions of $x_i$ and
$x_i^\eps$ in (\ref{eq:xi}) and (\ref{eq:xieps}) and
applying $t+1^{(a)}$, this term converges to 0
in the limits $n \to \infty$ followed by $\eps \to 0$.
The second term of (\ref{eq:triangle}) converges to 0 as $n \to \infty$ for any
fixed $\eps>0$, by Lemma \ref{lemma:main}. For the third term of
(\ref{eq:triangle}), note that as $\eps \to 0$, we have 
\[U_1^\eps \to U_1, \qquad (Z_1^\eps,\ldots,Z_t^\eps) \to (Z_1,\ldots,Z_t)\]
in the Wasserstein space $W_2$, where the second convergence follows from
$\|\bSigma_t^\eps-\bSigma_t\| \to 0$ in
$t^{(b)}$. Since the functions $u_2,\ldots,u_{t+1}$ are Lipschitz, this implies
$X^\eps \to X$ in $W_2$, so $\EE[f(X^\eps)] \to \EE[f(X)]$.
Combining these establishes (\ref{eq:Xconvergence}), and hence
\[(\u_1,\ldots,\u_{t+1},\z_1,\ldots,\z_t,\E) \toWtwo
(U_1,\ldots,U_{t+1},Z_1,\ldots,Z_t,E).\]
This implies the existence of the limits $\bDelta_{t+1}^\infty$.
Each function $u_s$ is Lipschitz and continuously-differentiable by assumption,
so each partial derivative $\partial_{s'} u_s$ is bounded and continuous.
Then this also implies the existence of $\bPhi_{t+1}^\infty$, and hence of
$\B_{t+1}^\infty$ and $\bSigma_{t+1}^\infty$. As $\eps \to 0$,
since $X^\eps \to X$ in $W_2$ as shown above, we also have
$\bDelta_{t+1}^{\eps,\infty} \to \bDelta_{t+1}^\infty$ and
$\bPhi_{t+1}^{\eps,\infty} \to \bPhi_{t+1}^\infty$, and hence
$\B_{t+1}^{\eps,\infty} \to \B_{t+1}^\infty$ and
$\bSigma_{t+1}^{\eps,\infty} \to \bSigma_{t+1}^\infty$. This concludes the
proof of $t+1^{(b)}$.

\section{Properties of empirical Wasserstein
convergence}\label{appendix:wasserstein}

In this appendix, we collect several properties of the notation $\toW$ and
$\toWp$ for empirical Wasserstein convergence from Section
\ref{sec:wasserstein}.

We will use below the following fact: To verify $\V \toWp \mathcal{L}$ where $\V
\in \RR^{n \times k}$, it suffices to check that (\ref{eq:LLN}) holds for every
function $f:\RR^k \to \RR$ satisfying, for some constant $C>0$,
the pseudo-Lipschitz condition
\begin{equation}\label{eq:pseudolipschitz}
|f(v)-f(v')| \leq C\Big(1+\|v\|^{p-1}+\|v'\|^{p-1}\Big)\|v-v'\|.
\end{equation}
This is because by \cite[Definition 6.7]{villani2008optimal}, it suffices to
check (\ref{eq:LLN}) for $f(v)=\|v\|^p$, together with the usual
weak convergence which is equivalent to (\ref{eq:LLN}) holding for bounded
Lipschitz functions. Note that this condition (\ref{eq:pseudolipschitz})
implies the polynomial growth condition (\ref{eq:growth}).

\begin{proposition}\label{prop:iid}
Fix any $p \geq 1$, $t \geq 1$, and $k \geq 0$. Let $\E \in \RR^{n \times k}$
be a deterministic matrix satisfying $\E \toWp E$, and let
$\V \in \RR^{n \times t}$ be random with i.i.d.\ rows equal in law to
$V \in \RR^t$, where $\EE[\|V\|^p]<\infty$. Then the joint convergence
\[(\V,\E) \toWp (V,E)\]
holds almost surely, where $V$ is independent of $E$ in the limit $(V,E)$.
\end{proposition}
\begin{proof}
For $k=0$, the result $\V \toWp V$ follows from the strong law of large
numbers applied to any function $f$ satisfying
(\ref{eq:growth}), as $\EE[|f(V)|] \leq C(1+\EE[\|V\|^p])<\infty$.

For $k>0$, we proceed by approximating $E$ with a discrete random variable, and
then applying the law of large numbers for each discrete value of $E$. In
detail: Fix any function $f$ satisfying (\ref{eq:pseudolipschitz}), and fix
any $\eps \in (0,1)$. Let $\{(v_i,e_i)\}_{i=1}^n$ be the rows of $(\V,\E)$.
Since (\ref{eq:pseudolipschitz}) implies (\ref{eq:growth}),
for any $R>0$ we have
\begin{equation}\label{eq:truncation}
\frac{1}{n}\sum_{i:\|e_i\|>R} |f(v_i,e_i)|
\leq \frac{C}{n}\sum_{i:\|e_i\|>R} (1+\|v_i\|^p+\|e_i\|^p).
\end{equation}
Note that as $\E \toWp E$, the uniform integrability condition of
\cite[Definition 6.7(iii)]{villani2008optimal} shows
\[\lim_{R \to \infty} \limsup_{n \to \infty} \frac{1}{n}
\sum_{i:\|e_i\|>R} 1 \leq 
\lim_{R \to \infty} \limsup_{n \to \infty} \frac{1}{n}
\sum_{i:\|e_i\|>R} \|e_i\|^p=0.\]
This bounds the first and third terms on the right side of
(\ref{eq:truncation}). For the middle term, we consider two cases.
If $\|V\| \leq K$ almost surely for some $K>0$, then by this and the convergence
$\V \toWp V$,
\[\lim_{n \to \infty} \frac{1}{n}\sum_{i=1}^n \min(\|v_i\|,2K)^p
=\EE[\min(\|V\|,2K)^p]=\EE[\|V\|^p]=
\lim_{n \to \infty} \frac{1}{n}\sum_{i=1}^n \|v_i\|^p.\]
Applying also $|\{i:\|v_i\|>2K\}|/n \to 0$, this implies
\[\lim_{n \to \infty} \frac{1}{n}\sum_{i:\|v_i\|>2K} \|v_i\|^p
=\lim_{n \to \infty} \left(\frac{1}{n}\sum_{i=1}^n \left(\|v_i\|^p
-\min(\|v_i\|,2K)^p\right)+(2K)^p \cdot \frac{|\{i:\|v_i\|>2K\}|}{n}\right)=0.\]
In this case, we may bound
\[\lim_{R \to \infty} \limsup_{n \to \infty} \frac{1}{n}\sum_{i:\|e_i\|>R}
\|v_i\|^p \leq \lim_{R \to \infty} \limsup_{n \to \infty} \frac{1}{n}
\left(\sum_{i:\|e_i\|>R} (2K)^p+\sum_{i:\|v_i\|>2K} \|v_i\|^p\right)=0.\]
Conversely, if the support of $V$ is unbounded, then let
$\|v\|_{(1)} \geq \ldots \geq \|v\|_{(n)}$ be the ordered values of
$\{\|v_i\|\}_{i=1}^n$. Note that for each $R>0$, we have
$|\{i:\|e_i\|>R\}|/n \to \delta(R)$ for some $\delta(R) \geq 0$, where
$\delta(R) \to 0$ as $R \to \infty$. Then
\[\lim_{R \to \infty} \limsup_{n \to \infty} \frac{1}{n}\sum_{i:\|e_i\|>R}
\|v_i\|^p \leq \lim_{\delta \to 0} \limsup_{n \to \infty}
\frac{1}{n}\sum_{i=1}^{\delta n} \|v\|_{(i)}^p.\]
Now applying $\V \toWp V$ and the corresponding uniform
integrability condition for $\V$,
\[\lim_{\delta \to 0} \limsup_{n \to \infty}
\frac{1}{n}\sum_{i=1}^{\delta n} \|v\|_{(i)}^p
=\lim_{R \to \infty} \limsup_{n \to \infty}
\frac{1}{n}\sum_{i:\|v_i\|>R} \|v_i\|^p=0.\]
Combining the above and applying this to (\ref{eq:truncation}),
\[\lim_{R \to \infty} \limsup_{n \to \infty}
\frac{1}{n}\sum_{i:\|e_i\|>R} |f(v_i,e_i)|=0.\]
So we may pick a bounded set $\mathcal{B} \subset \RR^k$ large enough
such that
\begin{equation}\label{eq:iid1}
\limsup_{n \to \infty}
\frac{1}{n}\sum_{i:e_i \notin \mathcal{B}} |f(v_i,e_i)|<\eps.
\end{equation}
Applying also
\[\EE\big[|f(V,E)| \cdot \1\{E \notin \mathcal{B}\}\big]
\leq \EE\big[C(1+\|V\|^p+\|E\|^p)\1\{E \notin \mathcal{B}\}\big]\]
and the integrability of $\|V\|^p$ and $\|E\|^p$,
we may pick $\mathcal{B}$ large enough such that
\begin{equation}\label{eq:iid2}
\EE\big[|f(V,E)| \cdot \1\{E \notin \mathcal{B}\}\big]<\eps.
\end{equation}

Now let $\{U_\alpha\}_{\alpha=1}^M$ be any finite partition of $\mathcal{B}$
such that each set $U_\alpha$ has diameter at most $\eps$, and the boundary of
$U_\alpha$ has probability 0 under the law of $E$.
(For example, take $\mathcal{B}=[-K,K]^k$ to be a hyperrectangle in
$\RR^k$, and construct this partition by dividing $[-K,K]$ along each axis
into small enough intervals. Take $-K$, $K$, and these
interval boundaries to have probability 0
under the univariate marginal distribution of each coordinate of $E$.) Pick a
point $u_\alpha \in U_\alpha$ for each $\alpha=1,\ldots,M$. For each $e \in
\mathcal{B}$, define $u(e)=u_\alpha$ where $\alpha$ is the index such that
$e \in U_\alpha$. Then applying (\ref{eq:pseudolipschitz}) and $\|u(e)\|^{p-1}
\leq C(\|e\|^{p-1}+1)$,
\[\frac{1}{n}\sum_{i:e_i \in \mathcal{B}} |f(v_i,e_i)-f(v_i,u(e_i))|
\leq \frac{C}{n}\sum_{i=1}^n
\Big(1+\|v_i\|^{p-1}+\|e_i\|^{p-1}\Big) \cdot \eps\]
for a constant $C>0$ independent of $\eps$. Since $\V \toWp V$
and $\E \toWp E$, this yields
\begin{equation}\label{eq:iid3}
\limsup_{n \to \infty} \frac{1}{n}\sum_{i:e_i \in \mathcal{B}}
|f(v_i,e_i)-f(v_i,u(e_i))| \leq C'\eps.
\end{equation}
Similarly,
\begin{equation}\label{eq:iid4}
\EE\Big[\big|f(V,E)-f(V,u(E))\big|
\cdot \1\{E \in \mathcal{B}\}\Big] \leq C'\eps.
\end{equation}

Finally, let us write
\[\frac{1}{n}\sum_{i:e_i \in \mathcal{B}} f(v_i,u(e_i))
=\sum_{\alpha=1}^M \frac{1}{n}\sum_{i:e_i \in U_\alpha} f(v_i,u_\alpha).\]
Observe that for each fixed $\alpha=1,\ldots,M$,
since the boundary of $U_\alpha$ has probability 0 under $E$,
by weak convergence
we have $|\{i:e_i \in U_\alpha\}|/n \to \PP[E \in U_\alpha]$.
Then by the law of large numbers applied to the function $f(\cdot,u_\alpha)$,
almost surely
\[\frac{1}{n}\sum_{i:e_i \in U_\alpha}
f(v_i,u_\alpha) \to \PP[E \in U_\alpha] \cdot \EE[f(V,u_\alpha)].\]
Summing over $\alpha=1,\ldots,M$ and applying the independence of $V$ and $E$,
\begin{align}
\frac{1}{n}\sum_{i:e_i \in \mathcal{B}} f(v_i,u(e_i))
&\to \sum_{\alpha=1}^M \PP[E \in U_\alpha] \cdot \EE[f(V,u_\alpha)]\nonumber\\
&=\sum_{\alpha=1}^M \EE[f(V,u_\alpha) \cdot \1\{E \in U_\alpha\}]
=\EE[f(V,u(E)) \cdot \1\{E \in \mathcal{B}\}].\label{eq:iid5}
\end{align}
Combining (\ref{eq:iid1}), (\ref{eq:iid2}), (\ref{eq:iid3}), (\ref{eq:iid4}),
and (\ref{eq:iid5}), we obtain
\[\limsup_{n \to \infty} \left|\frac{1}{n}\sum_{i=1}^n f(v_i,e_i)
-\EE[f(V,E)]\right| \leq C\eps\]
for a constant $C>0$ independent of $\eps$. As this holds for all $\eps>0$, this
shows $n^{-1}\sum_{i=1}^n f(v_i,e_i) \to \EE[f(V,E)]$, which concludes the
proof.
\end{proof}

\begin{proposition}\label{prop:composition}
Fix $p,p' \geq 1$ and $k,\ell \geq 1$.
If $\V \in \RR^{n \times k}$ satisfies $\V \overset{W_{p+p'}}{\to} V$, and
$g:\RR^k \to \RR^\ell$ is any continuous function satisfying
$\|g(v)\| \leq C(1+\|v\|^{p'})$ for some $C>0$ and all $v \in \RR^k$,
then $g(\V) \toWp g(V)$.
\end{proposition}
\begin{proof}
This follows from Definition \ref{def:toW}, since for any continuous function
$f:\RR^\ell \to \RR$ satisfying (\ref{eq:growth}) for the order
$p$, the composition $f \circ g:\RR^k \to \RR$ is continuous and
satisfies (\ref{eq:growth}) for the order $p+p'$.
\end{proof}

\begin{proposition}\label{prop:discontinuous}
Fix $p \geq 1$ and $k \geq 0$. Suppose $\V \in \RR^{n \times k}$ satisfies
$\V \toWp V$, and $f:\RR^k \to \RR$ is a function satisfying
(\ref{eq:growth}) that is continuous everywhere except on a set having
probability 0 under the law of $V$. Then
\[\frac{1}{n}\sum_{i=1}^n f(\V)_i \to \EE[f(V)].\]
\end{proposition}
\begin{proof}
Let $f$ be such a function. For any $M>0$, consider the bounded function
$f^M(v)=\max(-M,\min(f(v),M))$. Let $v_i$ be the $i^\text{th}$ row of $\V$.
The condition $\V \toWp V$ implies the usual weak
convergence of the empirical distribution of $\{v_i\}_{i=1}^n$ to $V$,
so $n^{-1} \sum_{i=1}^n f^M(v_i) \to \EE[f^M(V)]$ even when $f^M$ is
discontinuous on a set of probability 0 under $V$. Now taking $M \to \infty$,
we have $\EE[f^M(V)] \to \EE[f(V)]$ by the bound
$|f^M(v)| \leq C(1+\|v\|^p)$ and the dominated convergence theorem.
By this bound, we also have
\begin{align*}
\lim_{M \to \infty} \limsup_{n \to \infty} \frac{1}{n}\sum_{i=1}^n
|f^M(v_i)-f(v_i)|
&\leq \lim_{M \to \infty} \limsup_{n \to \infty} \frac{1}{n}\sum_{i:|f(v_i)|>M}
|f(v_i)|\\
&\leq \lim_{R \to \infty} \limsup_{n \to \infty}
\frac{1}{n}\sum_{i:\|v_i\|>R} C(1+\|v_i\|^p)=0,
\end{align*}
where the last limit is 0 by \cite[Definition 6.7(iii)]{villani2008optimal}.
Then $n^{-1} \sum_{i=1}^n f(v_i) \to \EE[f(V)]$ as desired.
\end{proof}

\begin{proposition}\label{prop:scalarprod}
Fix $p \geq 1$ and $k,\ell \geq 1$.
If $\V \in \RR^{n \times k}$, $\W \in \RR^{n \times \ell}$, and
$\M_n,\M \in \RR^{k \times \ell}$ satisfy
$\V \toWp (V_1,\ldots,V_k)$, $\W \toWp 0$, and $\M_n \to \M$
as $n \to \infty$, then
\[\V \M_n+\W \toWp \begin{pmatrix} V_1 & \ldots & V_k \end{pmatrix}
\M.\]
\end{proposition}
\begin{proof}
Let $f:\RR^\ell \to \RR$ satisfy (\ref{eq:pseudolipschitz}). Then
$(v_1\;\cdots\;v_k) \mapsto
f((v_1\;\cdots\;v_k)\M)$ is continuous and satisfies
(\ref{eq:growth}) with the order $p$, so by the convergence
$\V \toWp (V_1\;\cdots\;V_k)$, we have
\[\frac{1}{n}\sum_{i=1}^n f\big(\V\M\big)_i \to
\EE\big[f\big((V_1\;\cdots\;V_k)\M\big)\big].\]

Let $v_i,w_i$ be the $i^\text{th}$ rows of $\V$ and $\W$.
Note that $\V \toWp (V_1\;\cdots\;V_k)$ implies
$n^{-1}\sum_{i=1}^n \|v_i\|^p \to \EE[\|(V_1\;\cdots\;V_k)\|^p]<\infty$.
Similarly $n^{-1}\sum_{i=1}^n \|w_i\|^p \to 0$.
Then applying Jensen's inequality, Holder's inequality,
and the bound $\|v_i\M+w_i\|^{p-1} \leq C(\|v_i\|^{p-1}+\|w_i\|^{p-1})$,
we have for some constants $C,C'>0$ depending on $\M$ that
\begin{align*}
&\frac{1}{n}\sum_{i=1}^n |f(\V\M)_i-f(\V\M+\W)_i|\\
&\leq \frac{C}{n}\sum_{i=1}^n \Big(1+\|v_i\M\|^{p-1}+\|v_i\M+w_i\|^{p-1}\Big)\|w_i\|\\
&\leq C' \cdot \frac{1}{n}(\|w_i\|+\|v_i\|^{p-1}\|w_i\|
+\|w_i\|^p)\\
&\leq C'\left[\left(\frac{1}{n}\sum_{i=1}^n \|w_i\|^p\right)^{1/p}
+\left(\frac{1}{n}\sum_{i=1}^n \|v_i\|^p\right)^{(p-1)/p}
\left(\frac{1}{n}\sum_{i=1}^n \|w_i\|^p\right)^{1/p}
+\left(\frac{1}{n}\sum_{i=1}^n \|w_i\|^p\right)\right] \to 0.
\end{align*}
Similarly,
\begin{align*}
&\frac{1}{n}\sum_{i=1}^n |f\big(\V\M_n+\W\big)_i-f\big(\V\M+\W\big)_i|\\
&\leq \frac{C}{n}\sum_{i=1}^n
\Big(1+\|v_i\M_n+w_i\|^{p-1}+\|v_i\M+w_i\|^{p-1}\Big) \cdot \|v_i(\M_n-\M)\|\\
&\leq C'\|\M_n-\M\| \cdot \frac{1}{n}\sum_{i=1}^n
(\|v_i\|+\|v_i\|^p+\|w_i\|^{p-1}\|v_i\|) \to 0.
\end{align*}
Combining the above yields the proposition.
\end{proof}

The following is an empirical form of Stein's lemma.

\begin{proposition}\label{prop:stein}
Fix $p \geq 2$.
Suppose $(\z_1,\ldots,\z_t,\E) \in \RR^{n \times (t+k)}$ are such that
\[(\z_1,\ldots,\z_t,\E) \toWp (Z_1,\ldots,Z_t,E)\]
where, for some non-singular covariance matrix $\Sigma \in \RR^{t \times t}$,
$(Z_1,\ldots,Z_t) \sim \N(0,\Sigma)$ and this
is independent of $E$. Suppose $u:\RR^{t+k} \to \RR$ is weakly
differentiable in its first $t$ arguments and satisfies (\ref{eq:growth})
for the order $p-1$.
Then, almost surely as $n \to \infty$,
\[\frac{1}{n} \begin{pmatrix} \z_1^\top \\ \vdots \\ \z_t^\top \end{pmatrix}
u(\z_1,\ldots,\z_t,\E)
\to \Sigma \cdot \begin{pmatrix} \EE[\partial_1 u(Z_1,\ldots,Z_t,E)] \\
\vdots \\ \EE[\partial_t u(Z_1,\ldots,Z_t,E)] \end{pmatrix}.\]
\end{proposition}
\begin{proof}
Note that for each $s=1,\ldots,t$, the function $(z_1,\ldots,z_t,e)
\mapsto z_su(z_1,\ldots,z_t,e)$ is continuous and
satisfies (\ref{eq:growth}) with order $p$, so
\begin{equation}\label{eq:steinconvergence}
\frac{1}{n}(\z_1,\ldots,\z_t)^\top
u(\z_1,\ldots,\z_t,\E) \to \EE\Big[(Z_1,\ldots,Z_t) \cdot
u(Z_1,\ldots,Z_t,E)\Big].
\end{equation}
To show that the right side of (\ref{eq:steinconvergence}) is equivalent to the
given expression, we apply Stein's lemma:
Let us condition on a realization $E=e$ for any fixed $e \in \RR^k$,
and denote $Z=(Z_1,\ldots,Z_t)$. We may write
$Z=\Sigma^{1/2}X$ where $X \sim \N(0,\Id)$, and define
\[v_e(x)=u(\Sigma^{1/2}x,e).\]
Since $\Sigma$ is non-singular, the maps
$X \mapsto \Sigma^{1/2}X$ and $Z \mapsto \Sigma^{-1/2}Z$ are both Lipschitz.
Then by the chain rule for weak differentiability under bi-Lipschitzian
changes of coordinates, see \cite[Theorem 2.2.2]{ziemer2012weakly},
$v_e(x)$ is weakly differentiable with
\[\nabla v_e(x)=\Sigma^{1/2} \cdot \nabla u(\Sigma^{1/2}x,e)\]
a.e.\ over $x \in \RR^t$.
(We denote by $\nabla (\cdot)$ the vector of partial derivatives.)
Applying Stein's lemma for weakly differentiable functions, see \cite[Theorem
2.1]{fourdrinier2018shrinkage}, we have for each $s=1,\ldots,t$ that
$\EE[X_s v_e(X)]=\EE[\partial_s v_e(X)]$. Then
\[\EE[Z \cdot u(Z,e)]=\EE[\Sigma^{1/2} X \cdot v_e(X)]
=\Sigma^{1/2} \cdot \EE[\nabla v_e(X)]=\Sigma \cdot \EE[\nabla u(Z,e)].\]
Taking the expectation over $E$ and applying this to (\ref{eq:steinconvergence})
concludes the proof.
\end{proof}

\section{Auxiliary lemmas}

This appendix collects several auxiliary results that were used in the preceding
arguments.

\subsection{Properties of Haar-orthogonal matrices}\label{appendix:orthog}

The following result was established as \cite[Lemma 4]{rangan2019vector}.

\begin{proposition}\label{prop:orthogconditioning}
Fix $k \geq 1$, and let $\X,\Y \in \RR^{n \times k}$ be deterministic matrices
with rank $k$,
such that $\X=\Q\Y$ for some orthogonal matrix $\Q \in \RR^{n \times n}$. If
$\O \in \RR^{n \times n}$ is a random Haar-uniform orthogonal matrix,
then the law of $\O$ conditioned on $\X=\O\Y$ is equal to the law of
\[\X(\X^\top
\X)^{-1}\Y^\top+\proj_{\X^\perp}\tilde{\O}\proj_{\Y^\perp}^\top
=\X(\Y^\top\Y)^{-1}\Y^\top+\proj_{\X^\perp}\tilde{\O}\proj_{\Y^\perp}^\top.\]
Here, $\tilde{\O} \in \RR^{(n-k) \times (n-k)}$ is a Haar-uniform
orthogonal matrix, and
$\proj_{\X^\perp},\proj_{\Y^\perp} \in \RR^{n \times (n-k)}$
are matrices with orthonormal columns spanning the orthogonal complements of
the column spans of $\X$ and $\Y$.
\end{proposition}

\begin{proposition}\label{prop:orthognormal}
Fix any $p \geq 1$ and $k,\ell \geq 0$. Let $\O \in \RR^{(n-\ell) \times
(n-\ell)}$ be a
random Haar-uniform orthogonal matrix. Let
$\E \in \RR^{n \times k}$ and $\v \in \RR^{n-\ell}$ be deterministic
and satisfy $\E \toWp E$ and $n^{-1}\|\v\|^2 \to \sigma^2$,
and let $\proj \in \RR^{n \times (n-\ell)}$ be any deterministic
matrix with orthonormal columns.
\begin{enumerate}[(a)]
\item Almost surely as $n \to \infty$,
\[(\proj \O\v,\E) \toWp (Z,E)\]
where $Z \sim \N(0,\sigma^2)$ is independent of $E$.
\item Consider a second dimension $m$ such that $m,n \to \infty$ simultaneously.
Fix $j \geq 0$, and let $\F \in \RR^{m \times j}$ be deterministic and satisfy
$\F \toWp F$. Let $\check{\v} \in \RR^m$ be the first $m$ entries of $\Pi \O\v$
if $m \leq n$, or $\Pi\O\v$ extended by $m-n$ i.i.d.\ entries with
distribution $\N(0,\sigma^2)$ if $m>n$. Then almost
surely as $m,n \to \infty$,
\[(\check{\v},\F) \toWp (\check{Z},F)\]
where $\check{Z} \sim \N(0,\sigma^2)$ is independent of $F$.
\end{enumerate}
\end{proposition}
\begin{proof}
For part (a), observe that $\O\v$ is a
random vector uniformly distributed on the sphere of radius $\|\v\|$. Thus, we
may introduce a Gaussian vector $\z \sim \N(0,\Id_{n \times n})$
so that $\Pi^\top \z \sim \N(0,\Id_{(n-\ell) \times (n-\ell)})$, 
and write $\O\v=\Pi^\top \z \cdot \|\v\|/\|\Pi^\top \z\|$. Then
\begin{equation}\label{eq:orthognormal1}
\Pi \O \v=\Pi\Pi^\top \z \cdot \|\v\|/\|\Pi^\top \z\|
=\z \cdot \|\v\|/\|\Pi^\top \z\|-\Pi^\perp \z \cdot \|\v\|/\|\Pi^\top \z\|
\end{equation}
where $\Pi^\perp=\Id-\Pi\Pi^\top \in \RR^{n \times n}$ is a projection onto a subspace of fixed
dimension $\ell$. By Proposition \ref{prop:iid},
\begin{equation}\label{eq:orthognormal2}
(\z,\E) \toWp (\tilde{Z},E)
\end{equation}
where $\tilde{Z} \sim \N(0,1)$.
We have $n^{-1}\|\v\|^2 \to \sigma^2$ by assumption and
$n^{-1}\|\Pi^\top \z\|^2 \to 1$ almost surely, so
\begin{equation}\label{eq:orthognormal3}
\|\v\|/\|\Pi^\top \z\| \to \sigma.
\end{equation}
We also have the equality in law $\Pi^\perp \z=\u_1w_1+\ldots+\u_\ell w_\ell$
for some
orthonormal unit vectors $\u_1,\ldots,\u_\ell \in \RR^n$ spanning the range of
$\Pi^\perp$, and for $w_1,\ldots,w_\ell \overset{iid}{\sim} \N(0,1)$.
Letting $\{u_{ij}\}_{i=1}^n$ be the entries of $\u_j$,
for each $j=1,\ldots,\ell$ and any fixed $p \geq 1$, we have
\[\frac{1}{n}\sum_{i=1}^n |u_{ij} w_j|^p
\leq |w_j|^p \cdot \frac{1}{\sqrt{n}} \to 0\]
almost surely as $n \to \infty$. Thus also
\[\frac{1}{n}\sum_{i=1}^n \big|(\Pi^\perp \z)_i\big|^p \to 0,\]
so $\Pi^\perp \z \cdot \|\v\|/\|\Pi^\top \z\| \toWp 0$.
Combining this with (\ref{eq:orthognormal1}), (\ref{eq:orthognormal2}), and
(\ref{eq:orthognormal3}) and applying
Proposition \ref{prop:scalarprod}, we obtain part (a).

For part (b), let $\check{\z} \in \RR^m$ be the first $m$ entries of $\z$ if $m
\leq n$, or $\z$ extended by $m-n$ additional $\N(0,1)$ entries if $m>n$.
Let $\r_1 \in \RR^m$ be the first $m$ entries of 
$\proj^\perp \z \cdot \|\v\|/\|\Pi^\top \z\|$ if $m \leq n$, or this vector extended by
$m-n$ additional 0's if $m>n$. Let $\r_2 \in \RR^m$ be 0 if $m \leq n$, or equal
to 0 in the first $n$ entries and equal to
$\check{\z} \cdot (\|\v\|/\|\Pi^\top \z\|-\sigma)$ in the last $m-n$ entries
if $m>n$. Then we may write
\[\check{\v}=\check{\z} \cdot \|\v\|/\|\Pi^\top \z\|-\r_1-\r_2.\]
The same argument as in part (a) shows
\[(\check{\z},\F) \toWp (\tilde{Z},F)\]
where $\tilde{Z} \sim \N(0,1)$ is independent of $F$, and $\r_1 \toWp 0$.
When $m>n$, we also have
\[\frac{1}{m}\sum_{i=n+1}^m
\Big|\Big(\check{\z}(\|\v\|/\|\Pi^\top \z\|-\sigma)\Big)_i\Big|^p
\leq \Big|\|\v\|/\|\Pi^\top \z\|-\sigma\Big|^p \cdot
\frac{1}{m}\sum_{i=1}^m |(\check{\z})_i|^p \to 0\]
almost surely. So $\r_2 \toWp 0$. Then applying
Proposition \ref{prop:scalarprod} shows part (b).
\end{proof}

\subsection{Properties of moments and free cumulants}

\begin{proposition}\label{prop:cumulantbound}
Let $\Lambda$ be a random variable with finite moments of all orders, such that
$\EE[|\Lambda|^k] \leq M^k$ for some $M>0$ and all integers $k \geq 1$.
\begin{enumerate}[(a)]
\item Let $\{\kappa_k\}_{k \geq 1}$ be the free cumulants of $\Lambda$. Then
for all $k \geq 1$,
\[|\kappa_k| \leq (16M)^k.\]
Thus the R-transform of $\Lambda$ is analytic on the domain $|x|<1/(16M)$,
where it may be defined by the convergent series
\[R(x)=\sum_{k=1}^\infty \kappa_k x^{k-1}.\]
\item Let $\{\kappa_{2k}\}_{k \geq 1}$ be the rectangular free cumulants of
$\Lambda$ with aspect ratio $\gamma$. Then for all $k \geq 1$,
\[|\kappa_{2k}| \leq \max(\gamma^k,1) \cdot (16M)^{2k}.\]
Thus the rectangular R-transform of $\Lambda$ is analytic on the domain
$|x|<\min(\gamma^{-1},1)/(16M)^2$, where it may be defined by the convergent
series
\[R(x)=\sum_{k=1}^\infty \kappa_{2k} x^k.\]
\end{enumerate}
\end{proposition}
\begin{proof}
For part (a), the free cumulants may be expressed explicitly
by M\"obius inversion of the
moment-cumulant relations (\ref{eq:momentcumulant}), yielding
\[\kappa_k=\sum_{\pi \in \NC(k)} m_\pi \cdot \mu(\pi,1_k),
\qquad m_\pi=\prod_{S \in \pi} m_{|S|},\]
where $\mu(\cdot,\cdot)$ are the M\"obius functions on the non-crossing
partition lattice and $1_k$ is the trivial partition consisting of the single
set $\{1,\ldots,k\}$. We have $|\mu(\pi,1_k)|
\leq 4^k$ and $|\NC(k)| \leq 4^k$---see the proof of \cite[Proposition
13.15]{nica2006lectures}. Combining with $|m_\pi| \leq M^k$ for all $\pi \in
\NC(k)$, part (a) follows.

For part (b), we apply a similar argument in the rectangular probability
space $(\mathcal{A},p_m,p_n,\varphi_m,\varphi_n)$ from which the rectangular
free cumulants are defined---see \cite[Section 1.2]{benaych2009rectangular} for
definitions. Here, $p_m,p_n \in \mathcal{A}$ are orthogonal projections
satisfying $p_m+p_n=1$, and $\varphi_m$ and $\varphi_n$ are traces on
$p_m\mathcal{A}p_m$ and $p_n\mathcal{A}p_n$ that satisfy $\varphi_m(p_m)=1$,
$\varphi_n(p_n)=1$, and $\gamma/(1+\gamma) \cdot \varphi_m(xy)=1/(1+\gamma)
\cdot \varphi_n(yx)$ for $x \in p_m\mathcal{A}p_n$ and $y \in
p_n\mathcal{A}p_m$. Let $E:\mathcal{A} \to \mathcal{D}$ be the conditional
expectation onto the sub-algebra $\mathcal{D}$ generated by $(p_m,p_n)$, given
by $E(x)=\varphi_m(p_mxp_m)p_m+\varphi_n(p_nxp_n)p_n$.
For $k \geq 1$ and partitions $\pi \in \NC(k)$, let
$\kappa_\pi^{\mathcal{D}}$ be the $\mathcal{D}$-valued
free cumulants defined by the moment-cumulant relations
\[E(a_1\ldots a_k)=\sum_{\pi \in \NC(k)}
\kappa_\pi^{\mathcal{D}}(a_1,\ldots,a_k).\]

If $a \in p_m\mathcal{A}p_n$ is an element such that
$\varphi_m((aa^*)^k)=\EE[\Lambda^{2k}]$,
then the rectangular free cumulant $\kappa_{2k}$ of $\Lambda$ is given by
\[\kappa_{2k} \cdot p_m=\kappa_{1_{2k}}^\mathcal{D}(a,a^*,\ldots,a,a^*).\]
(Compare \cite[Eq.\ (2.5)]{benaych2009rectangularII} with
\cite[Eq.\ (8)]{benaych2009rectangular}, the latter being the definition of
rectangular free cumulants that we have reviewed in
Section \ref{sec:cumulantsrect} and used throughout this work.)
From the M\"obius inversion
\[\kappa_{1_{2k}}^\mathcal{D}(a,a^*,\ldots,a,a^*)
=\sum_{\pi \in \NC(2k)} m_\pi^\mathcal{D}(a,a^*,\ldots,a,a^*)
\cdot \mu(\pi,1_{2k})\]
where $m_\pi^\mathcal{D}$ is the $\mathcal{D}$-valued joint moment function
associated to $\pi$, we obtain
\[|\kappa_{2k}|=
\left|\varphi_m\Big(\kappa_{1_{2k}}^\mathcal{D}(a,a^*,\ldots,a,a^*)\Big)\right|
\leq 16^{2k} \max_{\pi \in \NC(2k)}
\Big|\varphi_m(m_\pi^\mathcal{D}(a,a^*,\ldots,a,a^*))\Big|.\]
Here, it may be checked that when
$\varphi_m(m_\pi^\mathcal{D}(a,a^*,\ldots,a,a^*))$ is non-zero, it must be a
product of $\varphi_m((aa^*)^{i_1}),\ldots,\varphi_m((aa^*)^{i_a})$ and
$\varphi_n((a^*a)^{j_1}),\ldots,\varphi_n((a^*a)^{j_b})$
where the elements of $\pi$ have cardinalities
$2i_1,\ldots,2i_a,2j_1,\ldots,2j_b$. Then applying
$\varphi_n((a^*a)^j)=\gamma \varphi_m((aa^*)^j)$ and
$\varphi_m((aa^*)^j)=\EE[\Lambda^{2j}] \leq M^{2j}$, we obtain
$|\varphi_m(m_\pi^\mathcal{D}(a,a^*,\ldots,a,a^*))| \leq
\max(\gamma^k,1)M^{2k}$, which yields part (b).
\end{proof}

\begin{proposition}\label{prop:conditionalcentralmoment}
Let $k$ be a positive integer. Then for any random variable $X$ and any
sigma-algebra $\cF$,
\[\EE[(X-\EE[X \mid \cF])^k] \leq 2^k\EE[|X|^k].\]
\end{proposition}
\begin{proof}
Write as shorthand $Y=\EE[X \mid \cF]$. We expand the left side and apply
H\"older's inequality to obtain
\[\EE[(X-Y)^k]=\sum_{j=0}^k \binom{k}{j} \EE[X^jY^{k-j}]\\
\leq \sum_{j=0}^k \binom{k}{j} \EE[|X|^k]^{j/k}\EE[|Y|^k]^{(k-j)/k}.\]
By Jensen's inequality,
\[\EE[|Y|^k]=\EE\Big[|\EE[X \mid \cF]|^k\Big] \leq \EE\Big[\EE[|X|^k \mid
\cF]\Big]=\EE[|X|^k],\]
and the result follows from $\sum_{j=0}^k \binom{k}{j}=2^k$.
\end{proof}

\section*{Acknowledgements}
I am grateful to my advisor Andrea Montanari, who first introduced me to the
beautiful worlds of both free probability and AMP. I would like to thank Keigo
Takeuchi and Galen Reeves for helpful discussions and pointers to related
literature, and Yufan Li for pointing out an error in a previous version
of the manuscript. This research is supported in part by NSF Grant DMS-1916198.

\bibliographystyle{alpha}
\bibliography{AMP_orthog}

\end{document}